\documentclass{amsart}
\usepackage{enumitem}
\usepackage{amssymb}
\usepackage{amsthm}
\usepackage{amsmath}
\usepackage{graphicx}
\usepackage[all]{xy}
\usepackage{comment}

\usepackage[cal=boondoxo]{mathalfa}

\usepackage[usenames]{color}

\usepackage[colorlinks=true,
pdfstartview=FitV, linkcolor=cyan, citecolor=magenta,
urlcolor=blue]{hyperref}

\theoremstyle{plain}
\newtheorem{proposition}{Proposition}[section]
\newtheorem{theorem}[proposition]{Theorem}
\newtheorem*{theorem*}{Theorem}
\newtheorem{lemma}[proposition]{Lemma}
\newtheorem{corollary}[proposition]{Corollary}
\theoremstyle{definition}
\newtheorem{example}[proposition]{Example}
\newtheorem{definition}[proposition]{Definition}

\newtheorem{introthm}{Theorem}
\newtheorem{introprop}[introthm]{Proposition}

\theoremstyle{remark}
\newtheorem{remark}[proposition]{Remark}
\newtheorem{notation}[proposition]{Notation}
\newtheorem{conjecture}[proposition]{Conjecture}

\numberwithin{equation}{section}

\let\oldtocsection=\tocsection

\let\oldtocsubsection=\tocsubsection

\let\oldtocsubsubsection=\tocsubsubsection

\renewcommand{\tocsection}[2]{\hspace{0em}\oldtocsection{#1}{#2}}
\renewcommand{\tocsubsection}[2]{\hspace{1em}\oldtocsubsection{#1}{#2}}
\renewcommand{\tocsubsubsection}[2]{\hspace{2em}\oldtocsubsubsection{#1}{#2}}

\DeclareMathOperator{\Hom}{Hom}

\DeclareMathOperator{\End}{End}

\DeclareMathOperator{\Gr}{Gr} 
\DeclareMathOperator{\id}{id}

\DeclareMathOperator{\Span}{Span}

\def\co{\colon\thinspace} 

\newcommand{\abs}[1]{\left|#1\right|}

\newcommand{\norm}[1]{\left\|#1\right\|}
\newcommand{\wt}[1]{\widetilde{#1}}
\newcommand{\wh}[1]{\widehat{#1}}

\newcommand{\vertiii}[1]{{\left\vert\kern-0.25ex\left\vert\kern-0.25ex\left\vert #1 
    \right\vert\kern-0.25ex\right\vert\kern-0.25ex\right\vert}}
	
	\newcounter{notes}

\begin{document}

\title{$d$--Pleated surfaces and their shear-bend coordinates}

\author[Maloni S.]{Sara Maloni}
\address{Department of Mathematics, University of Virginia}
\email{sm4cw@virginia.edu}
\urladdr{sites.google.com/view/sara-maloni/}

\author[Martone G.]{Giuseppe Martone}
\address{Department of Mathematics and Statistics, Sam Houston State University}
\email{gxm120@shsu.edu}
\urladdr{sites.google.com/view/giuseppemartone}

\author[Mazzoli F.]{Filippo Mazzoli}
\address{Department of Mathematics, University of California Riverside}
\email{filippo.mazzoli@ucr.edu}
\urladdr{filippomazzoli.github.io/index.html}

\author[Zhang T.]{Tengren Zhang}
\address{Department of Mathematics, National University of Singapore}
\email{matzt@nus.edu.sg}
\urladdr{sites.google.com/site/tengren85/}

\thanks{S.M. and F.M. were partially supported by U.S. National Science Foundation grant DMS-1848346 (NSF CAREER). G.M. acknowledges partial support by the AMS and the Simons Foundation. T.Z. was partially supported by the NUS-MOE grant R-146-000-270-133 and A-8000458-00-00. The authors also acknowledge support from the GEAR Network, funded by the National Science Foundation under grant numbers DMS 1107452, 1107263, and 1107367 (``RNMS: GEometric structures And Representation varieties").} 

\date{\today}

\begin{abstract}
In this article, we single out representations of surface groups into $\mathsf{PGL}_d(\mathbb{C})$ which generalize the well-studied family of pleated surfaces into $\mathsf{PGL}_2(\mathbb{C})$. Our representations arise as sufficiently generic $\lambda$--Borel Anosov representations, which are representations that are Borel Anosov with respect to a maximal geodesic lamination $\lambda$. For fixed $\lambda$ and $d$, we provide a holomorphic parameterization of the space $\mathcal R_d(\lambda)$ of $d$--pleated surfaces with pleating locus $\lambda$ which extends both work of Bonahon for pleated surfaces and Bonahon and Dreyer for Hitchin representations.  
\end{abstract}

\maketitle

\tableofcontents

\section{Introduction}\label{intro}

Let $S$ be a closed oriented hyperbolic surface and $\Gamma$ its fundamental group. Hitchin's work \cite{hit_lie} led to the discovery and study of a distinguished connected component in the space of conjugacy classes of representations $\Gamma \to \mathsf{PGL}_d(\mathbb{R})$ for any $d \geq 2$, commonly referred to as the \emph{Hitchin component} $\mathrm{Hit}_d(S)$. When $d = 2$, $\mathsf{PGL}_2(\mathbb{R})$ identifies with the group of isometries of the hyperbolic plane $\mathbb{H}^2$ and every element of $\mathrm{Hit}_2(S)$ is represented by the holonomy of a hyperbolic structure on $S$ or, in other words, by a \emph{Fuchsian representation}. Consequently, $\mathrm{Hit}_2(S)$ coincides with the Fricke-Klein model of Teichm\"uller space $\mathcal{T}(S)$. In fact, all group homomorphisms $\Gamma \to \mathsf{PGL}_d(\mathbb{R})$ whose conjugacy classes belong to the Hitchin component share many dynamical and geometric properties with Fuchsian representations. For example, the foundational papers by Labourie \cite{labourie-anosov} and Fock-Goncharov \cite{fock-goncharov-1} prove that such homomorphisms are discrete and faithful.

The natural inclusion of Lie groups $\mathsf{PGL}_2(\mathbb{R}) \subset \mathsf{PGL}_2(\mathbb{C})$ determines a totally geodesic embedding $f_0 : \mathbb{H}^2 \to \mathbb{H}^3$ of the hyperbolic plane inside hyperbolic $3$--space (as the corresponding symmetric spaces). The theory of pleated surfaces, introduced by Thurston \cite{thurston-notes} and further developed by Bonahon \cite{bonahon-toulouse}, allows to deform any Fuchsian representation $r_0 : \Gamma \to \mathsf{PGL}_2(\mathbb{R})$ by bending the map $f_0$ along a suitable collection of disjoint geodesics in $\mathbb{H}^2$. The process is carried out in such a way that the resulting \emph{pleated map} $f : \mathbb{H}^2 \to \mathbb{H}^3$ conjugates the actions of $r_0$ and of a ``bent" representation $r : \Gamma \to \mathsf{PGL}_2(\mathbb{C})$, namely $f \circ r_0(\gamma) = r(\gamma) \circ f$ for any $\gamma \in \Gamma$. The pleating locus of $f$ is the preimage in $\mathbb{H}^2$ of a closed subset $\lambda$ of $S = \mathbb{H}^2/r_0(\Gamma)$ obtained as a union of disjoint, simple (possibly closed) geodesics, called a \emph{geodesic lamination}. When the connected components of $S - \lambda$ are isometric to ideal hyperbolic triangles (or, equivalently, when the lamination $\lambda$ is \emph{maximal} with respect to the inclusion), the subset $\mathcal{R}(\lambda)$ of $\Hom(\Gamma,\mathsf{PGL}_2(\mathbb{C}))$, consisting of all representations that can be obtained by bending Fuchsian representations along $\lambda$, is open and contains the Fuchsian locus. Unlike the Fuchsian locus, the open set $\mathcal{R}(\lambda)$ contains representations that are not discrete (or not faithful).

In this paper, we develop a process of \emph{generalized bending} that, given a Hitchin representation $\rho_0 : \Gamma \to \mathsf{PGL}_d(\mathbb{R})$, deforms $\rho_0$ both along the geodesics and on the complementary regions of a maximal geodesic lamination $\lambda$ of $S$. The representations $\rho : \Gamma \to \mathsf{PGL}_d(\mathbb{C})$ that are obtained with this procedure are holonomies of \emph{$d$--pleated surfaces with pleating locus $\lambda$}, and they describe an open subset $\mathcal{R}_d(\lambda)$ of the representation variety $\Hom(\Gamma,\mathsf{PGL}_d(\mathbb{C}))$. We characterize these representations in terms of a dynamical property of the action of $\rho$ on the space $\mathcal{F}(\mathbb{C}^d)$ of complete flags of $\mathbb{C}^d$: they are \emph{$\lambda$--Borel Anosov}, a weakening of the notion of Anosov representation. When $d = 2$, this provides a new description of the image $\mathcal{R}_2(\lambda) = \mathcal{R}(\lambda)$ of Bonahon’s shear-bend parameterization from \cite{bonahon-toulouse} purely in terms of the dynamics of the action of representations $r : \Gamma \to \mathsf{PGL}_2(\mathbb{C})$ on the Riemann sphere $\mathbb{C}\mathrm{P}^1 \cong \mathcal{F}(\mathbb{C}^2)$. 

Furthermore, we provide analytic coordinates
\[
\mathcal{R}_d(\lambda) \longrightarrow \mathrm{Hit}_d(S) \times \mathcal{Y}_d(\lambda;\mathbb{R}/2\pi \mathbb{Z}) \times \mathsf{PGL}_d(\mathbb{C}),
\]
where $\mathcal{Y}_d(\lambda;\mathbb{R}/2\pi \mathbb{Z})$, called the space of \emph{$\lambda$--cocyclic pairs of dimension $d$ with values in $\mathbb{R}/2\pi \mathbb{Z}$}, is an Abelian Lie group that is homological in nature, and is the $\mathbb{R}/2\pi\mathbb{Z}$ analog of the space of shearing cocycles and triangle invariants from Bonahon-Dreyer \cite{BoD}. The parameterization of $\mathcal{R}_d(\lambda)$ above is equivariant with respect to the action of $\mathsf{PGL}_d(\mathbb{C})$ by conjugation on the left and by multiplication on the third component on the right, and it is a biholomorphism with respect to an appropriate complex structure on $\mathrm{Hit}_d(S) \times \mathcal{Y}_d(\lambda;\mathbb{R}/2\pi \mathbb{Z})$, which we describe below. 

We conjecture that, for every maximal geodesic lamination $\lambda$, the space $\mathcal{R}_d(\lambda)$ of $d$–pleated surfaces with pleating locus $\lambda$ is dense inside $\Hom(\Gamma,\mathsf{PGL}_d(\mathbb{C}))$. We hope to address this in the future.

\subsection{\texorpdfstring{$d$}{d}-pleated surfaces and the main statement}

The first important contribution of the present paper is the definition of \emph{$d$--pleated surfaces}, which we now describe. 

Let $\lambda$ be, as above, a fixed maximal geodesic lamination in the hyperbolic surface $S$, and let $\widetilde S$ denote its universal cover. Each geodesic inside the lift $\widetilde{\lambda} \subset \widetilde{S}$ of $\lambda$, also called a \emph{leaf} of $\widetilde{\lambda}$, is uniquely determined by its endpoints in the visual boundary $\partial\widetilde S$ of $\widetilde S$. Similarly, each complementary region of $\widetilde{\lambda}$, also called a \emph{plaque} of $\widetilde{\lambda}$, is isometric to an ideal triangle and, as such, is uniquely determined by its three ideal vertices in $\partial \widetilde S$. In what follows, we let $\partial \widetilde{\lambda} \subset \partial \widetilde{S}$ denote the set consisting of all endpoints of leaves of $\widetilde{\lambda}$. 

A \emph{$d$--pleated surface with pleating locus $\lambda$} is given by the data of a representation $\rho : \Gamma \to \mathsf{PGL}_d(\mathbb{C})$, and a $\rho$--equivariant map $\xi:\partial\wt\lambda\to\mathcal{F}(\mathbb{C}^d)$, called its \emph{$\lambda$--limit map}, from the set $\partial\widetilde{\lambda}$ into the space $\mathcal{F}(\mathbb{C}^d)$ of complete flags of $\mathbb{C}^d$, that satisfy the following properties:
\begin{description}
	\item[\textbf{$\lambda$--transversality}] for every leaf of $\wt\lambda$ with endpoints $x,y\in\partial \widetilde{\lambda}$, the pair of flags $\xi(x),\xi(y)$ is transverse.\vspace{0.5em}
	\item[\textbf{$\lambda$--hyperconvexity}] for every plaque of $\widetilde{\lambda}$ with vertices $x, y, z \in \partial \widetilde{\lambda}$, the triple of flags $\xi(x),\xi(y),\xi(z)$ is in general position (see Section \ref{sec: transversality} for the notion of general position that we use).\vspace{0.5em}
	\item[\textbf{$\lambda$--Borel Anosov}] Let $T^1 \lambda \subset T^1 S$ denote the set of unit tangent vectors to $S$ that are tangent to $\lambda$. Then the map $\xi$ determines a suitable collection of flat bundles over $T^1 \lambda$ with holonomy $\rho$, on which the natural lift of geodesic flow acts uniformly contracting/expanding across the fibers.
\end{description}
\noindent For a more detailed definition, see Section \ref{ssec:BorelAnosov}. We denote by $\mathcal{R}_d(\lambda)$ the set of representations that arise from some $d$--pleated surface with pleating locus $\lambda$. By classical results of hyperbolic dynamics (see e.g. \cite[Section 19]{Katok_Hasselblatt}, and Wang \cite{wang2021anosov}), $\mathcal{R}_d(\lambda)$ is an open subset of $\Hom(\Gamma,\mathsf{PGL}_d(\mathbb{C}))$ and, for any $\rho \in \mathcal{R}_d(\lambda)$, the map that associates to every leaf of $\widetilde{\lambda}$ with endpoints $x, y \in \partial \widetilde{S}$, the pair of transverse flags $\xi(x), \xi(y) \in \mathcal{F}(\mathbb{C}^d)$, is locally H\"older continuous and is uniquely determined by $\rho$. (For this reason, we will often confuse the data of a $d$--pleated surface $(\rho,\xi)$ with the representation $\rho$ itself.) For a general study of representations satisfying Anosov-like conditions on subflows of the unit tangent bundle $T^1 S$, which include representations arising from $d$--pleated surfaces, we refer to Wang \cite{wang2021anosov}. 

If a representation $\rho : \Gamma \to \mathsf{PGL}_d(\mathbb{C})$ is Borel Anosov and its $\rho$--equivariant limit curve $\partial\widetilde S \to \mathcal{F}(\mathbb{C}^d)$ is both $2$-- and $3$--hyperconvex in the sense of Labourie \cite{labourie-anosov} (for example, if $\rho$ is conjugate to a Hitchin representation in $\mathsf{PGL}_d(\mathbb{R})$), then for \emph{every} maximal geodesic lamination $\lambda$ of $S$, $\rho$ arises as a $d$--pleated surface with pleating locus $\lambda$. Hence, admitting an equivariant $\lambda$--limit map $\xi$ that satisfies the aforementioned properties can be considered as a weakening of the Borel Anosov and hyperconvex conditions from \cite{labourie-anosov} adapted to the lamination $\lambda$ (and to the complex Lie group $\mathsf{PGL}_d(\mathbb{C})$). It follows from work of Weil \cite{wei64} (see also Goldman \cite{goldman-symplectic}) that the set $\mathcal{R}_d(\lambda)$ is contained in the smooth locus of the representation variety (see Proposition \ref{prop: pleated surfaces smooth}), and thus is naturally a complex manifold. 

Our main result describes a biholomorphic parameterization of the space $\mathcal{R}_d(\lambda)$ of $d$--pleated surfaces with pleating locus $\lambda$ that extends the parameterization of the Hitchin component from the work of Bonahon and Dreyer \cite{BoD}. Drawing inspiration from the work of Fock and Goncharov \cite{fock-goncharov-1} on positive framed representations of punctured surface groups, Bonahon and Dreyer generalized Thurston's shearing coordinates of Teichm\"uller space $\mathcal{T}(S)$ (see \cite{thu_stretch,bonahon-toulouse}) to the context of Hitchin representations in $\mathsf{PGL}_d(\mathbb{R})$. In particular, they construct, for any maximal geodesic lamination $\lambda$ on $S$, a continuous and injective map
\[
\mathfrak{s}_d:\mathrm{Hit}_d(S) \longrightarrow \mathcal{Y}_d(\lambda;\mathbb{R})
\]
from the Hitchin component of $\mathsf{PGL}_d(\mathbb{R})$ into a certain real vector space $\mathcal{Y}_d(\lambda;\mathbb{R})$. An element of $\mathcal{Y}_d(\lambda;\mathbb{R})$ is a pair $(\alpha,\theta)$, where $\alpha$ is an $\mathbb{R}$--valued relative cocycle transverse to $\lambda$, and $\theta$ is a collection of $\mathbb{R}$--valued functions associated with the plaques of $\lambda$ that satisfy certain linear conditions involving $\alpha$, see Definition \ref{def_cocycle}. We refer to such pairs as \emph{$\mathbb{R}$--valued $\lambda$--cocyclic pairs of dimension $d$}. The real vector space $\mathcal{Y}_d(\lambda;\mathbb{R})$ has real dimension $(d^2 - 1)|\chi(S)| = \dim_{\mathbb{R}} \mathrm{Hit}_d(S)$, and the image of $\mathfrak s_d$ is fully characterized; it is equal to a distinguished open convex cone $\mathcal{C}_d(\lambda)$ inside $\mathcal{Y}_d(\lambda;\mathbb{R})$ cut out by finitely many linear inequalities. 

In analogy with the shear-bend coordinates for equivariant pleated surfaces from the work of Bonahon \cite{bonahon-toulouse}, our parameterization associates to every $d$--pleated surface $(\rho,\xi)$ with pleating locus $\lambda$, a $\lambda$--cocyclic pair of dimension $d$ with values in $\mathbb{C}/2\pi i \mathbb{Z}$, which we call the \emph{shear-bend coordinates of $(\rho,\xi)$}. The space of $\mathbb{C}/2\pi i \mathbb{Z}$--valued $\lambda$--cocyclic pairs of dimension $d$, denoted by $\mathcal{Y}_d(\lambda;\mathbb{C}/2\pi i \mathbb{Z})$, carries a natural structure of complex Abelian Lie group of \emph{complex} dimension $(d^2 - 1) \abs{\chi(S)}$ (see Section \ref{shear-bend}). The real (resp. imaginary) part of every element in $\mathcal{Y}_d(\lambda;\mathbb{C}/2\pi i \mathbb{Z})$ determines an $\mathbb{R}$--valued (resp. $\mathbb{R}/2 \pi \mathbb{Z}$--valued) $\lambda$--cocyclic pair of dimension $d$ and, hence, an element of $\mathcal{Y}_d(\lambda;\mathbb{R})$ (resp. $\mathcal{Y}_d(\lambda;\mathbb{R}/2 \pi \mathbb{Z})$). Our parameterization result can then be stated as follows:

\begin{introthm}[Theorem \ref{thm: main}]\label{thm: intro main}  
	There exists a $\mathsf{PGL}_d(\mathbb{C})$--equivariant map
	\[
	\mathcal{sb}_d \co \mathcal{R}_d(\lambda)\longrightarrow \mathcal{Y}_d(\lambda;\mathbb{C}/2\pi i\mathbb{Z})\times \mathsf{PGL}_d(\mathbb{C})
	\]
	that is a biholomorphism onto its image, which is equal to
	\[
	(\mathcal{C}_d(\lambda) + i\, \mathcal{Y}_d(\lambda;\mathbb{R}/2\pi \mathbb{Z}) )\times \mathsf{PGL}_d(\mathbb{C}),
	\]
	where $\mathcal{C}_d(\lambda)$ is the image of Bonahon-Dreyer's parameterization.
\end{introthm}

\noindent \emph{Comments on the main statement:}
\begin{enumerate}[wide, labelwidth=!]
	\item In Section \ref{sec: nonHausdorff} we show the following.
	\begin{introprop}[Corollary \ref{cor: GKW} and Proposition \ref{lem: non-Hausdorff 1}] When $d=2,3$, there exists a maximal lamination $\lambda$ on $S$ for which the space $\mathfrak R_d(\lambda)$ of conjugacy classes of $d$--pleated surfaces with pleating locus $\lambda$ intersects the locus of non-Hausdorff points of $\Hom(\Gamma,\mathsf{PGL}_d(\mathbb{C}))/\mathsf{PGL}_d(\mathbb{C})$. 
	\end{introprop}
	Because of this, Theorem \ref{thm: intro main} is phrased in terms of the subset $\mathcal{R}_d(\lambda)$ in the representation variety $\Hom(\Gamma,\mathsf{PGL}_d(\mathbb{C}))$ rather than its image inside the character variety. On the other hand, as a consequence of Theorem \ref{thm: intro main}, we deduce that the map $\mathcal{sb}_d$ descends to a homeomorphism from $\mathfrak R_d(\lambda)$ to $\mathcal{C}_d(\lambda) + i\, \mathcal{Y}_d(\lambda;\mathbb{R}/2\pi \mathbb{Z})$. In particular, $\mathfrak R_d(\lambda)$ admits the structure of a complex manifold. 
	
	\item We prove that, when $d = 2$, the data of a $2$--pleated surface is equivalent to the classical notion of a pleated surface from \cite{bonahon-toulouse,thurston-notes}, where the $\lambda$--limit map $\xi$ and the pleated map $f : \mathbb{H}^2 \to \mathbb{H}^3$ associated with a representation $\rho$ are related by the following property: for any leaf $g$ of $\widetilde{\lambda}$ with endpoints $x, y \in \partial \widetilde{S}$, the image $f(g)$ coincides with the geodesic in $\mathbb{H}^3$ whose endpoints are given by $\xi(x), \xi(y) \in \mathbb{C} \mathrm{P}^1 \cong \mathcal{F}(\mathbb{C}^2)$. Consequently, this provides a new characterization of the set of representations in $\Hom(\Gamma,\mathsf{PGL}_2(\mathbb{C}))$ that admit equivariant pleated surfaces with pleating locus $\lambda$ in the classical sense. Furthermore, Theorem \ref{thm: intro main} reduces to Bonahon's shear-bend parameterization of $\mathcal{R}_2(\lambda)$ from \cite{bonahon-toulouse}.
	
	\item Based on the regularity of the map from Theorem \ref{thm: intro main}, our result implies that Bonahon-Dreyer's parameterization of $\mathrm{Hit}_d(S)$ is, in fact, a real-analytic diffeomorphism (compare with \cite{BoD}).
	
	\item Based on the characterization of the image of the shear-bend parameterization above, the space $\mathcal{R}_d(\lambda)$ admits a real-analytic projection onto the Hitchin component $\mathrm{Hit}_d(S)$, which generalizes the notion of induced hyperbolic metric inherited from the data of an equivariant pleated map when $d = 2$. In geometric terms, Theorem \ref{thm: intro main} implies that any $d$--pleated surface with pleating locus $\lambda$ can be uniquely realized from a Hitchin representation by performing a process of ``generalized bending'' along the leaves and plaques of $\lambda$, according to the data provided by a $\lambda$--cocyclic pair in $\mathcal{Y}_d(\lambda;\mathbb{R}/2\pi\mathbb{Z})$. This process relies on a purely imaginary analog of the eruption flows considered in Wienhard-Zhang \cite{WZ} and Sun-Wienhard-Zhang \cite{SWZ}.
	
	\item In a subsequent work \cite{MMMZ2}, we prove that the space $\mathcal{Y}_d(\lambda;\mathbb{R}/2\pi \mathbb{Z})$ has exactly $d$ connected components, each of which is diffeomorphic to $(\mathbb{S}^1)^{(d^2-1) |\chi(S)|}$. In particular, this shows that, even though $\mathcal{R}_d(\lambda)$ as a subset of $\Hom(\Gamma,\mathsf{PGL}_d(\mathbb{C}))$ does depend on the lamination $\lambda$, its topology does not. Furthermore, each component of the image of $\mathcal{sb}_d$ corresponds to the intersection of $\mathcal{R}_d(\lambda)$ with precisely one of the $d$ connected components of the representation variety $\Hom(\Gamma,\mathsf{PGL}_d(\mathbb{C}))$, see \cite{MMMZ2} for details.
	
	\item In the case when $S$ has at least one puncture, a \emph{framed representation} from $\Gamma$ to $\mathsf{PGL}_d(\mathbb{C})$ is a pair $(\rho,\xi)$, where $\rho:\Gamma\to\mathsf{PGL}_d(\mathbb{C})$ is a representation and $\xi:\Lambda_p(\Gamma)\to\mathcal{F}(\mathbb{C}^d)$ is a $\rho$--equivariant map defined on the set of parabolic points $\Lambda_p(\Gamma)\subset\partial\widetilde S$. Fock and Goncharov considered the space $\mathcal{X}_{\mathsf{PGL}_d(\mathbb{C}),S}$ of framed representations from $\Gamma$ to $\mathsf{PGL}_d(\mathbb{C})$, and constructed, for any choice of ideal triangulation on $S$, a parameterization of an open dense subset of $\mathcal{X}_{\mathsf{PGL}_d(\mathbb{C}),S}$ by $(\mathbb{C}/2\pi i\mathbb{Z})^n$ for some positive integer $n$ \cite[Theorem 1.1]{fock-goncharov-1}. Theorem \ref{thm: intro main} can be viewed as an analog of this result in the case when $S$ is a closed surface, $\mathcal{X}_{\mathsf{PGL}_d(\mathbb{C}),S}$ is replaced with $\Hom(\Gamma,\mathsf{PGL}_d(\mathbb{C}))/\mathsf{PGL}_d(\mathbb{C})$, and the ideal triangulation is a replaced with a geodesic lamination $\lambda$. In our setting, the density of the open set in $\Hom(\Gamma,\mathsf{PGL}_d(\mathbb{C}))/\mathsf{PGL}_d(\mathbb{C})$ parameterized in Theorem \ref{thm: intro main} (after taking the quotient by $\mathsf{PGL}_d(\mathbb{C})$) is still a conjecture. 
\end{enumerate}

\subsection{The definition of the shear-bend parameterization.} The definition of the parameterization of Theorem \ref{thm: intro main} follows the construction of Bonahon and Dreyer's parameterization of the Hitchin component $\mathrm{Hit}_d(S)$, which, in turn, relies on the projective invariants for $k$--tuples of flags developed by Fock and Goncharov \cite{fock-goncharov-1} called {\em triple and double ratios}, see Section \ref{sec: double and triple}.

A key ingredient needed to define the shear--bend coordinates of a $d$--pleated surface $(\rho,\xi)$ with pleating locus $\lambda$ is its \emph{slithering map}. It provides a natural way to relate pairs of transverse flags associated to the endpoints of the leaves of $\widetilde{\lambda}$ and can be described as a function
\[
\Sigma : \widetilde{\Lambda}^2 \longrightarrow \mathsf{SL}_d(\mathbb{C}) ,
\]
where $\widetilde{\Lambda}$ denotes the set of unoriented leaves of $\widetilde{\lambda}$, which satisfies a series of natural properties with respect to the $\lambda$--limit map $\xi$. In particular, for any pair of leaves $g_1, g_2$ of $\widetilde{\lambda}$, the transformation $\Sigma_{g_1,g_2}:=\Sigma(g_1,g_2)$ sends the pair of flags associated to the endpoints of $g_2$ into the pair of flags associated to the endpoints of $g_1$, and it is unipotent whenever $g_1$ and $g_2$ share an endpoint. 
%This notion generalizes the definition of slithering map for Hitchin representations from Bonahon-Dreyer \cite[Section~5.1]{BoD} to the much wider class of representations that arise from $d$--pleated surfaces with pleating locus $\lambda$. 
To establish the existence of the slithering map, we need careful estimates on the chosen family of unipotent elements. These estimates hold since the limit map $\xi$ is locally $\lambda$--H\"older continuous, which is a consequence of $\rho$ being $\lambda$--Borel Anosov. When $d = 2$, the slithering map has a natural interpretation in terms of the geometry of $\widetilde{S}$, involving the \emph{horocyclic foliation} determined by $\widetilde{\lambda}$ and the leaves $g_1, g_2$, see \cite[Section~2]{bonahon-toulouse}. 

%Let us move now to the notions of triple and double ratios. Given an ordered triple of flags $(F_1, F_2, F_3)$ of $\mathbb{C}^d$ in general position (see Section \ref{flag} for the definition of general position), the \emph{triple ratios} of $(F_1,F_2,F_3)$ are a collection of non-zero complex numbers $T^{\bf j}(F_1,F_2,F_3)$ indexed by ${\bf j} = (j_1,j_2,j_3) \in \mathcal{B}$, where $\mathcal{B}$ denotes the set of triples of postive integers that sum up to $d$. The data of the triple ratios $(T^{\bf j}(F_1,F_2,F_3))_{{\bf j} \in \mathcal{B}}$ determines uniquely the $\mathsf{PGL}_d(\mathbb{C})$--orbit of the triple $(F_1,F_2,F_3)$. Moreover, given a quadruple $(G_1,G_2,H_1,H_2)$ of flags of $\mathbb{C}^3$ such that the triples $(G_1,G_2,H_1)$ and $(G_1,G_2,H_2)$ are in general position, the \emph{double ratios} of $(G_1,G_2,H_1,H_2)$ are a collection of non-zero complex numbers $D^{\bf i}(G_1,G_2,H_1,H_2)$ indexed by ${\bf i} \in \mathcal{A}$, where $\mathcal{A}$ denotes the set of pairs positive integers that sum up to $d$. The data of the triple ratios of the triples $(G_1,G_2,H_1)$ and $(G_1,G_2,H_2)$, combined with the collection of double ratios $D^{\bf i}(G_1,G_2,H_1,H_2)$, uniquely determine the $\mathsf{PGL}_d(\mathbb{C})$--orbit of the $4$--tuple of flags $(G_1,G_2,H_1,H_2)$. (See Section~\ref{sec: double and triple} for the precise definitions of the triple and double ratios.)

Then, the shear-bend map $\mathcal{sb}_d$ from Theorem \ref{thm: intro main} send any representation $\rho$ arising from a $d$--pleated surface $(\rho,\xi)$ with pleating locus $\lambda$, into a pair $((\alpha_\rho, \theta_\rho),h_\rho)$, where:
\begin{itemize}
	\item $\theta_\rho$ associates with every ordered triple of points $(x_1,x_2,x_3) \in (\partial\wt\lambda)^3$ that are vertices of some plaque of $\wt\lambda$ and with every triple ${\bf j}$ of positive integers that sum up to $d$ the element
	\[
	\theta_\rho((x_1,x_2,x_3), {\bf j}) := \log T^{\bf j}(\xi(x_1), \xi(x_2), \xi(x_3)) \in \mathbb{C}/2\pi i \mathbb{Z},
	\]
	where $T^{\bf j}$ is a triple ratio (see Section \ref{sec: double and triple} for details).
	\item $\alpha_\rho$ associates with every pair of distinct plaques $T_1, T_2$ of $\widetilde{\lambda}$ and with every pair ${\bf i}$ of positive integers that sum up to $d$, the element
	\[
	\alpha_\rho ((T_1,T_2), {\bf i}) := \log D^{{\bf i}}\left(\xi(y_2), \xi(y_1), \xi(y_3), \Sigma_{g_1, g_2}\left(\xi(z_3)\right)\right) \in \mathbb{C}/2\pi i \mathbb{Z},
	\]
	where $D^{\bf i}$ is a double ratio (see Section \ref{sec: double and triple} for details), $g_1, g_2$ denote the sides of $T_1, T_2$, respectively, that separate the interiors of the two plaques, and $(y_1,y_2,y_3)$, $(z_1,z_2,z_3)$ are the ordered triples of vertices of $T_1, T_2$, respectively, as in Figure~\ref{fig:shearing}. 
	
	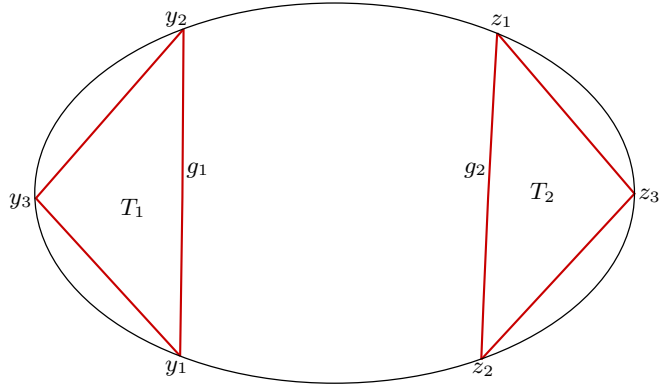
\begin{figure}[h!]
		\includegraphics[width=8cm]{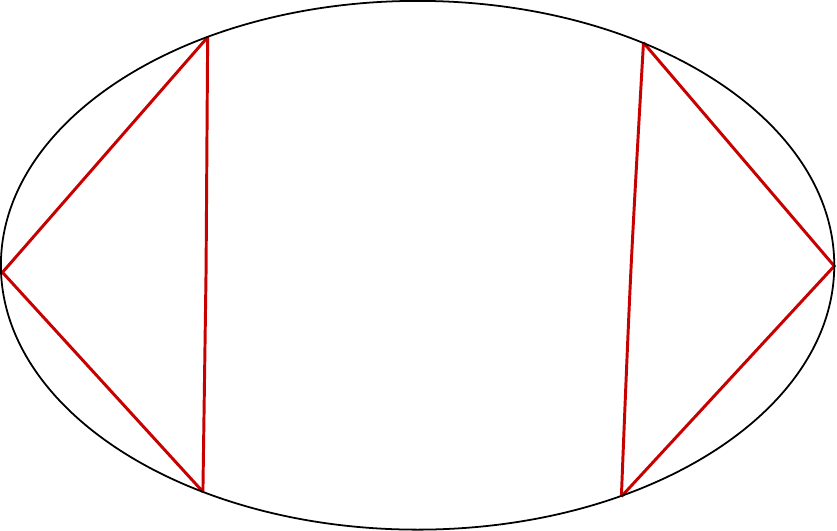}
		\put(-195,64){\small $T_1$}
		\put(-170,80){\small $g_1$}
		\put(-178,137){\small $y_2$}
		\put(-178,5){\small $y_1$}
		\put(-237,68){\small $y_3$}
		\put(-40,70){\small $T_2$}
		\put(-65,80){\small $g_2$}
		\put(-55,136){\small $z_1$}
		\put(-62,4){\small $z_2$}
		\put(1,70){\small $z_3$}
		\caption{\small The notation used in the definition of $\alpha_\rho((T_1,T_2),{\bf i)}$.}
		\label{fig:shearing}
	\end{figure}
	
	\item the projective transformation $h_\rho \in \mathsf{PGL}_d(\mathbb{C})$ is defined by the choice of a fixed base plaque of $\widetilde{\lambda}$  together with an ordering on its vertices, and of a fixed projective frame of $\mathrm{P}(\mathbb{C}^d)$ (which we omit here from the notation). It is designed so that it parameterizes the $\mathsf{PGL}_d(\mathbb{C})$ conjugacy class of the representation $\rho$. (Notice that $\theta_\rho$ and $\alpha_\rho$ are constant on $\mathsf{PGL}_d(\mathbb{C})$--conjugacy classes.)
\end{itemize}

%When $d = 2$, the space of flags $\mathcal{F}(\mathbb{C}^2)$ identifies with the Riemann sphere $\mathbb{C}\mathrm{P}^1$ and the double-ratio associated with $\mathcal{A} = \{{\bf i}=(1,1)\}$ reduces to the standard cross ratio in $\mathbb{C}\mathrm{P}^1$. Since $\mathsf{PGL}_2(\mathbb{C})$ acts simply transitively on the set of triples of distinct points $\mathbb{C}\mathrm{P}^1$, there is no need for triple ratios and the data of $\theta_\rho$ is empty. 

Because of the equivariance of $\xi$ with respect to $\rho$, combined with the properties of triple and double ratios, and of the slithering map $\Sigma$, the pair $(\alpha_\rho, \theta_\rho)$ satisfies a series of symmetries and (quasi-additive) relations, which make $(\alpha_\rho, \theta_\rho)$ a $\mathbb{C}/2\pi i \mathbb{Z}$--valued $\lambda$--cocyclic pair of dimension $d$ and, hence, an element of the space $\mathcal{Y}_d(\lambda;\mathbb{C}/2\pi i \mathbb{Z})$ (see Proposition~\ref{prop: shear bend map}). It follows from the construction that the $\lambda$--cocyclic pair associated with a Hitchin representation coincides with its Bonahon-Dreyer's coordinates from \cite{BoD}.

\subsection{The proof of Theorem \ref{thm: intro main}}\label{subsec:outline of the proof}

The first technical ingredient for the proof of our main result is the existence and uniqueness of a slithering map associated to any $d$--pleated surface $(\rho,\xi)$ with pleating locus $\lambda$, see Theorem \ref{thm: slitherable and slithering}. Given any pair of leaves $g_1, g_2$, the slithering map $\Sigma_{g_1,g_2}$ is obtained as limit of suitable compositions of finitely many unipotent transformations associated with the plaques that separate $g_1$ from $g_2$. Since the same technical estimates play an important role in other aspects of our work, we present the key convergence results in a more general framework: the study of what we call \emph{H\"older extendable maps}. Section~\ref{sec:productable} is dedicated to the study of extensions of such maps, and the existence (and uniqueness) of the slithering map compatible with a limit map $\xi$ is established later in Section~\ref{sec: slithering}. This construction naturally extends the definition of slithering map for Hitchin representations from Bonahon-Dreyer \cite[Section~5.1]{BoD}.

The proof of our main result can then be organized into four main steps:

\begin{itemize}
	\item[Step~1:] The shear-bend map $\mathcal{sb}_{d}$ is injective.
	\item[Step~2:] The image of $\mathcal{sb}_d$ contains $\big(\mathcal{C}_d(\lambda)+i\mathcal{Y}_d(\lambda;\mathbb{R}/2\pi\mathbb{Z})\big)\times \mathsf{PGL}_d(\mathbb{C})$.
	\item[Step~3:] The image of $\mathcal{sb}_d$ lies in $\big(\mathcal{C}_d(\lambda)+i\mathcal{Y}_d(\lambda;\mathbb{R}/2\pi\mathbb{Z})\big)\times \mathsf{PGL}_d(\mathbb{C})$.
	\item[Step~4:] The map $\mathcal{sb}_{d}$ is a biholomorphism onto its image.
\end{itemize}

The proof of the first step relies on the following key observation: the $\lambda$--limit map $\xi$ of a $d$--pleated surface determines the representation $\rho$, and is determined by the slithering map $\Sigma$ compatible with $\xi$. Thus, to prove that the map $\mathcal{sb}_d$ is injective, it suffices to show that the slithering map $\Sigma$ is uniquely determined by the shear-bend coordinates $\mathcal{sb}_d(\rho) = ((\alpha_\rho, \theta_\rho),h_\rho)$. We do so by expressing the slithering map of $(\rho,\xi)$ as a suitable infinite product of linear transformations that depend only on $(\alpha_\rho, \theta_\rho)$ (and the data of a suitable projective frame that depends only on $h_\rho$). In the case of Hitchin representations, a version of this expression was also developed by Bonahon and Dreyer \cite{BoD}. 

For the second main step, we observe that, given any
\[
((\alpha, \theta), h) \in (\mathcal{C}_d(\lambda) + i \mathcal{Y}_d(\lambda;\mathbb{R}/2\pi\mathbb{Z}))\times \mathsf{PGL}_d(\mathbb{C}) ,
\]
there exists a unique representative of the $\mathsf{PGL}_d(\mathbb{C})$--conjugacy class of some $\mathsf{PGL}_d(\mathbb{R})$--Hitchin representation $\rho_0$ with $\lambda$--limit map $\xi_0$ whose shear-bend coordinates $((\alpha_{\xi_0},\theta_{\xi_0}),h_{\rho_0})$ are given by $((\mathrm{Re}\,\alpha,\mathrm{Re}\,\theta),h)$. This relies on the characterization of the image of Bonahon and Dreyer's shearing parameterization of the Hitchin component. To see that the point $((\alpha,\theta), h)$ lies in the image of $\mathcal{sb}_{d}$, we introduce a process of \emph{generalized bending} that takes in input $(\rho_0,\xi_0)$ and the $\lambda$--cocyclic pair $(\mathrm{Im}\,\alpha, \mathrm{Im}\, \theta) \in \mathcal{Y}_d(\lambda;\mathbb{R}/2 \pi \mathbb{Z})$, and produces a $d$--pleated surface $(\rho,\xi)$ whose shear-bend coordinates are given by
\[
\mathcal{sb}_{d}(\rho,\xi) = ((\mathrm{Re}\,\alpha,\mathrm{Re}\,\theta) + i \, (\mathrm{Im}\,\alpha, \mathrm{Im}\, \theta),h) = ((\alpha, \theta),h) .
\]
This process uses ideas and technical tools developed for H\"older extendable maps in Section~\ref{sec:productable}, as well as the notions of complex shears, complex eruptions, and unipotent eruptions defined in Section~\ref{sec:eruption and shear}. The general outline is inspired Bonahon's bending process for classical pleated surfaces from \cite{bonahon-toulouse}, but the more general context of representations in $\mathsf{PGL}_d(\mathbb{C})$ rather than $\mathsf{PGL}_2(\mathbb{C})$ makes the analysis considerably more technical and subtle than the one in \cite{bonahon-toulouse}.

For the proof of the third step, we develop a homological interpretation of $\lambda$--cocyclic pairs that allows us to view $\mathrm{Re} \, \alpha_{\rho}$ as a relative cycle for a relative homology with values in $\mathbb{R}^{d-1}$. Hence, we may consider the function $\mathbb I([\mathrm{Re} \, \alpha_{\rho}],\cdot)$, where $\mathbb{I}$ is the algebraic intersection on relative homology classes with values in $\mathbb{R}^{d-1}$ (see Proposition~\ref{prop: duality}). By showing that the components of $\mathbb I([\mathrm{Re} \, \alpha_{\rho}],\cdot)$ coincide with the $d-1$ length functions associated with the $\lambda$--Borel Anosov representation $\rho$ (see Section \ref{subsec:length_functions}), we confirm that the coordinates of $d$--pleated surfaces satisfy the same positivity conditions as those of Hitchin representations. This, in turn, implies that every $d$--pleated surface with pleating locus $\lambda$ has shear-bend coordinates $\mathcal{sb}_d(\rho) = ((\alpha_\rho,\theta_\rho), h_\rho)$ that satisfy $(\mathrm{Re}\, \alpha_\rho, \mathrm{Re}\, \theta_\rho) \in \mathcal{C}_d(\lambda)$.

While the broad strategy of the proof of the third step follows the framework established by Bonahon and Dreyer \cite{BoD}, the transition to $d$--pleated surfaces introduces significant technical hurdles. Specifically, the $\lambda$--limit maps $\xi$ are defined only on the endpoints of the lamination $\partial\tilde{\lambda}$ and may lack the continuity or hyperconvexity properties inherent to the representations in the Hitchin component. In particular, the construction of a certain topological closed 1-form used by Bonahon and Dreyer to relate the intersection in homology to the length functions is more delicate in our setting. Furthermore, because $\mathbb{C}/2\pi i\mathbb{Z}$ is an additive group rather than a commutative ring with identity, the homological interpretation cannot rely on standard Poincaré duality, requiring subtle adaptations to the intersection pairings to accommodate the non-finitely generated nature of the coefficient group. See Section \ref{subsec:proof structure} for a more detailed discussion of the challenges in our setting.

Since the image $\mathcal{sb}_d(\mathcal{R}_d(\lambda))$ is an open subset of $\mathcal{Y}_d(\lambda;\mathbb{C}/2\pi i \mathbb{Z}) \times \mathsf{PGL}_d(\mathbb{C})$, which carries a natural structure of complex manifold, to prove the fourth step, it suffices to show that the inverse $\mathcal{sb}_{d}^{-1}$ of $\mathcal{sb}_{d}$ is holomorphic. We will prove that the holonomy $\rho$ of a $d$--pleated surface $(\rho,\xi)$ depends holomorphically on the data of its shear-bend parameters. To do so, it is enough to prove that the slithering map of $\rho$ depends holomorphically on $\mathcal{sb}_{d}(\rho)$, which will be deduced in Section~\ref{sec:holomorphicity} using the same explicit expression of the slithering map mentioned in the proof sketch for Step 1 (Lemma~\ref{lem: injectivity}).

\subsection{Organization of the paper}

In Section \ref{back}, we recall some important definitions and results about geodesic laminations and invariants of flags in $\mathbb{C}^d$. In Section~\ref{sec:d_pletaed}, we introduce the notions of $\lambda$--Borel Anosov representations, $d$--pleated surfaces and length associated to $\lambda$--Borel Anosov representations, and prove that classical pleated surfaces in $\mathsf{PGL}_2(\mathbb{C})$ and Hitchin representations are $\lambda$--Borel Anosov. In Section \ref{sec:cocycle_pair}, we define $\lambda$--cocyclic pairs, and how to define a $\lambda$--cocyclic pair associated to a $d$--pleated surface. We are then able to state precisely Theorem \ref{thm: intro main} and a summary  of the proof and the organization of the remaining part of the paper. In Section~\ref{sec:productable} we develop the general theory of H\"older extendable maps, which we then use in Sections \ref{sec: slithering} and \ref{sec:bending}  to construct slithering maps and generalized bending maps, respectively. 

In Section \ref{sec:eruption and shear} we prove the injectivity of the map $\mathcal{sb}_{d}$, while in Section \ref{length} we prove the surjectivity of the map, after having introduced in Section  \ref{sec: intersection} the necessary homological interpretation of $\lambda$--cocyclic pairs, which is necessary to define a length associated to  $\lambda$--cocyclic pairs and to compare it with the length of the associated $d$--pleated surface. In Section \ref{sec:holomorphicity} we prove the biholomorphicity of the map  $\mathcal{sb}_{d}$. Finally, in Section \ref{sec: nonHausdorff}  we discuss the phenomena of $\mathfrak{R}_d(\lambda)=\mathcal{R}_d(\lambda)/\mathsf{PGL}_d(\mathbb{C})$ intersecting the locus of non-Hausdorff points in $\Hom(\Gamma,\mathsf{PGL}_d(\mathbb{C}))/\mathsf{PGL}_d(\mathbb{C})$.

\subsection*{Acknowledgements} We would like to thank Francis Bonahon, Richard Canary, and Sean Lawton for enlightening conversations, and the referees for many useful comments which helped us to considerably improve the exposition. The authors acknowledge the support of the Institut Henri Poincaré (UAR 839 CNRS-Sorbonne Université) and LabEx CARMIN (ANR-10-LABX-59-01).

\section{Preliminaries}\label{back}

In this section, we discuss several preliminaries we will need in this paper. Specifically, in Section \ref{ssec:geodlam}, we collect background from the theory of geodesic laminations, and in Section \ref{flag}, we discuss flags and the notion of hyperconvexity.

\subsection{Geodesic laminations}\label{ssec:geodlam}
For a detailed treatment of geodesic laminations, see Canary-Epstein-Green \cite[Chapter 4]{CEG06} or Casson-Bleiler \cite[Chapter 3]{Casson}. 

Let $S$ be a closed, oriented, hyperbolic surface, and let $\Gamma$ denote its fundamental group. A {\em geodesic lamination} $\lambda$ is a closed subset of $S$ which can be decomposed as a disjoint union of simple (complete) geodesics, called {\em leaves}. A geodesic lamination is {\em maximal} if it is not properly contained in any other geodesic lamination. The complement of a maximal geodesic lamination has finitely many connected components, called \emph{plaques}, each of which is a hyperbolic ideal triangle in $S$. 

Henceforth, let $\lambda$ be a maximal geodesic lamination, let $\Lambda$ denote the collection of leaves of $\lambda$, and let $\Delta$ denote the set of plaques of $\lambda$. Let $\wt\lambda$, $\wt\Lambda$ and $\wt\Delta$ respectively denote the lifts of $\lambda$, $\Lambda$ and $\Delta$ to $\wt S\simeq\mathbb{H}^2$. We also refer to the elements in $\wt \Lambda$ (respectively, $\wt\Delta$) as \emph{leaves} (respectively, \emph{plaques}). 

Denote by $\partial\wt\lambda\subset\partial\wt S$ the set of endpoints of the leaves of $\wt\lambda$. Then define
\[\wt\Lambda^o:=\{(x,y)\in(\partial\wt\lambda)^2 \mid x\text{ and }y \text{ are the two endpoints of a leaf in }\wt\Lambda\},\]
\begin{align*}
	\wt\Delta^o:=\{(x,y,z)\in(\partial\wt\lambda)^3 \mid x,y,z \text{ are the three vertices of a plaque in }\wt\Delta\},
\end{align*}
and let $\Lambda^o:=\wt\Lambda^o/\Gamma$ and $\Delta^o:=\wt\Delta^o/\Gamma$. We will refer to the elements $(x, y, z)$ in $\wt\Delta^o$ as \textit{labelings} of plaques. Observe that the natural projections $\wt\Lambda^o\to\wt\Lambda$ and $\Lambda^o\to\Lambda$ are two-to-one, while $\wt\Delta^o\to\wt\Delta$ and $\Delta^o\to\Delta$ are six-to-one. 

One can think of the objects in $\wt\Delta^o$, $\Delta^o$, $\wt\Lambda^o$, and $\Lambda^o$ as oriented versions of the objects in $\wt\Delta$, $\Delta$, $\wt\Lambda$, and $\Lambda$ respectively. Indeed, the set of pairs of distinct points in $\partial\wt{S}$ is naturally identified with the set of oriented geodesics in $\wt{S}$ by identifying the pair $(x,y)$ with the oriented geodesic $\mathsf{g}$ in $\wt S$ whose forward and backward endpoints in $\partial\wt S$ are $x$ and $y$ respectively. In this case, we often denote by $\mathsf{g}^+:=x$ and $\mathsf{g}^-:=y$ the endpoints of $\mathsf{g}$, and by $g$ the underlying unoriented geodesic. With this point of view, we may then think of $\wt\Lambda^o$ (respectively, $\Lambda^o$) as the set of leaves of $\wt\lambda$ (respectively, $\lambda$) together with a choice of orientation. Similarly, the set of triples of pairwise disjoint points in $\partial\wt{S}$ is naturally identified with the set of ideal triangles in $\wt{S}$ with a labelling of their vertices. Then we may think of $\wt\Delta^o$ (respectively, $\Delta^o$) as the set of plaques in $\wt S\setminus \wt \lambda$ (respectively, $S\setminus \lambda$) with a labeling of their vertices. 

\subsubsection{Separation}\label{section: separation and coherence} 

Let $g_1$ and $g_2$ be any pair of geodesics in $\wt S$ that are disjoint or equal. We say that a geodesic $h$ in $\wt S$ \emph{separates} $g_1$ and $g_2$ if either $h=g_1$, $h=g_2$, or $g_1$ and $g_2$ lie in distinct connected components of $\widetilde S\setminus h$. Also, an ideal triangle $T$ in $\wt S$ \emph{separates} $g_1$ and $g_2$ if two of the three edges of $T$ separate $g_1$ and $g_2$. See Figure \ref{fig:notations1}.

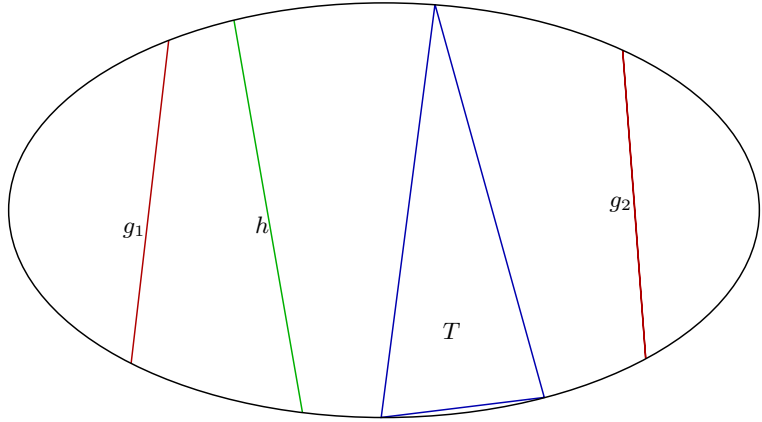
\begin{figure}[h!]
	\includegraphics[width=10cm]{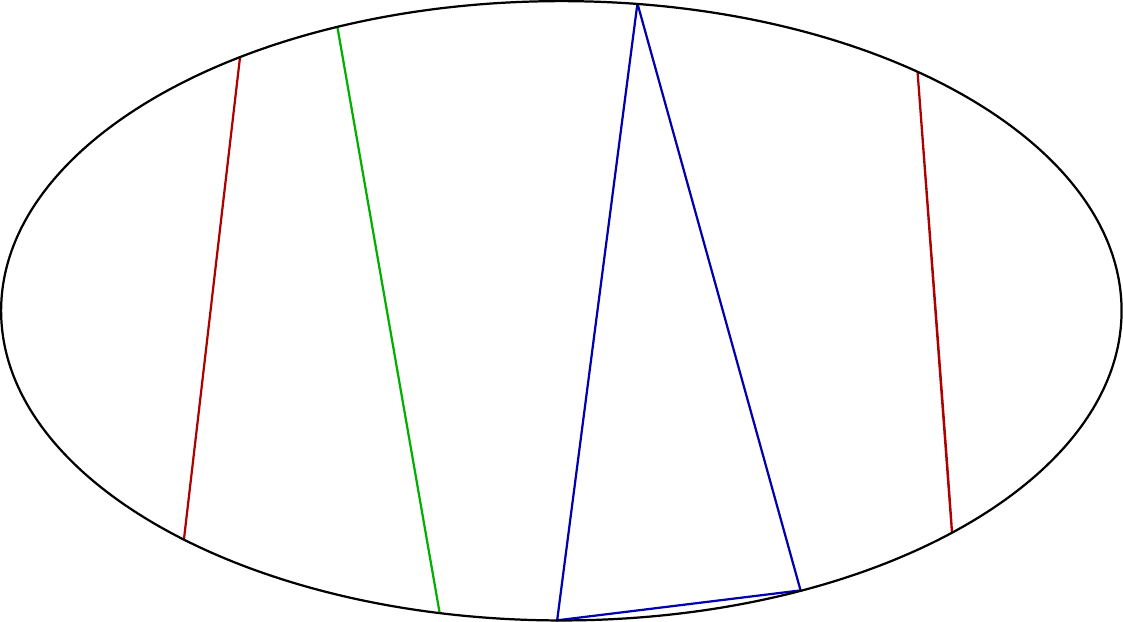}
	\put(-241,70){\small $g_1$}
	\put(-191,70){\small $h$}
	\put(-120,30){\small $T$}
	\put(-57,80){\small $ g_2$}
	\caption{\small The ideal triangle $T$ and the geodesic $h$ separate the geodesics $g_1$ and $g_2$.}
	\label{fig:notations1}
\end{figure}

If $T_1$ and $T_2$ are ideal triangles in $\wt S$ whose interiors do not intersect, then there are edges  $g_1$ and $g_2$ of $T_1$ and $T_2$ respectively such that no other edges of $T_1$ or $T_2$ separate $g_1$ and $g_2$ (it is possible that $g_1=g_2$). We say that $(g_1,g_2)$ is the \emph{separating pair} of edges for $(T_1,T_2)$. An ideal triangle or geodesic \emph{separates} $T_1$ and $T_2$ if it separates the separating pair of edges for $(T_1,T_2)$. 

In the case when $g_1$ and $g_2$ are leaves of $\wt\lambda$, let $Q(g_1,g_2)$ and $P(g_1,g_2)$ respectively denote the set of leaves and the set of plaques of $\wt\lambda$ that separate $g_1$ and $g_2$. We may define a total order $<$ on $Q(g_1,g_2)$ by $h_1\leq h_2$ if $h_1$ separates $g_1$ and $h_2$, or equivalently, $h_2$ separates $h_1$ and $g_2$, see Figure \ref{fig:notations2}. For any $T\in P(g_1,g_2)$, let $h_{T,-}$ and $h_{T,+}$ denote the two edges of $T$ that separate $g_1$ and $g_2$, such that $h_{T,-}<h_{T,+}$ in the total order on $Q(g_1,g_2)$. Then we may define a total order $<$ on $P(g_1,g_2)$ by writing $T<T'$ if $h_{T,+}\le h_{T',-}$. Observe that $Q(g_1,g_2)=Q(g_2,g_1)$ and $P(g_1,g_2)=P(g_2,g_1)$ as sets, but they have opposite total orders. 

\begin{figure}[h!]
	\includegraphics[width=10cm]{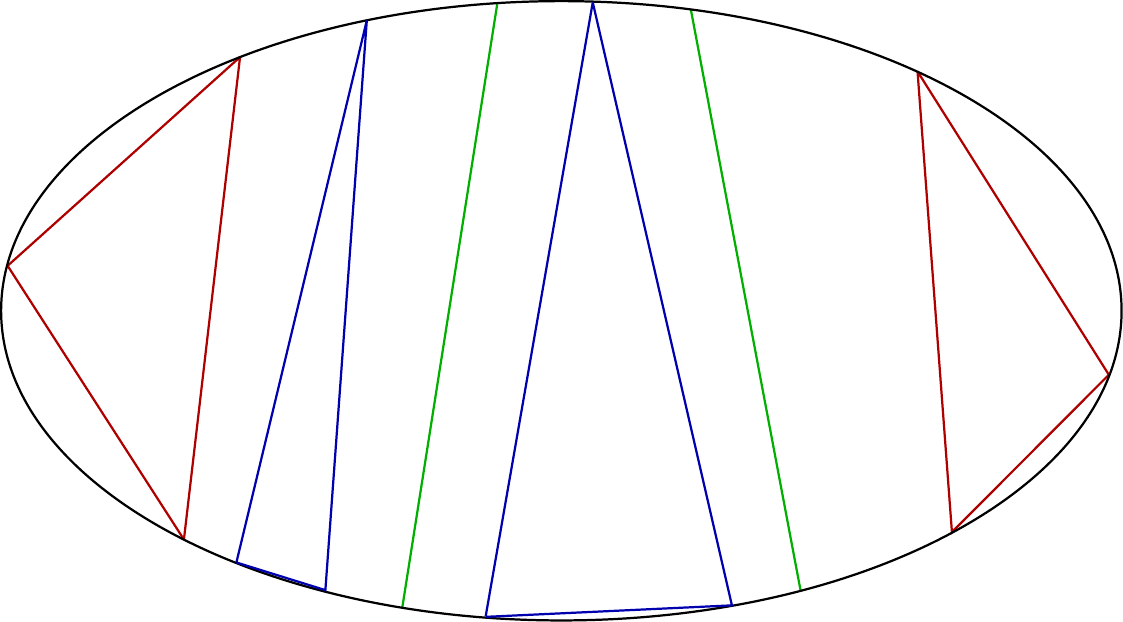}
	\put(-241,70){\small $g_1$}
	\put(-210,50){\small $T$}
	\put(-135,54){\small $T'$}
	\put(-260,75){\small $T_1$}
	\put(-171,70){\small $h_1$}
	\put(-195,93){\small $h_{T,+}$}
	\put(-155,33){\small $h_{T',-}$}
	\put(-103,70){\small $h_2$}
	\put(-57,80){\small $ g_2$}
	\put(-32,73){\small $ T_2$}
	\caption{\small The pair $(g_1,g_2)$ is a separating pair of edges for $(T_1,T_2)$. The geodesics $h_1$ and $h_2$ separate $T_1$ and $T_2$. If $g_1,g_2,h_1$, and $h_2$ are leaves of $\wt\lambda$, then $h_1,h_2\in Q(g_1,g_2)$ and $h_1\leq h_2$. The plaques $T$ and $T'$ in $P(g_1,g_2)$ are such that $T<T'$.}
	\label{fig:notations2}
\end{figure}
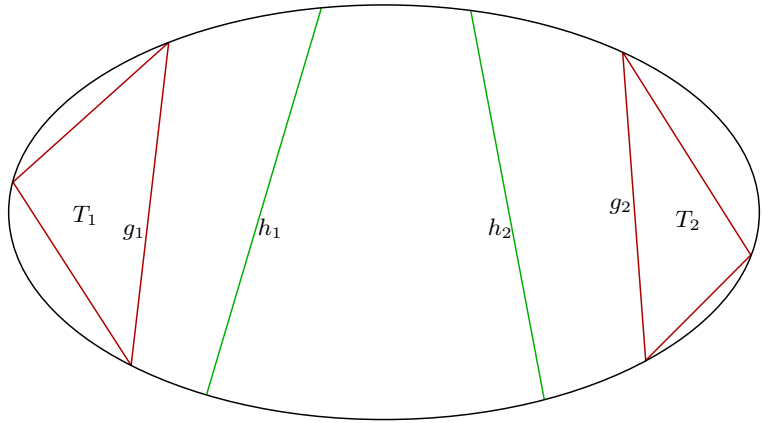

The following is a basic observation about the plaques of $\wt\lambda$ that we will use in Section \ref{sec:bending}. 

\begin{proposition}\label{prop: 3 leaves}
	Let $g_1$, $g_2$, and $g_3$ be pairwise distinct leaves of $\wt\lambda$ such that no leaf in $\{g_1,g_2,g_3\}$ separates the other two. Then there is a unique plaque that separates every distinct pair of leaves in $\{g_1,g_2,g_3\}$.
\end{proposition}

\begin{proof} For all $i=1,2,3$, let $h_i$ be the supremum of $Q(g_i,g_{i+1})\cap Q(g_i,g_{i-1})$ in the total order on $Q(g_i,g_{i+1})$ (or equivalently, $Q(g_i,g_{i-1})$). Since $Q(g_i,g_{i+1})$ and $Q(g_{i-1},g_i)$ are both compact subsets of $\wt\Lambda$ that contain $g_i$, $h_i$ exists and lies in $Q(g_i,g_{i+1})\cap Q(g_i,g_{i-1})$. It follows that 
	\begin{equation}\label{fdjkh}Q(g_i,h_i)=Q(g_i,g_{i+1})\cap Q(g_i,g_{i-1}).\end{equation} 
	Also, since $h_i$ separates $g_i$ from both $g_{i-1}$ and $g_{i+1}$, by assumption, $h_i\neq g_{i-1}$ and $h_i\neq g_{i+1}$. Thus, one of the two connected components of $\wt S-h_i$, call it $B_i$, contains $g_{i-1}$ and $g_{i+1}$ but not $g_i$.
	
	We first observe that for all $i=1,2,3$, $Q(g_i,h_i)$ and $Q(g_{i+1},h_{i+1})$ are disjoint. If this were not the case, then $h_{i+1}\le h_i$ in the total order on $Q(g_i,g_{i+1})$. Since $B_{i+1}$ contains $g_i$, and $B_i$ contains $g_{i+1}$, this implies that $B_i\cap B_{i+1}$ is empty. However, this is impossible because $g_{i-1}$ lies in both $B_i$ and $B_{i+1}$.
	
	Next, we show that for all $i=1,2,3$, if $h$ is a leaf of $\widetilde\lambda$ that separates $h_i$ and $h_{i+1}$, then either $h=h_i$ or $h=h_{i+1}$. If this were not the case, then there is some leaf $h$ of $\widetilde\lambda$ that separates $h_i$ and $h_{i+1}$, but is neither $h_i$ nor $h_{i+1}$. Since $Q(g_i,h_i)$ and $Q(g_{i+1},h_{i+1})$ are disjoint, one of the two connected components of $\wt S-h$ contains $Q(g_i,h_i)$, while the other contains $Q(g_{i+1},h_{i+1})$.  In particular, $h\notin Q(g_i,h_i)$ and $h\notin Q(g_{i+1},h_{i+1})$. By assumption, $g_{i-1}$ does not separate $g_i$ and $g_{i+1}$, so $g_{i-1}\neq h$. Thus, by switching $i$ and $i+1$ if necessary, we may assume that $g_{i-1}$ lies in the component of $\wt S-h$ that contains $g_{i+1}$, in which case $h$ separates $g_i$ and $g_{i-1}$. It follows that $h\in Q(g_i,g_{i+1})\cap Q(g_i,g_{i-1})$, but $h\notin Q(g_i,h_i)$, which contradicts \eqref{fdjkh}.
	
	Using this, we will now show that $h_1$, $h_2$, and $h_3$ are the edges of a plaque $T$ of $\widetilde\lambda$, i.e. that for all $i=1,2,3$, $h_i$ and $h_{i+1}$ share a common endpoint, but $h_1$, $h_2$, and $h_3$ do not share a common endpoint. For the former, notice that if $h_i$ and $h_{i+1}$ do not share a common endpoint, then the maximality of $\widetilde\lambda$ implies that there is a leaf $h$ of $\widetilde\lambda$ that separates $h_i$ and $h_{i+1}$, but is neither $h_i$ nor $h_{i+1}$. That contradicts the claim proven in the previous paragraph. For the latter, notice that if $h_1$, $h_2$, and $h_3$ share a common endpoint, then there is some $i\in\{1,2,3\}$ such that $h_i$ separates $h_{i-1}$ and $h_{i+1}$. Then by switching $i+1$ and $i-1$ if necessary, we may assume from the claim proven in the previous paragraph that $h_i=h_{i+1}$. However this means that either $B_i=B_{i+1}$ or $B_i\cap B_{i+1}$ is empty. Neither of these are possible because $g_i\subset B_{i+1}\setminus B_i$ and $g_{i-1}\subset B_i\cap B_{i+1}$. 
	
	It is straightforward to verify that $T$ is the unique plaque of $\widetilde\lambda$ that separates every distinct pair of leaves in $\{g_1,g_2,g_3\}$. 
	\end{proof}
	
	\subsubsection{Coherent orientations and the metric on $Q(g_1,g_2)$.}
	We will now specify, for any pair of leaves $g_1$ and $g_2$ of $\wt\lambda$, a metric $d_\infty$ on the set $Q(g_1,g_2)$ of leaves of $\wt\lambda$ that separate $g_1$ and $g_2$. The induced product metric on $Q(g_1,g_2)^2$, also denoted $d_\infty$, will be used in Sections \ref{sec:cocycle_pair}, \ref{sec:productable}, \ref{sec: slithering}, and \ref{sec:bending} to make sense of H\"older continuity of maps from $Q(g_1,g_2)^2$ to a normed vector space. 
	
	Let $\preceq$ denote the counter-clockwise cyclic order on $\partial\wt\lambda$ (with respect to the orientation on $S$).  If $\mathsf{g}_1$ and $\mathsf{g}_2$ are oriented geodesics in $\widetilde S$ that either do not intersect or are equal, then we say that $\mathsf{g}_1$ and $\mathsf{g}_2$ are \emph{oriented in parallel} if
	\[\mathsf{g}_2^+ \preceq \mathsf{g}_1^+ \prec \mathsf{g}_1^-\preceq \mathsf{g}_2^- \prec \mathsf{g}_2^+\quad\text{ or }\quad \mathsf{g}_1^+ \preceq \mathsf{g}_2^+ \prec \mathsf{g}_2^-\preceq \mathsf{g}_1^- \prec \mathsf{g}_1^+\]
	(recall that for any oriented geodesic $\mathsf{g}$ in $\widetilde S$, $\mathsf{g}^+$ and $\mathsf{g}^-$ respectively denote the forward and backward endpoints of $\mathsf{g}$, which lie in $\partial\wt S$).
	
	Fix a base point in $\widetilde{S}$, and equip $\partial\widetilde{S}$ with the visual metric $d_\infty$ induced by the base point. We abuse notation and also denote the induced product metric on $(\partial\widetilde{S})^2$ by $d_\infty$, i.e.
	\[d_\infty((x,y),(z,w))=\max\{d_\infty(x,z),d_\infty(y,w)\}\]
	for all $x,y,z,w\in\partial\wt S$.
	
	Now, suppose that $g_1$ and $g_2$ are leaves of $\widetilde\lambda$. Define a metric $d_\infty$ on $Q(g_1,g_2)$ by setting
	\[d_\infty(h_1,h_2):=d_\infty(\mathsf{h}_1,\mathsf{h}_2)=d_\infty((\mathsf{h}_1^+,\mathsf{h}_1^-),(\mathsf{h}_2^+,\mathsf{h}_2^-)),\]
	where $\mathsf{h}_1$ and $\mathsf{h}_2$ are some (any) pair of leaves in $\widetilde\Lambda^o$ oriented in parallel obtained by equipping the leaves $h_1$ and $h_2$ with orientations.
	
	\subsection{General position of flags}\label{flag}\label{sec: transversality}
	Let $\mathbb{K}$ be either $\mathbb{R}$ or $\mathbb{C}$, let $\mathcal{F}(\mathbb{K}^d)$ denote the space of complete flags in $\mathbb{K}^d$, and let $\Gr_k(\mathbb{K}^d)$ denote the Grassmannian of $k$--planes in $\mathbb{K}^d$. Fix once and for all a Riemannian metric on $\mathcal{F}(\mathbb{K}^d)$, and let $d_{\mathcal{F}}$ be the induced metric on $\mathcal{F}(\mathbb{K}^d)$. 
	
	For every $j\in\{1,\dots,d-1\}$, there is a natural projection
	\[\pi_j:\mathcal{F}(\mathbb{K}^d)\to\Gr_j(\mathbb{K}^d).\]
	For any flag $F\in\mathcal{F}(\mathbb{K}^d)$ and $j\in\{1,\dots,d-1\}$, we write $F^j:=\pi_j(F)$.
	Given a map $\xi:\partial\wt\lambda\to \mathcal{F}(\mathbb{K}^d)$, we denote
	\[\xi^j:=\pi_j\circ\xi:\partial\wt\lambda\to \Gr_j(\mathbb{K}^d).\]
	
	A pair of flags $F_1,F_2\in\mathcal{F}(\mathbb{K}^d)$ are \emph{transverse} if $F_1^j+F_2^{d-j}=\mathbb{K}^d$ for all $j\in\{1,\dots,d-1\}$. For any integer $m\geq 2$, we say that a collection of $m$ flags $F_1,\dots,F_m\in\mathcal F(\mathbb{K}^d)$ is in \emph{general position} if 
	\[F_1^{k_1}+\dots+F_m^{k_m}=\mathbb{K}^d\]
	for any non-negative integers $k_1,\dots,k_m$ such that $k_1 +\cdots + k_m = d$. Denote the set of $m$--tuples of flags in general position by $\mathcal F(\mathbb{K}^d)^{(m)}$. Note that an $m$--tuple of flags that is in general position is necessarily pairwise transverse. However, when $m>2$, an $m$--tuple of flags which is pairwise transverse is not necessarily in general position.
	
	We may extend these notions to flag maps.
	
	\begin{definition}\label{def: lambda transverse}
		A map $\xi:\partial\wt\lambda\to\mathcal{F}(\mathbb{K}^d)$ is 
		\begin{enumerate}
			\item \emph{$\lambda$--transverse} if $\xi(\mathsf{g})$ is transverse for all $\mathsf{g}\in\wt\Lambda^o$. 
			\item \emph{$\lambda$--hyperconvex} if $\xi({\bf x})$ is in general position for all ${\bf x} \in \wt\Delta^o$.
		\end{enumerate}
	\end{definition}
	
	\begin{remark}
		A map $\xi:\partial\wt\lambda\to\mathcal{F}(\mathbb{K}^d)$ is said to be \emph{transverse} if $\xi(x)$ and $\xi(y)$ are transverse for all distinct $x,y\in\partial\wt\lambda$. Notice that requiring $\xi$ to be $\lambda$--transverse is much weaker than requiring $\xi$ to be transverse. 
	\end{remark}

\section{\texorpdfstring{$d$}{d}--pleated surfaces}\label{sec:d_pletaed}

In this section, we introduce $d$--pleated surfaces, which are the main objects studied in this paper. They simultaneously generalize the classical notion of pleated surfaces, seen as representations from the fundamental group $\Gamma$ of the surface $S$ into $\mathsf{PGL}_2(\mathbb{C})$, and Hitchin representations of $\Gamma$ into $\mathsf{PGL}_d(\mathbb{R})$ (see Sections \ref{ssec:classical_pleated} and \ref{ssec:Hitchin} for details). Then, in Section \ref{subsec:length_functions}, we introduce length functions which will play an important role in the parameterization of the space of $d$--pleated surfaces.

\subsection{\texorpdfstring{$\lambda$}{l}--Borel Anosov representations}\label{ssec:BorelAnosov}
In order to define $d$--pleated surfaces, we first introduce the notion of a $\lambda$--Borel Anosov representation $\rho\colon\Gamma\to\mathsf{PGL}_d(\mathbb{K})$, where $\mathbb{K}$ is either $\mathbb{R}$ or $\mathbb{C}$. This is a weakening of the notion of Borel Anosov representation where we require the existence of an associated family of contracting line bundles over the unit tangent bundle to the lamination $\lambda$, as we now explain.

Let $T^1S$ denote the unit tangent bundle of $S$ (equipped with its auxiliary metric) and consider
\[
T^1\lambda:=\{v\in T^1S \mid v\text{ is tangent to a leaf of }\lambda\}
\] 
and its lift $T^1\wt \lambda\subset T^1\wt S$. Observe that the geodesic flows on $T^1\wt S$ and $T^1S$ restrict to flows on $T^1\wt\lambda$ and $T^1\lambda$ respectively. We will denote all these flows by $\varphi_t$, since the space they act on will be clear from context. Any $v\in T^1\wt\lambda$ is tangent to an oriented geodesic $\mathsf{g}$ and we let $v^\pm:=\mathsf{g}^\pm$ denote its endpoints. 

Now, let $\rho\colon\Gamma\to\mathsf{PGL}_d(\mathbb{K})$ be a representation, let $\xi\colon\partial\wt \lambda\to\mathcal F(\mathbb{K}^d)$ be a $\rho$--equivariant map, and recall that $\wt\Lambda^o$ denotes the set of oriented leaves of $\wt \lambda$.  With a slight abuse of notation, we let $\xi\colon\wt\Lambda^o\to\mathcal F(\mathbb{K}^d)^2$ denote the associated map on oriented leaves defined by $\xi(\mathsf{g})=(\xi(\mathsf{g}^+),\xi(\mathsf{g}^-))$. The $\rho$--equivariant maps that arise naturally in this work are in general not continuous, but their associated maps on oriented leaves satisfy certain regularity and transversality properties. In particular, we say that $\xi$ is
\begin{itemize}
	\item {\em $\lambda$--continuous} if its associated map on oriented leaves is continuous;
	\item {\em $\lambda$--transverse} if $\xi(\mathsf{g})$ is a transverse pair of flags for all oriented leaves $\mathsf{g}\in\wt\Lambda^o$.
\end{itemize}
Note that every $\rho$--equivariant, $\lambda$--continuous, and $\lambda$--transverse map 
\[\xi: \partial \wt \lambda \to \mathcal{F}(\mathbb{K}^d)\]
induces a continuous splitting of the product bundle $T^1\wt \lambda \times \mathbb{K}^d \to T^1\wt \lambda$ as 
\begin{align*}
	T^1\wt\lambda\times\mathbb{K}^d=L_1\oplus\dots\oplus L_d,
\end{align*} 
where $L_i|_v=\xi^i(v^+)\cap \xi^{d-i+1}(v^-)$ for all $v\in T^1\wt\lambda$ and all $i=1,\dots,d$. We refer to this as the \emph{splitting induced by $\xi$}. This in turn induces a splitting of the product bundle $T^1\wt\lambda\times\End(\mathbb{K}^d)\to T^1\wt\lambda$ as
\begin{align}\label{eqn: splitting'}
	T^1\wt\lambda\times\End(\mathbb{K}^d)=\bigoplus_{1\le i,j\le d}\Hom(L_i,L_j),
\end{align} 
which we also refer to as the \emph{splitting induced by $\xi$}. By the $\rho$--equivariance of $\xi$, the natural $\Gamma$--action on $T^1\wt\lambda\times\End(\mathbb{K}^d)$ (by deck transformations in the first factor and conjugation by $\rho(\Gamma)$ in the second factor) leaves the sub-bundle $\Hom(L_i,L_j)$ invariant for all $1\le i,j\le d$. In particular, for all pairs of positive integers ${\bf i}=(i_1,i_2)$ that sum to $d$, the line bundle
\[H_{\bf i}:=\Hom(L_{i_1+1},L_{i_1})\to T^1\wt\lambda\] 
descends to a line bundle $\widehat H_{\bf i}:=H_{\bf i}/\Gamma\to T^1\lambda$. We refer to $\wh H_{\bf i}$ as the \emph{${\bf i}$--th homomorphism bundle induced by $\xi$}.

The flow on $T^1\wt\lambda$ lifts trivially to $T^1\wt\lambda\times\End(\mathbb{K}^d)$. Since the splitting induced by $\xi$ is flow--invariant, we obtain a flow on $H_{\bf i}$ which is a linear map when restricted to each fiber and also commutes with the $\Gamma$--action on $H_{\bf i}$. As such, it descends to a flow on $\wh{H}_{\bf i}$ that covers the flow on $T^1\lambda$, and is linear restricted to each fiber of $\wh{H}_{\bf i}$. We will not emphasize the distinction between the flows introduced here and refer to all of them as $\varphi_t$, since the space they act on will be clear from context.

\begin{definition}\label{def:AnosovBorel}
	A representation $\rho\colon\Gamma\to\mathsf{PGL}_d(\mathbb{K})$ is \emph{$\lambda$--Borel Anosov} if the following holds:
	\begin{enumerate}
		\item There exists a $\rho$--equivariant, $\lambda$--continuous, $\lambda$--transverse map
		\[
		\xi\colon\partial\wt \lambda\to \mathcal{F}(\mathbb{K}^d).
		\]    
		\item For all pairs of positive integers ${\bf i}$ that sum to $d$ and for some (any) continuous family of norms $\norm{\cdot}_{v\in T^1\lambda}$ on the fibers of the ${\bf i}$--th homomorphism bundle $\wh H_{\bf i}$ induced by $\xi$, there exist $C,\alpha>0$ such that 
		\[
		\|\varphi_t(X)\|_{\varphi_t(v)}\leq Ce^{-\alpha t}\|X\|_v
		\]
		for all $X\in\wh H_{\bf i}\big|_v$, all $v\in T^1\lambda$, and all $t>0$.
	\end{enumerate}
	We refer to $\xi$ as a \emph{$\lambda$--limit map} of $\rho$.
\end{definition}
\begin{remark} \label{rem: d pleated} \phantom{a}
	\begin{enumerate}
		\item The contracting property (Definition \ref{def:AnosovBorel} Property (2)) guarantees that the $\lambda$--limit map of a $\lambda$--Borel Anosov representation is unique. See \cite[Theorem 1.1]{wang2021anosov} for details.
		\item The notion of a $\lambda$--Borel Anosov representation is a weakening of the more classical notion of a Borel Anosov representation that was introduced by Labourie \cite{labourie-anosov}. For Borel Anosov representations, the limit map is required to be defined on all of $\partial\wt S$ instead of just $\partial\wt\lambda$, and is also required to be continuous and transverse, i.e. $\xi(x)$ and $\xi(y)$ are transverse for all distinct points $x,y\in\partial\wt S$. This stronger limit map extends the splitting $T^1\wt\lambda\times\mathbb{K}^d=L_1\oplus\dots\oplus L_d$ to a splitting $T^1\wt S\times\mathbb{K}^d=L_1\oplus\dots\oplus L_d$, and so the homomorphism bundles $\wh H_{\bf i}$ can also be extended to bundles over $T^1S$. The contraction property is then required to hold for all vectors $v\in T^1S$. 
		\item Unlike Borel Anosov representations, $\lambda$--Borel Anosov representations do not necessarily have discrete image. 
		\item The notion of $\lambda$--Borel Anosov representation is a special case of the class of representations considered in Wang \cite{wang2021anosov}: in his language, $\lambda$--Borel Anosov representations are ``$k$--Anosov over $T^1\wt\lambda$ for all $k=1,\ldots,d$".
	\end{enumerate}
\end{remark}

	\subsection{\texorpdfstring{$d$}{d}--pleated surfaces: Definition and motivating examples}\label{ssec:d-pleated}
	
	We are now ready to define $d$--pleated surfaces.
	
	\begin{definition}\label{defn:d-pleated}
		A \emph{$d$--pleated surface with pleating locus $\lambda$} is a pair $(\rho,\xi)$, where $\rho:\Gamma\to\mathsf{PGL}_d(\mathbb{C})$ is a representation and $\xi:\partial\wt\lambda\to\mathcal{F}(\mathbb{C}^d)$ is a $\rho$--equivariant map such that 
		\begin{enumerate}
			\item $\rho$ is $\lambda$--Borel Anosov with $\lambda$--limit map $\xi$, and
			\item $\xi$ is {\em $\lambda$--hyperconvex}, that is, for any labeling of the vertices of plaques $\mathbf{x} = (x_1,x_2,x_3) \in\wt\Delta^o$, the triple of flags
			\[
			\xi(\mathbf x):= (\xi(x_1), \xi(x_2), \xi(x_3))
			\]
			is in general position.
		\end{enumerate}
		We denote the set of $d$--pleated surfaces with pleating locus $\lambda$ by $\mathcal{R}_d(\lambda)$.
	\end{definition}
	
	Given Remark \ref{rem: d pleated} item (1), we see that a $d$--pleated surface $(\rho,\xi)$ with pleating locus $\lambda$ is uniquely determined by the representation $\rho\colon\Gamma\to\mathsf{PGL}_d(\mathbb{C})$ (once $\lambda$ is fixed). Therefore, we have an embedding 
	\[\mathcal{R}_d(\lambda)\to\Hom(\Gamma,\mathsf{PGL}_d(\mathbb{C}))\]
	given by $(\rho,\xi)\mapsto\rho$. As such, we often view $\mathcal{R}_d(\lambda)$ as a subspace of $\Hom(\Gamma,\mathsf{PGL}_d(\mathbb{C}))$ with the induced topology.
	
	The two main motivating examples for our study of $d$--pleated surfaces are (abstract) pleated surfaces and Hitchin representation. We discuss these examples in Sections \ref{ssec:classical_pleated} and \ref{ssec:Hitchin} respectively.
	
	\subsubsection{Classical pleated surfaces}\label{ssec:classical_pleated}
	
	An abstract pleated surface with pleating locus $\lambda$ is the data of a hyperbolic metric $m$ on $S$ (which might not agree with the existing hyperbolic metric on $S$) and a pair $(\rho,f)$ where $\rho$ is a homomorphism from $\Gamma$ to $\mathsf{PGL}_2(\mathbb{C})$ and $f\colon\wt S\to\mathbb H^3$ is a $\rho$--equivariant path-isometry, with respect to the lift of $m$ to the universal cover, that sends leaves and plaques of $\wt\lambda\subset\wt S$ into geodesics and ideal triangles of $\mathbb H^3$. This notion was introduced by Thurston \cite{thurston-notes} in the context of hyperbolic 3--manifolds and further developed by Bonahon \cite{bonahon-toulouse}. 
	In this paper, we provide a new dynamical characterization of these classical objects: $(\rho,f)$ is a pleated surface with pleating locus $\lambda$ if and only if there is a $\rho$--equivariant map $\xi:\partial\wt\lambda\to\mathcal{F}(\mathbb{C}^d)$ such that $(\rho,\xi)$ is a $2$--pleated surface with pleating locus $\lambda$ in the sense of Definition \ref{defn:d-pleated}. In Proposition \ref{classical_pleated} we show that every classical pleated surface is 2--pleated and later, in Proposition \ref{prop:2-pleated is pleated}, we prove the converse of this statement.
	
	Observe that $(\rho,\xi)$ is a $2$--pleated surface with pleating locus $\lambda$ if and only if $\rho:\Gamma\to\mathsf{PGL}_2(\mathbb{C})$ is $\lambda$--Borel Anosov with $\lambda$--limit map $\xi:\partial\wt S\to\mathcal{F}(\mathbb{C}^2)$, since $\xi$ is vacuously $\lambda$--hyperconvex when $d=2$. Recall that $\mathsf{PGL}_2(\mathbb{C})$ is the group of orientation preserving isometries of $\mathbb{H}^3$, and so it naturally acts on its unit tangent bundle $T^1\mathbb{H}^3$ and its ideal boundary $\partial\mathbb{H}^3$. Fix a $\mathsf{PGL}_2(\mathbb{C})$--equivariant identification between $\partial\mathbb{H}^3$ and $\mathbb{P}(\mathbb{C}^2)=\mathcal{F}(\mathbb{C}^2)$. 	
	
	\begin{proposition}\label{classical_pleated}
		If $(\rho,f)$ is a pleated surface with pleating locus $\lambda$, then $\rho$ is $\lambda$--Borel Anosov with $\lambda$--limit map $\xi:\partial\wt\lambda\to\mathcal{F}(\mathbb{C}^2)\cong\partial\mathbb{H}^3$. In particular, $(\rho,\xi)$ is a $2$--pleated surface with pleating locus $\lambda$. Furthermore, $\xi$ has the property that for any oriented leaf $\mathsf{g}$ of $\wt\lambda$, $\xi(\mathsf{g}^+)$ and $\xi(\mathsf{g}^-)$ are the forward and backward endpoints of the geodesic $f(\mathsf{g})$ in $\mathbb{H}^3$.
	\end{proposition}	
	
	In order to prove Proposition \ref{classical_pleated}, we will need to record a preliminary result on the dynamics of the geodesic flow of $\mathbb H^3$.
	
	Observe that the trivial complex vector bundle $T^1\mathbb{H}^3\times\End(\mathbb{C}^2) \to T^1\mathbb{H}^3$ admits a natural $\mathsf{PGL}_2(\mathbb{C})$--action (by isometries in the first factor and conjugation in the second factor). Also, the geodesic flow $\psi_t$ on $T^1\mathbb{H}^3$ lifts trivially to a flow, also denoted $\psi_t$, on $T^1\mathbb{H}^3\times\End(\mathbb{C}^2)$, which commutes with the $\mathsf{PGL}_2(\mathbb{C})$--action. Furthermore, this trivial bundle admits a canonical complex line sub-bundle $K$ whose fiber above every $u\in T^1\mathbb{H}^3$ is given by
	\[
	K_{u}:= \Hom(L_{u}^-, L_{u}^+),
	\]
	where $L_u^+$ and $L_u^-$ are the complex lines in $\mathbb{C}^2$ associated to the forward and backward endpoints at infinity of the oriented geodesic in $\mathbb{H}^3$ that $u$ is tangent to. Observe that $K$ is $\mathsf{PGL}_2(\mathbb{C})$--invariant and $\psi_t$--invariant. The following lemma gives the contraction rate of the flow $\psi_t$ on the fibers of $K$.
	
	\begin{lemma}\label{lem:contracting flow classical}\
		\begin{enumerate}
			\item There exists a norm $\norm{\cdot}_{u \in T^1 \mathbb{H}^3}$ on the line bundle $K$ that is $\mathsf{PGL}_2(\mathbb{C})$--invariant, i.e. 
			\[
			\norm{a \cdot X}_{a(u)} = \norm{X}_u
			\]
			for every $a \in \mathsf{PGL}_2(\mathbb{C})$, $u \in T^1 \mathbb{H}^3$, and $X \in K_u$.
			\item Every $\mathsf{PGL}_2(\mathbb{C})$--invariant norm $\norm{\cdot}_{u \in T^1 \mathbb{H}^3}$ on the line bundle $K$ satisfies
			\[
			\norm{\psi_t(X)}_{\psi_t(u)} = e^{-t} \norm{X}_u
			\]
			for every $t \in \mathbb{R}, u \in T^1 \mathbb{H}^3$, and $X \in K_u$.
		\end{enumerate}
	\end{lemma}
	
	\begin{proof}
		Proof of (1). Fix a base point $u_0\in T^1\mathbb{H}^3$ and let $(b_1,b_2)$ be a basis of $\mathbb{C}^2$ that satisfies $b_1 \in L_{u_0}^+$ and $b_2 \in L_{u_0}^-$. Observe that in the basis $(b_1,b_2)$, every element in $\mathsf{PGL}_2(\mathbb{C})$ that fixes $u_0$ is represented by a matrix of the form
		\[
		\varepsilon_{i\vartheta}=
		\begin{bmatrix}
			e^{i\vartheta/2} & 0 \\
			0 & e^{-i\vartheta/2} 
		\end{bmatrix}, \quad\vartheta\in\mathbb{R}.
		\]
		An elementary computation shows that for all $X\in K_{u_0}$	
		\[
		\varepsilon_{i\vartheta} \cdot X = \varepsilon_{i\vartheta} \circ X \circ \varepsilon_{i\vartheta}^{-1} = e^{i\vartheta} X. 
		\]
		
		Choose now some Hermitian norm $\norm{\cdot}_{u_0}$ on the fiber $K_{u_0}$. By the identity above, the stabilizer of $u_0$ preserves the norm $\norm{\cdot}_{u_0}$. In particular, we can define a norm $\norm{\cdot}_u$ on the fiber $K_{u}$ over any point $u \in T^1 \mathbb{H}^3$ by setting
		\[
		\norm{X}_u:= \norm{a \cdot X}_{u_0}
		\]
		for some (any) $a \in \mathsf{PGL}_2(\mathbb{C})$ such that $a(u) = u_0$. By construction, the norm $\norm{\cdot}_{u \in T^1 \mathbb{H}^3}$ is $\mathsf{PGL}_2(\mathbb{C})$--invariant.
		
		Proof of (2). Let $u \in T^1 \mathbb{H}^3$ and let $(b_1,b_2)$ be a basis of $\mathbb{C}^2$ that satisfies $b_1 \in L_{u}^+$ and $b_2 \in L_{u}^-$. Observe that for every $t \in \mathbb{R}$ the element $\varepsilon_t \in \mathsf{PGL}_2(\mathbb{C})$, represented in the basis $(b_1,b_2)$ by
		\[
		\varepsilon_t =
		\begin{bmatrix}
			e^{t/2} & 0 \\
			0 & e^{-t/2} 
		\end{bmatrix}, 
		\]
		satisfies $\varepsilon_t(u) = \psi_t(u)$. Furthermore, we have
		\[
		\varepsilon_t^{-1} \cdot X = \varepsilon_t^{-1} \circ X \circ \varepsilon_t = e^{-t} X  \in K_{\psi_t(u)}
		\]
		for every $t \in \mathbb{R}$ and $X \in K_{u}$. Since the norm $\norm{\cdot}_{u \in T^1 \mathbb{H}^3}$ is $\mathsf{PGL}_2(\mathbb{C})$--invariant and $\psi_t$ acts trivially on the fibers, it follows that
		\[
		\norm{\psi_t(X)}_{\psi_t(u)}=\norm{X}_{\psi_t(u)}=\norm{\varepsilon_t^{-1}\cdot X}_{u}=e^{-t}\norm{X}_{u}.\qedhere
		\]
	\end{proof}
	
	We are now ready to prove Proposition \ref{classical_pleated}.
	\begin{proof}[Proof of Proposition \ref{classical_pleated}]
		Recall that $S$ is a hyperbolic surface, so we may lift the hyperbolic metric to $S$ a $\Gamma$--invariant hyperbolic metric $m_0$ on $\wt S$. On the other hand, by the definition of the pleated surface $(\rho,f)$, there is a $\Gamma$--invariant hyperbolic metric $m$ on $\wt S$ such that $\wt f:\wt S\to\mathbb{H}^3$ is a path-isometry that sends each leaf of $\wt\lambda$ to a complete geodesic in $\mathbb{H}^3$, and is a totally geodesic immersion on $\wt S-\wt\lambda$. While $m$ and $m_0$ might not agree, there is a geodesic lamination $\lambda_m$ on $S$ with respect to $m$, such that $\partial\wt\lambda_m=\partial\wt\lambda$. Then by classical results on the dynamics of the geodesic flow of compact negatively curved Riemannian manifolds (see e.g. \cite{gromov3remarks,matheus91}), $T^1\lambda_m$ and $T^1\lambda$ are H\"older orbit equivalent. Thus, for the purposes of proving that $\rho$ is $\lambda$--Borel Anosov, we may assume without loss of generality that $m_0=m$.
		
		Since $f$ is $1$--Lipschitz and sends leaves of $\widetilde{\lambda}$ to geodesics of $\mathbb{H}^3$, we may define a map 
		\[\xi \co\partial \widetilde{\lambda} \to \partial \mathbb{H}^3\cong\mathcal{F}(\mathbb{C}^2)\] 
		by setting $\xi(z)$ to be the forward endpoint of $f(\mathsf{g})$ for some (any) oriented leaf $\mathsf{g}$ of $\wt\lambda$ whose forward endpoint is $z$. Since geodesics in $\mathbb{H}^3$ have distinct endpoints in $\partial\mathbb{H}^3$, and points in $\mathcal{F}(\mathbb{C}^2)$ are transverse if and only if they are distinct, it follows that $\xi$ is $\lambda$--transverse. 
		
		It now suffices to show that $\rho$ is $\lambda$--Borel Anosov with $\lambda$--limit map $\xi$. Observe that $f$ determines a $\rho$--equivariant continuous map
		\[
		df : T^1 \widetilde{\lambda} \longmapsto T^1 \mathbb{H}^3 ,
		\]
		that sends every $v \in T^1 \widetilde{\lambda}$ into $\frac{\mathrm{d}}{\mathrm{d}t}(f \circ \gamma_v)(t) |_{t = 0}$, where $\gamma_v = \gamma_v(t)$ denotes the parameterized geodesic in $\widetilde{S}$ with initial conditions $v$. Notice that, 
		\begin{equation} \label{eq:pleated and flows}
			df(\varphi_t(v)) = \psi_t(df(v)) 
		\end{equation}
		for every $t \in \mathbb{R}$ and $v \in T^1 \widetilde{\lambda}$.
		
		Consider the pullback bundle $(df)^* K \to T^1 \widetilde{\lambda}$. By definition of $\xi$, $df$, and $K$, this pullback bundle coincides with the $(1,1)$--th homomorphism bundle induced by $\xi$,
		\[
		H = H_{(1,1)} \longrightarrow T^1 \widetilde{\lambda} ,
		\]
		see Section \ref{ssec:BorelAnosov}. Since $f$ is $\rho$--equivariant, the $\Gamma$--action on $H$ coincides with the pullback via $df$ of the $\rho(\Gamma)$--action on $K$. Furthermore, by relation \eqref{eq:pleated and flows} the flow $\varphi_t$ on $H$ is the pullback under $df$ of the flow $\psi_t$ on $K$.
		
		Fix a $\mathsf{PGL}_2(\mathbb{C})$--invariant norm on $K$ (part (1) Lemma \ref{lem:contracting flow classical} ensures the existence of such a norm), and let $\norm{\cdot}_H$ denote its pullback to $H$ via $df$. The $\rho$--equivariance of $df$ ensures that this pullback norm is $\Gamma$--invariant. Then by part (2) of Lemma~\ref{lem:contracting flow classical},
		\[
		\norm{\varphi_t(X)}_{H,\varphi_t(v)} = \norm{\psi_t((df)_*(X))}_{\psi_t(df(v))} = e^{-t} \norm{(df)_*(X)}_{df(v)} = e^{-t} \norm{X}_{H,v}
		\]
		for all $v \in T^1 \widetilde{\lambda}$, $X \in H_v$, and $t \in \mathbb{R}$. This proves that $\rho$ is $\lambda$--Borel Anosov with $\lambda$--limit map $\xi$, as desired.
	\end{proof}
	
	\subsubsection{Hitchin representations}\label{ssec:Hitchin}
	
	Another important example of $d$--pleated surfaces are {\em $d$--Hitchin representations}, which we now define.
	
	Recall that a representation $j:\Gamma\to\mathsf{PGL}_2(\mathbb{R})$ is \emph{Fuchsian} if it is discrete and faithful. The set of conjugacy classes of Fuchsian representations form a single connected component of 
	\[\mathfrak X(\Gamma,\mathsf{PGL}_2(\mathbb{R})):=\Hom(\Gamma,\mathsf{PGL}_2(\mathbb{R}))/\mathsf{PGL}_2(\mathbb{R}),\] 
	called the \emph{Teichm\"uller component}. 
	
	For any $d\geq 2$, there is a natural action of $\mathsf{GL}_2(\mathbb{R})$ on the space of $(d-1)$--th symmetric tensors ${\rm Sym}^{d-1}(\mathbb{R}^2)$ defined by requiring
	\[g\cdot (v_1\odot\dots\odot v_{d-1}):=(g\cdot v_1)\odot (g\cdot v_2)\odot\dots\odot(g\cdot v_{d-1}) ,
	\]
	and extending it by linearity to all elements of ${\rm Sym}^{d-1}(\mathbb{R}^2)$. Since this action is linear, it determines a group homomorphism
	\[\mathsf{GL}_2(\mathbb{R})\to\mathsf{GL}({\rm Sym}^{d-1}(\mathbb{R}^2))\simeq\mathsf{GL}_d(\mathbb{R}),\]
	where the last isomorphism is obtained by identifying the basis of ${\rm Sym}^{d-1}(\mathbb{R}^2)$ induced by the standard basis of $\mathbb{R}^2$, with the standard basis of $\mathbb{R}^d$. This homomorphism descends to an embedding
	\begin{align}\label{eqn: irred}
		\iota:\mathsf{PGL}_2(\mathbb{R})\to\mathsf{PGL}_d(\mathbb{R}),
	\end{align}
	which we may use to define an inclusion
	\[\hat\iota:\mathfrak X(\Gamma,\mathsf{PGL}_2(\mathbb{R}))\to \mathfrak X(\Gamma,\mathsf{PGL}_d(\mathbb{R})):=\Hom(\Gamma,\mathsf{PGL}_d(\mathbb{R}))/\mathsf{PGL}_d(\mathbb{R})\]
	by $[\rho]\mapsto[\iota\circ\rho]$. The connected component of $\mathfrak X(\Gamma,\mathsf{PGL}_d(\mathbb{R}))$ that contains the image of the Teichm\"uller component under $\hat\iota$ is the \emph{$d$--Hitchin component}, denoted $\mathrm{Hit}_d(S)$. The representations from $\Gamma$ to $\mathsf{PGL}_d(\mathbb{R})$ whose conjugacy classes lie in the $d$--Hitchin component are called \emph{$d$--Hitchin representations}.
	
	Often, we will consider a $d$--Hitchin representation $\rho:\Gamma\to\mathsf{PGL}_d(\mathbb{R})$ as a representation into $\mathsf{PGL}_d(\mathbb{C})$ by the natural embedding $\mathsf{PGL}_d(\mathbb{R})\subset\mathsf{PGL}_d(\mathbb{C})$. For this reason, we say that a representation $\rho:\Gamma\to\mathsf{PGL}_d(\mathbb{C})$ is a \emph{$d$--Hitchin representation} if it is conjugate in $\mathsf{PGL}_d(\mathbb{C})$ to a $d$--Hitchin representation $\rho:\Gamma\to\mathsf{PGL}_d(\mathbb{R})\subset\mathsf{PGL}_d(\mathbb{C})$. The following theorem is a seminal result of Labourie \cite[Theorem 1.4, Theorem 4.2]{labourie-anosov}.
	
	\begin{theorem}[\cite{labourie-anosov}]\label{thm: Labourie}
		If $\rho:\Gamma\to\mathsf{PGL}_d(\mathbb{R})$ is a $d$--Hitchin representation then it is a Borel Anosov representation.
	\end{theorem}
	
	While the limit map $\xi\colon\mathbb{S}^1\to\mathcal{F}(\mathbb{R}^d)$ of a Borel Anosov representation $\rho$ is in general $\rho$--equivariant, continuous and transverse, the limit maps of Hitchin representations satisfy an additional property. We say that a map $\xi:\mathbb{S}^1\to\mathcal{F}(\mathbb{R}^d)$ is \emph{hyperconvex} if for any pairwise distinct points $x_1,\dots,x_n\in\mathbb{S}^1$ the flags $\xi(x_1),\dots,\xi(x_n)$ are in general position. 
	
	\begin{theorem}[{\cite{labourie-anosov, gui_com}}]\label{thm: hyperconvex}
		If $\rho:\Gamma\to\mathsf{PGL}_d(\mathbb{R})$ is a Borel Anosov representation, then $\rho$ is a $d$--Hitchin representation if and only if its limit map is hyperconvex.
	\end{theorem}
	
	Theorems \ref{thm: Labourie} and \ref{thm: hyperconvex} give the following Corollary.
	
	\begin{corollary}
		If $\rho$ is a $d$--Hitchin representation with limit map $\xi$, then for any maximal geodesic lamination $\lambda$, the pair $(\rho,\xi|_{\partial\wt \lambda})$ is a $d$--pleated surface with pleating locus $\lambda$.
	\end{corollary}
	
	As a consequence of Theorem \ref{thm: hyperconvex}, one can also deduce that every $d$--Hitchin representation is irreducible. See \cite[Proposition~14]{gui_com} or \cite[Proposition 3.4]{labourie-anosov} for a proof using Higgs bundles. It then follows from \cite[Section 1.3]{goldman-symplectic} and \cite{johnson-millson} that the Hitchin component is a real-analytic manifold.
	
	\subsubsection{Regularity of $d$--pleated surfaces in $\mathsf{Hom}(\Gamma,\mathsf{PGL}_d(\mathbb C))$}
	
	A result of Wang \cite{wang2021anosov} and the finiteness of the set of plaques of $\lambda$ imply that $d$--pleated surfaces are stable under small deformations in $\mathsf{Hom}(\Gamma,\mathsf{PGL}_d(\mathbb C))$.
	
	\begin{theorem}[{\cite[Theorem 1.1]{wang2021anosov}}] The space $\mathcal R_d(\lambda)$ of $d$--pleated surfaces with pleating locus $\lambda$ is open in $\mathsf{Hom}(\Gamma,\mathsf{PGL}_d(\mathbb C))$.
	\end{theorem}
	
	The next proposition implies that $\mathcal{R}_d(\lambda)$ lies in the smooth part of the complex algebraic variety $\Hom(\Gamma,\mathsf{PGL}_d(\mathbb{C}))$, and so it is naturally a complex manifold. 
		
		\begin{proposition}\label{prop: pleated surfaces smooth}
			If $(\rho,\xi)$ is a $d$--pleated surface with pleating locus $\lambda$, then the centralizer in $\mathsf{PGL}_d(\mathbb{C})$ of the image of $\rho$ is trivial. In particular, the space $\mathcal{R}_d(\lambda)$ of $d$--pleated surfaces with pleating locus $\lambda$ lies in the smooth part of $\Hom(\Gamma,\mathsf{PGL}_d(\mathbb{C})).$ 
		\end{proposition}
		
		\begin{proof}
			Let $A \in \mathsf{PGL}_d(\mathbb{C})$ centralize the image of $\rho$, i.e.  
			\[c_A\circ \rho = \rho,\] 
			where $c_A:\mathsf{PGL}_d(\mathbb{C})\to\mathsf{PGL}_d(\mathbb{C})$ is conjugation by $A$. Since $(\rho,\xi)$ is a $d$--pleated surface, so is $(c_A\circ\rho, A \circ \xi)$. At the same time, since $c_A\circ\rho= \rho$, the uniqueness of $\xi$ given by Remark \ref{rem: d pleated} item (1) implies that $A \circ \xi = \xi$. Select now any triple of points ${\bf x}=(x_1,x_2,x_3)$ in $\partial\wt\lambda$ that are vertices of a plaque of $\wt\lambda$. Since $\xi$ is $\lambda$--hyperconvex, the fact that $A$ fixes $\xi(x_1)$, $\xi(x_2)$ and the line $\xi^1(x_3)$ implies that $A$ is the identity.
			
			To deduce the second statement from the first, we apply a result of Goldman \cite[Proposition~1.2]{goldman-symplectic}, who proved that a representation $\rho\in \Hom(\Gamma,\mathsf{PGL}_d(\mathbb{C}))$ lies in the smooth locus if and only if the centralizer of $\rho(\Gamma)$ in $\mathsf{PGL}_d(\mathbb{C})$ has zero dimension.
		\end{proof}
		
		\subsection{Lengths for \texorpdfstring{$\lambda$}{l}--Borel Anosov representations} \label{subsec:length_functions}
		We now define for every $\lambda$--Borel Anosov representation $\rho:\Gamma\to\mathsf{PGL}_d(\mathbb{K})$ and every pair of positive integers that sum to $d$, a length function which will play an important role in our parameterization of the space $\mathcal R_d(\lambda)$. The domain of these length functions is the space of transverse measures supported on the orientation cover of the lamination $\lambda$, which we now define.
		
		The \emph{orientation cover} $\lambda^o$ of the geodesic lamination $\lambda$ is the set of pairs $(x, \ast)$, where $x \in \lambda$ and $\ast$ is an orientation of the leaf of $\lambda$ that contains $x$. 
		
		Note that there is a canonical bijection $\lambda^o\to T^1\lambda$ that assigns every pair $(x,\ast)\in\lambda^o$ to the unit tangent vector based at $x$ that is tangent to the leaf of $\lambda$ that passes through $x$ and is equipped with the orientation $\ast$. This bijection defines a natural topology on $\lambda^o$, and with this topology, the projection $\lambda^o \to \lambda$ given by $(x, \ast) \to x$ is a $2$--fold cover.
		
		A \emph{transverse measure with support in  $\lambda^o$} is a $\Gamma$--invariant Radon measure $\mu$ on the space of oriented geodesics in $\wt S$ whose support lies in $\wt\Lambda^o$, which is the set of oriented leaves of $\wt\lambda^o$. The space of transverse measures with support in $\lambda^o$, denoted by $\mathsf{M}(\lambda^o)$, is a simplex in a finite dimensional real vector space, by a theorem of Katok \cite{Katok_measures}. 
		
		Since the set of flow lines of the geodesic flow $\varphi_t$ on $T^1S$ is exactly the set of oriented geodesics in $S$, we may construct a Borel measure $\mathrm{d}t\times\mu$ on $T^1S$ with the following defining property: Given a small open set $D\subset T^1S$,
		\[(\mathrm{d}t\times\mu)(D)=\int_{\mathsf{g}\in\mathcal{G}_D}\left(\int_{\mathsf{g}\cap D}{\rm d}t\right) \,\mu,\]
		where $\mathcal{G}_D$ denotes the set of flow lines of $T^1S$ that intersect $D$, and $\mathrm{d}t$ is the $1$--form along the flow lines of $T^1 S$ given by the arc length. Observe that $\mathrm{d}t\times\mu$ is invariant under $\varphi_t$, and is supported on $T^1\lambda$.
		
		Let $\xi$ be the $\lambda$--limit map of $\rho$, let ${\bf i}=(i_1,i_2)$ be a pair of positive integers that sum to $d$, and let $\wh H_{\bf i}\to T^1\lambda$ be the ${\bf i}$--th homomorphism bundle induced by $\xi$. Using a smooth partition of unity over $T^1 S$, we select a continuous norm $\norm{\cdot}_{u \in T^1\lambda}$ on the fibers of the line bundle $\wh H_{\bf i}$, whose restriction along the flow lines of the geodesic flow $\mathbb{R} \to T^1 S$ is continuously differentiable, and define the function $f^\rho_{\bf i} : T^1\lambda \to \mathbb{R}$ by setting
		\[
		f^\rho_{\bf i}(v):= - \left. \frac{\mathrm{d}}{\mathrm{d} s} \log \norm{\varphi_s(X)}_{\varphi_s(v)} \right|_{s = 0}
		\]
		for some (any) non-zero $X \in \wh H_{\bf i}|_v$. Since $\varphi_t$ acts linearly on the fibers of $\wh H_{\bf i}$, and since the bundle $\wh H_{\bf i}$ is a line bundle, the value $f^\rho_{\bf i}(v)$ does not depend on the choice of $X$. Then for all $\mu\in\mathsf{M}(\lambda^o)$, define 
		\[
		\ell_{\bf i}^\rho(\mu):=\int_{T^1\lambda} f^\rho_{\bf i} \ \mathrm{d}t\times \mu \in \mathbb R.
		\]
		
		We will see later (Remark \ref{rem: independence from norm}) that $\ell_{\bf i}^\rho$ does not depend on the choice of the continuous norm $\norm{\cdot}_{u \in T^1\lambda}$. Also, in the case when $\lambda^o$ contains a simple closed oriented geodesic $c$ and $\mu$ is the Dirac measure in $c$, then one can verify that 
		\[\ell_{\bf i}^\rho(\mu)=\log\frac{\lambda_{i_1}(\rho(\gamma))}{\lambda_{i_1+1}(\rho(\gamma))},\]
		where for all $g\in\mathsf{PGL}_d(\mathbb{K})$, $\lambda_1(g)\geq \dots\geq \lambda_d(g)$ denote the moduli of the eigenvalues of some linear representative $\bar{g}$ of $g$.
		
		We establish the following lemma by adapting Proposition 7.4 in \cite{BoD} to our context and using the contraction property of $\rho$ (see Definition \ref{def:AnosovBorel} item (2)). 
		\begin{lemma}\label{lem:lengths positive}
			Let $\rho$ be a $\lambda$--Borel Anosov representation, and let ${\bf i}=(i_1,i_2)$ be a pair of positive integers that sum to $d$. For any transverse measure $\mu\in\mathsf{M}(\lambda^o)$ supported on the orientation cover of $\lambda$, we have $\ell_{\bf i}^\rho(\mu) > 0$.
		\end{lemma}
		\begin{proof}
			
			Let $\wh H_{\bf i}$ be the ${\bf i}$--th homomorphism bundle induced by the $\lambda$--limit map $\xi$ of $\rho$. Given a (possibly non-continuous) nowhere zero section $X$ of $\widehat{H}_{\bf i}$, we can express the function $f_{\bf i}^\rho$ as $f_{\bf i}^\rho(v) =  -  \frac{\mathrm{d}}{\mathrm{d} s} \log \norm{\varphi_s(X(v))}_{\varphi_s(v)} |_{s = 0}$, for any $v \in T^1 \lambda$. Since the value $f_{\bf i}^\rho(v)$ is independent of the choice of $X(v)$, and since the norm $\| \cdot \|_{v \in T^1 \lambda}$ is continuous and continuously differentiable along the flow lines, the function $f_{\bf i}^\rho$ is continuous. We deduce that for every $t_0 > 0$
			\begin{align*}
				\ell_{\bf i}^\rho(\mu) &= \int_{T^1\lambda} f_{\bf i}^\rho \, \mathrm{d}t\times \mu \\& = \frac{1}{t_0} \int_0^{t_0} \left( \int_{T^1\lambda} f_{\bf i}^\rho \circ \varphi_s \, \mathrm{d}t\times \mu \right) \textrm{d}s \\
				& = \frac{1}{t_0} \int_{T^1\lambda} \left( \int_0^{t_0} f_{\bf i}^\rho \circ \varphi_s \, \textrm{d}s  \right) \mathrm{d}t\times \mu \\
				& = \frac{1}{t_0} \int_{T^1\lambda} \left( - \int_0^{t_0} \frac{\mathrm{d}}{\mathrm{d} s} \log \norm{\varphi_s(X(v))}_{\varphi_s(v)} \, \textrm{d}s  \right) (\mathrm{d}t\times \mu)(v) \\
				& = \frac{1}{t_0} \int_{T^1\lambda} \log \frac{\norm{X(v)}_v}{\norm{\varphi_{t_0}(X(v))}_{\varphi_{t_0}(v)}} \, (\mathrm{d}t\times \mu)(v) .
			\end{align*}
			Since $\rho$ is $\lambda$--Borel Anosov, there exist constants $C, \alpha > 0$ such that
			\[
			\frac{\norm{X(v)}_v}{\norm{\varphi_{t_0}(X(v))}_{\varphi_{t_0}(v)}} \geq C^{-1} e^{\alpha t_0}.
			\]
			Thus, by taking $t_0$ sufficiently large, we see that $\ell_{\bf i}^\rho(\mu)>0$.
		\end{proof}

\section{Main Theorem and applications} \label{sec:cocycle_pair}

In this section, we state our main result (Theorem \ref{thm: main}) which gives a parameterization of the space of $d$--pleated surfaces with pleating locus a fixed maximal geodesic lamination $\lambda$. In Section \ref{sec: application 2--pleated}, we use this result to complete the proof that the notions of 2--pleated surface and classical pleated surface coincide (Proposition \ref{prop:2-pleated is pleated}). Finally, in Section \ref{subsec:proof structure} we outline the main steps of the proof of Theorem \ref{thm: main}.

\subsection{The cocyclic pair of a pleated surface}\label{shear-bend}

We start by defining the notion of a $\lambda$--cocyclic pair, and then associate such an object to a $d$--pleated surface with pleating locus $\lambda$ via its $\lambda$--limit map and projective invariants of tuples of flags.

\subsubsection{$\lambda$--cocyclic pairs}

Recall from Section \ref{back} that $\wt\Delta$ denotes the set of plaques of $\wt\lambda$, and $\wt\Delta^o$ denotes the set of labelings, i.e. ordered triples of points in $\partial\wt\lambda$ that are the vertices of a plaque of $\wt\lambda$. Let $\wt\Delta^{2*}$ denote the set of distinct pairs of plaques of $\wt\lambda$. 
Observe that if $T_1$ and $T_2$ are distinct plaques with separating edges $(g_1,g_2)$, then the set $P(g_1,g_2)$ of plaques that separate $(g_1,g_2)$, coincides with the set of plaques that separate the pair $(T_1,T_2)$. Then, we denote by $h_{T,-}$ and $h_{T,+}$ the two edges of $T\in P(g_1,g_2)$ that separate $g_1$ and $g_2$, so that $h_{T,-}$ separates $g_1$ and $h_{T,+}$.
\begin{definition} Let $g_1$ and $g_2$ be leaves of $\wt\lambda$ and consider a plaque $T\in P(g_1,g_2)$. Then, the triple $(x_{T,1},x_{T,2},x_{T,3})\in\wt\Delta^o$ is a \emph{coherent labeling} of the vertices of $T$ with respect to $(g_1,g_2)$ if $x_{T,1}$, $x_{T,2}$, $x_{T,3}$ are vertices of $T$ and
	\begin{itemize}
		\item $x_{T,2}$ is the common endpoint of $h_{T,-}$ and $h_{T,+}$, 
		\item $x_{T,1}$ is the endpoint of $h_{T,-}$ that is not $x_{T,2}$, and
		\item $x_{T,3}$ is the endpoint of $h_{T,+}$ that is not $x_{T,2}$.
	\end{itemize}
	When $(g_1,g_2)$ is a separating pair of leaves for a pair of plaques $(T_1,T_2)$, we say that the triple $(x_{T,1},x_{T,2},x_{T,3})$ is a coherent labeling of the vertices of $T$ with respect to $(T_1,T_2)$. See Figure \ref{fig:notations3}.
\end{definition}

\begin{figure}[h!]
	\includegraphics[width=10cm]{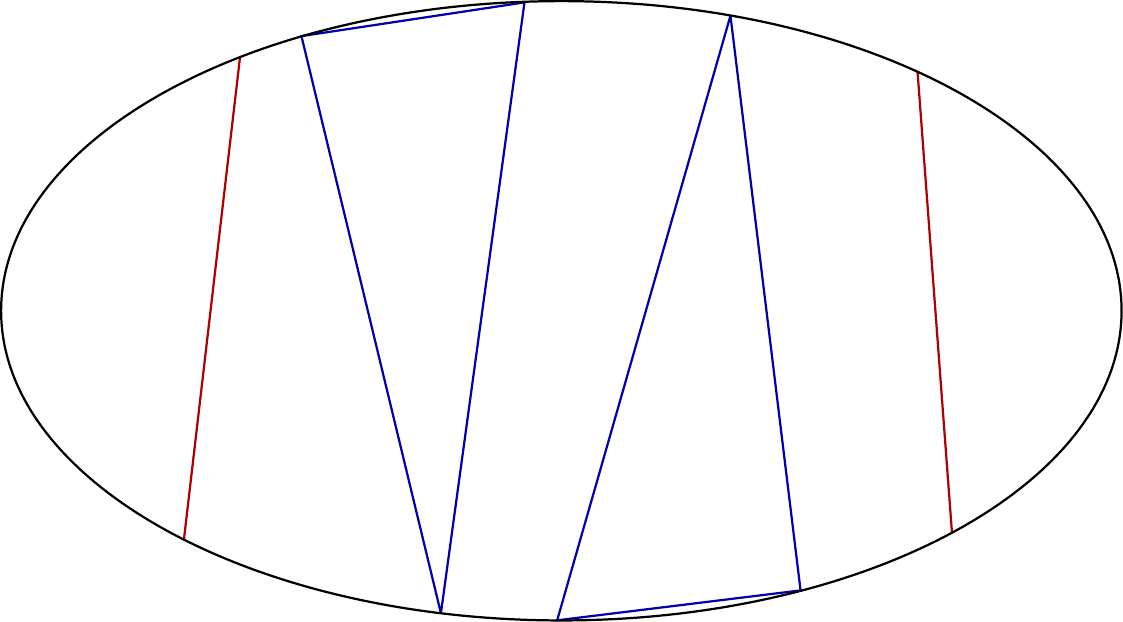}
	\put(-241,70){\small $g_1$}
	\put(-180,90){\small $T$}
	\put(-180,90){\small $T$}
	\put(-210,80){\small $h_{T,-}$}
	\put(-160,90){\small $h_{T,+}$}
	\put(-135,102){\small $h_{T',-}$}
	\put(-90,87){\small $h_{T',+}$}
	\put(-220,152){\small $x_{T,1}$}
	\put(-180,-5){\small $x_{T,2}$}
	\put(-149,-6){\small $x_{T',1}$}
	\put(-85,1){\small $x_{T',3}$}
	\put(-155,160){\small $x_{T,3}$}
	\put(-105,157){\small $x_{T',2}$}
	\put(-110,60){\small $T'$}
	\put(-57,80){\small $ g_2$}
	\caption{\small 
		The plaques $T$ and $T'$ are equipped with their coherent labelings.}
	\label{fig:notations3}
\end{figure}

Let $\mathcal{A}$ denote the set of pairs of positive integers that sum to $d$, and let $\mathcal{B}$ denote the set of triples of positive integers that sum to $d$. For any pair of plaques ${\bf T}=(T_1,T_2)\in\wt\Delta^2$ and any pair of positive integers ${\bf i}=(i_1,i_2)\in\mathcal{A}$ let 
\[\widehat{\bf T}:=(T_2,T_1)\quad\text{and}\quad\widehat{\bf i}:=(i_2,i_1).\] 
Also, for any labeling ${\bf x}=(x_1,x_2,x_3)\in\wt\Delta^o$ and any triple of positive integers ${\bf j}\in\mathcal{B}$, let 
\[\widehat{\bf x}:=(x_2,x_1,x_3),\quad{\bf x}_+:=(x_2,x_3,x_1),\quad{\bf x}_-:=(x_3,x_1,x_2),\]
\[\widehat{\bf j}:=(j_2,j_1,j_3),\quad{\bf j}_+:=(j_2,j_3,j_1),\quad\text{and}\quad{\bf j}_-:=(j_3,j_1,j_2).\]

\begin{definition}[$\mathcal{Y}_d(\lambda;G)$]\label{def_cocycle} For any Abelian group $G$, a $\lambda$--\textit{cocyclic pair of dimension $d$ with values in $G$} is a pair $(\alpha, \theta)$, where 
	$$\alpha \co \wt\Delta^{2*} \times \mathcal{A} \to G, \;\;\text{ and }\;\; \theta \co \wt\Delta^o  \times \mathcal{B} \to G$$
	are such that
	\begin{enumerate}
		\item(Symmetry of $\alpha$) for all distinct pairs of plaques ${\bf T}\in \wt\Delta^{2*}$ and all pairs of positive integers ${\bf i}\in \mathcal{A}$ that sum to $d$  $$\alpha\left({\bf T}, {\bf i}\right) = \alpha(\widehat{\bf T}, \widehat{\bf i});$$ 
		\item(Symmetry of $\theta$) for all triples ${\bf x}\in \wt\Delta^o$ that are the vertices of a plaque of $\wt\lambda$ and all triples ${\bf j} \in \mathcal{B}$ of positive integers that sum to $d$ $$\theta\left({\bf x}, {\bf j}\right) = \theta\left({\bf x}_+, {\bf j}_+\right)= \theta\left({\bf x}_-, {\bf j}_-\right)=-\theta(\widehat{\bf x},\widehat{\bf j});$$
		\item($\Gamma$--invariance of $\alpha$) for all ${\bf T} \in \wt\Delta^{2*}$, ${\bf i}\in \mathcal{A}$, and $\gamma\in\Gamma$, 
		$$\alpha\left(\gamma\cdot{\bf T}, {\bf i}\right) = \alpha({\bf T}, {\bf i});$$
		\item($\Gamma$--invariance of $\theta$) for all ${\bf x}\in \wt\Delta^o$, ${\bf j} \in \mathcal{B}$, and $\gamma\in\Gamma$, 
		$$\theta\left(\gamma\cdot {\bf x}, {\bf j}\right) = \theta({\bf x}, {\bf j});$$
		\item(Cocycle boundary condition) Let $T_1$, $T_2$, and $T$ be pairwise distinct plaques of $\wt\lambda$ such that $T$ separates $T_1$ and $T_2$, and let ${\bf x}_T = (x_{T,1}, x_{T,2}, x_{T,3})$ be a coherent labeling of the vertices of $T$ with respect to $(T_1,T_2)$. Then for all ${\bf i}\in \mathcal{A}$, we have that 
		\[
		\alpha((T_1,T_2), {\bf i}) = \alpha((T_1,T), {\bf i}) + \alpha((T,T_2), {\bf i}) + \sum_{{\bf j} \in \mathcal{B} : j_2 = i_1} \theta({\bf x}_T,{\bf j})
		\]	
		if $x_{T,3}\prec x_{T,2}\prec x_{T,1}$, and
		\[
		\alpha((T_1,T_2), {\bf i}) = \alpha((T_1,T), {\bf i}) + \alpha((T,T_2), {\bf i}) - \sum_{{\bf j} \in \mathcal{B} : j_2 = i_2} \theta({\bf x}_T,{\bf j})
		\]
		if $x_{T,1}\prec x_{T,2}\prec x_{T,3}$.
	\end{enumerate}	
	Let $\mathcal{Y}_d(\lambda;G)$ denote the set of $\lambda$--cocyclic pairs of dimension $d$ with values in $G$.
\end{definition} 

We will often write $\alpha^{\bf i}({\bf T}):=\alpha({\bf T},{\bf i})$ and $\theta^{\bf j}({\bf x}):=\theta({\bf x},{\bf j})$ when convenient. The motivation for the choice of indices and signs in the cocycle boundary condition will become apparent later (see Proposition \ref{prop: shear bend map}).

\begin{remark} Our $\lambda$--cocyclic pairs of dimension $d$ with values in $G$ are a repackaging and a generalization to all Abelian groups of the data that Bonahon and Dreyer use to parameterize the Hitchin component \cite{BoD}. 
\end{remark}

\begin{remark}\label{rmk: lambda cocyclic of dim 2} Note that when $d=2$, the set $\mathcal B$ of triples of positive integers that sum to $d$ is empty. In particular, in this case the map $\theta$ is defined trivially and the cocycle boundary condition becomes an additivity property of the map $\alpha$.
\end{remark}

\begin{remark}\label{rmk: lambda cocyclic modules} Observe that $\mathcal{Y}_d(\lambda;G)$ is naturally an Abelian group. If, in addition, $G$ has the structure of a vector space, then the space $\mathcal{Y}_d(\lambda;G)$ inherits this structure. We will use this observation in Section \ref{sec: intersection}
	when we provide a homological interpretation of $\lambda$--cocyclic pairs of dimension $d$ with values in $G$. 
\end{remark}

Recall that $\mathrm{Hit}_d(S)$ denotes the $d$--Hitchin component. The main goal of this section is to define a biholomorphism $\mathfrak{sb}_d$ between the space $\mathfrak R_d(\lambda)$ of conjugacy classes of $d$--pleated surfaces with pleating locus $\lambda$ and
\[
\mathrm{Hit}_d(S)+i\mathcal Y_d(\lambda;\mathbb R/2\pi\mathbb Z)
\]
seen as a subset of the space $\mathcal Y_d(\lambda;\mathbb C/2\pi\mathbb Z)
$ of $\lambda$--cocyclic pairs of dimension $d$ with values in $\mathbb C/2\pi\mathbb{Z}$.

\subsubsection{Triple and Double ratios of flags}\label{sec: double and triple}\label{xflag}

We associate a $\lambda$--cocylic pair to a $d$--pleated surface using projective invariants for triples or quadruples of flags in general position that were introduced by Fock and Goncharov \cite{fock-goncharov-1}. We now recal their definitions and basic properties.

Let $\mathbb{K}=\mathbb{R}, \mathbb{C}$ and fix a $\mathbb{K}$--linear isomorphism $\mathsf{\Lambda}^{d}(\mathbb{K}^d) \cong \mathbb{K}$. For any flag $F\in\mathcal{F}(\mathbb{K}^d)$ and any $k\in\{1,\dots,d-1\}$, choose a non-zero element $f^{k}\in \mathsf{\Lambda}^kF^{k}.$ With this, we define the following pair of projective invariants:
\begin{itemize}
\item Let ${\bf F}=(F_1,F_2,F_3)$ be an ordered triple of flags in $\mathcal{F}(\mathbb{K}^d)$ that are in general position, and let ${\bf j}=(j_1,j_2,j_3)$ be a triple of positive integers that sum to $d$. The ${\bf j}$--\textit{triple ratio} of ${\bf F}$ is 
$$T^{{\bf j}}({\bf F}):= \frac{f_1^{j_1+1}\wedge f_2^{j_2}\wedge f_3^{j_3-1}}{f_1^{j_1-1}\wedge f_2^{j_2}\wedge f_3^{j_3+1}} \frac{f_1^{j_1}\wedge f_2^{j_2-1}\wedge f_3^{j_3+1}}{f_1^{j_1}\wedge f_2^{j_2+1}\wedge f_3^{j_3-1}}  \frac{f_1^{j_1-1}\wedge f_2^{j_2+1}\wedge f_3^{j_3}}{f_1^{j_1+1}\wedge f_2^{j_2-1}\wedge f_3^{j_3}}.$$
\item Let $\mathcal F(\mathbb K^d)^{[4]}$ denote the set of quadruples $({\bf G},{\bf H})=(G_1, G_2, H_1,H_2)$ of flags in $\mathcal{F}(\mathbb{K}^d)$ such that $G_1$, $G_2$, and $H_k$ are in general position for both $k=1,2$. Then, let ${\bf i}=(i_1,i_2)$ be a pair of positive integers that sum to $d$. The ${\bf i}$--\textit{double ratio} of $({\bf G},{\bf H})$ is 
$$D^{\bf i}({\bf G},{\bf H}):= -\frac{g_1^{i_1}\wedge g_2^{i_2-1}\wedge h_1^{1}}{g_1^{i_1}\wedge g_2^{i_2-1}\wedge h_2^{1}} \frac{g_1^{i_1-1}\wedge g_2^{i_2}\wedge h_2^{1}}{g_1^{i_1-1}\wedge g_2^{i_2}\wedge h_1^{1}}.$$
Notice that $D^{\bf i}({\bf G}, {\bf H})$ depends only on the $1$--dimensional part of the flags $H_1, H_2$.
\end{itemize}

It is straightforward to verify that these invariants are well-defined, and take values in $\mathbb{K}-\{0\}$. As such, we may define
\begin{align}\label{eqn: taudef}\tau^{\bf j}({\bf F}):=\log T^{\bf j}({\bf F})\in\mathbb{R}+i\mathbb{R}/2\pi\mathbb{Z}\end{align}
and 
\begin{align}
\label{eqn: sigmadef}\sigma^{\bf i}({\bf G},{\bf H}):=\log D^{\bf i}({\bf G},{\bf H})\in\mathbb{R}+i\mathbb{R}/2\pi\mathbb{Z}.
\end{align}

The next proposition records some basic symmetries of the invariants $\tau^{\bf j}$ and $\sigma^{\bf i}$ which are straightforward consequences of the definitions. 
For any triple ${\bf F}=(F_1,F_2,F_3)$ of flags in $\mathcal{F}(\mathbb{K}^d)$, denote
\[\quad\wh{\bf F}:=(F_2,F_1,F_3),\quad{\bf F}_+:=(F_2,F_3,F_1),\quad{\bf F}_-:=(F_3,F_1,F_2).\] 
Similarly, for any pair ${\bf H}=(H_1,H_2)$ of flags in $\mathcal{F}(\mathbb{K}^d)$, denote 
\[\wh{\bf H}:=(H_2,H_1).\]
The properties of double and triple ratios, originally due to Goncharov \cite{goncharov94} and Fock-Goncharov \cite{fock-goncharov-1}, that will be used throughout the exposition are summarized in the following statement.
\begin{proposition}\label{prop_invariants}\
\begin{enumerate}
	\item If ${\bf j}$ is a triple of positive integers that sum to $d$ and ${\bf F}\in\mathcal F(\mathbb K^d)^{3}$ is a triple of flags in general position, then 
	\[\tau^{\bf j}({\bf F}) =\tau^{{\bf j}_+}({\bf F}_+)=\tau^{{\bf j}_-}({\bf F}_-)= - \tau^{\wh{\bf j}}(\wh{\bf F}).\]
	
	\item Let ${\bf F} =(F_1,F_2,F_3)$ and ${\bf F}' = (F_1',F_2',F_3') \in\mathcal F(\mathbb K^d)^3$ be triples of flags in general position. There exists a projective transformation $A \in \mathsf{PGL}_d(\mathbb{K})$ such that
	\[
	(A \, F_1, A \, F_2, A \, F_3) = (F_1',F_2',F_3')
	\]
	if and only if $\tau^{\bf j}({\bf F}) = \tau^{\bf j}({\bf F}')$ for every triple ${\bf j}$ of positive integers that sum to $d$.
	\item If ${\bf i}$ is a pair of positive integers that sum to $d$ and $({\bf G},{\bf H})\in\mathcal F(\mathbb K^d)^{[4]}$, then 
	\[\sigma^{\bf i}({\bf G},{\bf H}) = - \sigma^{\bf i}({\bf G},\wh{\bf H}) = - \sigma^{\wh{\bf i}}(\wh{\bf G},{\bf H}) = \sigma^{\wh{\bf i}}(\wh{\bf G},\wh{\bf H}).\]
	\item Let $({\bf G}, {\bf H}) = (G_1,G_2,H_1,H_2)$ and $({\bf G}', {\bf H}') = (G_1',G_2',H_1',H_2') \in \mathcal{F}(\mathbb{K})^{[4]}$. Then there exists a projective transformation  $A \in \mathsf{PGL}_d(\mathbb{K})$ such that
	\[
	(A \, G_1, A \, G_2, A \, H_1^1, A \, H_2^1) = (G_1', G_2', (H_1')^1, (H_2')^1)
	\]
	if and only if $\sigma^{\bf i}({\bf G}, {\bf H}) = \sigma^{\bf i}({\bf G}, {\bf H})$ for every pair ${\bf i}$ of positive integers that sum to $d$.
\end{enumerate}
\end{proposition}
(For a proof of Proposition \ref{prop_invariants}, see e.g. \cite[Lemmas 2.2.8, 2.3.7]{tengren_thesis}.)

\subsubsection{Positivity of Hitchin representations}
A seminal result of Fock and Goncharov \cite{fock-goncharov-1} characterizes Hitchin representations in terms of a certain positivity property of their limit maps, which can be stated in terms of positivity of triple and double ratios. We use the positivity of the limit map of Hitchin representations implicitly via results of Bonahon and Dreyer \cite{BoD}, but also to show existence of bending maps in Section \ref{sec:existence}.

\begin{definition}\label{def: positivity via invariants}
A map $\xi:\mathbb{S}^1\to\mathcal{F}(\mathbb{R}^d)$ is \emph{positive} if for any $\ell\geq 3$, any cyclically ordered tuple of points $x_1\prec\dots\prec x_l$ is mapped to a tuple of flags $(\xi(x_1),\dots,\xi(x_\ell))$ such that
\begin{itemize}
	\item the triple of flags $(\xi(x_{k_1}),\xi(x_{k_2}),\xi(x_{k_3}))$ is in general position and $T^{\bf j}(\xi(x_{k_1}),\xi(x_{k_2}),\xi(x_{k_3}))>0$ for all triples ${\bf j}$ of positive integers that sum to $d$ and all pairwise distinct $k_1,k_2,k_3\in\{1,\dots,\ell\}$,
	\item $D^{\bf i}(\xi(x_{k_1}),\xi(x_{k_2}),\xi(x_{k_3}),\xi(x_{k_4})) >0$ for all pairs ${\bf i}$ of positive integers that sum to $d$ and all $k_1<k_2<k_3<k_4$, where $<$ denotes the cyclic order on $\mathbb{Z}/\ell\mathbb{Z}$.
\end{itemize}
\end{definition}

\begin{theorem}[{Fock-Goncharov \cite[Theorem 1.15 and Theorem 9.1]{fock-goncharov-1}}]\label{thm: FG}
A representation $\rho:\Gamma\to\mathsf{PGL}_d(\mathbb{R})$ is a $d$--Hitchin representation if and only if it admits a $\rho$--equivariant, continuous, positive map $\xi:\partial\wt S\to\mathcal{F}(\mathbb{R}^d)$.
\end{theorem}

\subsubsection{The slithering map of a $d$--pleated surface}

The idea of a \emph{slithering map} was first introduced by Bonahon and Dreyer \cite{BoD} to parameterize the space of Hitchin representations. Later, in Sections \ref{sec:productable}, \ref{sec: slithering}, and \ref{sec:bending}, we show that such slithering maps exist in a much more general setting.

The existence of slithering maps rely on a regularity property of the $\lambda$--limit map of a $d$--pleated surface. We saw in Section \ref{ssec:BorelAnosov} that, unlike the case of Borel Anosov representations, the $\lambda$--limit maps of $\lambda$--Borel Anosov representations are not necessarily continuous, but only $\lambda$--continuous. However, just like for Borel Anosov representations, the contracting dynamics condition (Condition (2) of Definition~\ref{def:AnosovBorel}) in the definition of $\lambda$--Borel Anosov representations implies that their $\lambda$--limit maps satisfy an adapted notion of H\"older-continuity that we now describe precisely. 

Recall that $\wt\Lambda$ denotes the set of leaves of $\wt\lambda$, $\wt\Lambda^o$ denotes the set of leaves of $\wt\Lambda$ equipped with an orientation, and $\pi_{\wt\Lambda}:\wt\Lambda^o\to\wt\Lambda$ is the natural two-to-one map that forgets the orientation. Furthermore, $\mathsf{g}_1,\mathsf{g}_2\in\widetilde\Lambda^o$ are oriented in parallel if either
\[\mathsf{g}_2^+ \preceq \mathsf{g}_1^+ \prec \mathsf{g}_1^-\preceq \mathsf{g}_2^- \prec \mathsf{g}_2^+\quad\text{ or }\quad \mathsf{g}_1^+ \preceq \mathsf{g}_2^+ \prec \mathsf{g}_2^-\preceq \mathsf{g}_1^- \prec \mathsf{g}_1^+,\]
and $d_\infty$ is a metric on $\widetilde\Lambda^o$ induced by a Riemannian metric in $\partial\widetilde S$ (see Section~\ref{ssec:geodlam}). Let $d_{\mathcal{F}}$ be a Riemannian metric on the space of flags $\mathcal{F}(\mathbb{K}^d)$. 
As before, given $\xi:\partial\wt\lambda\to\mathcal{F}(\mathbb{K}^d)$, we will slightly abuse notation and let $\xi:\wt\Lambda^o\to\mathcal{F}(\mathbb{K}^d)^2$ denote the associated map given by $\xi(\mathsf{g})=(\xi(\mathsf{g}^+),\xi(\mathsf{g}^-))$.

\begin{definition}\label{def: lambda continuous}
A map $\xi:\partial\wt\lambda\to\mathcal{F}(\mathbb{K}^d)$ is \emph{locally $\lambda$--H\"older continuous} if its associated map $\xi:\wt\Lambda^o\to\mathcal{F}(\mathbb{K}^d)^2$ is locally H\"older continuous, i.e. for every oriented leaf $\mathsf{h}\in\wt\Lambda^o$, there is an open set $U\subset\widetilde\Lambda^o$ containing $\mathsf{h}$ and constants $A,\nu>0$ such that for all pairs of oriented in parallel leaves $\mathsf{g}_1$and $\mathsf{g}_2$ in $U$, we have
\[
d_{\mathcal{F}}(\xi(\mathsf{g}_1),\xi(\mathsf{g}_2))\leq A \, d_\infty(\mathsf{g}_1,\mathsf{g}_2)^{\nu}.
\]
\end{definition} 

Notice that since $\partial\widetilde S$ and $\mathcal{F}(\mathbb{K}^d)$ are compact, the notion of local $\lambda$--H\"older continuity for $\xi$ does not depend on the choice of Riemannian metrics on $\partial\widetilde S$ and $\mathcal{F}(\mathbb{K}^d)$. The following theorem is a consequence of the contracting property of Anosov representations, see Wang \cite[Theorem 1.2, and Remark~1.3]{wang2021anosov} for a proof.

\begin{theorem}[\cite{wang2021anosov}] \label{thm: Wang}
If $\rho:\Gamma\to\mathsf{PGL}_d(\mathbb{K})$ is $\lambda$--Borel Anosov, then the $\lambda$--limit map $\xi$ of $\rho$ is locally $\lambda$--H\"older continuous.
\end{theorem}
	
Generalizing Bonahon and Dreyer \cite{BoD}, we now define the notion of a slithering map compatible with a $\lambda$--H\"older continuous map $\xi\co\partial\widetilde{\lambda}\to \mathcal F(\mathbb{K}^d)$. Informally, a slithering map compatible with $\xi$ is a certain map that assigns to every pair of leaves of $\widetilde\Lambda$ an element in $\mathsf{SL}_d(\mathbb{K})$ sending the flags associated with the endpoints of the first leaf by $\xi$ to those associated with the endpoints of the second leaf. This is a higher rank version of the map defined by the horocylic foliation when $d=2$.
	
Recall that if $g_1$ and $g_2$ are leaves of $\wt\lambda$, then $Q(g_1,g_2)$ denotes the set of leaves of $\wt\lambda$ that separate $g_1$ and $g_2$.

	\begin{definition}\label{def: slithering map}
		Let $\xi\co\partial\widetilde{\lambda}\to \mathcal F(\mathbb{K}^d)$ be $\lambda$--transverse and locally $\lambda$--H\"older continuous. A map $$\Sigma\co\wt\Lambda^2\to\mathsf{SL}_d(\mathbb{K})$$ is a \emph{slithering map compatible with $\xi$} if the following hold:
		\begin{enumerate}
			\item $\Sigma(g,g)=\id$ for all $g\in\wt\Lambda$, $\Sigma(g_1,g_2)=\Sigma(g_2,g_1)^{-1}$ for all $g_1,g_2\in\wt\Lambda$, and $\Sigma(g_1,g_2)\Sigma(g_2,g_3)=\Sigma(g_1,g_3)$ for all $g_1,g_2,g_3\in \wt\Lambda$ such that $g_2$ separates $g_1$ and $g_3$.
			\item For all leaves $g_1$ and $g_2$ of $\wt\lambda$, $\Sigma|_{Q(g_1,g_2)}$ is H\"older continuous, i.e. there are constants $A,\nu>0$ (depending on $g_1$ and $g_2$) such that 
			\[\norm{\Sigma(\ell_1,\ell_2)-\Sigma(h_1,h_2)}\leq A\,d_\infty((\ell_1,\ell_2),(h_1,h_2))^\nu\]
			for all $\ell_1,\ell_2,h_1,h_2\in Q(g_1,g_2)$. Here, $\norm{\cdot}$ is the operator norm on $\End(\mathbb{K}^d)$ induced by the standard norm on $\mathbb{K}^d$.
			\item If $g_1,g_2\in\wt\Lambda$ share an endpoint, then $\Sigma(g_1,g_2)$ is unipotent. 
			\item If $\mathsf{g}_1, \mathsf{g}_2\in\wt\Lambda^o$ are oriented in parallel and $\pi_{\wt\Lambda}(\mathsf{g}_j) = g_j$ for $j = 1, 2$, then $\Sigma(g_1,g_2)$ sends $\xi(\mathsf{g}_2)$ to $\xi(\mathsf{g}_1)$.
		\end{enumerate}
		We write $\Sigma_{g_1,g_2}:=\Sigma(g_1,g_2)$ when convenient.
	\end{definition}
	
	In Section \ref{sec: slithering} we prove the following key result:
	
	\begin{theorem}\label{thm: slitherable and slithering}
		If $\xi\co\partial\wt\lambda\to\mathcal{F}(\mathbb{K}^d)$ is $\lambda$--transverse and locally $\lambda$--H\"older continuous, then it admits a unique compatible slithering map $\Sigma\co \wt\Lambda^2\to\mathsf{SL}_d(\mathbb{K})$. Furthermore, if there is a representation $\rho\co\Gamma\to\mathsf{PGL}_d(\mathbb{K})$ such that $\xi$ is $\rho$--equivariant, then $\Sigma$ is also $\rho$--equivariant.
	\end{theorem}
	
	In particular, Theorem \ref{thm: slitherable and slithering} implies that the $\lambda$--limit maps of $\lambda$--Borel Anosov representations admit a unique compatible slithering map that is $\rho$--equivariant.
	
	\subsubsection{Shear-bend coordinates in \texorpdfstring{$\mathsf{PGL}_d(\mathbb{C})$}{PSLd(C)}}\label{sec:shear_bend}
	
	Using slithering maps, we describe the main theorem of this paper, which is a parameterization of the space $\mathfrak{R}_d(\lambda)$ of $d$--pleated surfaces with pleating locus $\lambda$ by an explicitly described subset of the group $\mathcal{Y}_d(\lambda;\mathbb{C} /2\pi i\mathbb{Z})$ of $\lambda$--cocyclic pairs of dimension $d$ with values in $\mathbb{C} /2\pi i\mathbb{Z}$, see Theorem \ref{thm: main}. As we mentioned in the introduction, this is a generalization of the main result in Bonahon and Dreyer \cite{BoD}, who proved an analogous parameterization of the $d$--Hitchin component by a subset of $\mathcal{Y}_d(\lambda;\mathbb{R})$, see Theorem~\ref{thm: BD par}. 	
	
	First, we describe how one assigns a $\lambda$--cocyclic pair $(\alpha_\rho,\theta_\rho)\in\mathcal{Y}_d(\lambda;\mathbb{C} /2\pi i\mathbb{Z})$ to every $\rho\in\mathcal{R}_d(\lambda)$. This is a direct generalization of the way Bonahon and Dreyer \cite{BoD} assigned a $\lambda$--cocyclic pair in $\mathcal{Y}_d(\lambda;\mathbb{R})$ to every $d$--Hitchin representation.
	
	Recall that $\wt\Delta^{2*}$ denotes the set of distinct pairs of plaques of $\wt\lambda$ and $\mathcal{A}$ is the set of pairs of positive integers that sum to $d$. Given a locally $\lambda$--H\"older continuous, $\lambda$--transverse, and $\lambda$--hyperconvex map 
	\[\xi:\partial\wt\lambda\to\mathcal{F}(\mathbb{C}^d),\] 
	let $\Sigma$ be the slithering map compatible with $\xi$ (see Theorem \ref{thm: slitherable and slithering}). Then define
	\[\alpha_\xi \co \wt\Delta^{2*} \times \mathcal{A} \to \mathbb{C}/2\pi i\mathbb{Z}\]
	by 
	$$\alpha_\xi ({\bf T}, {\bf i})=\alpha_\xi^{\bf i} ({\bf T}) := \sigma^{{\bf i}}\left(\xi(y_2), \xi(y_1), \xi(y_3), \Sigma(g_1,g_2)\cdot\xi(z)\right),$$
	where $\sigma^{\bf i}$ is the logarithm of the double ratio (see Equation \eqref{eqn: sigmadef}), $(g_1,g_2)$ is the separating pair of edges for the pair of plaques ${\bf T}=(T_1,T_2)$, $z$ is the vertex of $T_2$ that is not an endpoint of $g_2$, and $(y_1,y_2,y_3)$ are the vertices $T_1$, enumerated so that $y_1$ and $y_2$ are the endpoints of $g_1$ and $y_1\prec y_2\prec y_3$.

	Recall that $\wt\Delta^o$ denotes the set of triples of points in $\partial\wt\lambda$ that are vertices of a plaque of $\wt\lambda$, and $\mathcal{B}$ is the set of triples of positive integers that sum to $d$. Then let
	\[\theta_\xi \co \wt\Delta^o  \times \mathcal{B} \to \mathbb{C}/2\pi i\mathbb{Z}\]
	be the map defined by
	$$\theta_\xi ({\bf x}, {\bf j})=\theta_\xi^{\bf j}({\bf x}) := \tau^{\bf j}\left(\xi({\bf x})\right),$$
	where, $\tau^{\bf j}$ is the logarithm of the triple ratio (see Equation \eqref{eqn: taudef}). This is well-defined because the triple of flags $\xi({\bf x})$ is in general position. 
	
	In the case where there is some $d$--pleated surface $\rho\in\mathcal{R}_d(\lambda)$ such that $\xi$ is the $\lambda$--limit map of $\rho$, we denote 
	\[\alpha_\rho({\bf T},{\bf i})=\alpha_\rho^{\bf i}({\bf T}):=\alpha_\xi^{\bf i}({\bf T})\,\,\text{ and }\,\,\theta_\rho({\bf x},{\bf j})=\theta^{\bf j}_\rho({\bf x}):=\theta^{\bf j}_\xi({\bf x})\]
	for all ${\bf i}\in\mathcal{A}$, ${\bf j}\in\mathcal{B}$, ${\bf T}\in\wt\Delta^{2*}$, and ${\bf x}\in\wt\Delta^o$.

	The following proposition verifies that $(\alpha_\rho,\theta_\rho)$ defined above is indeed a $\lambda$--cocyclic pair of dimension $d$ with values in $\mathbb{C}/2\pi i\mathbb{Z}$. The proof adapts ideas from Bonahon-Dreyer \cite[Section 5.2]{BoD} to our setting.
	
	\begin{proposition}\label{prop: shear bend map}
		Let $\xi:\partial\wt\lambda\to\mathcal{F}(\mathbb{C}^d)$ a locally $\lambda$--H\"older continuous, $\lambda$--transverse, and $\lambda$--hyperconvex map. Then the following hold:
		\begin{enumerate}
			\item For all pairs of distinct plaques ${\bf T}=(T_1,T_2)\in \wt\Delta^{2*}$ and all pairs ${\bf i}=(i_1,i_2)\in \mathcal{A}$ of positive integers that sum up to $d$, 
			$$\alpha_{\xi}\left({\bf T}, {\bf i}\right) = \alpha_{\xi}(\widehat{\bf T}, \widehat{\bf i}),$$ 
			where $\widehat{\bf T}:=(T_2,T_1)$ and $\widehat{\bf i}:=(i_2,i_1)$.
			\item For all triples ${\bf x}:=(x_1, x_2, x_3)\in \wt\Delta^o$ that are the vertices of a plaque of $\wt \lambda$ and all triples ${\bf j}:=(j_1, j_2, j_3) \in \mathcal{B}$ of positive integers that sum up to $d$, 
			$$\theta_{\xi}\left({\bf x}, {\bf j}\right) = \theta_{\xi}\left({\bf x}_+, {\bf j}_+\right)= \theta_{\xi}\left({\bf x}_-, {\bf j}_-\right)=-\theta_{\xi}(\widehat{\bf x},\widehat{\bf j}),$$
			where ${\bf x}_+:=(x_2,x_3,x_1)$, ${\bf x}_-:=(x_3,x_1,x_2)$, $\wh{\bf x}:=(x_2,x_1,x_3)$, ${\bf i}_+:=(i_2,i_3,i_1)$, ${\bf i}_-:=(i_3,i_1,i_2)$, and $\wh{\bf i}:=(i_2,i_1,i_3)$.
			\item Let $T_1$, $T_2$, and $T$ be pairwise distinct plaques of $\wt\lambda$ such that $T$ separates $T_1$ and $T_2$, and let ${\bf x}_T = (x_{T,1}, x_{T,2}, x_{T,3})$ be a coherent labeling of the vertices of $T$ with respect to $(T_1,T_2)$. Then for all ${\bf i}\in \mathcal{A}$, we have that 
			\[
			\alpha_\xi((T_1,T_2), {\bf i}) = \alpha_\xi((T_1,T), {\bf i}) + \alpha_\xi((T,T_2), {\bf i}) + \sum_{{\bf j} \in \mathcal{B} : j_2 = i_1} \theta_\xi({\bf x}_T,{\bf j})
			\]	
			if $x_{T,3}\prec x_{T,2}\prec x_{T,1}$, and
			\[
			\alpha_\xi((T_1,T_2), {\bf i}) = \alpha_\xi((T_1,T), {\bf i}) + \alpha_\xi((T,T_2), {\bf i}) - \sum_{{\bf j} \in \mathcal{B} : j_2 = i_2} \theta_\xi({\bf x}_T,{\bf j})
			\]
			if $x_{T,1}\prec x_{T,2}\prec x_{T,3}$.
			\item If $\rho\in \mathcal{R}_d(\lambda)$ is a $d$--pleated surface with pleating locus $\lambda$, then $(\alpha_\rho,\theta_\rho)$ is a $\lambda$--cocyclic pair of dimension $d$ with values in $\mathbb{C}/2\pi i\mathbb{Z}$.
		\end{enumerate}	
	\end{proposition}
	
	\begin{proof}	
		{\em Part (1).} Let $(g_1,g_2)$ be a separating pair of edges for ${\bf T}$. Then for both $\ell=1,2$, let $(y_{\ell,1},y_{\ell,2},y_{\ell,3})$ be the vertices $T_\ell$, enumerated so that $y_{\ell,1}$ and $y_{\ell,2}$ are the endpoints of $g_\ell$, and $y_{\ell,1}\prec y_{\ell,2}\prec y_{\ell,3}$. Let $\Sigma$ be the slithering map compatible with $\xi$, and observe that
		\[\Sigma(g_2,g_1)\cdot(\xi(y_{1,1}),\xi(y_{1,2}))=(\xi(y_{2,2}),\xi(y_{2,1})).\] 
		Thus, we have
		\begin{align*}
			\alpha_\xi^{\bf i}\left({\bf T}\right) &=\sigma^{{\bf i}}\left(\xi(y_{1,2}), \xi(y_{1,1}), \xi(y_{1,3}), \Sigma\left(g_1, g_2\right)\cdot\xi(y_{2,3})\right)\\
			&=\sigma^{{\bf i}}\left(\xi(y_{2,1}), \xi(y_{2,2}), \Sigma\left(g_2,g_1\right)\cdot\xi(y_{1,3}), \xi(y_{2,3})\right)\\
			&=\sigma^{\wh{\bf i}}\left(\xi(y_{2,2}),\xi(y_{2,1}),  \xi(y_{2,3}), \Sigma\left(g_2,g_1\right)\cdot\xi(y_{1,3})\right)= \alpha_\xi^{\wh{\bf i}}(\wh{\bf T})
		\end{align*}	
		for all ${\bf i}\in\mathcal{A}$, where the second and third equalities are parts (4) and (3) of Proposition \ref{prop_invariants} respectively.
		
		{\em Part (2).} This follows immediately from Part (1) of Proposition \ref{prop_invariants}.
		
		{\em Part (3).} We will only prove the case when $x_{T,3}\prec x_{T,2}\prec x_{T,1}$ since the case where $x_{T,1}\prec x_{T,2}\prec x_{T,3}$ is similar.
		
		Let $h_1$ and $h_2$ be the edges of $T$ such that $h_1<h_2$ with respect to $\left(g_1, g_2\right)$. For $\ell=1,2$, choose vectors $b_\ell^1\in \xi(y_{\ell,3})^1$. Also, for all $\ell=1,2,3$ and $k\in\{1,\dots,d-1\}$, choose non-zero vectors 
		\[a_\ell^{k}\in\mathsf{\Lambda}^k\xi(x_{T,\ell})^k\] 
		such that $\Sigma(h_1,h_2)\cdot a_1^{k}=a_3^{k}$.
		Since $\Sigma(h_1,h_2)$ is a unipotent endomorphism that fixes the flag $\xi(x_{T,2})$, it follows that $\Sigma(h_1,h_2)\cdot a_2^{k}=a_2^{k}$ for all $k$.
		
		For any ${\bf i}=(i_1,i_2)\in \mathcal{A}$,
		\begin{align*}
			\alpha^{\bf i}_\xi(T_1, T)&=\sigma^{{\bf i}}\left(\xi(y_{1,2}), \xi(y_{1,1}), \xi(y_{1,3}), \Sigma\left(g_1, h_1\right)\cdot\xi(x_{T,3})\right)\\
			&=\sigma^{{\bf i}}\left(\xi(x_{T,2}), \xi(x_{T,1}), \Sigma\left(h_1,g_1\right)\cdot\xi(y_{1,3}), \xi(x_{T,3})\right)\\
			&=
			\log\left(-\frac{a_2^{i_1}\wedge a_1^{i_2-1}\wedge  \Sigma\left(h_1,g_1\right)\cdot b_1^{1}}{a_2^{i_1}\wedge a_1^{i_2-1}\wedge  a_3^{1}}\frac{a_2^{i_1-1}\wedge a_1^{i_2}\wedge  a_3^{1}}{a_2^{i_1-1}\wedge a_1^{i_2}\wedge  \Sigma\left(h_1,g_1\right)\cdot b_1^{1}}\right)\\
			&=
			\log\left(-\frac{a_2^{i_1}\wedge a_3^{i_2-1}\wedge  \Sigma(h_2, g_1)\cdot b_1^{1}}{a_2^{i_1}\wedge a_1^{i_2-1}\wedge  a_3^{1}}\frac{a_2^{i_1-1}\wedge a_1^{i_2}\wedge  a_3^{1}}{a_2^{i_1-1}\wedge a_3^{i_2}\wedge  \Sigma(h_2, g_1)\cdot b_1^{1}}\right) .
		\end{align*}
		For the last equality, we apply the unipotent transformation $\Sigma(h_2, h_1)\in \mathsf{SL}_d(\mathbb{C})$ to the $d$--forms on the top left and the bottom right of the double ratio. Similarly, we have
		\begin{align*}
			\alpha_\xi^{\bf i}(T, T_2)
			&=
			\log\left(-\frac{a_2^{i_1}\wedge a_3^{i_2-1}\wedge  a_1^{1}}{a_2^{i_1}\wedge a_3^{i_2-1}\wedge  \Sigma(h_2, g_2)\cdot b_2^{1}}\frac{a_2^{i_1-1}\wedge a_3^{i_2}\wedge  \Sigma(h_2, g_2)\cdot b_2^{1}}{a_2^{i_1-1}\wedge a_3^{i_2}\wedge  a_1^{1}}\right).
		\end{align*}
		
		On the other hand, 
		\begin{align*}
			\sum_{\{{\bf  j}\in\mathcal{B}\colon j_2 = i_1\}} \theta_\xi^{\bf j}\left({\bf x}_T\right)=&
			\log\left( \frac{a_1^{i_2}\wedge a_2^{i_1}}{a_2^{i_1}\wedge a_3^{i_2}}\frac{a_1^{i_2-1}\wedge a_2^{i_1}\wedge a_3^{1}}{a_1^{1}\wedge a_2^{i_1}\wedge a_3^{i_2-1}}\frac{a_1^{1}\wedge a_2^{i_1-1}\wedge a_3^{i_2}}{a_1^{i_2}\wedge a_2^{i_1-1}\wedge a_3^{1}} \frac{a_2^{i_1+1}\wedge a_3^{i_2-1}}{a_1^{i_2-1}\wedge a_2^{i_1+1}}\right)\\
			=&\log\left((-1)^{d-1} \frac{a_1^{i_2-1}\wedge a_2^{i_1}\wedge a_3^{1}}{a_1^{1}\wedge a_2^{i_1}\wedge a_3^{i_2-1}}\frac{a_1^{1}\wedge a_2^{i_1-1}\wedge a_3^{i_2}}{a_1^{i_2}\wedge a_2^{i_1-1}\wedge a_3^{1}}
			\right).
		\end{align*}   
		The first equality holds by the cancellations that occur between the triple ratios involved in the sum, 
		and the second equality holds because for all ${\bf i}\in\mathcal{A} $
		\[(-1)^{i_1i_2}a_2^{i_1}\wedge a_3^{i_2}=a_3^{i_2}\wedge a_2^{i_1}=\Sigma(h_2, h_1)\cdot (a_1^{i_2}\wedge a_2^{i_1})=a_1^{i_2}\wedge a_2^{i_1}.\] 
		Combining the relations described above we deduce that
		\begin{align*}
			&\,\alpha_\xi^{\bf i}(T_1, T) + \alpha_\xi^{\bf i}(T, T_2) + \sum_{\{{\bf  j}\in\mathcal{B}\colon j_2 = i_1\}} \theta_\xi^{\bf j}\left({\bf x}_T\right)\\
			=&\,\log\left(-\frac{a_2^{i_1}\wedge a_3^{i_2-1}\wedge  \Sigma(h_2, g_1)\cdot b_1^{1}}{ a_2^{i_1}\wedge  a_3^{i_2-1}\wedge\Sigma(h_2, g_2)\cdot b_2^{1}}\frac{a_2^{i_1-1}\wedge a_3^{i_2}\wedge  \Sigma(h_2, g_2)\cdot b_2^{1}}{a_2^{i_1-1}\wedge a_3^{i_2}\wedge  \Sigma(h_2, g_1)\cdot b_1^{1}}\right)\\
			=&\,\,\sigma^{{\bf i}}\left(\xi(x_{T,2}), \xi(x_{T,3}), \Sigma(h_2, g_1)\cdot \xi(z_1), \Sigma(h_2, g_2)\cdot\xi(z_2)\right)=\alpha_\xi^{\bf i}\left({\bf T}\right).
		\end{align*}
		
		{\rm Part (4).} Let $\xi$ be the $\lambda$--limit map of $\rho$. Parts (1)--(3) imply that $(\alpha_\rho, \theta_\rho)$ satisfies conditions (1), (2), and (5) of Definition~\ref{def_cocycle}. By the $\rho$--equivariance of $\xi$, $(\alpha_\rho, \theta_\rho)$ also satisfies conditions (3) and (4) of Definition~\ref{def_cocycle}.
	\end{proof}

	\subsection{Parameterizing \texorpdfstring{$d$}{d}--pleated surfaces using \texorpdfstring{$\lambda$}{l}--cocyclic pairs}\label{sec: statement of main theorem}
	In the case when $\rho\in\mathcal{R}_d(\lambda)$ is a $d$--Hitchin representation, observe from Theorem \ref{thm: FG}, that $(\alpha_\rho,\theta_\rho)\in\mathcal{Y}_d(\lambda;\mathbb{R})$. Furthermore, if $\rho$ and $\rho'$ are $d$--Hitchin representations that are conjugate, then $(\alpha_\rho,\theta_\rho)=(\alpha_{\rho'},\theta_{\rho'})$. As such, we may define a map between the $d$--Hitchin component and the space of $\lambda$--cocyclic pairs of dimension $d$ with values in $\mathbb{R}$
	\[\mathfrak{s}_d:{\rm Hit}_d(S)\to\mathcal{Y}_d(\lambda;\mathbb{R})\]
	by $\mathfrak s_d([\rho])=(\alpha_\rho,\theta_\rho)$. Bonahon and Dreyer then proved the following theorem.
	
	\begin{theorem}[Bonahon-Dreyer \cite{BoD}]\label{thm: BD par}
		The map $\mathfrak s_{d}$ is a homeomorphism  onto its image with real-analytic inverse. 
		The image is an open polyhedral cone $\mathcal{C}_d(\lambda)$ contained in $\mathcal{Y}_d(\lambda;\mathbb{R})$.
	\end{theorem}
	We describe the cone $\mathcal{C}_d(\lambda)$ explicitly in Section \ref{length} (see Equation \eqref{polytope description}) using a homological interpretation of $\lambda$--cocyclic pairs and the length functions introduced in Section \ref{subsec:length_functions}.

As observed at the end of Section \ref{ssec:Hitchin}, ${\rm Hit}_d(S)$ is a real-analytic manifold, so the regularity claim in Theorem \ref{thm: BD par} makes sense. However, it is not a priori clear that $\mathfrak{R}_d(\lambda):=\mathcal{R}_d(\lambda)/\mathsf{PGL}_d(\mathbb{C})$ is even a Hausdorff topological space. In fact, using Theorem~\ref{thm: main}, in Section \ref{sec: nonHausdorff} we will show that $\mathfrak{R}(\lambda,3)$ intersects the locus of non-Hausdorff points in $\Hom(\Gamma,\mathsf{PGL}_3(\mathbb{C}))/\mathsf{PGL}_3(\mathbb{C})$. Namely, there exists $\rho$ in $\mathcal{R}_d(\lambda)$ whose $\mathsf{PGL}_3(\mathbb{C})$--orbit is not a closed subset of $\Hom(\Gamma,\mathsf{PGL}_3(\mathbb{C}))$.

On the other hand, by Proposition \ref{prop: pleated surfaces smooth}, $\mathcal{R}_d(\lambda)$ is a complex manifold contained in $\Hom(\Gamma,\mathsf{PGL}_3(\mathbb{C}))$. Thus, to state our main theorem, we let
\[\mathcal{N}(\mathbb{C}^d):=\{(F_1,F_2,p)\in\mathcal{F}(\mathbb{C}^d)^2\times\mathbb{P}(\mathbb{C}^d):F_1^{k-1}+F_2^{d-k}+p=\mathbb{C}^d\text{ for all }k=1,\dots,d\}\]
and define for every triple ${\bf x}=(x_1,x_2,x_3)\in\wt\Delta^o$, the map
\begin{align}\label{eqn: main thm}
	\begin{split}\mathcal{sb}_{d,{\bf x}}:\mathcal{R}_d(\lambda)&\to \mathcal{Y}_d(\lambda;\mathbb{C}/2\pi i\mathbb{Z})\times\mathcal{N}(\mathbb{C}^d),\\
		\rho&\mapsto\left((\alpha_\rho,\theta_\rho),(\xi(x_1),\xi(x_2),\xi(x_3)^1)\right),
	\end{split}
\end{align}
where $\xi$ is the $\lambda$--limit map of $\rho$. In other words, the map $\mathcal{sb}_{d,{\bf x}}$ depends on a normalization of the image of a fixed base plaque with (ordered) vertices ${\bf x}$. Note that the space $\mathcal{N}(\mathbb{C}^d)$
of projective frames in $\mathbb{C}^d$ is (non-canonically) isomorphic as $\mathsf{PGL}_d(\mathbb{C})$--spaces to $\mathsf{PGL}_d(\mathbb{C})$. 

As before, note that if $\rho$ and $\rho'$ are conjugate representations in $\mathcal{R}_d(\lambda)$, then $(\alpha_\rho,\theta_\rho)=(\alpha_{\rho'},\theta_{\rho'})$. Hence, $\mathcal{sb}_{d,{\bf x}}$ is equivariant with respect to the $\mathsf{PGL}_d(\mathbb{C})$--action on $\mathcal{R}_d(\lambda)$ by conjugation and the $\mathsf{PGL}_d(\mathbb{C})$--action on $\mathcal{Y}_d(\lambda;\mathbb{C}/2\pi i\mathbb{Z})\times\mathcal{N}(\mathbb{C}^d)$ by 
\[g\cdot ((\alpha,\theta),(F_1,F_2,p))= ((\alpha,\theta),(g\cdot F_1,g\cdot F_2,g\cdot p)).\]
Since the $\mathsf{PGL}_d(\mathbb{C})$--action on $\mathcal{N}(\mathbb{C}^d)$ is simply transitive, $\mathcal{sb}_{d,{\bf x}}$ descends to a homeomorphism 
\[\mathfrak{sb}_d:\mathfrak{R}_d(\lambda)\to \mathcal{Y}_d(\lambda;\mathbb{C}/2\pi i\mathbb{Z})\] 
which we call {\em the shear-bend map}. Then, our main theorem is as follows:

\begin{theorem}\label{thm: main} For any triple ${\bf x}\in\wt\Delta^o$ of vertices of a plaque of $\wt \lambda$, the map $\mathcal{sb}_{d,\bf x}$ is a biholomorphism onto its image, which is
	\[\left(\mathcal{C}_d(\lambda)+i\mathcal{Y}_d(\lambda;\mathbb{R}/2\pi\mathbb{Z})\right)\times\mathcal{N}(\mathbb{C}^d),\]
	where $\mathcal{C}_d(\lambda)$ is the image of the map $\mathfrak s_d$, see Theorem \ref{thm: BD par}.
	In particular, the shear-bend map $\mathfrak{sb}_d$ is a homeomorphism onto $\mathcal{C}_d(\lambda)+i\mathcal{Y}_d(\lambda;\mathbb{R}/2\pi\mathbb{Z})$, and it restricts to $\mathfrak{s}_d$ on the Hitchin component ${\rm Hit}_d(S)$.
\end{theorem}

As an immediate consequence of Theorem \ref{thm: main}, we see that $\mathfrak{R}_d(\lambda)$ is a Hausdorff topological space. 
This means that the $\mathsf{PGL}_d(\mathbb{C})$--action on $\mathcal{R}_d(\lambda)$ is proper, and that the complex structure on $\mathcal{R}_d(\lambda)$ descends to the structure of a complex manifold on $\mathfrak{R}_d(\lambda)$.

\subsection{2--pleated surfaces are classical pleated surfaces}\label{sec: application 2--pleated}

We now discuss the relations between Theorem \ref{thm: main} in the case $d = 2$ and Bonahon's shear-bend parameterization of classical pleated surfaces in $\mathbb{H}^3$. Recall from Proposition \ref{classical_pleated} that every classical pleated surface $(\rho,f)$ with pleating locus $\lambda$ determines a unique $2$--pleated surface $(\rho,\xi)$ with same holonomy and pleating locus. Moreover, $\xi$ and $f$ are related by the following property: for every oriented leaf $\mathsf{g} \in \widetilde{\Lambda}^o$, the geodesic $f(\mathsf{g})$ in $\mathbb{H}^3$ has forward and backward endpoints at infinity $\xi(\mathsf{g}^+), \xi(\mathsf{g}^-)$, respectively. We denote by
\[
\mathfrak{sb}_B : \mathfrak{R}(\lambda) \longrightarrow \mathcal{Y}(\lambda,2;\mathbb{C}/2 \pi i \mathbb{Z}).
\]
the Bonahon's shear-bend parameterization of the space $ \mathfrak{R}(\lambda)$ of classical pleated surfaces with pleating locus $\lambda$ (see Section \ref{intro}, and \cite[Theorem~D]{bonahon-toulouse}). The following proposition combined with Proposition \ref{classical_pleated} shows the equivalence between Theorem \ref{thm: main} and Bonahon's original work in the case $d=2$.

\begin{proposition}\label{prop:2-pleated is pleated}
	For every $2$--pleated surface $(\rho,\xi)$ with pleating locus $\lambda$, there exists a $\rho$--equivariant map $f : \widetilde{S} \to \mathbb{H}^3$ such that $(\rho,f)$ is a classical pleated surface with pleating locus $\lambda$ such that $\mathfrak{sb}_B(\rho,f) = \mathfrak{sb}_2(\rho,\xi)$.
\end{proposition}

\begin{proof}
	Let $(\alpha_\rho,\theta_\rho)= \mathfrak{sb}_2(\rho,\xi) \in \mathcal{Y}(\lambda,2;\mathbb{C}/2 \pi i \mathbb{Z})$. Since $d=2$, the map $\theta_\rho$ is trivial, see Remark \ref{rmk: lambda cocyclic of dim 2}, thus we will simplify notation and write $\alpha_\rho= \mathfrak{sb}_2(\rho,\xi)$. By Theorem \ref{thm: main} and \cite[Theorem D]{bonahon-toulouse}, the maps $\mathfrak{sb}_B$ and $\mathfrak{sb}_2$ have the same image inside $\mathcal{Y}(\lambda,2;\mathbb{C}/2 \pi i \mathbb{Z})$. In particular, the $\mathbb{R}$--valued cocycle $\text{Re}\, \alpha_\rho$ coincides with the shear(-bend) coordinates of a totally geodesic pleated surface $f_0 : \widetilde{S} \to \mathbb{H}^2 \subset \mathbb{H}^3$ with Fuchsian holonomy $\rho_0 : \Gamma \to \mathsf{PGL}_2(\mathbb{R})$. 
	
	In this case, the limit map of $\rho_0$ is a homeomorphism from $\partial \widetilde{S}$ to $\partial \mathbb{H}^2 \cong \mathbb{R}\mathrm{P}^1\subset\mathbb{C}\mathrm{P}^1$, which restricts to the $\lambda$--limit map $\xi_0 : \partial \widetilde{\lambda} \to \mathcal{F}(\mathbb{C}^2) \cong \mathbb{C}\mathrm{P}^1$ of $\rho_0$.  The slithering map $\Sigma^0$ of $\xi_0$ has a very clear geometric interpretation in terms of the maximal geodesic lamination $f_0(\widetilde{\lambda}) \subset \mathbb{H}^2$. Given any pair of distinct leaves $h_1, h_2$ of $\widetilde{\lambda}$, the transformation $\Sigma^0(h_1,h_2) \in \mathsf{SL}_2(\mathbb{R})$ sends the leaf $f_0(h_2)$ isometrically onto $f_0(h_1)$, and it is uniquely determined by a natural horocyclic foliation of the region in $\mathbb{H}^2$ delimited by $f_0(h_1), f_0(h_2)$ (compare in particular with the map $\theta^{h_1 h_2}$ from \cite[Section~2]{bonahon-toulouse}, and with \cite[Section~5.1]{BoD}). Thus, it follows from the definitions of the shear-bend coordinate maps $\mathfrak{sb}_B$ and $\mathfrak{sb}_2$ that
	\[
	\mathfrak{sb}_B(\rho_0,f_0) = \mathfrak{sb}_2(\rho_0,\xi_0) = \text{Re}\, \alpha_\rho .
	\]

	Let us now briefly recall the bending procedure described in \cite[Section 8]{bonahon-toulouse}, by which the author constructs, starting from $(\rho_0,f_0)$ and any cocycle $\beta \in \mathcal{Y}(\lambda,2;\mathbb{R}/2\pi \mathbb{Z})$, a classical pleated surface $(\rho', f)$ satisfying $\mathfrak{sb}_B(\rho',f) = \text{Re}\,\alpha_\rho + i \beta$. We fix ${\bf x} = (x_1,x_2,x_3) \in \widetilde{\Delta}^o$ the vertices of a plaque $T_{\bf x}$ of $\wt\lambda$ such that $x_1 \prec x_2, \prec x_3$. Choose an increasing sequence of finite subsets $(\mathcal{X}_n)_n$ of the set of plaques of $\widetilde{\lambda}$ whose union is equal to $\widetilde{\Delta}$. The map $f : \widetilde{S} \to \mathbb{H}^3$ is then defined as the limit of a family of maps $f_n = f_{\mathcal{X}_n} : \widetilde{S} \to \mathbb{H}^3$, each depending on the finite collection of plaques $\mathcal{X}_n$ and constructed as follows.

	For any leaf $h$ of $\wt\lambda$, let $g_{\bf x}^h$ be the edge of $T_{\bf x}$ that separates $T_{\bf x}$ and $h$, and let $\{T_1, \dots, T_k\}$ be an indexing of the set of plaques $\mathcal{X}_n \cap P(g_{\bf x}^h, h)$ so that $T_1 < \cdots < T_k$ with respect to the order of $P(g_{\bf x}^h, h)$. For every $i \in \{1, \dots, k\}$, denote by $\mathsf{g}_i^{\bf x}$ and $\mathsf{g}_i^h$ the oriented edges of $T_i$ that separate $T_{\bf x}$ and $T_i$ and separate $T_i$ and $h$ respectively, so that both $\mathsf{g}_i^{\bf x}$ and $\mathsf{g}_i^h$ are oriented from left to right of some (any) oriented geodesic segment with backward endpoint in $T_{\bf x}$ and forward endpoint in $h$. If $c_{\mathsf{g}_i^{\bf x}}(\vartheta), c_{\mathsf{g}_i^h}(\vartheta) \in \mathsf{PGL}_2(\mathbb{C})$ denote the counter-clockwise rotations by angle $\vartheta \in \mathbb{R}/2\pi \mathbb{Z}$ with axes $f_0(\mathsf{g}_i^{\bf x}), f_0(\mathsf{g}_i^h)$ respectively, we define
	\begin{align}\label{eqn: bending map dim 2}
		\Psi_n(g_{\bf x}^h, h):=c_{\mathsf{g}_1^{\bf x}}(\beta(T_{\bf x}, T_1)) \,& c_{\mathsf{g}_1^h}(-\beta(T_{\bf x}, T_1)) \, c_{\mathsf{g}_2^{\bf x}}(\beta(T_{\bf x}, T_2)) \, c_{\mathsf{g}_2^h}(-\beta(T_{\bf x}, T_2)) \cdots \nonumber\\
		& \cdots c_{\mathsf{g}_k^{\bf x}}(\beta(T_{\bf x}, T_k)) \, c_{\mathsf{g}_k^h}(-\beta(T_{\bf x}, T_k)) 
	\end{align}
	
	For any plaque $T$ of $\wt\lambda$ different from $T_{\bf x}$, we denote by $(g_{\bf x}^T, g_T^{\bf x})$ the separating pair of edges of $(T_{\bf x},T)$. Then the map $f_n$ is given by
	\[f_n(z)=\left\{\begin{array}{ll}
		f_0(z)&\text{if }z\in T_{\bf x},\\
		\Psi_n(g_{\bf x}^T, g_T^{\bf x})\, c_{\mathsf{g}_T^{\bf x}}(\beta(T_{\bf x}, T)) \cdot  f_0(z)&\text{if }z\in T \text{ for some plaque }T\in\mathcal{X}_n,\, T\neq T_{\bf x},\\
		\Psi_n(g_{\bf x}^T, g_T^{\bf x}) \cdot  f_0(z)&\text{if }z\in T \text{ for some plaque }T\notin\mathcal{X}_n,\, T\neq T_{\bf x},\\
		\Psi_n(g_{\bf x}^h, h) \cdot  f_0(z)&\text{if }z\in h \text{ for some leaf }h \text{ of }\wt\lambda.
	\end{array}\right.
	\]
	In other words, $f_n$ is the map $f_0$ bent along the edges of the finite collection of plaques in $\mathcal{X}_n$ according to the bending data given by $\beta$.	The maps $f_n$, while not equivariant, are still $1$--Lipschitz and send the leaves of $\wt\lambda$ to geodesics in $\mathbb{H}^3$. Thus, for each $n$, we may define the naturally associated limit maps 
	\[\xi_n : \partial \widetilde{\lambda} \to \partial\mathbb{H}^3\cong\mathbb{C} \mathrm{P}^1\]
	by setting $\xi_n(z)$ to be the forward endpoint of $f_n(\mathsf{g})$ for some (any) oriented leaf $\mathsf{g}$ of $\wt\lambda$ whose forward endpoint is $z$. 
	
	Since $\mathcal{X}_n$ is finite, the fact that $\xi_0$ is locally $\lambda$--H\"older continuous and $\lambda$--transverse implies that the same is true for $\xi_n$, so Theorem \ref{thm: slitherable and slithering} implies that $\xi_n$ admits a unique compatible slithering map $\Sigma^n$. Furthermore, it is straightforward to verify using the uniqueness part of Theorem \ref{thm: slitherable and slithering} that $\Sigma^n$ has the following description in terms of the slithering map $\Sigma^0$ of $\xi_0$ and $\Psi_n$: Given leaves $h_1$ and $h_2$ of $\wt\Lambda$,
	\begin{equation}\label{slithering step n}
		\Sigma^n(h_1,h_2) = \Psi_n(g_{\bf x}^{h_1}, h_1)\,\Sigma^0(h_1,h_2) \, \Psi_n(g_{\bf x}^{h_2}, h_2)^{-1}.
	\end{equation}
	In particular, if $T$ is a plaque of $\wt\lambda$ that is not $T_{\bf x}$, then
	\[\Sigma^n(g_{\bf x}^T, g_T^{\bf x}) = \Sigma^0(g_{\bf x}^T, g_T^{\bf x}) \, \Psi_n(g_{\bf x}^T, g_T^{\bf x})^{-1}.\]
	
	Let $T$ be any plaque of $\wt\lambda$ distinct from $T_{\bf x}$. We will now verify that if $n$ is large enough so that $T\in\mathcal{X}_n$, then the map $\xi_n : \partial \widetilde{\lambda} \to \mathbb{C}\mathrm{P}^1$ satisfies
	\begin{equation}\label{correct bending}
		\log(D^{(1,1)}(\xi_n(y_2),\xi_n(y_1), \xi_n(y_3), \Sigma^n(g_{\bf x}^T, g_T^{\bf x}) \, \xi_n(z_T))) = \text{Re}\,\alpha_\rho(T_{\bf x},T) + i \beta(T_{\bf x},T) ,
	\end{equation}
	where $(y_1,y_2,y_3)$ is a counter-clockwise labeling of the vertices of $T_{\bf x}$ such that $\mathsf{g}_{\bf x}^T$ has endpoints $y_1, y_2$, and $z_T$ is the vertex of $T$ that is not an endpoint of $g_{T}^{\bf x}$. Notice that the double ratio $D(a,b,c,d) = D^{(1,1)}(a,b,c,d)$ coincides (up to a sign) with the cross ratio of the quadruple of points $a,b,c,d$ in $\mathbb{C}\mathrm{P}^1$. An elementary computation shows that, if $c_{\mathsf{g}}(\vartheta)$ is the counter-clockwise rotation of angle $\vartheta$ around the oriented geodesic $\mathsf{g}$ in $\mathbb{H}^3$ with forward and backward endpoints $a$ and $b$, respectively, then
	\[
	D(a,b,c, c_{\mathsf{g}}(\vartheta) \, d) = e^{i \vartheta} \, D(a,b,c,d) .
	\]
	Therefore, by applying relation \eqref{slithering step n}, we deduce
	\begin{align*}
		\log(D^{(1,1)}&(\xi_n(y_2),\xi_n(y_1), \xi_n(y_3), \Sigma^n(g_{\bf x}^T, g_T^{\bf x}) \, \xi_n(z_T))) = \\
		& = \log(D^{(1,1)}(\xi_0(y_2),\xi_0(y_1), \xi_0(y_3), \Sigma^0(g_{\bf x}^T, g_T^{\bf x}) \, c_{\mathsf{g}_T^{\bf x}}(\beta(T_{\bf x},T)) \, \xi_0(z_T)) \\
		& = \log(D^{(1,1)}(\xi_0(y_2),\xi_0(y_1), \xi_0(y_3), c_{\mathsf{g}_{\bf x}^T}(\beta(T_{\bf x},T)) \, \Sigma^0(g_{\bf x}^T, g_T^{\bf x}) \, \xi_0(z_T)) \\
		& = \log(D^{(1,1)}(\xi_0(y_2),\xi_0(y_1), \xi_0(y_3), \Sigma^0(g_{\bf x}^T, g_T^{\bf x}) \, \xi_0(z_T)) + i \, \beta(T_{\bf x},T) \\
		& = \text{Re}\,\alpha_\rho(T_{\bf x},T) + i \, \beta(T_{\bf x},T) ,
	\end{align*}
	where in the first identity we noticed that $f_n|_{T_{\bf x}} = f_0|_{T_{\bf x}}$, and in the third identity we applied the relation $c_{\mathsf{g}_{\bf x}^T}(\vartheta) = \Sigma^0(\mathsf{g}_{\bf x}^T,\mathsf{g}_T^{\bf x}) \, c_{\mathsf{g}_T^{\bf x}}(\vartheta) \, \Sigma^0(\mathsf{g}_T^{\bf x},\mathsf{g}_{\bf x}^T)$.

	In \cite[Section~8]{bonahon-toulouse}, Bonahon showed that the maps $f_n$ converge uniformly on compact subsets to a pleated surface $f : \widetilde{S} \to \mathbb{H}^3$ which is equivariant with respect to some representation $\rho' : \Gamma \to \mathsf{PGL}_2(\mathbb{C})$, and such that $\mathfrak{sb}_B(\rho',f) = \text{Re}\,\alpha_\rho + i \beta$. It follows that $\xi_n$ converges to $\xi$ pointwise, and the induced maps 
	\[\wt\Lambda^o\to (\mathbb{C} \mathrm{P}^1)^2\]
	converges uniformly on compact subsets as well. Thus, the slithering maps $\Sigma^n$ converge uniformly on compact subsets to the slithering map $\Sigma$ compatible with $\xi$. Then by taking the limit as $n\to\infty$, we conclude that $\mathfrak{sb}_2(\rho',\xi) = \text{Re}\, \alpha_\rho + i \, \beta$. 
	
	If we apply the procedure described above to $\beta = \text{Im } \alpha_\rho$, then we deduce that there exists a classical pleated surface $(\rho',f)$ with associated $2$--pleated surface $(\rho',\xi)$ that satisfies
	\[
	\mathfrak{sb}_2(\rho',\xi) = \mathfrak{sb}_B(\rho',f) = \alpha_\rho = \mathfrak{sb}_2(\rho,\xi) .
	\] 
	Since the map $\mathfrak{sb}_2$ is injective, up to replacing $(\rho,\xi)$ with $(A \circ \xi,A \circ \rho(\cdot) \circ A^{-1})$ for some $A \in \mathsf{PGL}_2(\mathbb{C})$, we obtain that $(\rho,\xi)$ coincides with the $2$--pleated surface associated the classical pleated surface $(\rho',f)$, as desired.
\end{proof}

\subsection{Outline of the proof of the main theorem}\label{subsec:proof structure}

Recall from Section~\ref{sec:shear_bend} that the definition of the shear-bend coordinates $\mathcal{sb}_{d,{\bf x}}(\rho,\xi)$ of a $d$--pleated surface $(\rho,\xi)$ requires the existence of a slithering map $\Sigma : \widetilde{\Lambda}^2 \to \mathsf{SL}_d(\mathbb{C})$ compatible with the $\lambda$--limit map $\xi$ (see Definition \ref{def: slithering map}). 

For this reason, our first goal will be to prove Theorem \ref{thm: slitherable and slithering} which establishes existence and uniqueness of a slithering map compatible with the $\lambda$--limit map $\xi$.
	
The slithering map $\Sigma(g_1,g_2)$ between two leaves $g_1$, $g_2$ will be obtained as limit of a suitable composition of finitely many unipotent transformations. To establish the convergence of such compositions, we will need careful estimates on a specific family of unipotent elements, each associated with a plaque of $\widetilde{\lambda}$ that separate $g_1$ and $g_2$. This part of our analysis is a natural generalization of the construction of the slithering map for Hitchin representations from Bonahon-Dreyer \cite[Section~5.1]{BoD}. Since the same technical estimates will be needed in other parts of our work (see in particular Steps 1, 2, and 4 described below), we will present the key convergence results in a more general framework. This framework, which involves the study of so called \emph{extensions of H\"older extendable maps} (see the beginning of Section~\ref{sec:productable} for the necessary definitions), will make the results better suited for multiple applications. Section~\ref{sec:productable} is dedicated to the study of extensions of H\"older extendable maps. The existence (and uniqueness) of the slithering map compatible with a limit map $\xi$ will be established later in Section~\ref{sec: slithering}.
	
	The remainder of the paper is dedicated specifically to the proof of Theorem \ref{thm: main}. The structure of the argument can be divided into four main steps, each of which establishes one of the following properties:
	\begin{itemize}
		\item[Step~1:] The map $\mathcal{sb}_{d,{\bf x}}$ is injective.
		\item[Step~2:] The image of $\mathcal{sb}_{d,{\bf x}}$ contains $\big(\mathcal{C}_d(\lambda)+i\mathcal{Y}_d(\lambda;\mathbb{R}/2\pi\mathbb{Z})\big)\times \mathcal{N}(\mathbb{C}^d)$.
		\item[Step~3:] The image of $\mathcal{sb}_{d,{\bf x}}$ is contained in $\big(\mathcal{C}_d(\lambda)+i\mathcal{Y}_d(\lambda;\mathbb{R}/2\pi\mathbb{Z})\big)\times \mathcal{N}(\mathbb{C}^d)$.
		\item[Step~4:] The map $\mathcal{sb}_{d,{\bf x}}$ is a biholomorphism onto its image.
	\end{itemize}
	
	We now summarize the proofs of each step while providing references to the sections and statements where these are discussed.
	
	\subsubsection*{The proof of Step~1 (see Section~\ref{sec:eruption and shear})}
	As we will discuss in Section~\ref{sec: injectivity proof}, the $\lambda$--limit map $\xi$ of a $d$--pleated surface determines the representation $\rho$, and it is determined by the slithering map $\Sigma$ compatible with $\xi$. Hence, in order to prove that the map $\mathcal{sb}_{d,{\bf x}}$ is injective, it suffices to show that the slithering map $\Sigma$ is uniquely determined by the shear-bend coordinates $\mathcal{sb}_{d,{\bf x}}(\rho)$. We will do so by using H\"older extendable maps to find a formula for $\Sigma$ in terms of the shear-bend parameters $(\alpha_\xi,\theta_\xi)$ of $\xi$, the choice of a base plaque with labeling ${\bf x} = (x_1,x_2,x_3) \in \widetilde{\Delta}^o$, and of the projective frame determined by $(\xi(x_1),\xi(x_2),\xi(x_3)^1)$ (see Lemma \ref{lem: injectivity}). In the case of Hitchin representations, a version of this expression was also obtained by Bonahon and Dreyer \cite{BoD}. In the first part of Section~\ref{sec:eruption and shear} (namely, Sections \ref{sec: complex eruption}, \ref{sec: complex shear}, and \ref{sec: unipotent eruption}) we will introduce the notions of complex shears, complex eruptions, and unipotent eruptions. These tranformations will appear in the explicit expession of the slithering map that we mentioned above and will play a crucial role in the process of generalized bending (see Step~2 below).
	
	\subsubsection*{The proof of Step~2 (see Section \ref{sec:bending})}
	
	This step is inspired by ideas that Bonahon \cite{bonahon-toulouse} used to prove the surjectivity of the shear-bend parameterization $\mathfrak{sb}_B$ of the space of classical pleated surfaces with pleating locus $\lambda$. However, the much more general context of representations in $\mathsf{PGL}_d(\mathbb{C})$ rather than $\mathsf{PGL}_2(\mathbb{C})$ makes the analysis considerably more technical and subtle than the one in \cite{bonahon-toulouse}.
	
	Recall that $\mathcal C_d(\lambda)$ denotes the image of Bonahon and Dreyer's shearing map $\mathfrak{s}_d$ for the $\mathsf{PGL}_d(\mathbb{R})$--Hitchin component, see Theorem \ref{thm: BD par}. Hence, once we fix
	\[\left((\alpha,\theta), (F, G, p)\right)\in\mathcal{C}_d(\lambda)+i\mathcal{Y}_d(\lambda;\mathbb{R}/2\pi\mathbb{Z})\subset\mathcal{Y}_d(\lambda;\mathbb{C}/2\pi i\mathbb{Z}) \times \mathcal{N}(\mathbb{C}^d),\]
	there exists a unique representative of the $\mathsf{PGL}_d(\mathbb{C})$--conjugacy class of a $\mathsf{PGL}_d(\mathbb{R})$--Hitchin representation $\rho_0$ with $\lambda$--limit map $\xi_0$ whose shear-bend coordinates are given by
	\[
	(\alpha_{\xi_0},\theta_{\xi_0})=(\mathrm{Re}\,\alpha,\mathrm{Re}\,\theta)\,\,\text{ and }\,\,(\xi_0(x_1), \xi_0(x_2), \xi^1_0(x_3)) =  (F, G, p).
	\]
	To see that the point $((\alpha,\theta), (F, G, p))$ lies in the image of $\mathcal{sb}_{d,{\bf x}}$, we will introduce a suitable process of \emph{generalized bending} that takes in input $(\rho_0,\xi_0)$ and the $\lambda$--cocyclic pair $(\mathrm{Im}\,\alpha, \mathrm{Im}\, \theta) \in \mathcal{Y}_d(\lambda;\mathbb{R}/2 \pi i \mathbb{Z})$, and produces a $d$--pleated surface $(\rho,\xi)$ whose shear-bend parameters are given by
	\begin{equation}\label{eq:identity shear-bend}
		\mathcal{sb}_{d,{\bf x}}(\rho,\xi) = \mathcal{sb}_{d, {\bf x}}(\rho_0,\xi_0) + i \, (\mathrm{Im}\,\alpha, \mathrm{Im}\, \theta) = (\alpha, \theta) ,
	\end{equation}
	and such that $(\xi(x_1), \xi(x_2), \xi(x_3)^1) = (F,G,p)$. We will find such $(\rho,\xi)$ by constructing a function
	\[
	\Psi:\wt\Lambda\to\mathsf{SL}_d(\mathbb{C}) ,
	\] 
	called the \emph{$(\mathrm{Im}\,\alpha, \mathrm{Im}\,\theta)$--generalized bending map} (see Definition \ref{def: bending map}). Note that $\Psi$ depends on $(\rho_0, \xi_0)$ and a choice of labeling ${\bf x}$ of the vertices of the base plaque. The map $\Psi$ is uniquely characterized by a list of natural properties and its construction (described in Section~\ref{sec:existence}) heavily relies on the technical tools on H\"older extendable maps developed in Section~\ref{sec:productable} and on the notions of complex shears, complex eruptions, and unipotent eruptions from Section~\ref{sec:eruption and shear}. Such properties imply in particular that the map
	\[\xi:\partial\wt\lambda\to\mathcal{F}(\mathbb{C}^d) , \] 
	given by $\xi(z) = \Psi(g)\cdot\xi_0(z)$ for some (any) leaf $g$ that has $z$ as an endpoint, is the $\lambda$--limit map of a $d$--pleated surface $(\rho,\xi)$ that verifies condition \eqref{eq:identity shear-bend}. The existence and uniqueness of an $(\mathrm{Im}\,\alpha, \mathrm{Im}\,\theta)$--generalized bending map will be established in Theorem \ref{thm: bending Hitchin} and Proposition \ref{prop: uniqueness of bending}. In fact, we will show that generalized bending maps exist not only for representations $\rho_0$ that are conjugate to $\mathsf{PGL}_d(\mathbb{R})$--Hitchin representations, but also for a larger class of representations in $\mathcal{R}_d(\lambda)$ (see Definition~\ref{def:prop_star} and Theorem~\ref{thm_exist-psi} for details on this).
	
	\subsubsection*{The proof of Step~3}
	
	The proof of Step~3 formally follows the same strategy used by Bonahon and Dreyer \cite{BoD} to show that the image of the shearing map $\mathfrak{s}_d$ lies inside the cone $\mathcal{C}_d(\lambda) \subset \mathcal{Y}_d(\lambda;\mathbb{R})$. This said, it is worth mentioning that the differences between our setting and the case of Hitchin representations introduce several technical difficulties as one tries to adapt the argument from \cite{BoD}. We expand on this aspect after the description of the proof of Step~3.
	
	As we discussed in Section~\ref{subsec:length_functions}, for any $\lambda$--Borel Anosov representation $\rho:\Gamma\to\mathsf{PGL}_d(\mathbb{C})$, the contraction property of $\rho$ allows one to define, for each ${\bf i}\in\mathcal{A}$, a positive valued length function 
	\[\ell^\rho_{\bf i}:\mathsf{M}(\lambda^o)\to\mathbb{R}.\]
	These can be considered as a generalization of the length functions of Hitchin representations studied by Dreyer \cite{dreyer} and later deployed in \cite{BoD}. The crux of the proof of Step 3 will be the construction, for any $d$--pleated surface $(\rho,\xi)\in\mathcal{R}_d(\lambda)$ and for any ${\bf i}\in\mathcal{A}$, of a particular topological closed $1$--form $\Omega_{\bf i}^\rho$ on the oriented double cover $N^o$ of a train track neighborhood $N$ of $\lambda$ (see Section~\ref{sec: construction 1 form}). This topological closed $1$--form will define a cohomology class $[\Omega_{\bf i}^\rho]\in H^1(N^o;\mathbb{R})$. 
	
	On the one hand, the cohomology class $[\Omega_{\bf i}^\rho]\in H^1(N^o;\mathbb{R})$ is constructed so that the quantity $\ell^\rho_{\bf i}(\mu) > 0$ is obtained by integrating $[\Omega_{\bf i}^\rho]$ on a homology class $[\mu] \in H_1(N^o;\mathbb{R})$ naturally associated with the transverse measure $\mu \in \mathsf{M}(\lambda^o)$ (see Proposition~\ref{prop: second_extension_forms}). On the other hand, the same evaluation $\langle [\Omega_{\bf i}^\rho] , [\mu] \rangle$ can be interpreted as a cup pairing between the Poincar\'e dual of $[\mu]$ in $H^1(N^o,\partial N^o;\mathbb{R})$ and the cohomology class $[\Omega_{\bf i}^\rho]$. From a careful study of the integration of $\Omega_{\bf i}^\rho$ along the ties of the train track $N^o$ (see Proposition~\ref{prop: the rise of the shear}), we will relate $\langle [\Omega_{\bf i}^\rho] , [\mu] \rangle$ to the (${\bf i}$--th component of the) $\mathbb{R}^{\mathcal{A}}$--valued algebraic intersection $\mathbb{I}$ (see Proposition \ref{prop: duality}) between $[\mu]$ and a specific homology class $[\mathrm{Re}\,\alpha_\rho]$ associated with the real part of the shear-bend parameters $(\alpha_\rho, \theta_\rho)$ of $(\rho,\xi)$ (see Section~\ref{length_cyclic} for the definition of $[\mathrm{Re}\,\alpha_\rho]$). We will finally conclude that the image of the shear-bend coordinates map $\mathcal{sb}_{d, {\bf x}}$ lies inside
	\[
	\{(\alpha, \theta) \in \mathcal{Y}_d(\lambda;\mathbb{C}/2 \pi i \mathbb{Z}) \mid \mathbb{I}([\mu],[\mathrm{Re}\;\alpha])_{\bf i}> 0 \text{ for all ${\bf i} \in \mathcal{A}$} \} \times \mathcal{N}(\mathbb{C}^d) ,
	\]
	which is precisely equal to $\big(\mathcal{C}_d(\lambda)+i\mathcal{Y}_d(\lambda;\mathbb{R}/2\pi\mathbb{Z})\big)\times \mathcal{N}(\mathbb{C}^d)$. The homological interpretation of $\lambda$--cocyclic pairs and the definition of the algebraic intersection pairing will be discussed in Section~\ref{sec: intersection}, while the construction of the topological closed $1$--form and the proof of the identity $\ell^\rho_{\bf i}(\mu) = \mathbb{I}([\mu],[\mathrm{Re}\;\alpha])_{\bf i}$ are described in Section~\ref{length} (see in particular Theorem~\ref{thm: two lengths}).

	As we briefly mentioned above, there are some key differences between the context of $d$--pleated surfaces and of Hitchin representations that make several parts of the argument in Step~3 technically more challenging than their analogs in \cite{BoD}. In particular, our $\lambda$--limit maps $\xi : \partial \widetilde{\lambda} \to \mathcal{F}(\mathbb{C}^d)$ are not defined on the whole Gromov boundary $\partial \widetilde{S}$ and are not necessarily continuous nor hyperconvex. Hence, the line bundles involved in the definition of the length functions $\ell^\rho_{\bf i}$ are not defined on the whole unit tangent bundle $T^1 S$, but only on $T^1 \lambda$, which is a Hausdorff dimension $1$ subset. This makes the construction of the topological closed $1$--form $\Omega^\rho_{\bf i}$ more delicate than the one of its counterpart in \cite{BoD}. Furthermore, the space of $\lambda$--cocyclic pairs $\mathcal{Y}_d(\lambda;\mathbb{C}/2 \pi i \mathbb{Z})$ does not admit a natural structure of finitely generated module (over a ring), as opposed to the vector space $\mathcal{Y}_d(\lambda;\mathbb{R})$ appearing as target space of the shearing map from \cite{BoD}. Consequently, the homological interpretation of the space $\mathcal{Y}_d(\lambda;\mathbb{C}/2 \pi i \mathbb{Z})$ and the definition of the appropriate algebraic intersection pairing require subtle adaptations. Indeed, the analog arguments in \cite{BoD} often rely on the Poincaré duality theorem, which requires the group of coeffients of the cohomology/homology groups to be a commutative ring with identity and not just an additive group like $\mathbb{C}/2 \pi i \mathbb{Z}$ (which is not even finitely generated as a $\mathbb{Z}$--module).
	
	\subsubsection*{The proof of Step 4} 
	
	Steps 1, 2, and 3 will imply that the shear-bend map $\mathcal{sb}_{d,{\bf x}}$ is a continuous, bijective map from the space $\mathcal{R}_d(\lambda)$ onto
	\[
	\big(\mathcal{C}_d(\lambda)+i\mathcal{Y}_d(\lambda;\mathbb{R}/2\pi\mathbb{Z})\big)\times \mathcal{N}(\mathbb{C}^d) .
	\]
	Since the image of $\mathcal{sb}_{d,{\bf x}}$ carries a natural structure of complex manifold, it will be sufficient to show that the inverse $\mathcal{sb}_{d,{\bf x}}^{-1}$ is holomorphic. In other words, we will prove that the holonomy $\rho : \Gamma \to \mathsf{PGL}_d(\mathbb{C})$ of a $d$--pleated surface $(\rho,\xi)$ depends holomorphically on the data of its shear-bend parameters $(\alpha_\rho,\theta_\rho)$ and on the projective frame determined by $(\xi(x_1),\xi(x_2), \xi(x_3)^1)$. Since the representation $\rho$ can be described using the slithering map compatible with the $\lambda$--limit map $\xi$, it will be enough to prove that the slithering map depends holomorphically on the shear-bend parameters. This will be deduced in Section~\ref{sec:holomorphicity} from the same explicit expression of the slithering map of a $d$--pleated surface $(\rho,\xi)$ that we mentioned in the description of the proof of Step 1 (that is, Lemmas~\ref{lem: injectivity} and \ref{lem: uniform}).

\section{H\"older extendable maps}\label{sec:productable}

In this section, we develop tools needed to study the convergence properties of certain infinite products of matrices. These infinite products are described using the formalism of a \emph{H\"older extendable map} as defined below. This approach leads to a unified theory which encompasses the slithering map, defined in Section \ref{shear-bend} (see also Section \ref{sec: slithering}), and the bending map, defined in Section \ref{sec:bending}.

Let $\norm{\cdot}$ denote the operator norm on $\mathsf{SL}_d(\mathbb{K})$ with respect to the standard (Hermitian) inner product on $\mathbb{K}^d$. Fix a pair of leaves $g_1$ and $g_2$ of $\wt\lambda$, and recall that $P(g_1,g_2)$ denotes the set of plaques of $\wt\lambda$ that separate the leaves $g_1$ and $g_2$. Recall also that for any plaque $T\in P(g_1,g_2)$, we denote by $h_{T,-}$ and $h_{T,+}$ the two edges of $T$ that separate $g_1$ and $g_2$, so that $h_{T,-}$ separates $g_1$ and $h_{T,+}$. Then we say that $(x_{T,1}, x_{T,2}, x_{T,3})$ is a coherent labeling of the vertices of $T$ if $x_{T,2}$ is the common endpoint of $h_{T,-}$ and $h_{T,+}$, and $x_{T,1}$ and $x_{T,3}$ are respectively the endpoints of $h_{T,-}$ and $h_{T,+}$ that are not $x_{T,2}$.

\begin{definition}\label{def_productable}
	A map $$M\co P(g_1,g_2)\to \mathsf{SL}_d(\mathbb{K})$$ is \emph{H\"older extendable} if there exist constants $A,\nu > 0$ such that for all plaques $T\in P(g_1,g_2)$,
	\begin{equation}\label{prod}
		\norm{M(T)-\id}\le A\,d_\infty(x_{T,1},x_{T,3})^\nu,
	\end{equation}
	where $(x_{T,1},x_{T,2},x_{T,3})$ is a coherent labeling of the vertices of $T$ with respect to $(g_1,g_2)$.
\end{definition}

The next proposition is the main convergence result for H\"older extendable maps and it will be established in Section \ref{sec:existence enveloping}. To state it, we use the following notation and terminology:
\begin{itemize}
	\item Given a map $M\co P(g_1,g_2)\to \mathsf{SL}_d(\mathbb{K})$, we may define a map
	\[\overrightarrow{M}^F\co \{\text{finite subsets of }P(g_1,g_2)\}\to\mathsf{SL}_d(\mathbb{K})\]
	by $\overrightarrow{M}^F(\mathcal{X}):= M(T_1)\dots M(T_n)$, where $\mathcal{X}=\{T_1,\dots,T_n\}$ is enumerated so that $T_1<\dots<T_n$ in the total order on $P(g_1,g_2)$, defined in Section \ref{section: separation and coherence}.
	\item Let $h_1$ and $h_2$ be leaves of $\wt\lambda$ that separate $g_1$ and $g_2$. We say that a sequence of finite subsets $(\mathcal{X}_n)_{n=1}^\infty$ in $P(h_1,h_2)$ is \emph{exhausting} if $\mathcal{X}_n\subseteq\mathcal{X}_{n+1}$ for all $n$, and
	\[\bigcup_{n=1}^\infty\mathcal{X}_n=P(h_1,h_2).\]
\end{itemize}

\begin{proposition}\label{prop: existence of envelopes}
	Let $M \co P(g_1,g_2)\to \mathsf{SL}_d(\mathbb{K})$ be a H\"older extendable map. 
	If $h_1$ and $h_2$ are leaves of $\wt\lambda$ that separate $g_1$ and $g_2$, then the limit
	\[\lim_{n\to\infty}\overrightarrow{M}^F(\mathcal{X}_n)\] 
	exists for any exhausting sequence $(\mathcal{X}_n)_{n=1}^\infty$ in $P(h_1,h_2)$. Furthermore, this limit lies in $\mathsf{SL}_d(\mathbb{K})$, and does not depend on the exhausting sequence $(\mathcal{X}_n)_{n=1}^\infty$.
\end{proposition}

Recall that $Q(g_1,g_2)$ denotes the set of leaves of $\wt\lambda$ that separate $g_1$ and $g_2$. By Proposition \ref{prop: existence of envelopes}, we may define a map
\[\overrightarrow{M}:Q(g_1,g_2)^2\to\mathsf{SL}_d(\mathbb{K})\]
by
\[\overrightarrow{M}(h_1, h_2):=\left\{\begin{array}{ll}
	\displaystyle\lim_{n\to\infty}\overrightarrow{M}^F(\mathcal{X}_n)&\text{if }h_1\leq h_2,\\
	\displaystyle\left(\lim_{n\to\infty}\overrightarrow{M}^F(\mathcal{X}_n)\right)^{-1}&\text{if }h_2\leq h_1,\\
\end{array}\right.\]
where $\leq$ is the total order on $Q(g_1,g_2)$ and $(\mathcal{X}_n)_{n=1}^\infty$ is any exhausting sequence in $P(h_1,h_2)$. We refer to $\overrightarrow{M}$ as the \emph{extension} of $M$.

\begin{remark}\label{rem: obvious}
	If two leaves $g_3,g_4\in Q(g_1,g_2)$ satisfy $g_1\leq g_3\leq g_4\leq g_2$, then $Q(g_3,g_4)\subseteq Q(g_1,g_2)$ as ordered sets, and any H\"older extendable map $M\colon P(g_1,g_2)\to\mathsf{SL}_d(\mathbb{K})$ restricts to a H\"older extendable map $M'\colon P(g_3,g_4)\to\mathsf{SL}_d(\mathbb{K})$. It follows that $\overrightarrow{M}'(h_1,h_2)=\overrightarrow{M}(h_1,h_2)$ for any pair of leaves $h_1,h_2\in Q(g_3,g_4)$.
\end{remark}

The following proposition records several basic properties of extensions. Their proofs are omitted as they are straightforward consequences of the definitions.

\begin{proposition}\label{prop:cocycle}\label{prop: productable fixed}
	Let $M \co P(g_1,g_2)\to \mathsf{SL}_d(\mathbb{K})$ be a H\"older extendable map with extension $\overrightarrow{M}\co Q(g_1,g_2)\to \mathsf{SL}_d(\mathbb{K})$.
	\begin{enumerate} 
		\item For any leaf $h\in Q(g_1,g_2)$, 
		\[\overrightarrow{M}(h,h)=\id.\]
		\item For any $h_1,h_2,h_3\in Q(g_1,g_2)$,
		\[\overrightarrow{M}(h_1,h_3)=\overrightarrow{M}(h_1,h_2)\overrightarrow{M}(h_2,h_3).\]
		\item Let $g_3\in Q(g_1,g_2)$ and assume that there exists a flag $F\in\mathcal{F}(\mathbb{K}^d)$ that is fixed by $M(T)$ for every plaque $T\in Q(g_3,g_2)$. Then 
		\[\overrightarrow{M}(g_1,g_3)\cdot F=\overrightarrow{M}(g_1,g_2)\cdot F.\]
	\end{enumerate}
\end{proposition}

We also prove that the extensions of H\"older extendable maps are H\"older continuous:

\begin{proposition}\label{prop: Holder productable} 
	If $M \co P(g_1,g_2)\to \mathsf{SL}_d(\mathbb{K})$ is a H\"older extendable map, then its extension $\overrightarrow{M}$ is H\"older continuous, i.e. there exist constants $A',\nu'>0$ such that, given any two pairs of leaves $(h_1,h_2),(\ell_1,\ell_2)\in Q(g_1,g_2)^2$, 
	\[\norm{\overrightarrow{M}(h_1,h_2)-\overrightarrow{M}(\ell_1,\ell_2)}\leq A' d_\infty((h_1,h_2),(\ell_1,\ell_2))^{\nu'}.\]
\end{proposition}

Finally, we prove the following result, which gives us a way to construct new H\"older extendable maps from two existing ones. Given any plaque $T\in P(g_1,g_2)$, let $h_{T,-},h_{T,+}$ be the edges of $T$ that lie in $Q(g_1,g_2)$ such that $h_{T,-}<h_{T,+}$.

\begin{proposition}\label{prop: renormalization} 
	Let $M,N \co P(g_1,g_2)\to \mathsf{SL}_d(\mathbb{K})$ be H\"older extendable maps with extensions $\overrightarrow{M}$ and $\overrightarrow{N}$ respectively. Then the map
	\[N' \co P(g_1,g_2)\to \mathsf{SL}_d(\mathbb{K})\]
	defined by
	\[N'(T):=\overrightarrow{M}(g_1,h_{T,-})N(T)\overrightarrow{M}(h_{T,+},g_1)\]
	is H\"older extendable and its extension $\overrightarrow{N}'$ satisfies
	\[\overrightarrow{N}'(g_1,h)=\overrightarrow{N}(g_1,h)\overrightarrow{M}(h,g_1)\]
	for all $h\in Q(g_1,g_2)$.
\end{proposition}

The proofs of Propositions \ref{prop: Holder productable}, and \ref{prop: renormalization} will occupy Sections \ref{sec:Holder enveloping} and \ref{sec:combine productable}, respectively. 

\subsection{Hyperbolic geometry estimates}\label{train_tracks_estimates}
First, we discuss several important estimates that will be crucial for the later proofs. 

For any rectifiable curve $k\subset\wt S$, let $\ell(k)$ denote the hyperbolic length of $k$. The next lemma is a classical estimate due to Bonahon \cite{bonahon-toulouse}, see also Mazzoli-Viaggi \cite[Appendix~C]{MV22} for a proof.

\begin{lemma}[\cite{bonahon-toulouse} Lemma 4 and 5, \cite{BoD} Lemma 5.3]\label{lem: lamination property}
	Let $k$ be a geodesic segment in $\wt S$ that intersects both $g_1$ and $g_2$. Then there exists a function 
	\[r\co P(g_1,g_2)\to\mathbb{Z}^+\] 
	that satisfies the following properties: 
	\begin{enumerate}
		\item There are constants $E,E',\mu,\mu'>0$ such that 
		$$Ee^{-\mu r(T)}\leq\ell(k\cap T)\leq E'e^{-\mu'r(T)}$$
		for every plaque $T\in P(g_1,g_2)$.
		\item There exists some constant $E''>0$ such that $|r^{-1}(m)|<E''$ for all positive integers $m$.
	\end{enumerate}
	In particular, for any $\nu>0$, the sum
	\[\sum_{T\in P(g_1,g_2)}\ell(k\cap T)^\nu\]
	converges.
\end{lemma}

The following is a standard hyperbolic geometry fact.

\begin{lemma}\label{lem: hyperbolic geometry}
	Let $k$ be a geodesic arc in $\wt S$ that intersects $g_1$ and $g_2$. Then there exists a constant $C> 0$ with the following property: If $T$ is an ideal triangle in $\wt S$ that separates $g_1$ and $g_2$, $h_1$ and $h_2$ are the edges of $T$ that separate $g_1$ and $g_2$, and $x_1$ and $x_2$ are respectively the endpoints of $h_1$ and $h_2$ that are not the common endpoint of $h_1$ and $h_2$, then
	\[\frac{1}{C}d_\infty(x_1,x_2)\leq \ell(k\cap T)\leq Cd_\infty(x_1,x_2).\]
\end{lemma}

\begin{proof}
	Let $J\subset\partial\wt S$ denote the set of endpoints of geodesics that separate $g_1$ and $g_2$, and note that $J$ is the union of two disjoint, closed subintervals of $\partial\wt S$, call them $J_0$ and $J_1$. For every $y\in J$, let $a,b\in\partial\wt S$ be the points such that the geodesic in $\wt S$ with endpoints $y$ and $a$ and the geodesic in $\wt S$ with endpoints $y$ and $b$ each contain an endpoint of $k$. Let $I_y$ be the closed subinterval of $\partial\wt S$ that does not contain $y$, and has $a$ and $b$ as its endpoints. Then we may define
	\[f_y: I_y\to k\] 
	to be the map that sends every $x\in I_y$ to the intersection between $k$ and the geodesic with endpoints $y$ and $x$. Observe that $f_y$ is a diffeomorphism. Also, for both $j=0,1$, $y\in J_i$ implies that $J_{1-i}$ is a compact subset that lies in the interior of $I_y$. Thus, if $y\in J_i$, then the norm $\norm{(df_y)_x}_x$ of the differential $(df_y)_x$ has a uniform positive lower bound over all points $x\in J_{1-i}$ (here, the norms $\norm{\cdot}_x$ are with respect to the Riemannian metric on $\partial\wt S$ that induces $d_\infty$, and with respect to the restriction of the hyperbolic metric on $\wt S$ to $k$). Furthermore, for every $x\in J_{1-i}$, $(df_y)_x$ varies continuously over all $y\in J_i$. Thus, by the compactness of $J_0\times J_1$, there exists $C>1$ such that
	\[\frac{1}{C}\leq \norm{(df_y)_x}_x\leq C\]
	for all pairs $(x,y)\in (J_0\times J_1)\cup (J_1\times J_0)$. 
	
	Now, let $y_0$ denote the common endpoint of $h_1$ and $h_2$, and let $L$ be the subinterval in $J$ with endpoints $x_1$ and $x_2$. Then
	\[\frac{1}{C} d_\infty(x_1,x_2)=\frac{1}{C}\int_Ldx\leq\int_L\norm{(df_{y_0})_x}_x dx\le C\int_Ldx=C d_\infty(x_1,x_2).\]
	Since $\ell (k\cap T)=\int_L\norm{(df_{y_0})_x}_x dx$, the lemma follows.
\end{proof}

Lemmas \ref{lem: lamination property} and \ref{lem: hyperbolic geometry} together give the following consequence.
\begin{lemma}\label{lem: Cauchy}
	For all $T\in P(g_1,g_2)$, let $(x_{T,1},x_{T,2},x_{T,3})$ be a coherent labeling of the vertices of $T$ with respect to $(g_1,g_2)$. Then for all $A,\nu>0$ we have that:
	\[\sum_{T\in P(g_1,g_2)} d_\infty(x_{T,1},x_{T,3})^{\nu}<\infty\,\;\;\text{ and }\;\;\,\prod_{T\in P(g_1,g_2)}(1+A\,d_\infty(x_{T,1},x_{T,3})^{\nu})<\infty\]
	
\end{lemma}

\begin{proof} Let $k$ be a geodesic segment that intersects both $g_1$ and $g_2$. By Lemma \ref{lem: lamination property} and Lemma \ref{lem: hyperbolic geometry},
	\[\sum_{T\in P(g_1,g_2)}d_\infty(x_{T,1},x_{T,3})^{\nu}\leq C^\nu\sum_{T\in P(g_1,g_2)}\ell(k\cap T)^{\nu}<\infty,\]
	so the first inequality holds. The second inequality follows from the first by taking logarithms, applying the comparison test and observing that $\log(1+t)\le t$ for all $t>-1$.
\end{proof}

\subsection{Existence of enveloping maps}\label{sec:existence enveloping}

We first prove Proposition \ref{prop: existence of envelopes}. 
The proof relies on some key estimates for H\"older extendable maps given in the next lemma.
\begin{lemma}\label{lem: abstract}
	Let $M \co P(g_1,g_2)\to \mathsf{SL}_d(\mathbb{K})$ be a H\"older extendable map with constants $A,\nu>0$. For each plaque $T\in P(g_1,g_2)$, let $(x_{T,1},x_{T,2},x_{T,3})$ be a coherent labeling of the vertices of $T$ with respect to $(g_1,g_2)$. Then
	\begin{enumerate}
		\item For all $T\in P(g_1,g_2)$, we have
		\[1\leq \norm{M(T)}\leq 1+A\,d_\infty(x_{T,1}, x_{T,3})^\nu.\]
		\item For any finite subset of plaques $\mathcal{X}\subset P(g_1,g_2)$, 
		\[\prod_{T\in\mathcal{X}}\norm{M(T)}\le \prod_{T\in P(g_1,g_2)}\norm{M(T)}=:D<\infty.\] 
		Moreover, $D$ is bounded above by a constant that depends only on $A$ and $\nu$.
		\item For any finite subsets $\mathcal{X}\subset\mathcal{X}'\subset P(g_1,g_2)$,
		\[\norm{\overrightarrow{M}^F(\mathcal{X}')-\overrightarrow{M}^F(\mathcal{X})}\leq AD\sum_{T\in\mathcal{X}'-\mathcal{X}}d_\infty(x_{T,1}, x_{T,3})^{\nu}.\] 
		In particular, for any finite subset $\mathcal{X}\subset P(g_1,g_2)$,
		\[\norm{\overrightarrow{M}^F(\mathcal{X})-\id}\leq AD\sum_{T\in\mathcal{X}}d_\infty(x_{T,1}, x_{T,3})^{\nu}.\] 
	\end{enumerate}
\end{lemma}

\begin{proof}
	Claim (1): The upper bound for $\norm{M(T)}$ follows at once from the definition of H\"older extendable map and the triangle inequality. The lower bound for $\norm{M(T)}$ holds because $M(T)\in\mathsf{SL}_d(\mathbb{K})$, and so it has at least one singular value that is bounded below by $1$.
	
	Claim (2): By Claim (1) and Lemma \ref{lem: Cauchy},
	\[\prod_{T\in\mathcal{X}}\norm{M(T)}\le\prod_{T\in P(g_1,g_2)}\norm{M(T)}\leq\prod_{T\in P(g_1,g_2)}(1+A\,d_\infty(x_{T,1}, x_{T,3})^{\nu})<\infty . \]
	
	Claim (3): Choose $(\mathcal{X}_i)_i$ such that
	\[
	\mathcal{X} = \mathcal{X}_0 \subset \mathcal{X}_1 \subset \dots \subset \mathcal{X}_k = \mathcal{X}'
	\]
	and $\mathcal{X}_{i+1} = \mathcal{X}_i \cup \{ T_i \}$ for some plaque $T_i$ of $\wt\lambda$. Then,
	\begin{align*}\label{eq:finite_sets}
		\begin{split}
			\norm{\overrightarrow{M}^F(\mathcal{X}')-\overrightarrow{M}^F(\mathcal{X})} & \leq \sum_{i=1}^k\norm{\overrightarrow{M}^F(\mathcal{X}_i)-\overrightarrow{M}^F(\mathcal{X}_{i-1})}\\
			&\leq\sum_{i=1}^k\left(\norm{M(T_i)-\id}\prod_{T\in\mathcal{X}_{i-1}}\norm{M(T)}\right)\\
			& \leq AD\sum_{T \in \mathcal{X}' - \mathcal{X}} d_\infty(x_{T,1},x_{T,3})^{\nu}
		\end{split}
	\end{align*}
	where the third inequality follows from Claims (1) and (2).
\end{proof}

\begin{proof}[Proof of Proposition \ref{prop: existence of envelopes}]	
	Let $A,\nu>0$ be constants for the H\"older extendable map $M$. By Lemma \ref{lem: abstract} Claim (2), we may define
	\[D:=\prod_{T\in P(g_1,g_2)}\norm{M(T)}.\] 
	For any $\epsilon>0$, Lemma \ref{lem: Cauchy} ensures that there is a finite subset $\mathcal{X}_\epsilon\subset P(h_1,h_2)$ such that
	\[AD\sum_{T\in P(h_1,h_2)-\mathcal{X}_\epsilon}d_\infty(x_{T,1},x_{T,3})^{\nu}\leq \epsilon.\]
	
	Fix an exhausting sequence $(\mathcal{X}_n)_{n=1}^\infty$ of $P(h_1,h_2)$ and let $0<\epsilon<1$. Choose $N>0$ such that $\mathcal{X}_\epsilon\subseteq\mathcal{X}_n$ for all $n>N$. Let $m,n$ be positive integers such that $m,n>N$. Assume without loss of generality that $m\geq n$, in which case $\mathcal{X}_n\subseteq\mathcal{X}_m$. By Lemma \ref{lem: abstract} Claim (3),
	\begin{align*}
		\norm{\overrightarrow{M}^F(\mathcal{X}_m)-\overrightarrow{M}^F(\mathcal{X}_n)}
		\leq AD \sum_{T \in \mathcal{X}_m - \mathcal{X}_n} d_\infty(x_{T,1},x_{T,3})^{\nu} \leq \epsilon.
	\end{align*}
	This proves that $\left(\overrightarrow{M}^F(\mathcal{X}_n)\right)_{n=1}^\infty$ is a Cauchy sequence in $\End(\mathbb{K}^d)$. Since $\mathsf{SL}_d(\mathbb{K})$ is a closed subset of the complete normed space $(\End(\mathbb{K}^d), \norm{\cdot})$, the limit $\lim_{n\to\infty}\overrightarrow{M}^F(\mathcal{X}_n)$ exists and lies in $\mathsf{SL}_d(\mathbb{K})$. 
	
	Finally, we prove that the limits $\lim_{n\to\infty}\overrightarrow{M}^F(\mathcal{X}_n)$ agree for all exhausting sequences $(\mathcal{X}_n)_{n=1}^\infty$ in $P(h_1,h_2)$. To do so, it is sufficient to show that 
	\begin{align}\label{eqn: cauchy2}
		\lim_{ n \to \infty} \overrightarrow{M}^F(\mathcal{X}_n) = \lim_{ n \to \infty} \overrightarrow{M}^F(\mathcal{X}_n')
	\end{align}
	for any pair of exhausting sequences $(\mathcal{X}_n)_{n=1}^\infty$ and $(\mathcal{X}_n')_{n=1}^\infty$ such that $\mathcal{X}_n \subseteq \mathcal{X}_n'$ for all $n$. Let $\epsilon>0$, and choose $N>0$ such that $\mathcal{X}_\epsilon\subseteq\mathcal{X}_n$ for all $n>N$. Then by Lemma \ref{lem: abstract} Claim (3), 
	\begin{align*}
		\norm{\overrightarrow{M}^F(\mathcal{X}_n')-\overrightarrow{M}^F(\mathcal{X}_n)}
		&\leq AD \sum_{T \in \mathcal{X}_n' - \mathcal{X}_n} d_\infty(x_{T,1},x_{T,3})^{\nu} \leq \epsilon,
	\end{align*}
	so identity \eqref{eqn: cauchy2} holds.
\end{proof}

\subsection{H\"older continuity of enveloping maps}\label{sec:Holder enveloping}
Next, we prove Proposition \ref{prop: Holder productable}. To do so, we use the following estimate.
\begin{lemma}\label{lem: Holder estimate}
	Let $M \co P(g_1,g_2)\to \mathsf{SL}_d(\mathbb{K})$ be a H\"older extendable map. Then there exist constants $A'', \nu''>0$ such that for all  pairs of leaves $h_1,h_2\in Q(g_1,g_2)$,
	\[\norm{\overrightarrow{M}(h_1,h_2)-\id}\leq A''d_\infty(h_1,h_2)^{\nu''}.\]
\end{lemma}

\begin{proof}
	Fix a geodesic arc $k$ that intersects both $g_1$ and $g_2$, and let $C$ be the constant given in Lemma \ref{lem: hyperbolic geometry}. By Lemma \ref{lem: abstract} Claim (2), we may define 
	\[D:=\prod_{T\in P(g_1,g_2)}\norm{M(T)}.\] 
	If $h_1\leq h_2$ with respect to the total order of $Q(g_1,g_2)$, it follows from Lemma~\ref{lem: abstract} Claim (3) that
	\begin{align*}
		\norm{\overrightarrow{M}(h_1,h_2)-\id}\leq AD\sum_{T \in P(h_1,h_2)}d_\infty(x_{T,1}, x_{T,3})^{\nu}.
	\end{align*} 
	If $h_2\leq h_1$, then by Lemma~\ref{lem: abstract} Claims (2) and (3),
	\begin{align*}
		\norm{\overrightarrow{M}(h_1,h_2)-\id}\leq \norm{\overrightarrow{M}(h_1,h_2)}\norm{\overrightarrow{M}(h_2,h_1)-\id}\leq AD^2\sum_{T \in P(h_1,h_2)}d_\infty(x_{T,1}, x_{T,3})^{\nu}.
	\end{align*}
	By Lemma \ref{lem: abstract} Claim (1), $D\geq 1$. Thus, by Lemma \ref{lem: hyperbolic geometry}, we have
	\begin{align}\label{eqn: Holder 1}
		\norm{\overrightarrow{M}(h_1,h_2)-\id}&\leq AD^2\sum_{T\in P(h_1,h_2)}d_\infty(x_{T,1},x_{T,3})^{\nu}\\
		&\leq AD^2C^{\nu}\sum_{T\in P(h_1,h_2)}\ell(k\cap T)^{\nu}\nonumber
	\end{align}
	for arbitary leaves $h_1,h_2\in Q(g_1,g_2)$. 
	
	Let $r\co P(g_1,g_2)\to\mathbb{Z}^+$ be the function and $E,E',E'',\mu,\mu'>0$ be the constants given by Lemma~\ref{lem: lamination property}. Let $a\in\mathbb{Z}^+$ be the minimum value of the function $r$ over the set $P(h_1,h_2)$, let $T_0\in P(h_1,h_2)$ be a plaque that satisfies $r(T_0)=a$, and let ${\nu''}:=\frac{\nu\mu'}{\mu}$. Then 
	\begin{align}
		\sum_{T\in P(h_1,h_2)}\ell(k\cap T)^{\nu}&\leq (E')^{\nu}\sum_{T\in P(h_1,h_2)}e^{-\nu\mu'r(T)}\leq (E')^{\nu} E''\sum_{i=a}^\infty e^{-\nu\mu'i}\nonumber\\
		&\leq (E')^{\nu} E''\frac{e^{-\nu \mu'a}}{1-e^{-\nu\mu'}}= (E')^{\nu} \frac{E''}{1-e^{-\nu\mu'}}e^{-\nu''\mu a}\nonumber\\
		&\leq \frac{(E')^{\nu}}{E^{\nu''}} \frac{E''}{1-e^{-\nu\mu'}}\ell(k\cap T_0)^{\nu''}\label{eqn: Holder 2}\\
		&\leq \frac{(E')^{\nu}}{E^{\nu''}} \frac{E''}{1-e^{-\nu\mu'}}\left(\sum_{T\in P(h_1,h_2)}\ell(k\cap T)\right)^{\nu''}.\nonumber
	\end{align}
	
	Furthermore, Lemma \ref{lem: hyperbolic geometry} implies that 
	\begin{align}\label{eqn: Holder 3}
		\sum_{T\in P(h_1,h_2)}\ell(k\cap T)&\leq C\sum_{T\in P(h_1,h_2)}d_\infty(x_{T,1},x_{T,3})\leq 2C d_\infty(h_1,h_2).
	\end{align}
	Combining inequalities \eqref{eqn: Holder 1}, \eqref{eqn: Holder 2}, and \eqref{eqn: Holder 3} we have that there exists a constant $A''>0$ such that
	\[\norm{\overrightarrow{M}(h_1,h_2)-\id}\leq A''d_\infty(h_1,h_2)^{\nu''}\]
	and the lemma holds.
\end{proof}

\begin{proof} [Proof of Proposition \ref{prop: Holder productable}]
	As before, let 
	\[D:=\prod_{T\in P(g_1,g_2)}\norm{M(T)}\]
	and note that $D\geq 1$ by Lemma \ref{lem: abstract} Claim (1). Since $Q(g_1,g_2)$ is compact, it suffices to prove the proposition for pairs $(\ell_1,\ell_2)$ and $(h_1,h_2)$ such that $d_\infty((h_1,h_2),(\ell_1,\ell_2))<1$.
	By Proposition \ref{prop:cocycle} Claim (2), we have that
	\[\overrightarrow{M}(\ell_1,\ell_2)=\overrightarrow{M}(\ell_1,h_1)\overrightarrow{M}(h_1,h_2)\overrightarrow{M}(h_2,\ell_2).\]
	By Lemma~\ref{lem: abstract} Claim (2), $\norm{\overrightarrow{M}(h_1,h_2)},\norm{\overrightarrow{M}(\ell_1,h_2)}\leq D$. Furthermore, Lemma~\ref{lem: Holder estimate} implies that there are constants $A'',\nu''>0$ such that 
	\[\norm{\overrightarrow{M}(h_2,\ell_2)-\id}\le A''d_\infty(h_2,\ell_2)^{\nu''}\quad\text{and}\quad\norm{\overrightarrow{M}(\ell_1,h_1)-\id}\le A''d_\infty(\ell_1,h_1)^{\nu''}.\]
	Since $d_\infty(h_2,\ell_2)$ and $d_\infty(\ell_1,h_1)$ are bounded above by $d_\infty((h_1,h_2),(\ell_1,\ell_2))$ which is less than $1$, we may assume that $\nu''<1$. Then
	\begin{align*}
		\norm{\overrightarrow{M}(\ell_1,\ell_2)-\overrightarrow{M}(h_1,h_2)} &\leq \norm{\overrightarrow{M}(h_1,h_2)}\norm{\overrightarrow{M}(\ell_1,h_1)}\norm{\overrightarrow{M}(h_2,\ell_2)-\id}\\
		&\hspace{1cm}+\norm{\overrightarrow{M}(h_1,h_2)}\norm{\overrightarrow{M}(\ell_1,h_1)-\id}\\
		&\leq D^2\left(A''d_\infty(h_2,\ell_2)^{\nu''}+A''d_\infty(\ell_1,h_1)^{\nu''}\right)\\
		&\leq 2A''D^2d_\infty((h_1,h_2),(\ell_1,\ell_2))^{\nu''}.
	\end{align*}
	Thus, the inequality holds with $A':=2A''D^2$ and $\nu':=\nu''$.
\end{proof}

\subsection{Combining H\"older extendable maps}\label{sec:combine productable}
We will now prove Proposition \ref{prop: renormalization}.

\begin{proof}[Proof of Proposition \ref{prop: renormalization}]
	Since $M$ and $N$ are H\"older extendable, there are constants $A,\nu>0$ such that 
	\[\norm{M(T)^{-1}-\id},\,\,\norm{N(T)-\id}\leq A\,d_\infty(x_{T,1},x_{T,3})^\nu\]
	for all plaques $T\in P(g_1,g_2)$, where $(x_{T,1},x_{T,2},x_{T,3})$ is a coherent labeling of the vertices of $T$ with respect to $(g_1,g_2)$. Since $M$ is H\"older extendable, Lemma \ref{lem: Holder estimate} implies that by enlarging $A$ and $\nu$ if necessary, we may assume that 
	\[\norm{\overrightarrow{M}(h_1,h_2)-\id}\leq A\,d_\infty(h_1,h_2)^{\nu}\]
	for all $h_1,h_2\in Q(g_1,g_2)$. Also, by Proposition \ref{prop:cocycle} parts (1) and (2),
	\[\overrightarrow{M}(g_1,h_{T,\pm})\overrightarrow{M}(h_{T,\pm},g_1)=\id\,\,\text{ and }\,\,\overrightarrow{M}(g_1,h_{T,-})M(T)=\overrightarrow{M}(g_1,h_{T,+})\]
	for all plaques $T\in P(g_1,g_2)$. It follows that 
	\[
	N'(T)=\overrightarrow{M}(g_1,h_{T,-})N(T)M(T)^{-1}\overrightarrow{M}(h_{T,-},g_1).
	\]
	Thus, if we set $B:=A\,d_\infty(g_1,g_2)^{\nu}+1$, then
	\begin{align*}
		\norm{N'(T)-\id}&\leq \norm{\overrightarrow{M}(g_1,h_{T,-})}\norm{\overrightarrow{M}(h_{T,-},g_1)}\left(\norm{N(T)}\norm{M(T)^{-1}-\id}+\norm{N(T)-\id}\right)\\
		&\leq 2B^2(1+B)A \, d_\infty(x_{T,1},x_{T,3})^\nu.
	\end{align*}
	This proves that the map $N'$ is H\"older extendable.
	
	Next, we show that the extension of $N'$ satisfies
	\[\overrightarrow{N}'(g_1,h)=\overrightarrow{N}(g_1,h)\overrightarrow{M}(h,g_1)\]
	for all $h\in Q(g_1,g_2)$. Let $(\mathcal{X}_n)_{n=1}^\infty$ be an exhausting sequence in $P(g_1,h)$. For every positive integer $n$, we enumerate $\mathcal{X}_n=\{T_{n,1},\dots,T_{n,m_n}\}$ so that $T_{n,1}<\dots<T_{n,m_n}$ in $P(g_1,g_2)$, we denote $h_{n,i,\pm}:=h_{T_{n,i},\pm}$ for all $i\in\{1,\dots,m_n\}$, and we set $h_{n,0,+}:=g_1$. Then
	\begin{align*}
		\overrightarrow{N'}^F(\mathcal{X}_n)&=N'(T_{n,1})\ldots N'(T_{n,m_n})\\
		&=\overrightarrow{M}(h_{n,0,+},h_{n,1,-})N(T_{n,1})\overrightarrow{M}(h_{n,1,+},h_{n,2,-})N(T_{n,2})\overrightarrow{M}(h_{n,2,+},h_{n,3,-})\\
		&\hspace{.5cm} N(T_{n,3})\dots\overrightarrow{M}(h_{n,m_n-1,+},h_{n,m_n,-})N(T_{n,m_n})\overrightarrow{M}(h_{n,m_n,+},g_1).
	\end{align*}
	Since $\overrightarrow{M}(h_{n,m_n,+},g_1)\to\overrightarrow{M}(h,g_1)$ as $n\to\infty$, it suffices to show that
	\begin{align}\label{eqn: alt prod}\lim_{n\to\infty}\overrightarrow{M}(h_{n,0,+},h_{n,1,-})N(T_{n,1})\dots\overrightarrow{M}(h_{n,m_n-1,+},h_{n,m_n,-})N(T_{n,m_n})=\overrightarrow{N}(g_1,h).
	\end{align}
	
	Lemma \ref{lem: abstract} part (2) implies that there is some $D>0$ such that for all $n$,
	\[\prod_{i=0}^{m_n-1}\norm{\overrightarrow{M}(h_{n,i,+},h_{n,i+1,-})},\,\,\prod_{T\in \mathcal{X}_n}\norm{N(T)}\leq D.\]
	As such,
	\begin{align*}
		&\hspace{.5cm}\norm{\overrightarrow{M}(h_{n,0,+},h_{n,1,-})N(T_{n,1})\dots\overrightarrow{M}(h_{n,m_n-1,+},h_{n,m_n,-})N(T_{n,m_n})-\overrightarrow{N}(g_1,h)}\\
		&\leq\sum_{j=0}^{m_n-1}\left(\left(\prod_{T\in \mathcal{X}_n}\norm{N(T)}\right)\norm{\overrightarrow{M}(h_{n,j,+},h_{n,j+1,-})-\id}\left(\prod_{i=j+1}^{m_n-1}\norm{\overrightarrow{M}(h_{n,i,+},h_{n,i+1,-})}\right)\right)\\
		&\hspace{.5cm}+\norm{\overrightarrow{N}^F(\mathcal{X}_n)-\overrightarrow{N}(g_1,h)}\\
		&\leq D^2A\sum_{j=0}^{m_n-1}d_\infty(h_{n,j,+},h_{n,j+1,-})^\nu+\norm{\overrightarrow{N}^F(\mathcal{X}_n)-\overrightarrow{N}(g_1,h)}.
	\end{align*}
	Since $(\mathcal{X}_n)_{n=1}^\infty$ exhausts $P(g_1,h)$, 
	\[\lim_{n\to\infty}\norm{\overrightarrow{N}^F(\mathcal{X}_n)-\overrightarrow{N}(g_1,h)}= 0,\] 
	and Lemma \ref{lem: Cauchy} implies that
	\[\lim_{n\to\infty}\sum_{j=0}^{m_n-1}d_\infty(h_{n,j,+},h_{n,j+1,-})^\nu=0.\]
	It follows that equation \eqref{eqn: alt prod} holds.
\end{proof}

\section{Slithering maps}\label{sec: slithering}

		Recall that in Section \ref{shear-bend} we introduced the slithering map to associate a $\lambda$--cocylic pair to a $d$--pleated surface with pleating locus $\lambda$. More precisely, given a $\lambda$--transverse and locally $\lambda$--H\"older continuous map $\xi\co\partial\wt\lambda\to\mathcal{F}(\mathbb{K}^d)$ we defined its compatible slithering map $\Sigma\co \wt \Lambda^2\to\mathsf{SL}_d(\mathbb{K})$ (see Definition \ref{def: slithering map}). This section is dedicated to the proof of Theorem \ref{thm: slitherable and slithering}, that is the existence and uniqueness of the slithering map compatible to $\xi$. In addition, we show that if $\xi$ is equivariant with respect to a representation $\rho\co \Gamma\to\mathsf{PGL}_d(\mathbb{K})$, then so is $\Sigma$. We divide the proof of Theorem \ref{thm: slitherable and slithering} into four broad steps:
		\begin{enumerate}
			\item[\underline{Step 1:}] For any leaves $g_1$ and $g_2$ of $\wt\lambda$, use $\xi$ to define a map $M_{{\rm slither}}\co P(g_1,g_2)\to\mathsf{SL}_d(\mathbb{K})$ and show that it is H\"older extendable (see Definition \ref{def_productable}).
			\item[\underline{Step 2:}] Using the extensions of $M_{{\rm slither}}$, define a map $\Sigma\co \wt\Lambda^2\to\mathsf{SL}_d(\mathbb{K})$ and show that it is a slithering map compatible with $\xi$.
			\item[\underline{Step 3:}] Prove that the map $\Sigma$ is uniquely characterized by its properties.
			\item[\underline{Step 4:}] Show that if there exists $\rho\co\Gamma\to\mathsf{PGL}_d(\mathbb{K})$ such that $\xi$ is $\rho$--equivariant, then $\Sigma$ is also $\rho$--equivariant.
		\end{enumerate}
		These steps are proved in the following four subsections.

		\subsection{Define a H\"older extendable map \texorpdfstring{$M_{{\rm slither}}$}{Mslither}.} \label{subsec:define productable}
		Fix leaves $g_1$ and $g_2$ of $\wt\lambda$. Recall that for any plaque $T\in P(g_1,g_2)$, we denote by $h_{T,-}$ and $h_{T,+}$ the two leaves of $T$ that separate $g_1$ and $g_2$, so that $h_{T,-}$ separates $g_1$ and $h_{T,+}$. Then we say that $(x_{T,1}, x_{T,2}, x_{T,3})$ is a coherent labeling of the vertices of $T$ if $x_{T,2}$ is the common endpoint of $h_{T,-}$ and $h_{T,+}$, and $x_{T,1}$ and $x_{T,3}$ are respectively the endpoints of $h_{T,-}$ and $h_{T,+}$ that are not $x_{T,2}$.
		
		Recall that $\mathcal{F}(\mathbb{K}^d)$ denotes the space of flags in $\mathbb{K}^d$. Let
		\[\mathfrak{T}:=\{(G_1,F,G_2)\in\mathcal{F}(\mathbb{K}^d)^3: F\text{ and }G_i\text{ are transverse for both }i=1,2\},\]
		and let 
		\[U:\mathfrak{T}\to\mathsf{SL}_d(\mathbb{K})\] 
		be the map that sends each triple of flags $(G_1,F,G_2)$ to the unique unipotent element that fixes $F$ and that sends $G_2$ to $G_1$. Then define
		\[M_{\rm slither}:=M_{{\rm slither},g_1,g_2}\co P(g_1,g_2)\to\mathsf{SL}_d(\mathbb{K})\] 
		by setting $M_{{\rm slither}}(T):=U\left(\xi(\mathbf{x}_T)\right)$ for all plaques $T \in P(g_1,g_2)$, where $\mathbf{x}_T = (x_{T,1}, x_{T,2}, x_{T,3})$ is the coherent labeling of the vertices of $T$ with respect to $(g_1, g_2)$.
		
		\begin{proposition}\label{prop: slithering estimate}
			The map 
			\[M_{\rm slither}\co P(g_1,g_2)\to\mathsf{SL}_d(\mathbb{K})\] 
			is H\"older extendable. 
		\end{proposition}
		
		\begin{proof}
			The case $g_1=g_2$ is trivial, so we may assume $g_1\neq g_2$. We need to prove that there are constants $A,\nu>0$ such that 
			\[\norm{M_{\rm slither}(T)-\id}\le A \, d_\infty(x_{T,1}, x_{T,3})^\nu\]
			for all $T\in P(g_1,g_2)$, where ${\bf x}_T=(x_{T,1},x_{T,2},x_{T,3})$ is a coherent labeling of the vertices of $T$ with respect to $(g_1, g_2)$. 
			
			Let $\mathsf{g}_1, \mathsf{g}_2\in\wt\Lambda^o$ be oriented in parallel and satisfy $\pi_\Lambda(\mathsf{g}_j) = g_j$ for $j = 1, 2$. Since $\xi$ is locally $\lambda$--H\"older continuous, there exist constants $A',\nu'>0$ such that 
			\[d_{\mathcal{F}}(\xi(\mathsf{h}^+_1),\xi(\mathsf{h}^+_2)), d_{\mathcal{F}}(\xi(\mathsf{h}^-_1),\xi(\mathsf{h}^-_2)) \leq A'\,d_\infty(\mathsf{h}_1,\mathsf{h}_2)^{\nu'}\]
			for all pairs of leaves $\mathsf{h}_1,\mathsf{h}_2$ oriented in parallel that separate $g_1$ from $g_2$. In particular, if $T \in P(g_1,g_2)$, then
			\[d_{\mathcal{F}}(\xi(x_{T,1}),\xi(x_{T,3}))\leq A'\,d_\infty(x_{T,1},x_{T,3})^{\nu'}\]
			
			Observe that the domain of definition $\mathfrak{T}$ of $U$ is an open subset of the manifold $\mathcal{F}(\mathbb{K}^d)^3$. Since $U\co\mathfrak{T}\to\mathsf{SL}_d(\mathbb{K})$ is continuously differentiable and $Q(g_1, g_2)$ is compact, there is some constant $C>0$ such that
			\[\norm{U(G_1,F,G_2)-\id}\leq Cd_{\mathcal{F}}(G_1,G_2)\]
			whenever the pairs of flags $(F,G_1)$ and $(F,G_2)$ are both inside $\xi(Q(g_1, g_2))$. Thus, for any plaque $T\in P(g_1,g_2)$,
			\[\norm{M_{\rm slither}(T)-\id}\leq Cd_{\mathcal{F}}(\xi(x_{T,1}),\xi(x_{T,3}))\leq CA'd_\infty(x_{T,1}, x_{T,3})^{\nu'}.\qedhere\]
		\end{proof}
		
		\subsection{Define a slithering map \texorpdfstring{$\Sigma$}{S} compatible with \texorpdfstring{$\xi$}{xi}.}
		
		By Proposition \ref{prop: existence of envelopes} and Proposition \ref{prop: slithering estimate}, we may now define 
		\[\Sigma\co \wt\Lambda^2\to\mathsf{SL}_d(\mathbb{K})\] 
		to be the map given by $\Sigma(g_1,g_2):=\overrightarrow{M}_{{\rm slither},h_1,h_2}(g_1,g_2)$ for some (equivalently, any) leaves $h_1,h_2\in\widetilde\Lambda$ such that $g_1,g_2\in Q(h_1,h_2)$. The fact that $\Sigma$ does not depend on the choice of $h_1$ and $h_2$ is a consequence of Remark \ref{rem: obvious} and the observation that $U(G_1,F,G_2) = U(G_2,F,G_1)^{-1}$ for all triples of flags $(G_1,F,G_2)\in\mathfrak T$.
		
		\begin{proposition}\label{prop: sigma slithering}
			The map $\Sigma$ is a slithering map compatible with $\xi$.
		\end{proposition}
		
		\begin{proof}
			We will verify that $\Sigma$ satisfies the conditions (1) -- (4) of Definition \ref{def: slithering map}.
			
			Condition (1): This follows easily from properties (1) and (2) of Proposition~\ref{prop:cocycle}. 
			
			Condition (2): This is a special case of Proposition \ref{prop: Holder productable}.
			
			Condition (3): If $g_1,g_2\in\wt\Lambda$ share a common endpoint $p$, then $x_{T,2}=p$ for all plaques $T\in P(g_1,g_2)$, where ${\bf x}_T=(x_{T,1},x_{T,2},x_{T,3})$ is a coherent labeling of the vertices of $T$ with respect to $(g_1, g_2)$. As such, for all $T\in P(g_1,g_2)$, the transformation $M_{\rm slither}(T)=U(\xi({\bf x}_T))$ is a unipotent element of $\mathsf{SL}_d(\mathbb{K})$ that fixes $\xi(p)$, so the same is true for $\Sigma(g_1,g_2)$.
			
			Condition (4): Let $(\mathcal{X}_n)_{n=1}^\infty$ be an exhausting sequence of $P(g_1,g_2)$ (see the beginning of Section \ref{sec:productable} for the necessary terminology). For all positive integers $n$, enumerate
			\begin{align*}\mathcal{X}_n=\{T_{n,1},\dots,T_{n,m_n}\}
			\end{align*}
			so that $T_{n,1}<\dots<T_{n,m_n}$ in $P(g_1,g_2)$. Then for all $i=1,\dots,m_n$, we simplify notation by setting
			\begin{align}h_{n,i,\pm}:=h_{T_{n,i},\pm},\label{eqn: hni}\end{align}
			where as usual, $h_{T_{n,i},-}<h_{T_{n,i},+}$ are the two edges of $T_{n,i}$ that separate $g_1$ and $g_2$. Set $h_{n,0,+}:=g_1$ and $h_{n,m_n+1,-}:=g_2$. Let
			$\mathsf{h}_{n,i, \pm}\in\widetilde\Lambda^o$ be the oriented leaf such that $\pi_{\wt\Lambda}(\mathsf{h}_{n,i, \pm})=h_{n,i, \pm}$ and such that the pair $\mathsf{g}_1,\mathsf{h}_{n,i, \pm}$ (equivalently, $\mathsf{g}_2,\mathsf{h}_{n,i, \pm}$) are oriented in parallel.

			The strategy of the proof of Condition (4) is to find (for sufficiently large $n$) elements $L_{n,i}\in\mathsf{SL}_d(\mathbb{K})$ that send $\xi(\mathsf{h}_{n,i+1,-})$ to $\xi(\mathsf{h}_{n,i,+})$, such that the product
			\[L(\mathcal{X}_n):=L_{n,0} M_{\rm slither}(T_{n,1}) L_{n,1} \dots M_{\rm slither}(T_{n,m_n-1}) L_{n,m_n-1} M_{\rm slither}(T_{n,m_n})L_{n,m_n} \]
			converges to $\Sigma(g_1,g_2)$ as $n$ goes to $\infty$. If we can do so, then the fact that $M_{\rm slither}(T_{n,i})$ sends $\xi(\mathsf{h}_{n,i,+})$ to $\xi(\mathsf{h}_{n,i,-})$, implies that $L(\mathcal{X}_n)$ sends $\xi(\mathsf{g}_2)$ to $\xi(\mathsf{g}_1)$, so Condition (4) holds. 
			
			To define $L_{n,i}$, we use the following transversality result. Let $Q(\mathsf{g}_1,\mathsf{g}_2)$ denote the set of oriented leaves $\mathsf{g}\in\widetilde\Lambda^o$ such that the unoriented leaf $\pi_{\widetilde\Lambda}(\mathsf{g})$ belongs to $Q(g_1,g_2)$ and the pair $\mathsf{g},\mathsf{g}_1$ (equivalently, $\mathsf{g},\mathsf{g}_2$) are oriented in parallel. Let 
			\[Q(\mathsf{g}_1,\mathsf{g}_2)_\pm:=\{\mathsf{g}^\pm:\mathsf{g}\in Q(\mathsf{g}_1,\mathsf{g}_2)\}.\] 
			
			\begin{lemma}\label{lem: close transverse}
				There exists $\delta>0$ such that for any oriented leaf $\mathsf{h}\in Q(\mathsf{g}_1, \mathsf{g}_2)$ and for any endpoints $p\in Q(\mathsf{g}_1, \mathsf{g}_2)_+$ and $q\in Q(\mathsf{g}_1, \mathsf{g}_2)_-$ that satisfy $d_\infty(( \mathsf{h}^+,\mathsf{h}^-),(p,q))\leq\delta$, the flags $\xi(p)$ and $\xi(q)$ are transverse. 
			\end{lemma}
			
			\begin{proof}
				Suppose not, then there are sequences $(\mathsf{h}_i)_{i=1}^\infty$ in $Q(\mathsf{g}_1, \mathsf{g}_2)$, $(p_i)_{i=1}^\infty$ in $Q(\mathsf{g}_1,\mathsf{g}_2)_+$, and $(q_i)_{i=1}^\infty$ in $Q(\mathsf{g}_1,\mathsf{g}_2)_-$ such that for all $i$, $d_\infty((\mathsf{h}_i^+,\mathsf{h}_i^-),(p_i,q_i))\leq\frac{1}{i}$ and $\xi(p_i)$ and $\xi(q_i)$ are not transverse. Since $Q(\mathsf{g}_1,\mathsf{g}_2)$ is compact, by taking a subsequence, we may assume that $(\mathsf{h}_i)_{i=1}^\infty$ converges to some $\mathsf{h}_\infty\in Q(\mathsf{g}_1,\mathsf{g}_2)$. Then $p_i\to \mathsf{h}_\infty^+$ and $q_i\to \mathsf{h}_\infty^-$, so the $\lambda$--continuity of $\xi$ implies that $\xi(p_i)\to \xi(\mathsf{h}_\infty^+)$ and $\xi(q_i)\to \xi(\mathsf{h}_\infty^-)$. Since the flags $\xi(\mathsf{h}_\infty^+)$ and $\xi(\mathsf{h}_\infty^-)$ are transverse (which is an open condition), $\xi(p_i)$ and $\xi(q_i)$ are transverse for sufficiently large $i$, which is a contradiction. 
			\end{proof}
			
			Let $\delta>0$ be the constant given in Lemma \ref{lem: close transverse}. By  Lemma \ref{lem: lamination property} (with $\nu = 1$) and Lemma~\ref{lem: hyperbolic geometry}, there exists some positive integer $N$ such that
			\begin{align}\label{eqn: 4.2}
				\sum_{T\in P(g_1,g_2)-\mathcal{X}_N}d_\infty(x_{T,1}, x_{T,3})\leq\delta.
			\end{align}
			Hence, by Lemma \ref{lem: close transverse}, the flags $\xi(\mathsf{h}_{n,i, +}^+)$ and $\xi(\mathsf{h}_{n,i+1, -}^-)$ are transverse, so we may define 
			\begin{align*}
				U_{n,i,-}&:=U\left(\xi(\mathsf{h}_{n,i, +}^-),\xi(\mathsf{h}_{n,i, +}^+),\xi(\mathsf{h}_{n,i+1, -}^-)\right),\\
				U_{n,i,+}&:=U\left(\xi(\mathsf{h}_{n,i, +}^+),\xi(\mathsf{h}_{n,i+1, -}^-),\xi(\mathsf{h}_{n,i+1, -}^+)\right),\\
				L_{n,i}&:=U_{n,i,-}U_{n,i,+}
			\end{align*}
			for all $n>N$ and $i=0,\dots,m_n$. Observe that $L_{n,i}$ sends $\xi(\mathsf{h}_{n,i+1,-})$ to $\xi(\mathsf{h}_{n,i,+})$, see Figure \ref{fig:extendplaques}.
			\begin{figure}[h!]
				\includegraphics[width=\textwidth]{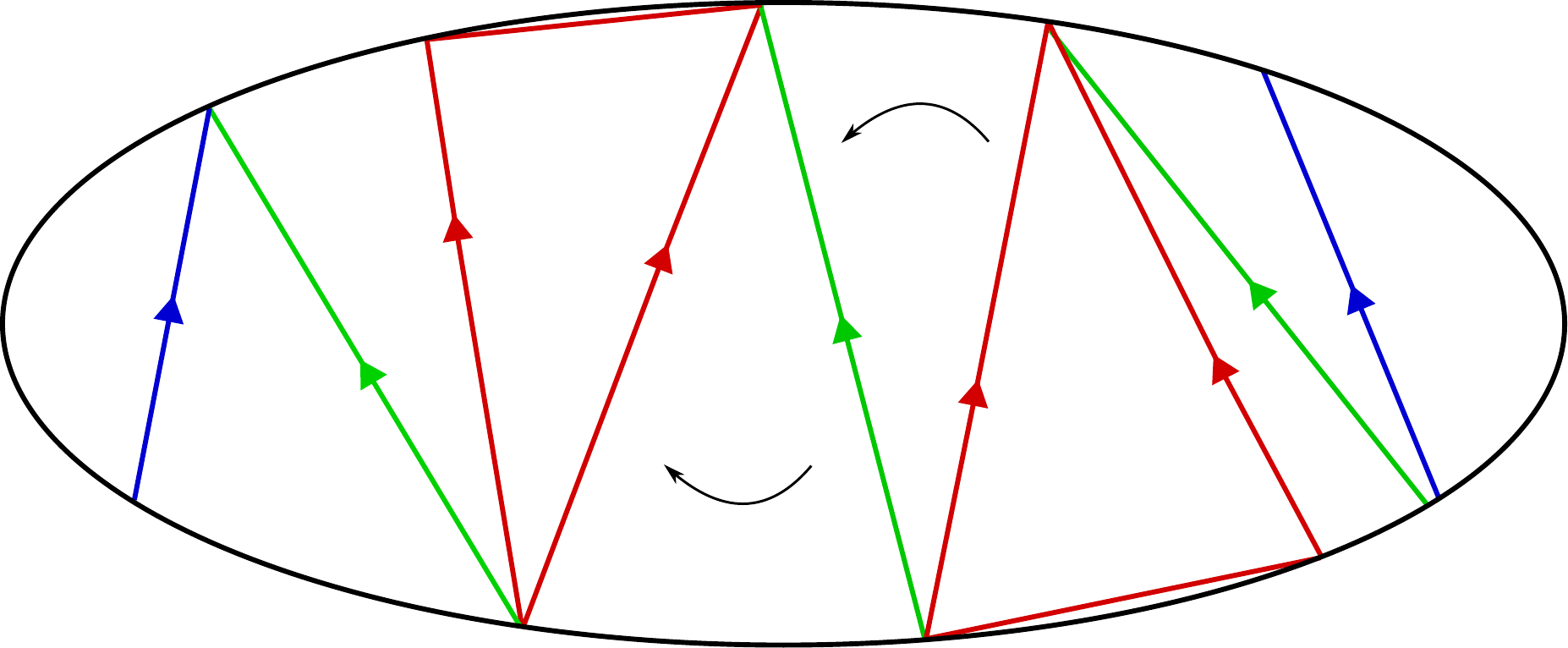}
				\put(-327,90){\small $\sf g_1$}
				\put(-241,110){\small $T_{n,1}$}
				\put(-251,70){\small $\sf h_{n,1,-}$}
				\put(-198,153){\small $\sf h_{n,1,+}^+$}
				\put(-241,-4){\small $\sf h_{n,1,+}^-$}
				\put(-200,43){\small $U_{n,1,-}$}
				\put(-160,112){\small $U_{n,1,+}$}
				\put(-148,-6){\small $\mathsf h_{n,2,-}^-$}
				\put(-125,149){\small $\sf h_{n,2,-}^+$}
				\put(-115,35){\small $T_{n,2}$}
				\put(-101,60){\small $\sf h_{n,2,+}$}
				\put(-45,80){\small $\mathsf g_2$}
				\caption{{\small In blue the oriented geodesics $\sf g_1,\sf g_2$. In red, the plaques $T_{n,1}$ and $T_{n,2}$. The green geodesics are {\em not} leaves of $\wt \lambda^o$. However, they illustrate the main idea behind the proof of Proposition \ref{prop: sigma slithering} and the role of the unipotent matrices $U_{n,i,\pm}$.}}
				\label{fig:extendplaques}
			\end{figure}
			
			It remains to show that $L(\mathcal{X}_n)$ converges to $\Sigma(g_1,g_2)$, or equivalently,
			\begin{align}\label{eqn: slither limit}
				\lim_{n\to\infty}\norm{L(\mathcal{X}_n)-\overrightarrow{M}_{\rm slither}^F(\mathcal{X}_n)}=0.
			\end{align}
			(Recall that $\overrightarrow{M}_{\rm slither}^F(\mathcal{X}_n)=M_{\rm slither}(T_{n,1})\dots M_{\rm slither}(T_{n,m_n}) $.) By enlarging $N$ if necessary, we may assume that $d_{\infty}\left(\mathsf{h}_{n,i, +},\mathsf{h}_{n,i+1, -}\right)\leq 1$ for all $n>N$ and $i=0,\dots,m_n$. Since $\xi$ is locally $\lambda$--H\"older continuous, there exist constants $A,\nu>0$ such that 
			\begin{align}\label{eqn: standard assumption}
				d_{\mathcal{F}}(\xi(\mathsf{h}_1),\xi(\mathsf{h}_2))\leq A\, d_\infty(\mathsf{h}_1,\mathsf{h}_2)^\nu
			\end{align}
			for all oriented leaves $\mathsf{h}_1,\mathsf{h}_2\in Q(\mathsf{g}_1, \mathsf{g}_2)$. Since $Q(\mathsf{g}_1, \mathsf{g}_2)$ has finite diameter, by enlarging $A$ if necessary we may and will assume that $\nu \in (0,1)$. 
			
			The following lemma gives the necessary estimates to prove identity \eqref{eqn: slither limit}.
			
			\begin{lemma}\label{lem: gap products}\
				\begin{enumerate}
					\item There exists $A'>0$ such that 
					\[\norm{L_{n,i}-\id}\leq A'd_{\infty}\left(\mathsf{h}_{n,i, +},\mathsf{h}_{n,i+1, -}\right)^\nu\]
					for all $n>N$ and all $i=0,\dots,m_n$. In particular, 
					\[\sum_{i=0}^{m_n}\norm{L_{n,i}-\id}\leq A'\sum_{T\in P(g_1,g_2)-\mathcal{X}_N}d_\infty(x_{T,1},x_{T,3})^\nu\]
					\item There exists $C>0$ such that for all $n>N$, we have
					\[\prod_{i=0}^{m_n}\norm{L_{n,i}}\leq C.\]
				\end{enumerate}
			\end{lemma}
			
			\begin{proof}
				Part (1): Since the map $U\co\mathfrak{T}\to\mathsf{SL}_d(\mathbb{K})$ is continuously differentiable (see Section \ref{subsec:define productable} for the definition of $U$) and $Q(\mathsf{g}_1,\mathsf{g}_2)$ is compact, there exists $C'>0$ such that
				\[\norm{U_{n,i,-}-\id}<C'd_{\mathcal{F}}\left(\xi(\mathsf{h}_{n,i, +}^-),\xi(\mathsf{h}_{n,i+1, -}^-)\right)\leq C'd_{\mathcal{F}}\left(\xi(\mathsf{h}_{n,i, +}),\xi(\mathsf{h}_{n,i+1, -})\right)\]
				and
				\[\norm{U_{n,i,+}-\id}<C'd_{\mathcal{F}}\left(\xi(\mathsf{h}_{n,i, +}^+),\xi(\mathsf{h}_{n,i+1, -}^+)\right)\leq C'd_{\mathcal{F}}\left(\xi(\mathsf{h}_{n,i, +}),\xi(\mathsf{h}_{n,i+1, -})\right)\]
				for all $n>N$ and $i=0,\dots,m_n$. Then by inequality \eqref{eqn: standard assumption},
				\begin{align*}
					\norm{L_{n,i}-\id}&\leq\norm{L_{n,i}-U_{n,i,-}}+\norm{U_{n,i,-}-\id}\\
					&\leq \norm{U_{n,i,+}-\id}\norm{U_{n,i,-}}+\norm{U_{n,i,-}-\id}\\
					&\leq AC'd_{\infty}\left(\mathsf{h}_{n,i, +},\mathsf{h}_{n,i+1, -}\right)^\nu\left(AC'd_{\infty}\left(\mathsf{h}_{n,i, +},\mathsf{h}_{n,i+1, -}\right)^\nu+1\right)\\
					&\leq AC'(AC'+1)d_{\infty}\left(\mathsf{h}_{n,i, +},\mathsf{h}_{n,i+1, -}\right)^\nu.
				\end{align*}
				The first statement of part (1), then holds with $A':=AC'(AC'+1)$. Now,
				\[\sum_{i=0}^{m_n}\norm{L_{n,i}-\id}\leq A'\sum_{i=0}^{m_n}d_{\infty}\left(\mathsf{h}_{n,i, +},\mathsf{h}_{n,i+1, -}\right)^\nu\leq A'\sum_{T\in P(g_1,g_2)-\mathcal{X}_N}d_\infty(x_{T,1},x_{T,3})^\nu,\]
				where the second inequality holds because $x\mapsto x^\nu$ is a concave function on $\mathbb{R}^+$ whenever $\nu\in(0,1)$.

				Part (2): By Lemma \ref{lem: Cauchy},
				\[C:=\prod_{T\in P(g_1,g_2)}\left(1+A'd_\infty(x_{T,1}, x_{T,3})^\nu\right)<\infty.\]
				Thus, by part (1),
				\[\prod_{i=1}^{m_n}\norm{L_{n,i}}\leq\prod_{i=1}^{m_n}\left(1+A'd_{\infty}\left(\mathsf{h}_{n,i, +},\mathsf{h}_{n,i+1, -}\right)^\nu\right)\leq C.\]
				The last inequality holds because $x\mapsto 1+A'x^\nu$ is a concave function on $\mathbb{R}^+$ when $\nu\in(0,1)$.
			\end{proof}

			Pick any $0<\epsilon<\min\{1,\delta\}$, where $\delta > 0$ denotes the constant guaranteed by Lemma \ref{lem: close transverse}. By Lemma \ref{lem: lamination property} and Lemma~\ref{lem: hyperbolic geometry}, we may assume
			\begin{align}\label{eqn: 4.99}
				\sum_{T\in P(g_1,g_2)-\mathcal{X}_N}d_\infty(x_{T,1},x_{T,3})^\nu<\epsilon ,
			\end{align}
			up to further enlarging $N$. Also, part (2) of Lemma \ref{lem: abstract}  gives
			\begin{align}\label{eqn: 4.88}
				\prod_{T\in P(g_1,g_2)}\norm{M_{\rm slither}(T)}=:D<\infty.
			\end{align}
			For every $n>0$ and $i \in\{ 1, \dots, m_n - 1\}$, let
			\begin{align*}V_{n,i}:=\big(L_{n,0}M_{\rm slither}(T_{n,1})L_{n,1}  \ldots M_{\rm slither}&(T_{n,m_n - i}) L_{n,m_n - i}\big) \\
				&\cdot\big(M_{\rm slither}(T_{n,m_n - i + 1})\dots M_{\rm slither}(T_{n,m_n})\big),
			\end{align*}
			and let
			\begin{align*}
				V_{n,m_n}&:=L_{n,0}M_{\rm slither}(T_{n,1})M_{\rm slither}(T_{n,2}) \ldots  M_{\rm slither}(T_{n,m_n}).
			\end{align*}
			Then Lemma \ref{lem: gap products} and relations \eqref{eqn: 4.99} and \eqref{eqn: 4.88} imply that there exist $A',C>0$ such that for all $n>N$,
			\begin{align*}
				\norm{L(\mathcal{X}_n)-\overrightarrow{M}_{\rm slither}^F(\mathcal{X}_n)}&\leq \norm{L(\mathcal{X}_n)-V_{n,1}}+\sum_{i=1}^{m_n-1}\norm{V_{n,i}-V_{n,i+1}}+\norm{V_{n,m_n}-\overrightarrow{M}_{\rm slither}^F(\mathcal{X}_n)} \\
				&\leq\left(\prod_{T\in\mathcal{X}_n}\norm{M_{\rm slither}(T)}\right)\sum_{i=0}^{m_n}\left(\norm{L_{n,m_n - i}-\id}\prod_{j = 0}^{m_n - i - 1}\norm{ L_{n,j}}\right) \\
				&\leq DA'C \epsilon.
			\end{align*}
			Since $\epsilon$ can be chosen to be arbitrarily small, equation \eqref{eqn: slither limit} follows.
		\end{proof}
		
		\subsection{The uniqueness of \texorpdfstring{$\Sigma$}{S}}
		We now prove the uniqueness of $\Sigma$.
		
		\begin{proposition}\label{prop: uniqueness}
			If $\Sigma$ and $\Sigma'$ are slithering maps that are compatible with $\xi$, then $\Sigma=\Sigma'$. 
		\end{proposition}

		\begin{proof}
			Let $g_1,g_2\in\wt\Lambda$ and let $(\mathcal{X}_n)_{n=1}^\infty$ be an exhausting sequence of $P(g_1,g_2)$. For all $n>0$, let $m_n$ denote the cardinality of the set $\mathcal{X}_n$, and for all indices $i\in\{1,\dots,m_n\}$, let $h_{n,i,\pm}$ be the leaf defined in Step 2 (see equation \eqref{eqn: hni}), with $h_{n,0,+}:=g_1$ and $h_{n,m_n+1,-}:=g_2$. By Condition (1) of Definition \ref{def: slithering map}, we may write
			\begin{align*}
				\Sigma(g_1,g_2)= \Sigma(h_{n,0, +},h_{n,1,-}) \Sigma&(h_{n,1, -},h_{n,1, +}) \dots \\
				&\Sigma(h_{n,m_n, -},h_{n,m_n, +})\Sigma(h_{n,m_n, +},h_{n,m_n+1, -})
			\end{align*}
			and 
			\begin{align*}
				\Sigma'(g_1,g_2)= \Sigma'(h_{n,0, +},h_{n,1,-}) \Sigma'&(h_{n,1, -},h_{n,1, +}) \dots \\
				&\Sigma'(h_{n,m_n, -},h_{n,m_n, +})\Sigma'(h_{n,m_n, +},h_{n,m_n+1, -}).
			\end{align*}
			
			It follows from conditions (3) and (4) of Definition \ref{def: slithering map} that for all indices $i\in\{1,\dots,m_n\}$,
			\begin{align}\label{eqn: unique 3}
				\Sigma(h_{n,i,-},h_{n,i,+})=\Sigma'(h_{n,i,-},h_{n,i,+}).
			\end{align} 
			On the other hand, to compare the transformations $\Sigma(h_{n,i+1,+},h_{n,i,-})$ and $\Sigma'(h_{n,i+1,+},h_{n,i,-})$, we apply the following lemma. 
			
			\begin{lemma}\label{lem: hyp geom}\
				For every $\epsilon>0$, there exists some positive integer $N$ such that if $n>N$, then
				\[\sum_{i=0}^{m_n}\norm{\Sigma(h_{n,i+1,+},h_{n,i,-})-\Sigma'(h_{n,i+1,+},h_{n,i,-})}\leq \epsilon.\]
			\end{lemma}
			
			\begin{proof}
				By Definition \ref{def: slithering map} Condition (2), there exist constants $A,\nu>0$ such that
				\begin{align}\label{eqn: 4.3}
					\norm{\Sigma(h_1,h_2)-\id},\norm{\Sigma'(h_1,h_2)-\id}\leq A \,d_\infty(h_1,h_2)^\nu
				\end{align}
				for all leaves $h_1,h_2\in Q(g_1,g_2)$. We may assume that $\nu\in(0,1)$ by enlarging $A$ if necessary. By Lemma \ref{lem: Cauchy}, there exists a sufficiently large integer $N$ such that 
				\[\sum_{T\in P(g_1,g_2)-\mathcal{X}_N}A\,d_\infty(x_{T,1},x_{T,3})^\nu\leq\frac{\epsilon}{2}.\]
				Thus, for all $n>N$, 
				\[\sum_{i=0}^{m_n}\norm{\Sigma(h_{n,i+1,+},h_{n,i,-})-\Sigma'(h_{n,i+1,+},h_{n,i,-})}\leq 2\sum_{i=0}^{m_n}A\,d_\infty(h_{n,i, +},h_{n,i+1, -})^\nu\leq \epsilon,\]
				where the last inequality uses the concavity of the function $x\mapsto x^\nu$ on $\mathbb{R}^+$ given that $\nu\in(0,1)$.
			\end{proof}

			For all positive integers $n$ and all indices $i\in\{1,\dots,m_n\}$, let
			\begin{align*}
				V_{n,i}&:=\Sigma(h_{n,0, +},h_{n,1,-})\Sigma(h_{n,1, -},h_{n,1, +}) \dots \Sigma(h_{n,i-1, +},h_{n,i,-}) \Sigma(h_{n,i, -},h_{n,i, +})  , \\
				V_{n,i}'&:=\Sigma'(h_{n,i, +},h_{n,i+1, -})\Sigma'(h_{n,i+1, -},h_{n,i+1, +})\dots\\
				&\hspace{5cm}\Sigma'(h_{n,m_n, -},h_{n,m_n,+})  \Sigma'(h_{n,m_n, +},h_{n,m_n+1, -}) ,
			\end{align*}
			and let $W_{n,i}:=V_{n,i} V_{n,i}'$. Then by equation \eqref{eqn: unique 3},
			\begin{align*}
				&\norm{\Sigma(g_1,g_2)-\Sigma'(g_1,g_2)}\\
				\leq&\norm{\Sigma'(g_1,g_2)-W_{n,1}}+\norm{W_{n,1}-W_{n,2}}+\dots+\norm{W_{n,m_n}-\Sigma(g_1,g_2)}\\
				\leq&\sum_{i=0}^{m_n}\left(\prod_{j=0}^{i-1}\norm{\Sigma(h_{n,j,+},h_{n,j+1,-})}\right)\left(\prod_{j=i+1}^{m_n}\norm{\Sigma'(h_{n,j,+},h_{n,j+1,-})}\right)\\
				&\hspace{2cm}\cdot\norm{\Sigma(h_{n,i, +},h_{n,i+1, -})-\Sigma'(h_{n,i, +},h_{n,i+1, -})}\left(\prod_{j=1}^{m_n}\norm{\Sigma(h_{n,j,-},h_{n,j,+}))}\right).
			\end{align*}
			
			Lemma \ref{lem: Cauchy} implies that 
			\[B:=\prod_{T\in P(g_1,g_2)}\left(1+A\,d_\infty(x_{T,1},x_{T,3})^\nu\right)<\infty,\]
			so it follows from inequality \eqref{eqn: 4.3} and the concavity of the function $x\mapsto 1+Ax^\nu$ on $\mathbb{R}^+$ that for all $n\in\mathbb{Z}^+$ and all $i\in\{0,\dots,m_n\}$, we have
			\begin{align*}
				&\prod_{j=0}^{i-1}\norm{\Sigma(h_{n,j, +},h_{n,j+1, -})}\prod_{j=i+1}^{m_n}\norm{\Sigma'(h_{n,j, +},h_{n,j+1, -})}\prod_{j=1}^{m_n}\norm{\Sigma(h_{n,j, -},h_{n,j,+})}\\
				\leq&\prod_{j=0}^{m_n}(1+A\,d_\infty(h_{n,j, +},h_{n,j+1, -})^\nu)\prod_{j=1}^{m_n}(1+A\,d_\infty(h_{n,j, -},h_{n,j, +})^\nu) \leq B.
			\end{align*}
			Thus, by Lemma \ref{lem: hyp geom}, there exists a positive integer $N$ such that for all $n>N$, 
			\[\norm{\Sigma(g_1,g_2)-\Sigma'(g_1,g_2)}\leq B\epsilon.\]
			Since $\epsilon$ can be chosen to be arbitrarily small, $\Sigma(g_1,g_2)=\Sigma'(g_1,g_2)$.
		\end{proof}
		
		\subsection{The \texorpdfstring{$\rho$}{r}-equivariance of \texorpdfstring{$\Sigma$}{S}}
		
		Finally, we deduce that if $\xi$ is $\rho$--equivariant for some representation $\rho:\Gamma\to\mathsf{PGL}_d(\mathbb{K})$, then $\Sigma$ is also $\rho$--equivariant.

		To see this, observe that for all projective transformations $A\in\mathsf{PGL}_d(\mathbb{K})$, the map $\xi_A:\partial\wt\lambda\to\mathcal{F}(\mathbb{K}^d)$ given by
		\[\xi_A:p\mapsto A\cdot\xi(p)\]
		is also $\lambda$--transverse and locally $\lambda$--H\"older continuous, and the map 
		\[\Sigma_A:\wt\Lambda^2\to\mathsf{SL}_d(\mathbb{K})\] 
		given by $\Sigma_A:(g_1,g_2)\mapsto A\Sigma(g_1,g_2) A^{-1}$ is a slithering map compatible with $\xi_A$. In the case where $A=\rho(\gamma)$ for some $\gamma\in\Gamma$, the $\rho$--equivariance of $\xi$ implies that $\xi_A=\xi$. Thus, $\Sigma_A$ and $\Sigma$ are both slithering maps compatible with $\xi$, so Proposition~\ref{prop: uniqueness} implies that $\Sigma_A=\Sigma$. Since this is true for all $\gamma\in\Gamma$, $\Sigma$ is $\rho$--equivariant. This finishes the proof of Theorem \ref{thm: slitherable and slithering}.

\section{Complex eruptions, complex shears and the injectivity of \texorpdfstring{$\mathcal{sb}_{d,{\bf x}}$}{sb-d-x}} \label{sec:eruption and shear}
Recall that $\partial\widetilde\lambda$ denotes the set of endpoints of the leaves of $\widetilde\lambda$, and $\wt\Delta^o$ denotes the set of labelings, i.e. triples of points in $\partial\wt\lambda$ that are the vertices of a plaque of $\wt\lambda$. The main objective of the current section is the proof of the following statement, which constitutes Step 1 in the outline of the proof of Theorem~\ref{thm: main}, see Section~\ref{subsec:proof structure}.

\begin{theorem} \label{thm: injectivity}
	For all labelings ${\bf x}\in\wt\Delta^o$, the map
	\[\mathcal{sb}_{d,{\bf x}}:\mathcal{R}_d(\lambda)\to \mathcal{Y}_d(\lambda;\mathbb{C}/2\pi i\mathbb{Z})\times\mathcal{N}(\mathbb{C}^d)\] 
	defined by identity \eqref{eqn: main thm} is injective.
\end{theorem} 

The key step in the proof of Theorem \ref{thm: injectivity} is the construction, for each locally $\lambda$--H\"older continuous, $\lambda$--transverse, and $\lambda$--hyperconvex map, $\xi:\partial\widetilde\lambda\to\mathcal{F}(\mathbb{C}^d)$, of a certain H\"older extendable map $L$ whose dependence on the shear-bend coordinates of $\xi$ is explicit, and whose extension gives the slithering map compatible with $\xi$, see Lemma \ref{lem: injectivity} for a more precise statement. The map $L$ will also be used later to prove the holomorphicity of $\mathcal{sb}_{d,{\bf x}}$.

In order to construct the map $L$, we will introduce the notions of \emph{complex eruption} and \emph{complex shear}, which are the natural complex analogues of two classes of projective transformations previously used by Sun-Wienhard-Zhang \cite{SWZ} (and Wienhard-Zhang \cite{WZ} in the $d=3$ case) to define the eruption flows and shearing flows on the space of Frenet curves. The map $L$ will then be defined using a suitable combinations of eruptions and shears, which we call \emph{unipotent eruptions}. Complex shears, complex eruptions, and unipotent eruptions are also essential ingredients used later in the construction of the generalized bending map in Section~\ref{sec:bending}.

The definitions and basic properties of the complex eruptions, complex shears, and unipotent eruptions are discussed in Sections \ref{sec: complex eruption}, \ref{sec: complex shear}, and \ref{sec: unipotent eruption}, respectively. Section \ref{sec: explicit expression} is dedicated to the construction of $L$ and a discussion of some of its important properties, which will be later applied in Section \ref{sec: injectivity proof} to prove Theorem~\ref{thm: injectivity}.

\subsection{Complex eruptions}\label{sec: complex eruption}
The first type of elementary deformations are the \emph{complex eruptions}. These transformations allow us to deform a triple of flags in general position into one with prescribed triangle invariants. 

Recall that $\mathcal{B}$ denotes the set of triples of positive integers that sum to $d$. For any triple of flags ${\bf F}=(F_1,F_2,F_3)$ in $\mathcal{F}(\mathbb{C}^d)$ that is in general position, for any triple of integers ${\bf j}=(j_1,j_2,j_3)\in\mathcal{B}$, and for any $\delta\in\mathbb{C}/2\pi i\mathbb{Z}$, we define the \emph{${\bf j}$--elementary eruption of ${\bf F}$ by $\delta$}, denoted
\[a^{\bf j}_{\bf F}(\delta)\in\mathsf{PGL}_d(\mathbb{C}),\]
to be the projectivization of the linear map $A^{\bf j}_{\bf F}(\delta)\in\mathsf{GL}_d(\mathbb{C})$ that acts as scaling by $e^\delta$ on $F_2^{j_2}$ and the identity on $F_{1}^{j_1}+F_{3}^{j_3}$.

If ${\bf F}=(F_1,F_2,F_3)$ is a triple of flags in $\mathcal{F}(\mathbb{C}^d)$ and ${\bf j}=(j_1,j_2,j_3) \in \mathcal{B}$, we denote
\[{\bf F}_+:=(F_2,F_3,F_1),\quad{\bf F}_-:=(F_3,F_1,F_2),\quad{\bf j}_+:=(j_2,j_3,j_1),\,\,\text{ and }\,\,{\bf j}_-:=(j_3,j_1,j_2).\]
The following lemma summarizes the properties of elementary eruptions:
\begin{lemma}\label{lem: properties simple eruption}
	Let ${\bf F}=(F_1,F_2,F_3)$ be a triple of flags in general position. Then for any triple of integers ${\bf j}=(j_1,j_2,j_3)\in\mathcal{B}$ and $\delta\in\mathbb{C}/2\pi i \mathbb{Z}$, we have:
	\begin{enumerate}
		\item $a^{{\bf j}}_{{\bf F}}(\delta)$, $a^{{\bf j}_+}_{{\bf F}_+}(\delta)$, $a^{{\bf j}_-}_{{\bf F}_-}(\delta)$ pairwise commute, and $a^{{\bf j}}_{{\bf F}}(\delta) \, a^{{\bf j}_+}_{{\bf F}_+}(\delta) \, a^{{\bf j}_-}_{{\bf F}_-}(\delta)=\id$.
		\item If $G_3 \in\mathcal{F}(\mathbb{C}^d)$ is a flag such that ${\bf G}:=(F_1,F_2,G_3)$ is in general position and $G_3^{j_3}=F_3^{j_3}$, then $a^{\bf j}_{{\bf G}}(\delta)=a^{\bf j}_{\bf F}(\delta)$.
		\item The transformation $a^{{\bf j}}_{{\bf F}}(\delta)$ fixes the flag $F_2$.
		\item The triples $(F_1,F_2,a_{{\bf F}}^{{\bf j}}(\delta)\cdot F_3)$, $(a_{{\bf F}_+}^{{\bf j}_+}(\delta)\cdot F_1,F_2,F_3)$, and $(F_1,a_{{\bf F}_-}^{{\bf j}_-}(\delta)\cdot F_2,F_3)$ are projectively equivalent.
		\item The triple of flags $(F_1,F_2,a_{\bf F}^{\bf j}(\delta)\cdot F_3)$ is in general position, and
		\[\tau^{\bf k}(F_1,F_2,a_{\bf F}^{\bf j}(\delta) \cdot F_3)=\left\{\begin{array}{ll}
			\tau^{\bf k}(F_1,F_2,F_3)&\text{if }{\bf k}\neq{\bf j},\\
			\tau^{\bf k}(F_1,F_2,F_3)+\delta&\text{if }{\bf k}={\bf j}.
		\end{array}\right.\]
		for all triples of integers ${\bf k}\in\mathcal{B}$.
		\item For any $\delta'\in\mathbb{C}/2\pi i \mathbb{Z}$ and for any triple of integers ${\bf k}=(k_1,k_2,k_3)\in\mathcal{B}$ such that $k_3=j_3$, we have
		\[a^{{\bf k}}_{\bf F}(\delta')\,a_{\bf F}^{\bf j}(\delta)=a_{\bf F}^{\bf j}(\delta)\,a^{{\bf k}}_{\bf F}(\delta').\]
	\end{enumerate}
\end{lemma}

\begin{proof}
	Parts (1) and (2) follow directly from the definition. We focus on the proof of the remaining parts.
	
	\textit{Part (3).} Let $t\in\{1,\dots,d-1\}$. For any vector $v\in F_2^{t}$, we may decompose $v=w+u$, where $w\in F_2^{j_2}$ and $u\in F_1^{j_1}+F_3^{j_3}$. (If $t\leq j_2$, then $w=v$ and $u=0$.) Then
	\[A^{{\bf j}}_{{\bf F}}(\delta)\cdot v=e^\delta w+u=(e^\delta-1)w+v.\] 
	It follows that $a^{{\bf j}}_{{\bf F}}(\delta)\cdot F_2^{t}=F_2^{t}$. Since $t$ is arbitrary, part (3) holds. 
	
	\textit{Part (4).} We will only prove that $(F_1,a_{{\bf F}_-}^{{\bf j}_-}(\delta)\cdot F_2,F_3)$ and $(F_1,F_2,a_{{\bf F}}^{{\bf j}}(\delta)\cdot F_3)$ are projectively equivalent; the other cases are similar. By parts (1) and (3), 
	\begin{align*}
		(F_1,a_{{\bf F}_-}^{{\bf j}_-}(\delta)\cdot F_2,F_3)
		&=a_{{\bf F}_-}^{{\bf j}_-}(\delta)\cdot (F_1,F_2,a_{{\bf F}}^{{\bf j}}(\delta)a_{{\bf F}_+}^{{\bf j}_+}(\delta)\cdot F_3)\\
		&=a_{{\bf F}_-}^{{\bf j}_-}(\delta)\cdot (F_1,F_2,a_{{\bf F}}^{{\bf j}}(\delta)\cdot F_3).
	\end{align*}
	
	\textit{Part (5).} To prove that the triple $(F_1,F_2,a_{\bf F}^{\bf j}(\delta)\cdot F_3)$ is in general position, we need to show that 
	\begin{align}\label{eqn: wedge}
		F_1^{k_1}+ F_2^{k_2}+ a_{\bf F}^{\bf j}(\delta) \cdot F_3^{k_3}
	\end{align}
	is a direct sum for all ${\bf k}=(k_1,k_2,k_3)\in\mathcal{B}$. 
	Choose an identification $\mathsf{\Lambda}^d\mathbb{C}^d\simeq\mathbb{C}$. For all indices $\ell\in\{1,2,3\}$ and $m\in\{1,\dots,d-1\}$, select a nonzero vector
	\[f_\ell^{m}\in \mathsf{\Lambda}^m F_\ell^{m}.\] 
	Fix a triple of integers ${\bf k}=(k_1,k_2,k_3)\in\mathcal{B}$, and observe that there is some $\ell\in\{1,2,3\}$ such that $k_\ell\leq j_\ell$. By part (4), we may assume that $k_3\leq j_3$ up to relabeling $F_1$, $F_2$, and $F_3$ if necessary. Then
	\begin{align}\label{eqn: wedge2}
		f_1^{k_1}\wedge f_2^{k_2}\wedge\left(A_{{\bf F}}^{{\bf j}}(\delta)\cdot f_3^{k_3}\right)=f_1^{k_1}\wedge f_2^{k_2}\wedge f_3^{k_3}.
	\end{align}
	Since the triple of flags ${\bf F}$ is in general position, this implies that 
	\[f_1^{k_1}\wedge f_2^{k_2}\wedge\left(A_{{\bf F}}^{{\bf j}}(\delta)\cdot  f_3^{k_3}\right)\neq 0,\]
	which in turn implies that identity \eqref{eqn: wedge} holds.
	
	Next, we calculate $\tau^{\bf k}(F_1,F_2,a_{\bf F}^{\bf j}(\delta)\cdot F_3)$ when ${\bf k}\neq{\bf j}$. In this case, there is some index $\ell\in\{1,2,3\}$ such that $k_\ell<j_\ell$. Again by part (4), we may assume that $k_3<j_3$. Under these hypotheses, the equality
	\[
	\tau^{\bf k}(F_1,F_2,a_{\bf F}^{\bf j}(\delta)\cdot F_3) = \tau^{\bf k}(F_1,F_2,F_3)
	\]
	follows from a straightforward computation using equation \eqref{eqn: wedge2} and the definition of the triple ratio. 
	
	Finally, we calculate $\tau^{\bf j}(F_1,F_2,a_{\bf F}^{\bf j}(\delta)\cdot F_3)$. Observe that the $(j_3 + 1)$--form $f_3^{j_3 + 1}$ can be expressed as $f_3^{j_3} \wedge w$, for some vector $w \in F_3^{j_3 + 1} \cap (F_1^{j_1}+F_2^{j_2})$ and we can decompose $w=w_1+w_2$, where $w_1\in F_1^{j_1}$ and $w_2\in F_2^{j_2}$. Then 
	\begin{align*}
		f_1^{j_1}\wedge f_2^{j_2-1}\wedge (A_{{\bf F}}^{{\bf j}}(\delta)\cdot  f_3^{j_3+1})&=f_1^{j_1}\wedge f_2^{j_2-1}\wedge (A_{{\bf F}}^{{\bf j}}(\delta)\cdot  f_3^{j_3})\wedge (A_{{\bf F}}^{{\bf j}}(\delta)\cdot w)\\
		&=f_1^{j_1}\wedge f_2^{j_2-1}\wedge f_3^{j_3}\wedge (e^\delta w_2)\\
		&=e^\delta f_1^{j_1}\wedge f_2^{j_2-1}\wedge f_3^{j_3+1}.
	\end{align*}
	Similarly,
	\[		f_1^{j_1-1}\wedge f_2^{j_2}\wedge \left(A_{{\bf F}}^{{\bf j}}(\delta)\cdot  f_3^{j_3+1}\right)=f_1^{j_1-1}\wedge f_2^{j_2}\wedge f_3^{j_3+1}.
	\]
	Therefore, by definition the quantity $T^{{\bf j}}(F_1,F_2,a_{\bf F}^{\bf j}(\delta)\cdot F_3)$ coincides with
	\begin{align*}
		\frac{f_1^{j_1+1}\wedge f_2^{j_2}\wedge f_3^{j_3-1}}{f_1^{j_1-1}\wedge  f_2^{j_2}\wedge f_3^{j_3+1}} \frac{e^\delta f_1^{j_1}\wedge f_2^{j_2-1}\wedge f_3^{j_3+1}}{f_1^{j_1}\wedge f_2^{j_2+1}\wedge f_3^{j_3-1}}  \frac{f_1^{j_1-1}\wedge f_2^{j_2+1}\wedge f_3^{j_3}}{f_1^{j_1+1}\wedge f_2^{j_2-1}\wedge f_3^{j_3}} = e^{\delta} \, T^{{\bf j}}(F_1,F_2,F_3)
	\end{align*}
	as desired.	
	
	\textit{Part (6).} Since $j_3 = k_3$, we may assume that $j_1\leq k_1$ and $j_2\geq k_2$; the other case is similar. Consider a basis $(b_1,b_2,\dots, b_d)$ of $\mathbb C^d$ that satisfies
	\begin{align*}
		\mathrm{Span}(b_1,\dots, b_{j_1})=F_1^{j_1}, & \qquad \mathrm{Span}(b_{j_1+1},\dots, b_{k_1})=(F_1^{k_1}+F_3^{k_3})\cap F_2^{j_2}, \\
		\mathrm{Span}(b_{k_1+1},\dots, b_{j_1+j_2})=F_2^{k_2},&\qquad  \mathrm{Span}(b_{j_1+j_2+1},\dots, b_d)=F_3^{j_3}=F_3^{k_3}.
	\end{align*}
	Such a basis exists because $(F_1,F_2,F_3)$ is in general position. It follows from their definitions that $A^{\bf k}_{\bf F}(\delta')$ and $A_{\bf F}^{\bf j}(\delta)$ are diagonal in this basis, so they commute.
\end{proof}

Given any triple of flags ${\bf F}=(F_1,F_2,F_3)$ in $\mathcal{F}(\mathbb{C}^d)$ that is in general position, and any ${\zeta}:=(\zeta^{\bf j})_{{\bf j}\in\mathcal{B}}\in(\mathbb{C}/2\pi i\mathbb{Z})^{\mathcal{B}}$, Lemma \ref{lem: properties simple eruption} part (6) implies that
\[a^n_{\bf F}({\zeta}):=\prod_{{\bf j}\in\mathcal{B}:j_3=n}a^{\bf j}_{\bf F}(\zeta^{\bf j})\quad\text{and}\quad A^n_{\bf F}({\zeta}):=\prod_{{\bf j}\in\mathcal{B}:j_3=n}A^{\bf j}_{\bf F}(\zeta^{\bf j})\]
is a commuting product for all $n\in\{1,\dots,d-2\}$. We then define the \emph{complex eruption of ${\bf F}$ by $\zeta$} as 
\[a_{\bf F}({\zeta}):=a^1_{\bf F}({\zeta})\cdot\ldots\cdot a^{d-2}_{\bf F}({\zeta})\in\mathsf{PGL}_d(\mathbb{C}),\]
which is the projectivization of 
\[A_{\bf F}({\zeta}):=A^1_{\bf F}({\zeta})\cdot\ldots\cdot A^{d-2}_{\bf F}({\zeta})\in\mathsf{GL}_d(\mathbb{C}).\]

If ${\bf F}=(F_1,F_2,F_3)$ is a triple of flags in $\mathcal{F}(\mathbb{C}^d)$ and ${\bf j}=(j_1,j_2,j_3) \in \mathcal{B}$, we denote
\[\wh{\bf F}:=(F_2,F_1,F_3),\quad{\bf F}':=(F_3,F_2,F_1),\quad\wh{\bf j}:=(j_2,j_1,j_3),\,\,\text{ and }\,\,{\bf j}':=(j_3,j_2,j_1).\]
Then for any ${\zeta}=(\zeta^{\bf j})_{{\bf j}\in\mathcal{B}}\in(\mathbb{C}/2\pi i \mathbb{Z})^{\mathcal{B}}$, let $\zeta_+,\zeta_-,\wh{\zeta},\zeta'\in(\mathbb{C}/2\pi i \mathbb{Z})^{\mathcal{B}}$ be defined by
\begin{equation}\label{eqn: permute}
	(\zeta_+)^{\bf j}:= \zeta^{{\bf j}_+},\quad  (\zeta_-)^{\bf j}:= \zeta^{{\bf j}_-},\quad (\wh{\zeta})^{\bf j}:= - \zeta^{\wh{\bf j}} , \,\,\text{ and }\,\,  (\zeta')^{\bf j}:= - \zeta^{{\bf j}'}
\end{equation}
for all ${\bf j}  \in \mathcal{B}$. The complex eruption $a_{\bf F}({\zeta})$ satisfies the following properties.

\begin{proposition}\label{prop: eruption combination}
	Let ${\bf F}=(F_1,F_2,F_3)$ be a triple of flags in $\mathcal{F}(\mathbb{C}^d)$ that are in general position, and let ${\zeta}=(\zeta^{\bf j})_{{\bf j}\in\mathcal{B}}\in(\mathbb{C}/2\pi i \mathbb{Z})^{\mathcal{B}}$. Then:
	\begin{enumerate}
		\item The transformation $a_{\bf F}({\zeta})$ is the unique element in $\mathsf{PGL}_d(\mathbb{C})$ that fixes the flag $F_2$, the lines $F_1^1$ and $F_3^1$, and satisfies the property that the triple of flags $(F_1,F_2,a_{\bf F}({\zeta}) \cdot F_3)$ is in general position and
		\[\tau^{\bf j}(F_1,F_2,a_{\bf F}({\zeta})\cdot F_3)=\tau^{\bf j}(F_1,F_2,F_3)+\zeta^{\bf j}\]
		for every triple of integers ${\bf j}\in\mathcal{B}$.
		\item Let $(b_1,\dots,b_d)$ be a basis of $\mathbb{C}^d$ such that $F_2^{m}=\Span_{\mathbb{C}}(b_1,\dots,b_m)$ for all $m\in\{1,\dots,d-1\}$. The linear representative $A_{\bf F}({\zeta})\in\mathsf{GL}_d(\mathbb{C})$ of $a_{\bf F}({\zeta})$, when written in this basis, is an upper triangular matrix whose $n$--th diagonal entry is 
		\[\prod_{{\bf j}\in\mathcal{B}:j_2\geq n}\exp\left(\zeta^{\bf j}\right),\]
		and whose $(i,j)$--th entry for all $i\le j$ is of the form
		\[P_{i,j}(\{\exp(\zeta^{\bf j}):{\bf j}\in\mathcal{B}\})\]
		where $P_{i,j}$ is a polynomial of $\abs{\mathcal{B}}$ variables that is of degree at most $1$ in each variable.
		\item $a_{\widehat{\bf F}}(\widehat{\zeta}) \cdot F_3 = a_{\bf F}(\zeta) \cdot F_3$.	
		\item $a_{\bf F'}(\zeta') = a_{\bf F}(\zeta)^{-1}$.
		\item $a_{\bf F}(\zeta) \, a_{{\bf F}_+}(\zeta_+) \, a_{{\bf F}_-}(\zeta_-) = \mathrm{id}$.
	\end{enumerate}
\end{proposition}

\begin{proof}
	
	\noindent \textit{Part (1).} We first check that the transformation $a_{\bf F}(\zeta)$ fixes $F_2$, $F_1^1$, and $F_3^1$. Notice that for every ${\bf j} \in \mathcal{B}$ the transformation $a_{\bf F}^{\bf j}(\zeta^{\bf j})$ fixes the lines $F_1^1, F_3^1$ and, by Lemma \ref{lem: properties simple eruption} part (3), the flag $F_2$. In particular the same holds for $a_{\bf F}(\zeta)$. 
	
	We now show that $(F_1,F_2,a_{\bf F}(\zeta) \cdot F_3)$ is in general position and we determine the value of $\tau^{\bf j}(F_1,F_2,a_{\bf F}(\zeta)\cdot F_3)$ for any triple of integers ${\bf j} \in \mathcal{B}$. Set
	\[H_n:=a^n_{\bf F}(\zeta)\, a^{n+2}_{\bf F}(\zeta)\, \cdots \, a^{d-2}_{\bf F}(\zeta)\cdot F_3,\,\, \text{ if } n \in\{1,\dots, d-2\}\]
	and $H_{d-1}:=F_3$, and denote ${\bf H}_n:=(F_1,F_2,H_n)$ for all $n\in\{1,\dots,d-1\}$. Observe from the definition of the elementary eruptions that for all $k\in\{1,\dots,d-2\}$ and for all triples of integers ${\bf j}=(j_1,j_2,j_3)\in\mathcal{B}$ such that $j_3\ge k$, we have that $a_{\bf F}^{\bf j}(\zeta^{\bf j})$ fixes $F_3^k$. It follows that $H_{n}^{k}=F_3^{k}$ for every $k\in\{1,\dots,n\}$ and all $n\in\{1,\dots,d-1\}$. Thus, for all integers $n\in\{1,\dots,d-2\}$, Lemma~\ref{lem: properties simple eruption} part (2) gives $a_{{\bf H}_{n+1}}^{n}(\zeta)=a_{{\bf F}}^{n}(\zeta)$, and iterative applications of Lemma~\ref{lem: properties simple eruption} part (5) imply that 
	\[{\bf H}_n=(F_1,F_2,a_{{\bf H}_{n+1}}^{n}(\zeta)\cdot H_{n+1})\] 
	is in general position. In particular, ${\bf H}_1 = (F_1,F_2,a_{\bf F}(\zeta) \cdot F_3)$ is in general position and
	\[a_{\bf F}(\zeta)\cdot F_3=a^1_{{\bf H}_2}(\zeta) \, a^2_{{\bf H}_3}(\zeta)\,\cdots \, a^{d-2}_{{\bf H}_{d-1}}(\zeta)\cdot H_{d-1}.\]
	Thus, to prove that $\tau^{\bf j}(F_1,F_2,a_{\bf F}({\zeta})\cdot F_3)=\tau^{\bf j}(F_1,F_2,F_3)+\zeta^{\bf j}$ for every ${\bf j}\in\mathcal{B}$, it now suffices to show that for any $n\in\{1,\dots,d-2\}$
	\begin{align*}
		\tau^{\bf j}(F_1,F_2,H_n) & = \tau^{\bf j}(F_1,F_2,a_{{\bf H}_{n+1}}^{n}(\zeta)\cdot H_{n+1}) \\
		& = \begin{cases}
			\tau^{\bf j}(F_1,F_2,H_{n+1})+\zeta^{\bf j} & \text{if $j_3=n$,}\\
			\tau^{\bf j}(F_1,F_2,H_{n+1}) & \text{if $j_3\neq n$.}
		\end{cases}
	\end{align*}
	This follows easily from Lemma \ref{lem: properties simple eruption} parts (2) and (5).
	
	For the uniqueness, suppose that $a , b \in \mathsf{PGL}_d(\mathbb{C})$ are projective transformations that fix the flag $F_2$, the lines $F_1^1$, $F_3^1$, and such that the triples of flags $(F_1,F_2,a \cdot F_3)$ and $(F_1,F_2,b \cdot F_3)$ are in general position and satisfy
	\[
	\tau^{\bf j}(F_1,F_2, a\cdot F_3)  = \tau^{\bf j}(F_1, F_2, b\cdot F_3)
	\]
	for all ${\bf j} \in \mathcal{B}$. Then $(F_1,F_2,a \cdot F_3)$ and $(F_1,F_2,b \cdot F_3)$ are projectively equivalent and 
	\[
	(F_1,F_2,a \cdot F_3^1) = (F_1,F_2,F_3^1) = (F_1,F_2,b \cdot F_3^1) ,
	\]
	so $(F_1,F_2,a \cdot F_3)=(F_1,F_2,b \cdot F_3)$. In particular, $a \cdot F_3 = b \cdot F_3$. Hence the transformation $a^{-1} b$ fixes the flags $F_2$, $F_3$, and the line $F_1^1$, so $a^{-1} b = \mathrm{id} \in \mathsf{PGL}_d(\mathbb{C})$, as desired.

	\textit{Part (2).} For each ${\bf j}\in\mathcal{B}$, choose a basis $(b_{1,{\bf j}},\dots,b_{d,{\bf j}})$ such that 
	\begin{itemize}
		\item $F_2^m=\Span_{\mathbb{C}}(b_{1,{\bf j}},\dots,b_{m,{\bf j}})$ for all $m\in\{1,\dots,j_2\}$, 
		\item $b_{m,{\bf j}}\in F_2^m\cap (F_1^{j_1+j_2+1-m}+F_3^{j_3})$ for all $m\in\{j_2+1,\dots,j_1+j_2\}$, and 
		\item $b_{m,{\bf j}}\in F_2^m\cap (F_3^{d+1-m})$ for all $m\in\{j_1+j_2+1,\dots,d\}$.
	\end{itemize}
	Notice that in this basis, the linear transformation $A^{\bf j}_{\bf F}(\zeta^{\bf j})$ is a diagonal matrix whose first $j_2$ entries along the diagonal are $\exp\zeta^{\bf j}$ and whose remaining $j_1+j_3$ entries along the diagonal are $1$. Also, the linear map that sends $(b_1,\dots,b_d)$ to $(b_{1,{\bf j}},\dots,b_{d,{\bf j}})$ is upper triangular (in either basis). Thus, in the basis $(b_1,\dots,b_d)$, $A^{\bf j}_{\bf F}(\zeta^{\bf j})$ is an upper triangular matrix whose first $j_2$ entries along the diagonal are $\exp\zeta^{\bf j}$, whose remaining $j_1+j_3$ entries along the diagonal are $1$, and whose strict upper triangular entries are polynomials of $\exp\zeta^{\bf j}$ of degree at most $1$. The claim now follows from the definition of $A_{\bf F}(\zeta)$.	
	
	\textit{Part (3).} Since $a_{\widehat{\bf F}}(\widehat{\zeta})$ and $a_{\bf F}(\zeta)$ both fix $F_3^1$, to prove that $a_{\widehat{\bf F}}(\widehat{\zeta}) \cdot F_3 = a_{\bf F}(\zeta) \cdot F_3$, it suffices to verify that the triples $(F_1,F_2,a_{\widehat{\bf F}}(\widehat{\zeta}) \cdot F_3)$ and $(F_1,F_2,a_{\bf F}(\zeta) \cdot F_3)$ are projectively equivalent or, equivalently, that they have the same triple ratios. For every ${\bf j} \in \mathcal{B}$, we have:
	\begin{align*}
		\tau^{\bf j}(F_1,F_2,a_{\widehat{\bf F}}(\widehat{\zeta}) \cdot F_3) & = - \tau^{\widehat{\bf j}}(F_2,F_1,a_{\widehat{\bf F}}(\widehat{\zeta}) \cdot F_3) \\
		& = - \tau^{\widehat{\bf j}}(F_2,F_1,F_3) - (\widehat{\zeta})^{\widehat{\bf j}} \\
		& = \tau^{\bf j}(F_1,F_2,F_3) + \zeta^{\bf j} \\
		& = \tau^{\bf j}(F_1,F_2,a_{\bf F}(\zeta) \cdot F_3),
	\end{align*}
	where the first and third equalities are a consequence of the symmetries of the triple ratio, and the second and fourth equalities follow from part (1).
	
	\textit{Part (4).} By part (1), it suffices to show that the transformation $a_{\bf F'}(\zeta')^{-1}$ satisfies the following properties:
	\begin{enumerate}[label=(\alph*)]
		\item $a_{\bf F'}(\zeta')^{-1}\cdot F_1^1 = F_1^1$, $a_{\bf F'}(\zeta')^{-1}\cdot F_3^1 = F_3^1$, and $a_{\bf F'}(\zeta')^{-1}\cdot F_2 = F_2$,
		\item $(F_1,F_2,a_{\bf F'}(\zeta')^{-1}\cdot F_3)$ is in general position and 
		\[\tau^{\bf j}(F_1,F_2,a_{\bf F'}(\zeta')^{-1} \cdot F_3) = \tau^{\bf j}(F_1,F_2,F_3) + \zeta^{\bf j}\] 
		for every ${\bf j} \in \mathcal{B}$.
	\end{enumerate}	
	The equalities in (a) follow from part (1) applied to the transformation $a_{{\bf F}'}(\zeta')$. For (b), observe that by Lemma \ref{lem: properties simple eruption} part (3),
	\[a_{\bf F'}(\zeta')\cdot(F_1,F_2,a_{\bf F'}(\zeta')^{-1}\cdot F_3)=(a_{\bf F'}(\zeta')\cdot F_1,F_2, F_3).\]
	Thus, by part (1), $(F_1,F_2,a_{\bf F'}(\zeta')^{-1}\cdot F_3)$ is in general position, and 
	\begin{align*}
		\tau^{\bf j}(F_1,F_2,a_{\bf F'}(\zeta')^{-1}\cdot F_3) 		& = \tau^{\bf j}(a_{\bf F'}(\zeta') \cdot F_1,F_2,F_3) \\
		& = - \tau^{{\bf j}'}(F_3,F_2,a_{\bf F'}(\zeta') \cdot F_1) \\
		& = - \tau^{{\bf j}'}(F_3,F_2,F_1) - (\zeta')^{{\bf j}'} \\
		& = \tau^{\bf j}(F_1,F_2,F_3) + \zeta^{\bf j} 
	\end{align*}
	for every ${\bf j} \in \mathcal{B}$. Here, the first equality is a consequence of the projective invariance of the triple ratio, the second and fourth equality hold by the symmetries of the triple ratio, while the third equality holds by part (1).
	
	\textit{Part (5).}  By part (1), both transformations $a_{{\bf F}_-}(\zeta_-)^{-1} a_{{\bf F}_+}(\zeta_+)^{-1}$ and $a_{\bf F}(\zeta)$ fix $F_1^1$. Thus, it suffices to show that 	
	\[a_{{\bf F}_-}(\zeta_-)^{-1} a_{{\bf F}_+}(\zeta_+)^{-1}\cdot (F_2,F_3)=a_{\bf F}(\zeta)\cdot (F_2,F_3).\]
	
	Note that $\wh{({\bf F}_+)'}={\bf F}_-$ and $\widehat{(\zeta_+)'} = \zeta_-$.
	Thus,
	\begin{align*}
		a_{{\bf F}_-}(\zeta_-)^{-1} a_{{\bf F}_+}(\zeta_+)^{-1}\cdot F_2 & = a_{{\bf F}_-}(\zeta_-)^{-1} \, a_{({\bf F}_+)'}((\zeta_+)') \cdot F_2 \\
		& = a_{{\bf F}_-}(\zeta_-)^{-1} \, a_{(\wh{{\bf F}_+)'}}(\wh{(\zeta_+)'}) \cdot F_2 \\
		&=F_2\\
		& =a_{\bf F}(\zeta)\cdot F_2,
	\end{align*}
	where the first, second, and fourth equalities hold by parts (4), (3), and (1) respectively.

	Similarly, $\wh{({\bf F}_-)'}={\bf F}$ and $\widehat{(\zeta_-)'} = \zeta$. Thus,
	\begin{align*}
		a_{{\bf F}_-}(\zeta_-)^{-1} a_{{\bf F}_+}(\zeta_+)^{-1}\cdot F_3 & = a_{{\bf F}_-}(\zeta_-)^{-1} \cdot F_3 \\
		& = a_{({\bf F}_-)'}((\zeta^-)') \cdot F_3 \\
		& = a_{\wh{({\bf F}_-)'}}(\wh{(\zeta^-)'}) \cdot F_3 \\
		& = a_{\bf F}(\zeta) \cdot F_3 .
	\end{align*}
	where the first, second, and third equalities hold by parts (1), (4), and (3) respectively.
\end{proof}

\subsection{Complex shears} \label{sec: complex shear}
The second kind of building blocks are the \emph{complex shears}. Recall that $\mathcal{A}$ denotes the set of pairs of positive integers that sum to $d$.
For any pair of transverse flags ${\bf E}=(E_1,E_2)$ in $\mathcal{F}(\mathbb{C}^d)$, any pair of integers ${\bf i}=(i_1,i_2)\in\mathcal{A}$, and any $\varepsilon\in\mathbb{C}/2\pi i \mathbb{Z}$, the \emph{${\bf i}$--elementary shear of ${\bf E}$ by $\varepsilon$}, denoted by
\[c^{\bf i}_{\bf E}(\varepsilon)\in\mathsf{PGL}_d(\mathbb{C}).\] 
is the projectivization of the linear map $C^{\bf i}_{\bf E}(\epsilon)$ that acts as scaling by $e^{\varepsilon}$ on $E_1^{i_1}$ and the identity on $E_2^{i_2}$.

If ${\bf E}=(E_1,E_2)$ is a pair of flags in $\mathcal{F}(\mathbb{C}^d)$ and ${\bf i}=(i_1,i_2) \in \mathcal{A}$, we denote
\[\wh{\bf E}:=(E_2,E_1)\,\,\text{ and }\,\,\wh{\bf i}:=(i_2,i_1).\]
The elementary shears satisfy the following properties. 
\begin{lemma}\label{prop: shear double ratio}
	Let ${\bf E}=(E_1,E_2)$ be a transverse pair of flags in $\mathcal{F}(\mathbb{C}^d)$, and let $F,G \in \mathcal{F}(\mathbb{C}^d)$ be two flags such that the triples $(E_1, E_2, F)$ and $(E_1,E_2,G)$ are both in general position. Then:
	\begin{enumerate}
		\item $c^{\mathbf{i}}_{\mathbf{E}}(\varepsilon) = c^{\hat{\mathbf{i}}}_{\widehat{\mathbf{E}}}(-\varepsilon) = c^{\mathbf{i}}_{\mathbf{E}}(- \varepsilon)^{-1}$ for any ${\bf i}\in \mathcal{A}$ and $\varepsilon\in\mathbb{C}/2\pi i \mathbb{Z}$.
		\item For all ${\bf i}\in \mathcal{A}$ and $\varepsilon\in\mathbb{C}/2\pi i \mathbb{Z}$, the triple of flags $(E_1,E_2,c_{{\bf E}}^{{\bf i}}(\varepsilon)\cdot G)$ is in general position, and
		\[\sigma^{\bf k}(E_1,E_2,F,c_{{\bf E}}^{{\bf i}}(\varepsilon)\cdot G)=\left\{\begin{array}{ll}
			\sigma^{\bf k}(E_1,E_2,F,G)&\text{if }{\bf k}\neq{\bf i},\\
			\sigma^{\bf k}(E_1,E_2,F,G)+\varepsilon&\text{if }{\bf k}={\bf i}.
		\end{array}\right.\]
		for all ${\bf k}\in\mathcal{A}$.
	\end{enumerate}
\end{lemma}

We omit the proof of Lemma \ref{prop: shear double ratio} because part $(1)$ is immediate from the definition of the elementary shears, while the proof of part $(2)$ is very similar to the proof of part (5) of Lemma \ref{lem: properties simple eruption}.

Next, given a pair of transverse flags ${\bf E}=(E_1,E_2)$ in $\mathcal{F}(\mathbb{C}^d)$ and $\upsilon:=(\upsilon^{\bf i})_{{\bf i}\in\mathcal{A}}\in(\mathbb{C}/2\pi\mathbb{Z})^{\mathcal{A}}$, we may define the \emph{complex shear of ${\bf E}$ by $\upsilon$} to be the commuting product
\[c_{\bf E}({\upsilon}):=\prod_{{\bf i}\in\mathcal{A}}c^{\bf i}_{\bf E}(\upsilon^{\bf i}),\]
which is the projectivization of 
\[C_{\bf E}({\upsilon}):=\prod_{{\bf i}\in\mathcal{A}}C^{\bf i}_{\bf E}(\upsilon^{\bf i}).\]

For any $\upsilon:=(\upsilon^{\bf i})_{{\bf i}\in\mathcal{A}}\in(\mathbb{C}/2\pi\mathbb{Z})^{\mathcal{A}}$, let $\wh\upsilon\in(\mathbb{C}/2\pi\mathbb{Z})^{\mathcal{A}}$ be defined by
\begin{equation} \label{eqn: permute 2}
	(\wh\upsilon)^{\bf i}:=-\upsilon^{\wh{\bf i}}
\end{equation}
for all ${\bf i}\in\mathcal{A}$. The following proposition is the analog of Proposition~\ref{prop: eruption combination}. We omit its proof because the proofs of parts (1) and (2) are very similar to the proofs of part (1) and (2) of Proposition \ref{prop: eruption combination} (using Lemma \ref{prop: shear double ratio} in place of Lemma~\ref{lem: properties simple eruption}), and part (3) follows immediately from Lemma \ref{prop: shear double ratio} part (1).

\begin{proposition}\label{prop: shear combination}
	Let ${\bf E}=(E_1,E_2)$ be a transverse pair of flags in $\mathcal{F}(\mathbb{C}^d)$, and let ${\upsilon}:=(\upsilon^{\bf i})_{{\bf i}\in\mathcal{A}}\in(\mathbb{C}/2\pi i \mathbb{Z})^{\mathcal{A}}$. Then:
	\begin{enumerate}
		\item The projective transformation $c_{\bf E}({\upsilon})$ is the unique element of $\mathsf{PGL}_d(\mathbb{C})$ that fixes the flags $E_1$ and $E_2$, and satisfies the property that if $F,G \in \mathcal{F}(\mathbb{C}^d)$ are two flags such that $(E_1,E_2,F)$ and $(E_1,E_2,G)$ are in general position, then the triple of flags $(E_1,E_2,c_{\bf E}({\upsilon}) \cdot G)$ is in general position and 
		\[\sigma^{\bf i}(E_1,E_2,F,c_{\bf E}({\upsilon})\cdot G)=\sigma^{\bf i}(E_1,E_2,F,G)+\upsilon^{\bf i}\]
		for every ${\bf i}\in\mathcal{A}$.
		\item Let $(b_1,\dots,b_d)$ be a basis of $\mathbb{C}^d$ such that $E_1^{m}=\Span_{\mathbb{C}}(b_1,\dots,b_m)$ for all $m\in\{1,\dots,d-1\}$. The linear representative $C_{\bf E}({\upsilon})\in\mathsf{GL}_d(\mathbb{C})$ of $c_{\bf E}({\upsilon})$, when written in this basis, is an upper triangular matrix whose $n$--th diagonal entry is 
		\[\prod_{{\bf i}\in\mathcal{A}:i_1\geq n}\exp\left(\upsilon^{\bf i}\right),\]
		and whose $(i,j)$--th entry for all $i\le j$ is of the form
		\[P_{i,j}(\{\exp(\upsilon^{\bf i})\mid {\bf i}\in\mathcal{A}\})\]
		where $P_{i,j}$ is a polynomial of $\abs{\mathcal{A}}$ variables that is of degree at most $1$ in each variable. Furthermore, if $E_2^{(m)}=\Span_{\mathbb{C}}(b_{d-m+1},\dots,b_d)$ for all $m\in\{1,\dots,d-1\}$, then  $C_{\bf E}({\upsilon})$ is a diagonal matrix in the basis $(b_1,\dots,b_d)$.

		\item $c_{\mathbf{E}}(\upsilon) = c_{\widehat{\mathbf{E}}}(\widehat{\upsilon}) = c_{\mathbf{E}}(-\upsilon)^{-1}$ and $c_{\mathbf{E}}(\upsilon + \omega) = c_{\mathbf{E}}(\upsilon) \, c_{\mathbf{E}}(\omega)$ for every $\omega \in (\mathbb{C}/2\pi i \mathbb{Z})^{\mathcal{A}}$.
	\end{enumerate}
\end{proposition}

\subsection{Unipotent eruptions} \label{sec: unipotent eruption}
Using the complex eruptions and the complex shears, we may now define the unipotent eruptions.

Given ${\zeta}=(\zeta^{{\bf j}})_{{\bf j}\in\mathcal{B}}\in(\mathbb{C}/2\pi i \mathbb{Z})^{\mathcal{B}}$, let
\[
s_\zeta = (s_\zeta^\mathbf{i})_{\mathbf{i} \in \mathcal{A}} \in (\mathbb{C}/2\pi i \mathbb{Z})^{\mathcal{A}}
\]
be defined by
\begin{equation}\label{eqn: sum zeta}
	s_\zeta^\mathbf{i}:=\sum_{{\bf j}=(j_1,j_2,j_3) \in \mathcal{B} : j_2=i_1}\zeta^{\bf j} \in \mathbb{C}/2\pi i \mathbb{Z} 
\end{equation}
for all pairs of integers $\mathbf{i}=(i_1,i_2) \in \mathcal{A}$.
Then for any triple of flags ${\bf F}=(F_1,F_2,F_3)$ in $\mathcal{F}(\mathbb{C}^d)$ that is in general position, any ${\zeta}\in(\mathbb{C}/2\pi i \mathbb{Z})^{\mathcal{B}}$, and any ${\upsilon}\in(\mathbb{C}/2\pi i \mathbb{Z})^{\mathcal{A}}$, we may define the \emph{unipotent eruption of ${\bf F}$ by $(\zeta,\upsilon)$} to be the product
\begin{equation}\label{eq: unipotent eruption}
	u_{\bf F}(\zeta,\upsilon):=c_{\bf E}({\upsilon})\, a_{{\bf F}}({\zeta})\, c_{{\bf G}}(-{\upsilon}-s_\zeta) = c_{\widehat{\mathbf{E}}}(\widehat{\upsilon})\, a_{{\bf F}}({\zeta})\, c_{\widehat{\mathbf{G}}}(-\widehat{\upsilon}-\widehat{s_\zeta}),
\end{equation}
where ${\bf E} :=(F_2,F_1) $ and $ {\bf G}:=(F_2,F_3)$. The equality here follows from the relations of Proposition \ref{prop: shear combination} part (3).  

Recall that an element in $\mathsf{PGL}_d(\mathbb{C})$ is \emph{unipotent} if it has a linear representative in $\mathsf{SL}_d(\mathbb{C})$ that is a unipotent. The following proposition gives a geometric description of the unipotent eruption. These unipotent eruptions play an important role in our study of the generalized bending map in Section \ref{sec:bending}. 

\begin{proposition}\label{prop: lift}
	Let ${\bf F}=(F_1,F_2,F_3)$ be a triple of flags in $\mathcal{F}(\mathbb{C}^d)$ that is in general position, let ${\bf E}:=(F_2,F_1)$, let ${\zeta}\in(\mathbb{C}/2\pi i \mathbb{Z})^{\mathcal{B}}$, and let ${\upsilon}\in(\mathbb{C}/2\pi i \mathbb{Z})^{\mathcal{A}}$. The unipotent eruption $u_{\bf F}(\zeta,\upsilon)$ is the unique unipotent projective transformation that fixes the flag $F_2$, sends $F_3^1$ to $c_{\bf E}(\upsilon)\cdot F_3^1$, and satisfies the property that the triple of flags $(F_1,F_2,u_{\bf F}(\zeta,\upsilon)\cdot F_3)$ is in general position and
	\[
	\tau^{\bf j}(F_1,F_2,u_{\bf F}(\zeta,\upsilon)\cdot F_3) = \tau^{\bf j}(F_1,F_2,F_3) + \zeta^{\bf j}
	\]
	for every ${\bf j} \in \mathcal{B}$.
\end{proposition}

\begin{proof}
	First, we prove that $u_{\bf F}(\zeta,\upsilon)$ satisfies the stated properties. Notice that
	\[U_{\bf F}(\zeta,\upsilon):=C_{\bf E}(\upsilon)A_{\bf F}(\zeta)C_{\bf G}(-\upsilon)C_{\bf G}(-s_\zeta)\]
	is a linear representative of $u_{\bf F}(\zeta,\upsilon)$. It follows from the descriptions of the linear map $A_{\bf F}(\zeta)$ in Proposition \ref{prop: eruption combination} part (2) and the linear maps $C_{\bf E}(\upsilon)$, $C_{\bf G}(-\upsilon)$, and $C_{\bf G}(-s_\zeta)$ in Proposition \ref{prop: shear combination} part (2) that $u_{\bf F}(\zeta,\upsilon)$ is unipotent and fixes $F_2$.
	The fact that $u_{\bf F}(\zeta,\upsilon)\cdot F_3^1=c_{\bf E}({\upsilon})\cdot F_3^1$ is an immediate consequence of the observation that $c_{{\bf G}}(-\upsilon-s_\zeta)$ and $a_{{\bf F}}({\zeta})$ both fix $F_3^1$. Since $c_{{\bf G}}(-\upsilon-s_\zeta)$ fixes $F_3$, while $c_{{\bf E}}(\upsilon)$ fixes $F_2$ and $F_1$, it follows that
	\[(F_1,F_2,u_{\bf F}(\zeta,\upsilon)\cdot F_3)=c_{\bf E}({\upsilon})\cdot (F_1,F_2,a_{{\bf F}}({\zeta})\cdot F_3).\]
	Thus, by Proposition~\ref{prop: eruption combination} part (1), $(F_1,F_2,u_{\bf F}(\zeta,\upsilon)\cdot F_3)$ is in general position, and 
	\[
	\tau^{\bf j}(F_1,F_2,u_{\bf F}(\zeta,\upsilon)\cdot F_3)=\tau^{\bf j}(F_1,F_2,a_{\bf F}(\zeta)\cdot F_3) = \tau^{\bf j}(F_1,F_2,F_3) + \zeta^{\bf j}
	\]
	for every ${\bf j} \in \mathcal{B}$.
	
	Next, we prove the uniqueness of $u_{\bf F}(\zeta,\upsilon)$. Suppose that there is some $u\in\mathsf{PGL}_d(\mathbb{C})$ that also fixes $F_2$, sends $F_3^1$ to $c_{{\bf E}}(\upsilon)\cdot F_3^1$, and satisfies the property that the triple of flags $(F_1,F_2,u\cdot F_3)$ is in general position and
	\[
	\tau^{\bf j}(F_1,F_2,u_\cdot F_3) = \tau^{\bf j}(F_1,F_2,F_3) + \zeta^{\bf j}
	\]
	for every ${\bf j} \in \mathcal{B}$. Then $(F_1,F_2,u\cdot F_3)$ and $(F_1,F_2,u_{\bf F}(\zeta,\upsilon)\cdot F_3)$ are projectively equivalent and $u\cdot F_3^1=u_{\bf F}(\zeta,\upsilon)\cdot F_3^1$, so $u\cdot F_3=u_{\bf F}(\zeta,\upsilon)\cdot F_3$. But this means that $u^{-1}u_{\bf F}(\zeta,\upsilon)$ is unipotent and fixes both $F_3$ and $F_2$, so $u=u_{\bf F}(\zeta,\upsilon)$. 
\end{proof}

For any triple of flags ${\bf F}=(F_1,F_2,F_3)$ in $\mathcal{F}(\mathbb{C}^d)$ that is in general position, recall from Section \ref{subsec:define productable} that $U({\bf F})\in\mathsf{SL}_d(\mathbb{C})$ denotes the unipotent linear map that fixes $F_2$ and sends $F_1$ to $F_3$. Using unipotent eruptions, one can also obtain a formula for $U({\bf F})$.

\begin{proposition}\label{formula}
	Let ${\bf F}=(F_1,F_2,F_3)$ be a triple of flags in $\mathcal{F}(\mathbb{C}^d)$ that is in general position, and let $(b_1,\dots,b_d)$ of $\mathbb{C}^d$ be a basis such that $b_m\in F_2^{m}\cap F_3^{d-m+1}$ for all $m\in\{1,\dots,d\}$. For any $\zeta\in(\mathbb{C}/2\pi i \mathbb{Z})^{\mathcal{B}}$ and $\upsilon\in(\mathbb{C}/2\pi i \mathbb{Z})^{\mathcal{A}}$, let $F=F(\zeta,\upsilon)\in\mathcal{F}(\mathbb{C}^d)$ be the flag such that $(F_3,F_2,F)$ is in general position, $\tau^{\bf j}(F_3,F_2,F)=\zeta^{\bf j}$ for all ${\bf j}\in\mathcal{B}$, and $\sigma^{\bf i}(F_2,F_3,F_1,F)=\upsilon^{\bf i}$ for all ${\bf i}\in\mathcal{A}$. In the basis $(b_1,\dots,b_d)$, the unipotent linear transformation $U(F_3,F_2,F)$ is an upper triangular matrix whose $(i,j)$--th entry for all $i<j$ is of the form
	\[\prod_{{\bf i}\in\mathcal{A}\colon i\le i_1<j}\exp(\upsilon^{\bf i})R_{i,j}(\{\exp(\zeta^{\bf j})\mid{\bf j}\in\mathcal{B}\}),\]
	where $R_{i,j}$ is a Laurent polynomial of $|\mathcal{B}|$ variables.
\end{proposition}

\begin{proof}
	Observe that the proposition uses only $F_1^1$ and not the rest of the flag $F_1$, so we may assume without loss of generality that $\tau^{\bf j}({\bf F})=0$ for all ${\bf j}\in\mathcal{B}$. Then let $W_{\bf F}\in\mathsf{SL}_d(\mathbb{C})$ be the unipotent linear map that fixes $F_2$ and sends $F_1$ to $F_3$. The following lemma is a well-known fact whose proof we will give a brief sketch of.

	\begin{lemma}\label{lem: Pascal}
		For all ${\bf i}\in\mathcal{A}$, $\sigma^{\bf i}(F_2,F_3,F_1,W_{\bf F}\cdot F_3)=0\in\mathbb{C}/2\pi i\mathbb{Z}$.
	\end{lemma}
	
	\begin{proof}[Sketch of proof] The standard irreducible representation $\iota_d:\mathsf{PGL}_2(\mathbb{R})\to\mathsf{PGL}_d(\mathbb{R})$ admits a $\iota_d$--equivariant map $V_d:\mathbb{RP}^1\cong\mathbb{R}\cup\{\infty\}\to\mathcal{F}(\mathbb{R}^d)$ that has the following well-known properties:
		\begin{enumerate}
			\item A triple of flags ${\bf H}=(H_1,H_2,H_3)$ in $\mathcal{F}(\mathbb{C}^d)$ is in general position and satisfies $\tau^{\bf j}({\bf H})=0$ for all ${\bf j}\in\mathcal{B}$ if and only if there is some $g\in\mathsf{PGL}_d(\mathbb{C})$ such that $g({\bf H})=(V_d(-1),V_d(\infty),V_d(0))$.
			\item For all ${\bf i}\in\mathcal{A}$, $\sigma^{\bf i}(V_d(\infty),V_d(0),V_d(-1),V_d(1))=0$.
		\end{enumerate}
		
		By (1), we may assume that ${\bf F}=(V_d(-1),V_d(\infty),V_d(0))$. Since 
		\[
		u:=\begin{bmatrix} 1&1\\0&1\end{bmatrix}\in\mathsf{PGL}_2(\mathbb{R})
		\]
		fixes $\infty$ and sends $--1$ to $0$, it follows that (the projectivization of) $W_{\bf F}$ is $\tau_d(u)$. Thus, for all ${\bf i}\in\mathcal{A}$,
		\begin{align*}
			\sigma^{\bf i}(F_2,F_3,F_1,W_{\bf F}\cdot F_3)&=\sigma^{\bf i}(V_d(\infty),V_d(0),V_d(-1),\tau_d(u)\cdot V_d(0))\\
			&=\sigma^{\bf i}(V_d(\infty),V_d(0),V_d(-1),V_d(1))=0.\qedhere
		\end{align*}
	\end{proof}
	
	Let $U_{\bf F}(\zeta,\upsilon)\in\mathsf{SL}_d(\mathbb{C})$ be the unipotent linear representative of $u_{\bf F}(\zeta,\upsilon)$. Using Lemma \ref{lem: Pascal}, we may relate $U(F_3,F_2,F)$ to $U_{\bf F}(\zeta,\upsilon)$.
	
	\begin{lemma}\label{lem: U description}
		$U(F_3,F_2,F)=W_{\bf F}U_{\bf F}(\zeta,\upsilon)$.
	\end{lemma}
	
	\begin{proof}
		By Proposition \ref{prop: lift}, $W_{\bf F}U_{\bf F}(\zeta,\upsilon)$ is a unipotent linear map that fixes $F_2$. Since the same is true for $U(F_3,F_2,F)$, it suffices to show that $W_{\bf F}U_{\bf F}(\zeta,\upsilon)\cdot F_3=F$.
		
		By Proposition \ref{prop: lift}, the triple of flags $\left(F_1,F_2,U_{{\bf F}}(\zeta,\upsilon)\cdot F_3\right)$ is in general position, and its triangle invariants satisfy
		\[\tau^{\bf j}\left(F_1,F_2,U_{{\bf F}}(\zeta,\upsilon)\cdot F_3\right)=\tau^{\bf j}({\bf F})+\zeta^{\bf j}=\zeta^{\bf j}.\]
		Thus, the triple 
		\[\left(F_3,F_2,W_{\bf F}U_{\bf F}(\zeta,\upsilon)\cdot F_3\right)=W_{\bf F}\cdot \left(F_1,F_2,U_{\bf F}(\zeta,\upsilon)\cdot F_3\right)\] 
		is in general position, and 
		\[\tau^{\bf j}\left(F_3,F_2,W_{\bf F}U_{\bf F}(\zeta,\upsilon)\cdot F_3\right)=\zeta^{\bf j}=\tau^{\bf j}\left(F_3,F_2,F\right),\] 
		so the triples $\left(F_3,F_2,W_{\bf F}U_{\bf F}(\zeta,\upsilon)\cdot F_3\right)$ and $\left(F_3,F_2,F\right)$ are projectively equivalent. Hence, it suffices to verify that $W_{\bf F}U_{\bf F}(\zeta,\upsilon)\cdot F_3^1=F^1$, or equivalently, that
		\[\sigma^{\bf i}\left(F_2,F_3,F_1,W_{\bf F}U_{\bf F}(\zeta,\upsilon)\cdot F_3\right)=\sigma^{\bf i}\left(F_2,F_3,F_1,F\right)\]
		for all ${\bf i}\in\mathcal{A}$. 
		
		As before, set ${\bf E}:=(F_2,F_1)$ and ${\bf G}:=(F_2,F_3)$. Then we have
		\begin{align*}
			\sigma^{\bf i}\left(F_2,F_3,F_1,W_{\bf F}U_{\bf F}(\zeta,\upsilon)\cdot F_3\right)
			&=\sigma^{\bf i}\left(F_2,F_1,W_{\bf F}^{-1}\cdot F_1,U_{{\bf F}}(\zeta,\upsilon)\cdot F_3\right)\\
			&=\sigma^{\bf i}\left(F_2,F_1,W_{\bf F}^{-1}\cdot F_1,C_{{\bf E}}(\upsilon)\cdot F_3\right)\\
			&=\sigma^{\bf i}\left(F_2,F_1,W_{\bf F}^{-1}\cdot F_1,F_3\right)+\upsilon^{\bf i}\\
			&=\upsilon^{\bf i}=\sigma^{\bf i}\left(F_2,F_3,F_1,F\right)
		\end{align*}
		for all ${\bf i}\in\mathcal{A}$. Here, the first equality follows from the equality
		\begin{align*}
			\left(F_2,F_3,F_1,W_{\bf F}U_{\bf F}(\zeta,\upsilon)\cdot F_3\right)=W_{\bf F}\cdot\left(F_2,F_1,W_{\bf F}^{-1}\cdot F_1,U_{{\bf F}}(\zeta,\upsilon)\cdot F_3\right),
		\end{align*}
		the second equality holds because $C_{{\bf G}}(-\upsilon-s_\zeta)$ and $A_{{\bf F}}({\zeta})$ both fix $F_3^1$, the third holds by Proposition~\ref{prop: shear combination} part (1), and the fourth is Lemma~\ref{lem: Pascal}.
	\end{proof}
	
	For the remainder of this proof, all matrices are written in the basis $(b_1,\dots,b_d)$. Since $U(F_3,F_2,F)$ fixes $F_2$, it is upper triangular, so it suffices to prove the given description of the upper triangular entries of $U(F_3,F_2,F)$. By Lemma \ref{lem: U description}, we may do so for $W_{\bf F}U_{\bf F}(\zeta,\upsilon)$ in place of $U(F_3,F_2,F)$.

	Set ${\bf E}:=(F_2,F_1)$ and ${\bf G}:=(F_2,F_3)$, and notice that $W_{\bf F}\cdot{\bf E}={\bf G}$, so
	\[W_{\bf F}U_{\bf F}(\zeta,\upsilon)=W_{\bf F}C_{\bf E}(\upsilon)A_{\bf F}(\zeta)C_{\bf G}(-\upsilon-s_\zeta)=C_{\bf G}(\upsilon)W_{\bf F}A_{\bf F}(\zeta)C_{\bf G}(-s_\zeta)C_{\bf G}(-\upsilon).\]	
	By the descriptions of the linear maps $A_{\bf F}(\zeta)$ in Proposition \ref{prop: eruption combination} part (2) and $C_{\bf G}(-s_\zeta)$ in Proposition \ref{prop: shear combination} part (2), we see that $A_{\bf F}(\zeta)C_{\bf G}(-s_\zeta)$ is an upper triangular matrix, and all the upper triangular entries of $A_{\bf F}(\zeta)C_{\bf G}(-s_\zeta)$ are Laurent polynomial of the variables $\{\exp(\zeta^{\bf j}):{\bf j}\in\mathcal{B}\}$. Since $W_{\bf F}$ does not depend on $\zeta$ and $\upsilon$, the same is also true for $W_{\bf F}A_{\bf F}(\zeta)C_{\bf G}(-s_\zeta)$. 
	
	Again by Proposition \ref{prop: shear combination} part (2), $C_{\bf G}(\upsilon)=C_{\bf G}(-\upsilon)^{-1}$ is diagonal, and its $n$--th diagonal entry is 
	\[\prod_{{\bf i}\in\mathcal{A}:i_1\geq n}\exp\left(\upsilon^{\bf i}\right).\] 
	It follows that for all $i<j$, the $(i,j)$--th entry of $W_{\bf F}U_{\bf F}(\zeta,\upsilon)$ is the $(i,j)$--th $W_{\bf F}A_{\bf F}(\zeta)C_{\bf G}(-s_\zeta)$ rescaled by
	\[\prod_{{\bf i}\in\mathcal{A}:i\le i_1<j}\exp\left(\upsilon^{\bf i}\right).\qedhere\]
\end{proof}

\subsection{Explicit formula for the slithering map}\label{sec: explicit expression}
For each locally $\lambda$--H\"older continuous, $\lambda$--transverse, and $\lambda$--hyperconvex map $\xi:\partial\widetilde\lambda\to\mathcal{F}(\mathbb{C}^d)$, let $(\alpha_\xi,\theta_\xi)$ be its shear-bend coordinate (see Section \ref{sec:shear_bend}) and let $\Sigma=\Sigma_\xi:\widetilde\Lambda^2\to\mathsf{SL}_d(\mathbb{C})$ be its compatible slithering map (see Theorem~\ref{thm: slitherable and slithering}). As we mentioned at the start of the section, the key result needed to prove Theorem~\ref{thm: injectivity} is a way to realize $\Sigma_\xi$ as extensions of H\"older extendable maps whose dependence on $(\alpha_\xi,\theta_\xi)$ is explicit. The objective of this subsection is to construct said map, and to establish some of its properties that we will use in the proof of the injectivity and holomorphicity of the shear-bend map.

For the remainder of this section, fix a labelling ${\bf x}=(x_1,x_2,x_3)\in\widetilde\Delta^o$, and let $T_{\bf x}$ denote the plaque of $\widetilde\lambda$ whose vertices are $x_1$, $x_2$, and $x_3$.  Recall that $\widetilde\Lambda$ denotes the set of leaves of the geodesic lamination $\widetilde\lambda$. Also, for any pair of leaves $g_1$ and $g_2$ of $\widetilde\lambda$, recall that $P(g_1,g_2)$ denotes the set of plaques of $\widetilde\lambda$ that separate $g_1$ and $g_2$.

\begin{figure}[h!]
	\includegraphics[width=10cm]{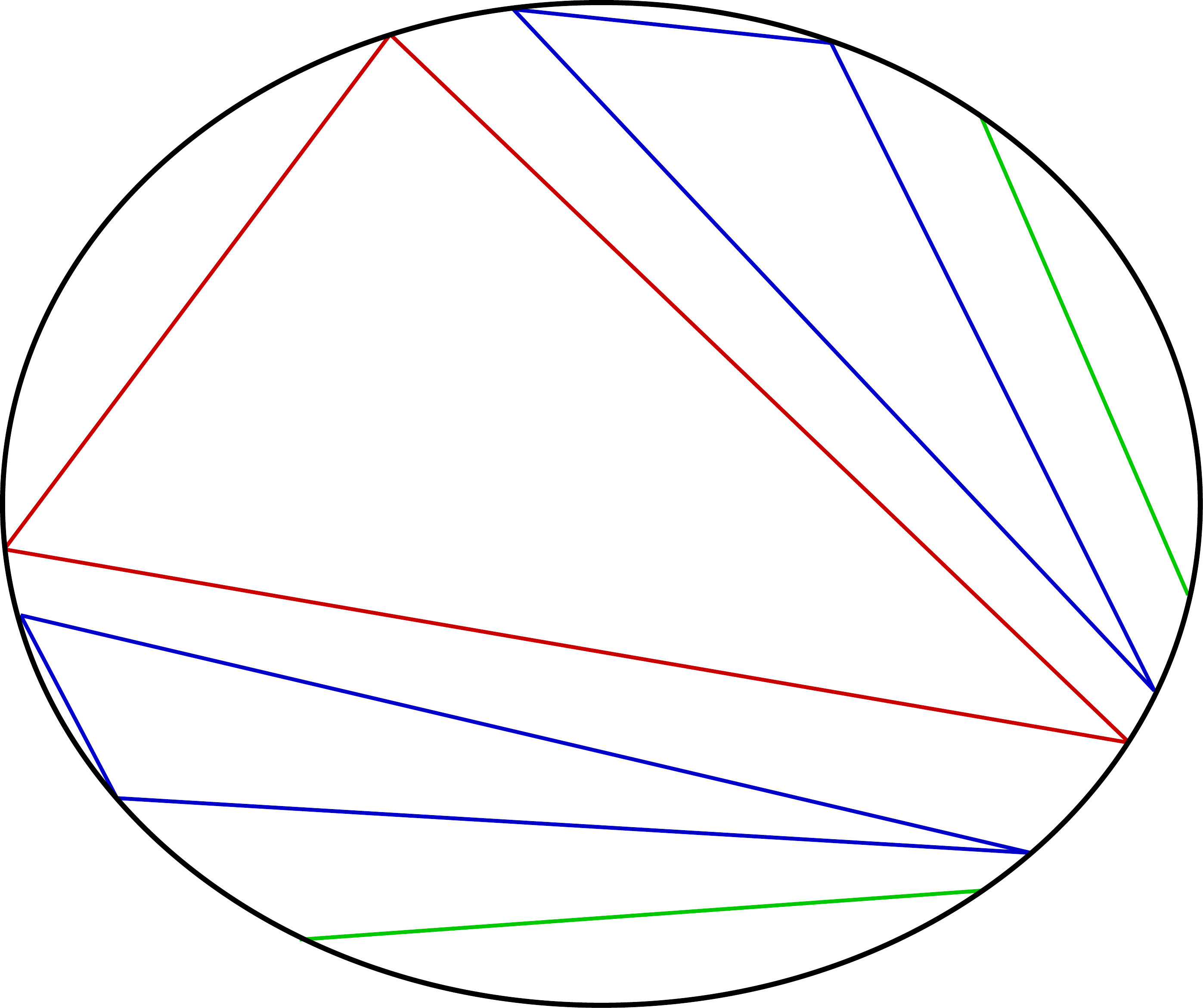}
	\put(-19,135){\small $h$}
	\put(-80,165){\small $T$}
	\put(-8,75){\small $x_{T,2}$}
	\put(-16,60){\small $y_1^h=y_2^T=y_2^{T'}=y_2^{h'}$}
	\put(-39,33){\small $x_{T',2}$}
	\put(-120,15){\small $h'$}
	\put(-160,50){\small $T'$}
	\put(-275,45){\small $x_{T',3}$}
	\put(-300,90){\small $x_{T',1}$}
	\put(-185,135){\small $T_{\bf x}$}
	\put(-185,95){\small $g^{h'}_{\bf x}$}
	\put(-115,140){\small $g^{h}_{\bf x}$}
	\put(-365,109){\small $y_3^h=y_3^T=y_1^{T'}=y_1^{h'}$}
	\put(-270,234){\small $y_2^h=y_1^T=y_3^{T'}=y_3^{h'}$}
	\put(-170,240){\small $x_{T,1}$}
	\put(-90,230){\small $x_{T,3}$}
	\caption{\small The notation needed in the proof of Lemma \ref{lem: injectivity}.}
	\label{fig:unipotenteruption9}
\end{figure}

For each leaf $h$ of $\widetilde\lambda$, let $g_{\bf x}^h$ denote the edge of $T_{\bf x}$ that separates $T_{\bf x}$ from $h$. Then for each plaque $T\in P(g_{\bf x}^h,h)$, and each  locally $\lambda$--H\"older continuous, $\lambda$--transverse, and $\lambda$--hyperconvex map $\xi:\partial\widetilde\lambda\to\mathcal{F}(\mathbb{C}^d)$, we will use the following notation:
\begin{itemize}
	\item Let ${\bf y}^h=(y_1^h,y_2^h,y_3^h)\in\widetilde\Delta^o$ be the labelling of the vertices of $T_{\bf x}$ such that $y_1^h\prec y_2^h\prec y_3^h$ and such that $y_1^h$ and $y_2^h$ are the endpoints of $g_{\bf x}^h$, see Figure~\ref{fig:unipotenteruption9}.
	\item Let ${\bf y}^T=(y_1^T,y_2^T,y_3^T)\in\widetilde\Delta^o$ be the labeling of the vertices of $T_{\bf x}$ such that $y_2^T$ and $y_1^T$ are the endpoints of $g_{\bf x}^h$, and $y_1^T\prec y_2^T\prec y_3^T$ if $x_{T,3}\prec x_{T,2}\prec x_{T,1}$ while $y_3^T\prec y_2^T\prec y_1^T$ if $x_{T,1}\prec x_{T,2}\prec x_{T,3}$, see Figure \ref{fig:unipotenteruption9}. Here, recall that ${\bf x}_T=(x_{T,1},x_{T,2},x_{T,3})$ denotes the coherent labeling of the vertices of $T$ i.e. if $h_{T,-}$ and $h_{T,+}$ are the two edges of $T$ that separate $g^h_{\bf x}$ and $h$, so that $h_{T,-}$ separates $g^h_{\bf x}$ and $h_{T,+}$, then $x_{T,2}$ is the common endpoint of $h_{T,-}$ and $h_{T,+}$, and $x_{T,1}$ and $x_{T,3}$ are respectively the endpoints of $h_{T,-}$ and $h_{T,+}$ that are not $x_{T,2}$, see Figure \ref{fig:notations3}. Notice that $y_3^T=y_3^h$, and $(y_1^T,y_2^T)=(y_1^h,y_2^h)$ if  $x_{T,3}\prec x_{T,2}\prec x_{T,1}$ while $(y_1^T,y_2^T)=(y_2^h,y_1^h)$ if $x_{T,1}\prec x_{T,2}\prec x_{T,3}$.
	\item Let $B_h$ be a basis of vectors in $\mathbb{C}^d$ whose $i$--th basis vector lies in $\xi(y_2^h)^i\cap\xi(y_1^h)^{d-i+1}$, and so that $\xi(y_3^h)^1$ is the span of the sum of all the basis vectors in $B_h$.
	\item Let $H_T\in\mathcal{F}(\mathbb{C}^d)$ be the flag such that $(\xi(y_1^T),\xi(y_2^T),H_T)$ is in general position, $\tau^{\bf j}(\xi(y_1^T),\xi(y_2^T),H_T)=\theta_\xi^{\bf j}({\bf x}_T)$ for all ${\bf j}\in\mathcal{B}$, and 
	\[\sigma^{\bf i}(\xi(y_2^T),\xi(y_1^T),\xi(y_3^T),H_T)=\left\{\begin{array}{ll}\alpha_\xi^{\bf i}(T_{\bf x},T)&\text{if }x_{T,3}\prec x_{T,2}\prec x_{T,1}\\
		-\alpha_\xi^{\wh{\bf i}}(T_{\bf x},T)&\text{if }x_{T,1}\prec x_{T,2}\prec x_{T,3}\\
	\end{array}\right.\] 
	for all ${\bf i}\in\mathcal{A}$.
\end{itemize}
\begin{remark}\label{fjkfjsjkfskdj}
	One should keep in mind that $B_h$ is unique up to scaling, and depends on $\xi({\bf x})$, while $H_T$ depends on $\xi({\bf x})$, $\alpha_\xi(T_{\bf x},T)$, and $\theta_\xi({\bf x}_T)$. However, we suppress these dependences from the notation, but instead record all other dependences.
\end{remark} 
Recall again that for any triple of flags ${\bf F}=(F_1,F_2,F_3)$ in $\mathcal{F}(\mathbb{C}^d)$ that is in general position, $U({\bf F})\in\mathsf{SL}_d(\mathbb{C})$ denotes the unipotent linear map that fixes $F_2$ and sends $F_1$ to $F_3$. Then define
\[L=L_{(\alpha_\xi,\theta_\xi),\xi({\bf x}),h}:P(g_{\bf x}^h,h)\to \mathsf{SL}_d(\mathbb{C}) \] 
by $L(T)= U(\xi(y_1^T),\xi(y_2^T),H_T)$. Notice from Remark \ref{fjkfjsjkfskdj} that for each $T\in P(g_{\bf x}^h,h)$, $L(T)$ depends on $\xi({\bf x})$, $\alpha_\xi(T_{\bf x},T)$, and $\theta_\xi({\bf x}_T)$. 

The following lemma implies that $L$ is the required H\"older extendable map whose extension gives the slithering map.

\begin{lemma}\label{lem: injectivity}
	For each locally $\lambda$--H\"older continuous, $\lambda$--transverse, and $\lambda$--hyperconvex map $\xi:\partial\wt\lambda\to\mathcal{F}(\mathbb{C}^d)$, the map
	\[
	L=L_{(\alpha_\xi,\theta_\xi),\xi({\bf x}),h}:P(g_{\bf x}^h,h)\to \mathsf{SL}_d(\mathbb{C})
	\] 
	has the following properties:
	\begin{enumerate}
		\item Fix a leaf $h\in\widetilde\lambda$ and let $T\in P(g_{\bf x}^h,h)$. If $x_{T,3}\prec x_{T,2}\prec x_{T,1}$ (resp. $x_{T,1}\prec x_{T,2}\prec x_{T,3}$), then in the basis $B_h$, the linear map $L(T)\in\mathsf{SL}_d(\mathbb{C})$ is an upper (resp. lower) triangular unipotent matrix whose $(i,j)$--th entry for all $i< j$ (resp. $i> j$) is of the form
		\begin{align*}
			&\left(\prod_{{\bf i}:i\le i_1<j-1}\exp({\alpha_\xi^{\bf i}(T_{\bf x},T)})\right)R_{i,j}(\{\exp(\theta_\xi^{\bf j}({\bf x}_T))\mid{\bf j}\in\mathcal{B}\})\\
			\Bigg(\text{resp.}\quad &\left(\prod_{{\bf i}: j\le i_1<i-1}\exp(-\alpha_\xi^{\bf i}(T_{\bf x},T))\right)R_{i,j}(\{\exp(\theta_\xi^{\bf j}({\bf x}_T))\mid{\bf j}\in\mathcal{B}\})\Bigg)
		\end{align*}
		where $R_{i,j}$ does not depend on $T$, and is a Laurent polynomial of $\abs{\mathcal B}$ variables. 
		\item The map $L$ is H\"older extendable, and its extension $\overrightarrow{L}$ of $L$ satisfies
		\[\overrightarrow{L}(g_{\bf x}^h,h)=\Sigma(h,g_{\bf x}^h),\]
		where $\Sigma=\Sigma_{\xi}:\wt\Lambda^2\to\mathsf{SL}_d(\mathbb{C})$ denotes the slithering map associated to $\xi$. In particular, for fixed $h$, the linear transformation $\Sigma(g_{\bf x}^h,h)\in\mathsf{SL}_d(\mathbb{C})$ is determined by $(\alpha_{\xi},\theta_\xi)$, and $\left(\xi(x_1),\xi(x_2),\xi(x_3)^1\right)$. 
	\end{enumerate}
\end{lemma}
\begin{proof} \textit{Part (1).} In the case when $T$ satisfies $x_{T,3}\prec x_{T,2}\prec x_{T,1}$, we have $(y_1^T,y_2^T)=(y_1^h,y_2^h)$. Then the basis $B_h=(b_1,\dots,b_d)$ satisfies $b_m\in \xi(y_2^T)^{m}\cap \xi(y_1^T)^{d-m+1}$ for all $m\in\{1,\dots,d-1\}$. Since $L(T)=U(\xi(y_2^T),\xi(y_1^T),H_T)$, Proposition~\ref{formula} (and the definition of $H_T$) implies the required claim. 
	
	On the other hand, when $T$ satisfies $x_{T,1}\prec x_{T,2}\prec x_{T,3}$, we have $(y_1^T,y_2^T)=(y_2^h,y_1^h)$. Unlike the previous case, we have that the basis $B_h=(b_1,\dots,b_d)$ satisfies $b_m\in \xi(y_1^T)^{m}\cap \xi(y_2^T)^{d-m+1}$ for all $m\in\{1,\dots,d-1\}$. Since $L(T)=U(\xi(y_2^T),\xi(y_1^T),H_T)$,  the claim in this case follows from first applying Proposition~\ref{formula}, and then conjugating by the matrix with $1$'s along the anti-diagonal and $0$'s everywhere else. 
	
	\textit{Part (2).} Notice that $H_T=\Sigma(g_{\bf x}^h,h_{T,-})\cdot \xi(x_{T,3})$. It thus follows that 
	\[\Sigma(g_{\bf x}^h,h_{T,-})\,\Sigma(h_{T,+},g_{\bf x}^h)=\Sigma(g_{\bf x}^h,h_{T,-})\,\Sigma(h_{T,+},h_{T,-})\,\Sigma(h_{T,-},g_{\bf x}^h)\]
	is unipotent, fixes $\xi(y_2^T)$, and sends $\xi(y_1^T)$ to $H_T$. Thus,
	\[\Sigma(g_{\bf x}^h,h_{T,-})\,\Sigma(h_{T,+},g_{\bf x}^h)=U(\xi(y_1^T),\xi(y_2^T),H_T)=L(T).\]
	The claim now follows from Proposition \ref{prop: renormalization}. \end{proof}

\begin{remark}
	Notice that in the proof of Theorem \ref{thm: injectivity}, we have already constructed a H\"older extendable map 
	\[M=M_{\rm slither}:P(g_{\bf x}^h,h)\to\mathsf{SL}_d(\mathbb{C})\] 
	whose extension $\overrightarrow{M}$ satisfies $\overrightarrow{M}(g_{\bf x}^h,h)=\Sigma(g_{\bf x}^h,h)$. However, given Lemma~\ref{lem: injectivity}, it is much easier to deduce injectivity and holomorphicity of $\mathcal{sb}_{{\bf x},d}$ from $L$ instead of $M$. 
\end{remark}

In the case when $\xi$ is the $\lambda$--limit map of a $d$--pleated surface $\rho$, recall that we denote $(\alpha_\xi,\theta_\xi)=(\alpha_\rho,\theta_\rho)$, which is now an element in $\mathcal{Y}_d(\lambda;\mathbb{C}/2\pi i\mathbb{Z})$. Before turning to the proof of Theorem \ref{thm: injectivity}, we prove the following lemma, which gives uniformity properties of $L=L_{(\alpha_\rho,\theta_\rho),\xi({\bf x}),h}$ when $h$ is fixed while $(\alpha_\rho,\theta_\rho)$ and $\xi({\bf x})$ varies compactly. This will not be used in the proof of Theorem \ref{thm: injectivity}, but will be key in the proof of the holomorphicity of $\mathcal{sb}_{d,{\bf x}}$, see Proposition \ref{prop:sb is holomorphic}.

\begin{lemma}\label{lem: uniform}
	Fix a leaf $h\in\widetilde\lambda$. Let $E\subset \mathcal{Y}_d(\lambda;\mathbb{C}/2\pi i\mathbb{Z})$ and $E'\subset\mathcal N(\mathbb{C}^d)$ be compact sets.
	\begin{enumerate}
		\item The H\"older constants for the H\"older extendable maps
		\[L_{(\alpha_\rho,\theta_\rho),\xi({\bf x}),h}:P(g_{\bf x}^h,h)\to \mathsf{SL}_d(\mathbb{C})\] 
		can be chosen to be uniform as $(\alpha_\rho,\theta_\rho)$ varies in $E$ and $\xi({\bf x})$ varies in $E'$.
		\item For any exhausting sequence $(\mathcal X_n)_{n=1}^\infty$ of $P(g_{\bf x}^h,h)$, the convergence
		\[\overrightarrow{L}_{(\alpha_\rho,\theta_\rho),\xi({\bf x}),h}^F(\mathcal X_n)\to \overrightarrow{L}_{(\alpha_\rho,\theta_\rho),\xi({\bf x}),h}(g_{\bf x}^h,h)\]
		is uniform over all $(\alpha_\rho,\theta_\rho)\in E$ and $\xi({\bf x})\in E'$.
	\end{enumerate}
\end{lemma}

\begin{proof}\textit{Part (1).}  Since there are only finitely many $\Gamma$--orbits of plaques of $\wt\lambda$, there is an upper bound on the set
	\[\{\exp(\theta_\rho^{\bf j}({\bf x}_T))\mid(\alpha_\rho,\theta_\rho)\in E, \,T\in P(g_{\bf x}^h,h),\text{ and }{\bf j}\in\mathcal{B}\}.\]
	This implies the finiteness of the supremum
	\[C:=\sup\{R_{i,j}(\{\exp(\theta_\rho^{\bf j}({\bf x}_T)) \mid {\bf j}\in\mathcal{B}\})\mid 1\le i\neq j\le d,(\alpha_\rho,\theta_\rho)\in E, T\in P(g^h_{\bf x},h)\},\]
	where $R_{i,j}$ are the Laurent polynomials with finitely many terms given by Lemma~\ref{lem: injectivity} part (1). It then follows from said lemma that if $T$ satisfies $x_{T,3}\prec x_{T,2}\prec x_{T,1}$ (resp. $x_{T,1}\prec x_{T,2}\prec x_{T,3}$), then in the basis $B_h$, the upper (resp. lower) triangular entries of $L(T)$ have absolute values that are uniformly bounded above by 
	\[C\max_{{\bf i}\in\mathcal{A}}\exp({\rm Re}\;\alpha_\rho^{{\bf i}}(T_{\bf x},T))^{d-1}\quad\bigg(\text{resp.}\quad C\max_{{\bf i}\in\mathcal{A}}\exp(-{\rm Re}\;\alpha_\rho^{{\bf i}}(T_{\bf x},T))^{d-1}\bigg)\]
	We may now apply Lemma 8.8 of \cite{BoD} to deduce that there are uniform constants $A',B>0$ (as $(\alpha_\rho,\theta_\rho)$ varies in $E$ and as $T$ varies in $P(g^h_{\bf x},h)$) such that these entries are all bounded above by $A'e^{-Br(T)}$, where $r(T)$ is the function given by Lemma~\ref{lem: lamination property}. Thus, by Lemmas \ref{lem: lamination property} and \ref{lem: hyperbolic geometry},
	\[\norm{L(T)-\text{id}}_{\xi({\bf x})}\leq  A \, d(x_{T,1},x_{T,3})^\nu\]
	for all $(\alpha_\rho,\theta_\rho)\in E$ and $T\in P(g^h_{\bf x},h)$, where $\norm{\cdot}_{\xi({\bf x})}$ is the matrix norm with respect to the basis $B_h$ and $A:=A'\frac{d(d-1)}{2}$. 
	
	Pick any norm $\norm{\cdot}$ on $\mathbb{C}^d$. Since the basis $B_h$ varies continuously with $\xi({\bf x})$, it follows that $\norm{\cdot}$ and $\norm{\cdot}_{\xi({\bf x})}$ are uniformly bi-Lipschitz as $\xi({\bf x})$ varies in $E'$. Thus, by increasing the constant $A$ if necessary, we have that
	\[\norm{L(T)-\text{id}}\leq  A \, d(x_{T,1},x_{T,3})^\nu\]
	for all $(\alpha_\rho,\theta_\rho)\in E$, $\xi({\bf x})\in E'$, and $T\in P(g^h_{\bf x},h)$.
	
	\textit{Part (2).} Let $(\rho_k,\xi_k)_k$ be a sequence of $d$--pleated surfaces that converges to a limiting $d$--pleated surface $(\rho_\infty,\xi_\infty)$ such that $(\alpha_k,\theta_k):=(\alpha_{\rho_k},\theta_{\rho_k})\in E$ and $\xi_k({\bf x})\in E'$ for all positive integers $k$. Then for each $k\in\mathbb{N}\cup\{\infty\}$, denote 
	\[\overrightarrow{L}^F_k:=\overrightarrow{L}^F_{(\alpha_k,\theta_k),\xi_k({\bf x})}\quad\text{and}\quad\overrightarrow{L}_k:=\overrightarrow{L}_{(\alpha_k,\theta_k),\xi_k({\bf x})}.\]
	It suffices to show that 
	\begin{align}\label{eqn: limit}
		\overrightarrow{L}^F_n(\mathcal X_n)\to \overrightarrow{L}_\infty(h,g^h_{\bf x})\quad\text{as }n\to\infty.
	\end{align}
	
	To do so, fix a norm $\norm{\cdot}$ on $\mathbb{C}^d$. Together, part (1) and Lemma \ref{lem: abstract} part (2) imply that there exists some a constant $D>1$ such that for all positive integers $k$ and $n$, 
	\[
	\|\overrightarrow{L}^F_k(\mathcal X_n)\|\leq D.
	\]
	Also, Lemma \ref{lem: abstract} part (3) implies that there exists some constant $A>0$ such that for all positive integers $k$ and $m\leq n$, 
	\begin{equation*}\label{eq: alphathetan}
		\|\overrightarrow{L}^F_k(\mathcal X_n)-\overrightarrow{L}^F_k(\mathcal X_m)\|\leq AD\sum_{T\in\mathcal X_n-\mathcal X_m}d_{\infty}(x_{T,1},x_{T,3})^\nu.
	\end{equation*}
	Note that for any finite set $\mathcal X\subset P(g^h_{\bf x},h)$, $\overrightarrow{L}^F_k(\mathcal X)\to \overrightarrow{L}^F_\infty(\mathcal X)$ as $k\to\infty$. Thus, taking the limit as $k\to\infty$ and then the limit as $n\to\infty$ gives
	\[
	\|\overrightarrow{L}_\infty(h,g_{\bf x}^h)-\overrightarrow L^{F}_\infty(\mathcal X_m)\|\leq AD\sum_{T\in P(g^h_{\bf x},h)-\mathcal X_m}d_{\infty}(x_{T,1},x_{T,3})^\nu.
	\]
	
	Combining everything, we have
	\begin{align*}
		\|\overrightarrow{L}_\infty(h,g^h_{\bf x})-\overrightarrow L^{F}_n(\mathcal X_n)\|&\leq  \|\overrightarrow{L}_\infty(h,g^h_{\bf x})-\overrightarrow L^{F}_\infty(\mathcal X_m)\|+\|\overrightarrow L^{F}_\infty(\mathcal X_m)-\overrightarrow L^{F}_n(\mathcal X_m)\|\\
		&\qquad+\|\overrightarrow L^F_n(\mathcal X_m)-\overrightarrow L^F_n(\mathcal X_n)\|\\
		&\leq 2AD\sum_{T\in P(g^h_{\bf x},h)-\mathcal X_m}d(x_{T,1},x_{T,3})^\nu\\
		&\qquad+\|\overrightarrow L^{F}_\infty(\mathcal X_m)-\overrightarrow L^F_n(\mathcal X_m)\|.
	\end{align*}
	Since $\|\overrightarrow L^{F}_\infty(\mathcal X_m)-\overrightarrow L^F_n(\mathcal X_m)\|\to 0$ as $n\to\infty$, for fixed $m\in\mathbb N$ we have:
	\begin{equation*}\label{eq: last}
		\limsup_{n\to \infty}\|\overrightarrow{L}_\infty(h,g^h_{\bf x})-\overrightarrow L^{F}_n(\mathcal X_n)\|\leq  2AD\sum_{T\in P(g^h_{\bf x},h)-\mathcal X_m}d(x_{T,1},x_{T,3})^\nu.
	\end{equation*}
	Finally, taking the limit as $m\to\infty$, we conclude that the limit \eqref{eqn: limit} holds.
\end{proof}

\subsection{Injectivity of \texorpdfstring{$\mathcal{sb}_{d,{\bf x}}$}{sbdx}}\label{sec: injectivity proof}
We may now use Lemma \ref{lem: injectivity} to prove Theorem \ref{thm: injectivity}. As an intermediate step, we prove the following lemma.

\begin{lemma}\label{lem: injxi}
	If $\xi,\xi':\partial\wt\lambda\to\mathcal{F}(\mathbb{C}^d)$ are locally $\lambda$--H\"older continuous, $\lambda$--transverse, and $\lambda$--hyperconvex maps such that 
	\begin{enumerate}[label=(\Roman*)]
		\item $\big(\xi(x_1),\xi(x_2),\xi(x_3)^1\big)=\big(\xi'(x_1),\xi'(x_2),\xi'(x_3)^1\big)$, 
		\item $\alpha_{\xi}(T_{\bf x},T)=\alpha_{\xi'}(T_{\bf x},T)$ for all plaques $T\in\wt\Delta-\{T_{\bf x}\}$, and 
		\item $\theta_\xi({\bf y})=\theta_{\xi'}({\bf y})$ for all labelings ${\bf y}\in\wt\Delta^o$, 
	\end{enumerate}
	then $\xi=\xi'$.
\end{lemma}

\begin{proof}
	The assumptions (I) and (III) imply that 
	\begin{align}\label{eqn: frame equal}
		\xi({\bf x})=\xi'({\bf x}).
	\end{align} 
	(Compare with Proposition \ref{prop_invariants} part (2).)
	Also, if we let $\Sigma$ and $\Sigma'$ respectively denote the slithering maps compatible to $\xi$ and $\xi'$ respectively, then by Lemma \ref{lem: injectivity} part (2),
	\begin{align}\label{eqn: sigma equal}
		\Sigma(g_{\bf x}^h,h)=\Sigma'(g_{\bf x}^h,h)
	\end{align}
	for any $h\in\wt\Lambda$.
	
	Now, pick any $z\in\partial\wt\lambda$, let $h\in\wt\Lambda$ be a leaf that has $z$ as its endpoint, and let $g_{\bf x}^h$ be the edge of $T_{\bf x}$ that separates $T_{\bf x}$ and $h$. Equip $h$ and $g_{\bf x}^h$ with orientations so that they are oriented in parallel, and that $z$ is the forward endpoint of $h$, and let $x\in\{x_1,x_2,x_3\}$ denote the forward endpoint of $g_{\bf x}^h$. Then 
	\[\xi(z)=\Sigma(g_{\bf x}^h,h)\cdot \xi(x)\,\,\text{ and }\,\,\xi'(z)=\Sigma'(g_{\bf x}^h,h)\cdot \xi'(x),\]
	so equations \eqref{eqn: frame equal} and \eqref{eqn: sigma equal} imply that $\xi(z)=\xi'(z)$. Since $z$ is arbitrary, $\xi=\xi'$.
\end{proof}

\begin{proof}[Proof of Theorem \ref{thm: injectivity}]
	Let $\rho$ and $\rho'$ be representations in $\mathcal{R}_d(\lambda)$, and let $\xi$ and $\xi'$ be their $\lambda$--limit maps respectively. By the definition of $\mathcal{R}_d(\lambda)$ and Theorem \ref{thm: Wang}, $\xi$ and $\xi'$ are both locally $\lambda$--H\"older continuous, $\lambda$--transverse, and $\lambda$--hyperconvex. Also, by the definition of the map $\mathcal{sb}_{d,{\bf x}}$, the assumption that $\mathcal{sb}_{d,{\bf x}}(\rho)=\mathcal{sb}_{d,{\bf x}}(\rho')$ implies that assumptions (I), (II), and (III) of Lemma \ref{lem: injxi} hold for the maps $\xi$ and $\xi'$. So, by Lemma~\ref{lem: injxi}, $\xi=\xi'$.
	
	Pick any $\gamma\in\Gamma$. Observe that $\rho(\gamma)\in\mathsf{PGL}_d(\mathbb{C})$ is the unique element that sends $\xi({\bf x})$ to $\xi(\gamma\cdot{\bf x})$. Since $\xi({\bf x})=\xi'({\bf x})$ and $\xi(\gamma\cdot{\bf x})=\xi'(\gamma\cdot{\bf x})$, it follows that $\rho'(\gamma)$ also sends $\xi({\bf x})$ to $\xi(\gamma\cdot{\bf x})$, so $\rho'(\gamma)=\rho(\gamma)$. Since $\gamma$ is arbitrary, $\rho'=\rho$.
\end{proof}

\section{Generalized bending} \label{sec:bending}

The main goal of this section is to define a procedure to deform Hitchin representations that we call \emph{generalized bending}. Recall that $\widetilde\Delta^o$ denotes the set of labelings of the vertices of plaques of $\widetilde\lambda$ (see Section \ref{ssec:geodlam}), and for any Abelian group $G$, $\mathcal{Y}_d(\lambda;G)$ denotes the Abelian group of $\lambda$--cocyclic pairs of dimension $d$ with values in $G$ (see Definition~\ref{def_cocycle}). Then generalized bending is a procedure that takes as inputs 
\begin{itemize}
	\item a $d$--Hitchin representation $\rho_0:\Gamma\to\mathsf{PGL}_d(\mathbb{R})$ with limit map $\xi_0:\partial\widetilde S\to\mathcal{F}(\mathbb{R}^d)$, 
	\item a $\lambda$--cocyclic pair $(\beta,\eta)\in\mathcal Y_d(\lambda;\mathbb R/2\pi\mathbb{Z})$,
	\item a labeling ${\bf x}=(x_1,x_2,x_3)\in\widetilde\Delta^o$,
\end{itemize}
and returns as output a $d$--pleated surface $\rho:\Gamma\to\mathsf{PGL}_d(\mathbb{C})$ with $\lambda$--limit map $\xi:\partial\widetilde\lambda\to\mathcal{F}(\mathbb{C}^d)$ such that 
\[(\xi(x_1),\xi(x_2),\xi(x_3)^1)=(\xi_0(x_1),\xi_0(x_2),\xi_0(x_3)^1)\quad\text{and}\quad\mathfrak{sb}_d([\rho])=\mathfrak{s}_d([\rho_0])+(i\beta,i\eta),\] 
where $\mathfrak{sb}_d:\mathfrak R_d(\lambda)\to \mathcal{Y}_d(\lambda;\mathbb C/2\pi i\mathbb{Z})=\mathcal{Y}_d(\lambda;\mathbb{R})+i\mathcal{Y}_d(\lambda;\mathbb R/2\pi \mathbb{Z})$ is the shear-bend map defined in Section~\ref{sec: statement of main theorem}, and $\mathfrak{s}_d:{\rm Hit}_d(S)\to \mathcal C_d(\lambda)$ is the Bonahon-Dreyer parameterization of the Hitchin component by a convex polytope $\mathcal C_d(\lambda)\subset\mathcal{Y}_d(\lambda;\mathbb{R})$, see Theorem \ref{thm: BD par}. Using the map 
\[\mathcal{sb}_{d,{\bf x}}:\mathcal{R}_d(\lambda)\to \mathcal{Y}_d(\lambda;\mathbb{C}/2\pi i\mathbb{Z})\times\mathcal{N}(\mathbb{C}^d)\] 
defined by identity \eqref{eqn: main thm} in Section~\ref{sec: statement of main theorem}, the conditions on $\rho$ can be re-written  as
\[\mathcal{sb}_{d,{\bf x}}(\rho)=(\mathfrak s_d([\rho_0])+(i\beta,i\eta),(\xi_0(x_1),\xi_0(x_2),\xi_0(x_3)^1)).\]

The existence of generalized bending implies the following theorem, which provide us information about the images of the maps $\mathcal{sb}_{d,{\bf x}}$ and $\mathfrak{sb}_d$. This proves Step~2 in the outline of the proof of Theorem \ref{thm: main}, see Section \ref{subsec:proof structure}.

\begin{theorem} \label{thm: surjectivity}
	For all labelings ${\bf x}\in\wt\Delta^o$, the image of
	$\mathcal{sb}_{d,{\bf x}}$ contains
	\[\left(\mathcal{C}_d(\lambda)+i\mathcal{Y}_d(\lambda;\mathbb R/2\pi\mathbb{Z})\right)\times\mathcal{N}(\mathbb{C}^d).\]
	In particular, the image of $\mathfrak{sb}_d$ contains $\mathcal{C}_d(\lambda)+i\mathcal{Y}_d(\lambda;\mathbb R/2\pi\mathbb{Z})$.
\end{theorem} 

\begin{proof}
	Pick a $\lambda$--cocyclic pair
	\[(\alpha,\theta)=(\mathrm{Re}\,\alpha+i\:\mathrm{Im}\;\alpha,\mathrm{Re}\,\theta+i\:\mathrm{Im}\;\theta)\in \mathcal{C}_d(\lambda)+i\mathcal{Y}_d(\lambda;\mathbb{R}/2\pi\mathbb{Z})\]
	and a point $(F_1,F_2,p)\in\mathcal{N}(\mathbb{C}^d)$. By Theorem \ref{thm: BD par}, there is a $d$--Hitchin representation $\rho_0:\Gamma\to\mathsf{PGL}_d(\mathbb{R})$ with limit map $\xi_0:\widetilde S\to\mathcal{F}(\mathbb{R}^d)$ such that $\mathfrak{s}_{d}([\rho_0])=(\mathrm{Re}\,\alpha,\mathrm{Re}\,\theta)$. We may then apply generalized bending with inputs $(\rho_0,\xi_0,\:\mathrm{Im}\;\alpha,\:\mathrm{Im}\;\theta,{\bf x})$ to get a $d$--pleated surface $(\rho,\xi)$ that satisfies
	\begin{align*}
		\mathcal{sb}_{d,{\bf x}}(\rho)&=(\mathfrak s_d([\rho_0])+(i\:\mathrm{Im}\;\alpha,i\:\mathrm{Im}\;\theta),(\xi_0(x_1),\xi_0(x_2),\xi_0(x_3)^1))\\
		&=((\alpha,\theta),(\xi_0(x_1),\xi_0(x_2),\xi_0(x_3)^1)).
	\end{align*}
	Since $\mathsf{PGL}_d(\mathbb{C})$ acts transitively on $\mathcal{N}(\mathbb{C}^d)$, there is some $g\in\mathsf{PGL}_d(\mathbb{C})$ such that $g\,(\xi_0(x_1),\xi_0(x_2),\xi_0(x_3)^1)=(F_1,F_2,p)$. Then the conjugate of $\rho$ by $g$ is sent by $\mathcal{sb}_{d,\mathbf x}$ to $((\alpha,\theta),(F_1,F_2,p))$.
\end{proof}

The construction of generalized bending is motivated by Bonahon's construction \cite{bonahon-toulouse} of the bending deformations in $\mathsf{PGL}_2(\mathbb{C})$, which was described in the proof of Theorem \ref{prop:2-pleated is pleated}. The idea is to deform the Hitchin representation $\rho_0$ by deforming its limit map $\xi_0$ in a $\Gamma$--invariant manner using the \emph{purely imaginary eruptions} along each plaque of $\widetilde\lambda$ by a magnitude specified by $\eta$, and \emph{purely imaginary shears} along each leaf of $\widetilde\lambda$ by a magnitude specified by $\beta$. These purely imaginary eruptions and shears are special cases of the complex eruptions and shears described in Section \ref{sec:eruption and shear}, where we restrict the variables to take values in $i\mathbb{R}/2\pi\mathbb{Z}$ instead of $\mathbb{C}/2\pi i\mathbb{Z}$.

More specifically, we define the notion of a \emph{$(\beta,\eta)$--generalized bending map for $(\rho_0,\xi_0)$ with base ${\bf x}$}
\[\Psi= \Psi_{\rho_0, \xi_0, \beta, \eta, {\bf x}}\co \wt\Lambda \to\mathsf{PGL}_d(\mathbb{C}),\]
see Definition \ref{def: bending map}. This notion, while technical, is designed so that the map
\[\xi:\partial\wt\lambda\to\mathcal{F}(\mathbb{C}^d)\]
given by $\xi(z):=\Psi(h)\cdot\xi_0(z)$, where $h$ is some (any) leaf of $\wt\lambda$ that has $z$ as an endpoint, is well-defined and has the following properties:
\begin{itemize}
	\item $\xi$ is locally $\lambda$--H\"older continuous and $\lambda$--transverse. In particular, there is a slithering map $\Sigma:\wt\Lambda^2\to\mathsf{SL}_d(\mathbb{C})$ compatible with $\xi$.
	\item $(\xi(x_1),\xi(x_2),\xi(x_3)^1)=(\xi_0(x_1),\xi_0(x_2),\xi_0(x_3)^1)$.
	\item $\xi$ is $\lambda$--hyperconvex, and 
	\[(\alpha_\xi,\theta_\xi)=\mathfrak{s}_d([\rho_0])+(i\beta,i\eta).\]
\end{itemize}
Here, recall that the pair $(\alpha_\xi,\theta_\xi)$ was defined in Section \ref{shear-bend} (it is the analog of $\lambda$--cocyclic pairs that one can associate to any locally $\lambda$--H\"older continuous, $\lambda$--transverse, and $\lambda$--hyperconvex map).

We then show that one can uniquely construct such generalized bending maps by taking the extension of a certain H\"older extendable map defined from purely imaginary eruptions and shears:

\begin{theorem}\label{thm: bending Hitchin}
	For all $(\rho_0, \xi_0, \beta, \eta, {\bf x})$, there exists a unique $(\beta,\eta)$--generalized bending map for $(\rho_0,\xi_0)$ with base ${\bf x}$. 
\end{theorem}

Then, using the uniqueness of generalized bending maps, we prove that the bent limit maps are limit maps of $d$--pleated surfaces in $\mathcal{R}_d(\lambda)$:

\begin{theorem}\label{thm: bending Hitchin properties}
	For every generalized bending map $\Psi$, there exists a (necessarily unique) $d$--pleated surface $\rho\in\mathcal{R}_d(\lambda)$ such that the $\Psi$--bent limit map is the $\lambda$--limit map of $\rho$.
\end{theorem}

The $d$--pleated surface $\rho$ given by Theorem \ref{thm: bending Hitchin properties} is the required output of generalized bending applied to the inputs $(\rho_0,\xi_0,\beta,\eta,{\bf x})$. 

In Section \ref{ssec: def}, we will define generalized bending maps and prove that they are unique, see Proposition \ref{prop: uniqueness of bending}. Then, we prove Theorem \ref{thm: bending Hitchin} in Section~\ref{sec:existence}, and Theorem \ref{thm: bending Hitchin properties} in Section \ref{sec: bending cocycle}.

\subsection{The generalized bending map} \label{ssec: def}
We will now define the generalized bending map. As before, $\mathcal{A}$ denotes the set of pairs of positive integers that sum to $d$ and $\mathcal{B}$ denotes the set of triples of positive integers that sum to $d$. Recall that for any $d$--pleated surface $\rho:\Gamma\to\mathsf{PGL}_d(\mathbb{C})$ with pleating locus $\lambda$, 
\[(\alpha_\rho,\theta_\rho):=\mathfrak{sb}_d([\rho])\in\mathcal{Y}_d(\lambda;\mathbb{C}/2\pi i\mathbb{Z})\] 
denotes the $\lambda$--cocyclic pair associated to $\rho$, and $\widetilde\Delta^o$ denotes the set of labelings, i.e. triples of points in $\partial\wt\lambda$ that are the vertices of a plaque of $\wt\lambda$. In the rest of this section, we will also use the following notation: 

\begin{notation}\label{not:section10} Let ${\bf x} = (x_1,x_2,x_3), {\bf y} = (y_1,y_2,y_3) \in \widetilde{\Delta}^o$.
	\begin{enumerate}
		\item Denote ${\bf x}_-:=(x_3,x_1,x_2)$, ${\bf x}_+:=(x_2,x_3,x_1)$, and $\wh{\bf x}:=(x_2,x_1,x_3)$. 
		\item Denote by
		\begin{itemize}
			\item $T_{\bf x}$ the plaque of $\widetilde{\lambda}$ with vertices $x_1,x_2,x_3$.
			\item $h_{\bf x}$ the edge of $T_{\bf x}$ with endpoints $x_1$ and $x_2$.
		\end{itemize}
		In this notation, $h_{\bf x}$, $h_{{\bf x}_-}$, and $h_{{\bf x}_+}$ are the three edges of $T_{\bf x}$. 
		\item 
		Given any leaf $h$ and any plaque $T$ of $\widetilde{\lambda}$, we denote by $g_T^h$ the edge of $T$ that separates $T$ and $h$. Also, given a pair of distinct plaques $T$ and $T'$ of $\widetilde{\lambda}$, we let $(g_T^{T'},g_{T'}^T)$ denote a separating pair for $(T,T')$. In the case when $T=T_{\bf x}$, we also sometimes ease notation as follows:
		\[g_{\bf x}^T:=g_{T_{\bf x}}^T, \quad g_T^{\bf x}:=g_T^{T_{\bf x}}\,\,\text{ and }\,\,g_{\bf x}^h:=g_{T_{\bf x}}^h.\]
		Similarly, if $T=T_{\bf x}$ and $T'=T_{\bf y}$, we denote
		\[g_{\bf x}^{\bf y}:=g_{T_{\bf x}}^{T_{\bf y}}.\]
		
		\begin{figure}[h!]
			\includegraphics[width=10cm]{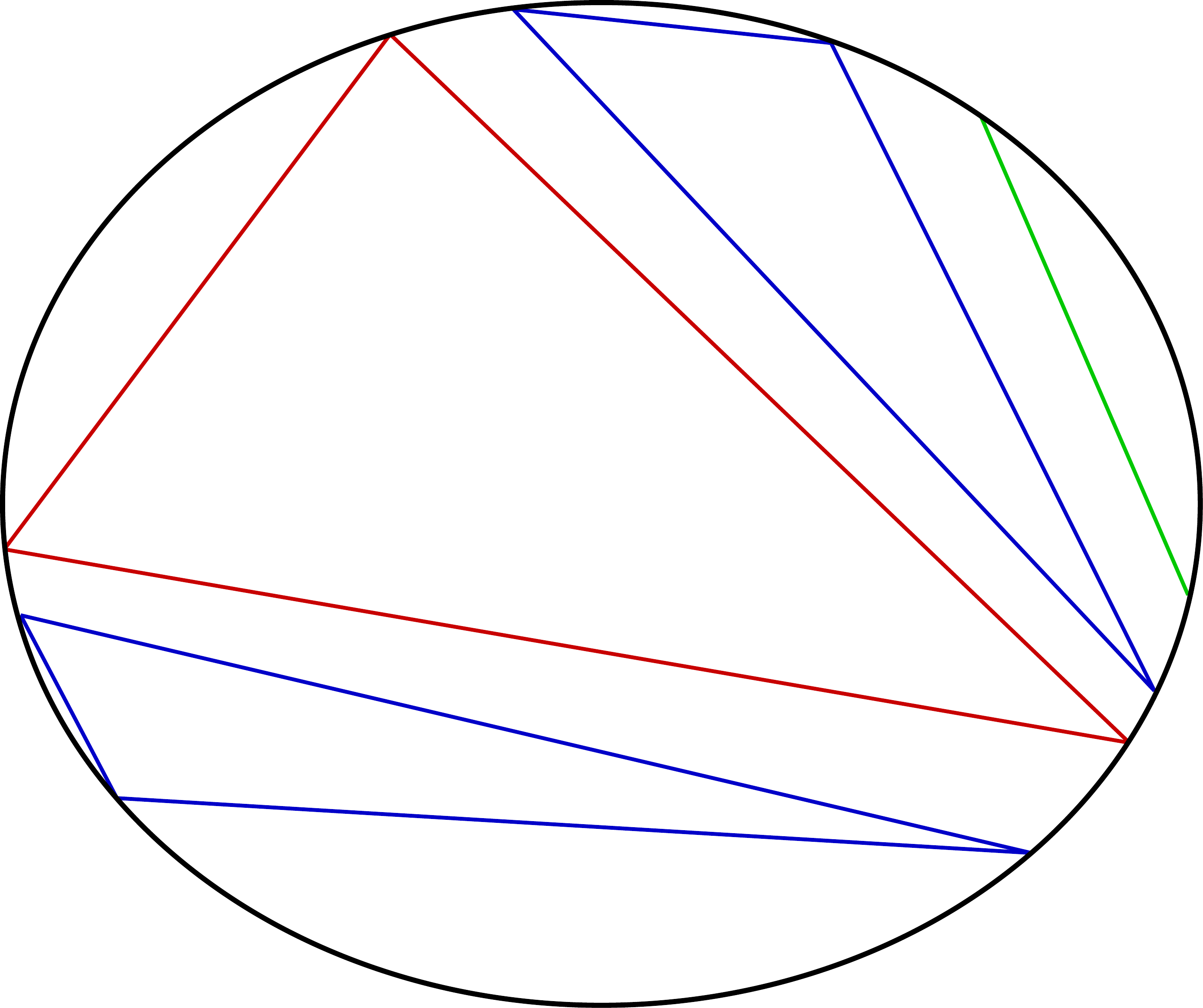}
			\put(-19,135){\small $h$}
			\put(-80,165){\small $T$}
			\put(-83,190){\small $g_T^h$}
			\put(-17,60){\small $x_1$}
			\put(-160,50){\small $T'$}
			\put(-185,135){\small $T_{\bf x}$}
			\put(-185,95){\small $g_{\bf x}^{T'}=h_{\bf x}$}
			\put(-197,63){\small $g_{T'}^{T}$}
			\put(-135,140){\small $g_{\bf x}^h=h_{\bf x_-}$}
			\put(-295,109){\small $x_2$}
			\put(-200,234){\small $x_3$}
			\caption{\small The notation introduced in Notation \ref{not:section10}.}
			\label{fig:unipotenteruption10}
		\end{figure}
	\end{enumerate}
\end{notation}

Let $(\rho_0, \xi_0)$ be  a $d$--pleated surface with pleating locus $\lambda$, let $(\beta, \eta)$ be a $\lambda$--cocyclic pair in $\mathcal{Y}_d(\lambda; \mathbb{R} / 2\pi \mathbb{Z})$, and let ${\bf x} \in \wt\Delta^o$. Denote 
\[\eta({\bf y}):=(\eta^{\bf j}({\bf y}))_{{\bf j}\in\mathcal{B}}\,\,\text{ and }\,\,\beta({\bf T}):=(\beta^{\bf i}({\bf T}))_{{\bf i}\in\mathcal{A}}\]
for all labelings ${\bf y}=(y_1,y_2,y_3)\in\wt\Delta^o$ and all distinct pairs of plaques ${\bf T}=(T_1,T_2)$ of $\wt\lambda$. 

\begin{definition}\label{def: bending map}
	A \emph{$(\beta, \eta)$--generalized bending map} for $(\rho_0, \xi_0)$ with base ${\bf x}=(x_1,x_2,x_3)$ is a map
	\[\Psi = \Psi_{\rho_0, \xi_0, \beta, \eta, {\bf x}}\co \wt\Lambda \to\mathsf{PGL}_d(\mathbb{C})\] 
	that satisfies the following properties:
	\begin{enumerate}[label=(\roman*)]
		\item For any two leaves $h_1, h_2 \in \wt\Lambda$ that share an endpoint $z \in \partial \wt\lambda$, we have that 
		$$\Psi(h_1) \cdot \xi_0(z) = \Psi(h_2)\cdot\xi_0(z).$$
	\end{enumerate}
	This allows us to define the \emph{$\Psi$--bent limit map} 
	\[\xi = \xi_{\Psi}\co\partial\wt\lambda\to\mathcal{F}(\mathbb{C}^d)\] 
	by setting $\xi(z):= \Psi(h)\cdot\xi_0(z)$, where $h \in \wt{\Lambda}$ is some (any) leaf that has $z$ as an endpoint. 
	\begin{enumerate}[resume, label=(\roman*)]
		\item The map $\xi$ is locally $\lambda$--H\"older continuous. 
	\end{enumerate}
	Since the map $\xi_0$ is $\lambda$--transverse, so is $\xi$. Then, Theorem \ref{thm: slitherable and slithering} and this assumption implies that there is a unique slithering map $\Sigma$ associated to $\xi$.
	\begin{enumerate}[resume, label=(\roman*)]
		
		\item For any leaf $h$ of $\wt\lambda$, 
		$$\Psi(h) = \Sigma(h,g^h_{\bf x})\, v({\bf x},h)\, \Sigma^0(g^h_{\bf x}, h),$$ 
		where $\Sigma^0$ denotes the slithering map associated to $\xi_0$, and \begin{equation}\label{eq: definition of v}
			v({\bf x},h):= \begin{cases}
				a_{\xi_0(\widehat{\bf x})}(i\eta(\widehat{\bf x})) & \text{if $g_{\bf x}^h = h_{{\bf x}_-}$,} \\
				a_{\xi_0({\bf x})}(i\eta({\bf x})) & \text{if $g_{\bf x}^h = h_{{\bf x}_+}$,} \\
				\mathrm{id} & \text{if $g_{\bf x}^h =  h_{\bf x}$.}
			\end{cases}
		\end{equation}
		\item $(\xi(x_1),\xi(x_2),\xi(x_3)^1) = (\xi_0(x_1),\xi_0(x_2),\xi_0(x_3)^1)$.
		\item For any labeling ${\bf y}\in\wt\Delta^o$ and ${\bf j} \in \mathcal{B}$, the triple of flags $\xi({\bf y})$ is in general position and satisfies
		\begin{align*}
			\theta_\xi^{\bf j}({\bf y}) 
			&= \theta_{\rho_0}^{\bf j}(\mathbf{y}) + i \eta^{\bf j}(\mathbf{y}).
		\end{align*}

		\item For any pair of distinct plaques ${\bf T} = (T_1,T_2)$ of $\wt\lambda$ and any ${\bf i} \in \mathcal{A}$, 
		\begin{align*}
			\alpha_{\xi}^{\bf i}({\bf T})
			&= \alpha_{\rho_0}^{\bf i}({\bf T})+ i \beta^{\bf i}({\bf T}).
		\end{align*}
		
	\end{enumerate}
	We refer to $\xi$ as the \emph{$\lambda$--limit map associated to $\Psi$}.
\end{definition}

Note that in Definition \ref{def: bending map}, Property (i) ensures that we can use the bending map $\Psi$ to define its associated $\lambda$--limit map $\xi$, while Properties (ii), (iv), (v), and (vi) are the required properties of $\xi$ described before the statement of Theorem \ref{thm: bending Hitchin}. Condition (iii) on the other hand, is a natural condition on the bending map itself that guarantees its uniqueness (if it exists), as we now show.

\begin{proposition}\label{prop: uniqueness of bending}
	If $\Psi$ and $\Psi'$ are two $(\beta,\eta)$--bending maps for $(\rho_0,\xi_0)$ with base ${\bf x}=(x_1,x_2,x_3)$, then $\Psi=\Psi'$.
\end{proposition}

\begin{proof}
	Let $\xi$ and $\xi'$ denote the $\lambda$--limit maps associated to $\Psi$ and $\Psi'$ respectively, and let $\Sigma$ and $\Sigma'$ denote the slithering maps compatible with $\xi$ and $\xi'$ respectively. By Properties (iv), (v), and (vi) of Definition \ref{def: bending map}, we have
	\[(\xi(x_1),\xi(x_2),\xi(x_3)^1) = (\xi'(x_1),\xi'(x_2),\xi'(x_3)^1),\quad \theta_\xi^{\bf j} = \theta_{\xi'}^{\bf j},\,\,\text{ and }\,\,\alpha_\xi^{\bf i}=\alpha_{\xi'}^{\bf i}.\]
	So, we may apply Part (2) of Lemma \ref{lem: injectivity} to deduce that $\Sigma(g_{\bf x}^h,h)=\Sigma'(g_{\bf x}^h,h)$ for any leaf $h\in\wt\Lambda$. Then by Property (iii) of Definition \ref{def: bending map}, 
	\[\Psi(h) = \Sigma(h,g^h_{\bf x})\, v({\bf x},h)\, \Sigma^0(g^h_{\bf x}, h)=\Sigma'(h,g^h_{\bf x})\, v({\bf x},h)\, \Sigma^0(g^h_{\bf x}, h)=\Psi'(h).\]
	for any leaf $h\in\wt\Lambda$.
\end{proof}

\subsection{Existence of the bending map}\label{sec:existence}
In this section, we will prove Theorem~\ref{thm: bending Hitchin}. In fact, we will prove a more general version of Theorem \ref{thm: bending Hitchin} (see Theorem \ref{thm_exist-psi}) where $(\rho_0,\xi_0)$ is allowed to be any $d$--pleated surface with pleating locus $\lambda$ that satisfies the following genericity condition. 

\begin{definition}\label{def:prop_star}
	A triple of flags ${\bf F} = (F_1,F_2,F_3)$ in $\mathcal{F}(\mathbb{C}^d)$ satisfies \textit{property-$\star$} if the following hold:
	\begin{enumerate}
		\item The pairs of flags $F_1,F_2$ and $F_2, F_3$ are transverse. 
		\item For any $i \in \{1, \dots, d\}$, let $L_i:=F_2^i\cap F_1^{d-i+1}$ and let $\pi_i : \mathbb{C}^d \to L_i$ be the linear projection with kernel $\bigoplus_{j \neq i} L_j$. Then for every $i \in \{1, \dots, d - 1\}$ the restriction of $\pi_i$ to the line $L_{i+1}' = F_2^{i+1} \cap F_3^{d - i}$ is a linear isomorphism.
	\end{enumerate}
	We say that a $d$--pleated surface $(\rho_0, \xi_0)$ with pleating locus $\lambda$ \textit{satisfies property-$\star$} if for all ${\bf y} \in \wt\Delta^o$ the triple of flags $\xi_0({\bf y})$ satisfies property-$\star$. 
\end{definition}

Note that property-$\star$ is invariant under conjugation, i.e. if $(\rho_0,\xi_0)$ is a $d$--pleated surface with pleating locus $\lambda$ that satisfies property-$\star$, then the same is true for $(g\cdot \rho_0\cdot g^{-1},g\circ \xi_0)$ for any $g\in\mathsf{PGL}_d(\mathbb{C})$.

\begin{remark}
	It is possible for a triple of flags in $\mathcal{F}(\mathbb{C}^d)$ to be in general position but not satisfy property-$\star$. For example, if $e_1$, $e_2$, and $e_3$ is a basis of $\mathbb{C}^3$, then for all $i=1,2,3$, let $F_i\in\mathcal{F}(\mathbb{C}^3)$ be the flag such that 
	\[F_i^1=[e_i]\,\,\text{ and }\,\,F_i^2=\Span_{\mathbb{C}}(e_i,e_1+e_2+e_3).\] 
	It is straightforward to verify that the triple of flags $(F_1,F_2,F_3)$ is in general position but does not satisfy property-$\star$. Also, requiring a $d$--pleated surface to satisfy property-$\star$ is a non-trivial condition. For example, in Proposition \ref{lem: non-Hausdorff 1} we construct a $3$--pleated surface $\rho$ such that $\theta_\rho({\bf x})=i\pi$ for all ${\bf x}\in\wt{\Delta}^o$, and thus, $\rho$ does not satisfy property-$\star$. 
\end{remark}

\begin{remark}\label{rem: star reformulation}
	Condition (1) of Definition \ref{def:prop_star} implies that there is a unique unipotent transformation $u_{\bf F}\in\mathsf{PGL}_d(\mathbb{C})$ that fixes the flag $F_2$ and sends $F_1$ to $F_3$. Then
	Condition (2) of Definition \ref{def:prop_star} is equivalent to the following: If $(b_1, \dots, b_d)$ is a basis of $\mathbb{C}^d$ such that $b_i \in L_i$ for all $i \in \{1, \dots, d\}$, then the $(i,i+1)$--entry of the upper triangular unipotent matrix that represents $u_{\bf F}$ in this basis is non-zero for every $i \in \{1, \dots, d - 1\}$. 
\end{remark}

The following theorem is the main result of this subsection.

\begin{theorem}\label{thm_exist-psi}
	Let $(\rho_0, \xi_0)$ be a $d$--pleated surface with pleating locus $\lambda$, let $(\beta, \eta) \in \mathcal{Y}_d(\lambda; \mathbb{R} / 2\pi \mathbb{Z})$, and let ${\bf x} \in \wt\Delta^o$ be a labeling of a plaque of $\widetilde{\lambda}$. If $(\rho_0,\xi_0)$ satisfies property-$\star$, then there is a unique $(\beta, \eta)$--generalized bending map for $(\rho_0,\xi_0)$ with base ${\bf x}$.
\end{theorem}

Before proving Theorem \ref{thm_exist-psi}, we remark that the positivity of Hitchin representations implies that they satisfy the property-$\star$.  Given a basis $(b_1,\dots,b_d)$ of $\mathbb{R}^d$, we denote by $\mathcal{U}_{>0}(b_1,\dots,b_d)$ the set of unipotent elements in $\mathsf{SL}_d(\mathbb{R})$ that are represented in the basis $(b_1,\dots,b_d)$ by an upper triangular matrix $M_u$ that is \emph{totally positive}, i.e. the minors of $M_u$ (except those that are forced to be zero by virtue of $M_u$ being upper triangular) are positive. Fock and Goncharov \cite{fock-goncharov-1} proved that a triple of flags $(F_1,F_2,F_3)$ in $\mathcal{F}(\mathbb{R}^d)$ is positive (see Definition \ref{def: positivity via invariants}) if and only if there is a basis $(b_1,\dots,b_d)$ of $\mathbb{R}^d$ such that $b_i\in F_2^i\cap F_1^{d-i+1}$ and  $u\in \mathcal{U}_{>0}(b_1,\dots,b_d)$ such that $F_3=u\cdot  F_1$. Thus, Theorem~\ref{thm: bending Hitchin} follows immediately from Theorem \ref{thm_exist-psi} and Remark \ref{rem: star reformulation}.

The remainder of this subsection is dedicated to the proof of Theorem \ref{thm_exist-psi}. Note that the uniqueness of the generalized bending map was already proven in Proposition \ref{prop: uniqueness of bending}, so we focus on the existence. Fix a $d$--pleated surface $(\rho_0, \xi_0)$ with pleating locus $\lambda$ that satisfies property-$\star$, a $\lambda$--cocyclic pair $(\beta, \eta) \in \mathcal{Y}_d(\lambda; \mathbb{R} / 2\pi \mathbb{Z})$, and a labeling ${\bf x} = (x_{1}, x_{2}, x_3) \in \wt\Delta^o$. 

Recall that:
\begin{itemize}
	\item $\wt\Lambda^o$ denotes the set of leaves of $\wt\lambda$ equipped with the choice of an orientation. 
	\item For any leaves $g_1$ and $g_2$ of $\widetilde\lambda$, we denote by $P(g_1,g_2)$ the set of plaques of $\widetilde\lambda$ that separate $g_1$ and $g_2$, and by $Q(g_1,g_2)$ the set of leaves of $\widetilde\lambda$ that separate $g_1$ and $g_2$. 
	\item For any plaque $T\in P(g_1,g_2)$, $h_{T,-}$ and $h_{T,+}$ denote the edges of $T$ that separate $g_1$ and $g_2$ such that $h_{T,-}$ separates $g_1$ and $h_{T,+}$, and ${\bf x}_T = (x_{T,1}, x_{T,2}, x_{T,3}) \in \widetilde{\Delta}^o$ denotes the coherent labeling of the vertices of $T$ with respect to the pair $(g_1,g_2)$, i.e. $x_{T,2}$ is the common endpoint of $h_{T,-}$ and $h_{T,+}$, and $x_{T,1}$ and $x_{T,3}$ are respectively the endpoints of $h_{T,-}$ and $h_{T,+}$ that are not $x_{T,2}$. 
\end{itemize}

We will also use the following simplification of notation.

\begin{notation}\label{not: notations and identities}\
	Given any oriented edge $\mathsf{g}\in\wt\Lambda^o$ any labeling ${\bf y}\in\wt\Delta^o$, any ${\upsilon}\in(\mathbb{R}/2\pi \mathbb{Z})^{\mathcal{A}}$, and any ${\zeta}\in(\mathbb{R}/2\pi \mathbb{Z})^{\mathcal{B}}$, we set
	\[
	c_{\mathsf{g}}(\upsilon):= c_{\xi_0({\mathsf{g}})}(i \upsilon),\quad a_{\bf y}(\zeta):= a_{\xi_0({\bf y})}(i \zeta), \,\,\text{ and }\,\,u_{{\bf y}}(\zeta,\upsilon):=u_{\xi_0({\bf y})}(i \zeta,i \upsilon),
	\] 
	where $c_{\bullet}(\bullet)$, $a_{\bullet}(\bullet)$, and $u_{\bullet}(\bullet,\bullet)$ are the complex shear, complex eruption, and unipotent eruption defined in Section \ref{sec:eruption and shear}.
\end{notation}

If $h$ is a leaf in $Q(g_1,g_2)$, the \emph{right-to-left orientation of $h$ with respect to $(g_1,g_2)$} is the orientation on $h$ that passes from the right to the left of some (any) arc that intersects every leaf of $\widetilde\lambda$ transversely and at most once, and whose backward and forward endpoints lie in $g_1$ and $g_2$ respectively.

For any leaf $h$ of $\widetilde{\lambda}$, let  
\begin{equation}\label{eqn: Nh}
N_h : P(g_{\bf x}^h,h) \to \mathsf{SL}_d(\mathbb{C})
\end{equation}
be the map that assigns to each plaque $T\in P(g_{\bf x}^h,h)$ the unipotent linear map in $\mathsf{SL}_d(\mathbb{C})$ that represents the unipotent eruption 
\begin{align*}
u_T&:=\left\{\begin{array}{ll}
	u_{{\bf x}_T}\big(\eta({\bf x}_T), \beta(T_{\bf x},T)\big)&\text{if }x_{T,3}\prec x_{T,2}\prec x_{T,1},\\
	u_{{\bf x}_T}\big(\eta({\bf x}_T),\wh{\beta(T_{\bf x},T)}\big)&\text{if }x_{T,1}\prec x_{T,2}\prec x_{T,3},
\end{array}\right.\\
&=c_{\mathsf{h}_{T,-}}(\beta(T_{\bf x},T))\,a_{{\bf x}_T}(\eta({\bf x}_T))\,c_{\mathsf{h}_{T,+}}(-\beta(T_{\bf x},T)-\nu({\bf x}_T)),
\end{align*} 
where 
\begin{itemize}
\item ${\bf x}_T=(x_{T,1},x_{T,2},x_{T,3})$ is the coherent labeling of the vertices of $T$ with respect to $(g_{\bf x}^h,h)$, 
\item $\mathsf{h}_{T,-}$ and $\mathsf{h}_{T,+}$ are the two edges of $T$ that separate $g_{\bf x}^h$ and $h$, equipped with the right-to-left orientation with respect to $(g_{\bf x}^h,h)$, and so that $\mathsf{h}_{T,-}$ separates $g_{\bf x}^h$ and $\mathsf{h}_{T,+}$, 
\item $\widehat{\cdot}:\mathbb{R}^{\mathcal{A}}\to\mathbb{R}^{\mathcal{A}}$ is as defined by identity \eqref{eqn: permute 2},
\item $\nu:\wt\Delta^o\to(\mathbb{R}/2\pi\mathbb Z)^{\mathcal{A}}$ is the map defined by
\begin{align}\label{eqn: nu}
	\nu({\bf y}) := 
	\begin{cases}
		s_{\eta({\bf y})} & \text{if $y_{3} \prec y_{2} \prec y_{1}$,} \\
		\widehat{s_{\eta({\bf y})}} & \text{if $y_{1} \prec y_{2} \prec y_{3}$,}
	\end{cases}
\end{align}
and $s_{\eta({\bf y})}$ is as defined by equation \eqref{eqn: sum zeta}.
\end{itemize}
For convenience, we will often abuse notation and also denote 
\[N_h(T):=u_T.\] 
The following proposition, which we prove in Section \ref{subsub:bending productable}, is a key result needed to define the required bending map.

\begin{proposition}\label{prop:bending productable}
The map $N_h$ is H\"older extendable 
for all leaves $h$ of $\wt\lambda$.
\end{proposition}

By Proposition \ref{prop:bending productable}, we may then define for any leaf $h$ of $\widetilde\lambda$, the extension 
\[\overrightarrow{N_h}:Q(g_{\bf x}^h,h)\to\mathsf{SL}_d(\mathbb{C})\] 
of the H\"older extendable map $N_h$. 
For convenience, we will often abuse notation and also denote by $\overrightarrow{N_h}$ the map
\[\pi\circ \overrightarrow{N_h}:Q(g_{\bf x}^h,h) \to \mathsf{PGL}_d(\mathbb{C}),\]
where $\pi:\mathsf{SL}_d(\mathbb{C})\to\mathsf{PGL}_d(\mathbb{C})$ is the obvious projection. Then define 
\begin{align}\label{eqn: Psi}
\Psi:\wt\Lambda\to\mathsf{PGL}_d(\mathbb{C}) , \quad \Psi(h):= v({\bf x}, h) \, \overrightarrow{N_h}(g_{\bf x}^h,h) ,
\end{align}
where $v({\bf x},h)$ was given by identity \eqref{eq: definition of v}.

\begin{remark}
In order to justify geometrically the definition of $\Psi$, it is instructive to consider the following (non-generic) example. Assume that:
\begin{itemize}
	\item $h$ is a leaf of $\widetilde{\lambda}$ that is not an edge of $T_{\bf x}$ such that $g_{\bf x}^h = h_{\bf x}$, and
	\item the set $P(g_{\bf x}^h,h)$ is finite. 
\end{itemize}
The first condition implies that $v({\bf x},h)=\id$, so $\Psi(h)=\overrightarrow{N_h}(h_\mathbf x,h)$. The second condition implies that we can enumerate $P(g_{\bf x}^h,h)=\{T_1, T_2, \dots, T_N\}$ according to the order on $P(g_{\bf x}^h,h)$, in which case $\Psi(h)=N_h(T_1)\dots N_h(T_N)$. We would like to understand the reason for this expression.

Set $T_0:= T_{\bf x}$, and observe that for every integer $i \in \{0, \dots, N-1\}$, the plaques $T_i$ and $T_{i + 1}$ share a common edge $g_i$. Set $g_{N} = h$, and for each integer $i\in\{1,\dots,N\}$, denote by $\mathsf{g}_i$ the right-to-left orientation of $g_i$ with respect to the pair $(g_{\bf x}^h, h)$. We now want to modify the limit map $\xi_0$ to build a new $d$--pleated surface $(\rho,\xi)$ with pleating locus $\lambda$  such that 
\[(\alpha_\rho,\theta_\rho)=(\alpha_{\rho_0} + i  \beta, \theta_{\rho_0} + i  \eta).\] 

Following Sun-Wienhard-Zhang \cite{SWZ}, one may expect to be able to ``bend'' the map $\xi_0$ along the edges and plaques using the complex shear and complex eruption. More precisely, the new limit map $\xi$ on the endpoints of the leaves $g_i$ for all integers $i\in\{1,\dots,N\}$ should be given by
\begin{equation}\label{eqn: bending justification}
	\xi(\mathsf{g}_i) = c_{\mathsf{g}_0}(\beta(T_0,T_1)) \, a_{{\bf x}_1}(\eta({\bf x}_1))\dots c_{\mathsf{g}_{i-1}}(\beta(T_{i-1},T_i)) \, a_{{\bf x}_i}(\eta({\bf x}_i)) \cdot \xi_0(\mathsf{g}_i).
\end{equation}
where ${\bf x}_j = (x_{j,1}, x_{j,2}, x_{j,3})$ denotes the coherent labeling of the vertices of $T_j$ with respect to the pair $(g_{\bf x}^h, h)$. One should think of the above formula as applying each of the complex eruptions $a_{{\bf x}_j}(\eta({\bf x}_j))$ and each of the complex shears $c_{\mathsf{g}_{j-1}}(\beta(T_{j-1},T_j))$ to $\xi_0(\mathsf{g}_k)$ for $k\ge j$, while not changing $\xi_0(\mathsf{g}_k)$ for $k<j$.

By the properties of the complex eruption and complex shear discussed in Section~\ref{sec:eruption and shear}, we see
\[\alpha_\rho(T_{i-1},T_i)=\alpha_{\rho_0}(T_{i-1},T_i)+i\beta(T_{i-1},T_i)\,\,\text{ and }\,\,\theta_\rho({\bf x}_i)=\theta_{\rho_0}({\bf x}_i)+i\eta({\bf x}_i)\]
for all integers $i\in\{1,\dots,N\}$.	

On the other hand, since $(\beta,\eta)$ is a $\lambda$--cocyclic pair of dimension $d$, it satisfies
\[
\beta(T_0,T_{j+1}) = \beta(T_0,T_j) + \beta(T_j,T_{j+1}) + \nu({\bf x}_j),
\]
for all integers $j\in\{1,\dots,N-1\}$, where $\nu({\bf x}_j)$ is defined by identity \eqref{eqn: nu}. This identity, combined with Proposition \ref{prop: shear combination} part (3) and the observation that  $c_{\mathsf{g}_i}(-\beta(T_0,T_i)-\nu({\bf x}_{T_i}))$ fixes the flag $\xi_0(\mathsf{g}_i)$, implies that equation \eqref{eqn: bending justification} can be written as 
\[\xi(\mathsf{g}_i)=N_h(T_1)\dots N_h(T_i)\cdot \xi_0(\mathsf{g}_i),\]
where recall that (with the usual abuse of notation)
\[N_h(T_j)=c_{\mathsf{g}_{j-1}}(\beta(T_0,T_j))\, a_{{\bf x}_j}(\eta({\bf x}_j))\,c_{\mathsf{g}_j}(-\beta(T_0,T_j)-\nu({\bf x}_{j})).\]
This rearrangement of the right-hand side of equation \eqref{eqn: bending justification} has the following convenience: each $N_h(T_j)$ depends only on the flags $\xi_0$ evaluated on the vertices of $T_j$, the values of $\eta({\bf x}_j)$, and the values $\beta(T_0,T_j)$. In particular, the definition of each term appearing in the product on the right-hand side can be generalized to any plaque that separates $g_{\bf x}^h$ and $h$, even without the assumption that the set $P(g_{\bf x}^h,h)$ is finite.
\end{remark}

To finish the proof of Theorem \ref{thm_exist-psi}, it now suffices to verify that the map $\Psi$ defined above is indeed a $(\beta, \eta)$--bending map for $(\rho_0,\xi_0)$ with base ${\bf x}$, and prove Proposition \ref{prop:bending productable}. These will be done in Sections \ref{sec: psi is bending map} and \ref{subsub:bending productable} respectively.

\subsubsection{$\Psi$ is a bending map}\label{sec: psi is bending map}
We will verify that the map $\Psi$ defined in equation \eqref{eqn: Psi} satisfies every property appearing in Definition \ref{def: bending map}.

We will need the following observation. Let $h$ be a leaf of $\wt\lambda$. Observe that if $h'\in Q(g_{\bf x}^h,h)$, then $g_{\bf x}^{h'}=g_{\bf x}^h$, so $v({\bf x},h)=v({\bf x},h')$. Furthermore, $\overrightarrow{N_h}(h_1,h_2)=\overrightarrow{N_{h'}}(h_1,h_2)$ for all leaves $h_1,h_2\in Q(g_{\bf x}^h,h')$, see Remark \ref{rem: obvious}. Thus, if we further assume that $h$ and $h'$ are the distinct edges of a plaque $T\neq T_{\bf x}$ of $\wt\lambda$, then by Proposition~\ref{prop: productable fixed} part (2), we have
\begin{equation}\label{eq: bending map on triangles}
\Psi(h) = v({\bf x}, h) \, \overrightarrow{N_h}(g_{\bf x}^h,h)=v({\bf x}, h') \, \overrightarrow{N_{h'}}(g_{\bf x}^h,h')\,\overrightarrow{N_h}(h',h)=\Psi(h') \, N_h(T).
\end{equation}

\noindent \textbf{Property (i).} We need to show that if $h_1$ and $h_2$ are leaves of $\wt\lambda$ that share an endpoint $z \in \partial \widetilde{\lambda}$, then $\Psi(h_1)^{-1} \Psi(h_2)$ fixes the flag $\xi_0(z)$. The proof will proceed by considering the following three possible cases separately:
\begin{itemize}
\item[Case 1:] $g_{\bf x}^{h_1}=g_{\bf x}^{h_2}$, and $h_1\in Q(g_{\bf x}^{h_1},h_2)$ up to switching the roles of $h_1$ and $h_2$,
\item[Case 2:] $g_{\bf x}^{h_1}=g_{\bf x}^{h_2}$, $h_1\notin Q(g_{\bf x}^{h_1},h_2)$, and $h_2\notin Q(g_{\bf x}^{h_1},h_1)$,
\item[Case 3:] $g_{\bf x}^{h_1}\neq g_{\bf x}^{h_2}$.
\end{itemize}

{\it Case 1.} Since $h_1\in Q(g_{\bf x}^{h_1},h_2)$, every plaque $T$ that separates $h_1$ and $h_2$ has $z$ as a common vertex. Hence, by Proposition \ref{prop: lift} part (1), $N_h(T) \cdot \xi_0(z) = \xi_0(z)$ for every plaque $T \in P(h_1,h_2)$, so Proposition \ref{prop: productable fixed} part (3) implies that $\overrightarrow{N_{h_2}}(h_1,h_2)$ fixes the flag $\xi_0(z)$.  At the same time, Proposition \ref{prop: productable fixed} part (2) implies that
\begin{align*}
\Psi(h_1)^{-1} \Psi(h_2)&= \overrightarrow{N_{h_1}}(g_{\bf x}^{h_1},h_1)^{-1}v({\bf x}, h_1)^{-1}v({\bf x}, h_2)\,\overrightarrow{N_{h_2}}(g_{\bf x}^{h_2},h_2), \\
&= \overrightarrow{N_{h_1}}(g_{\bf x}^{h_1},h_1)^{-1}\overrightarrow{N_{h_2}}(g_{\bf x}^{h_2},h_1)\,  \overrightarrow{N_{h_2}}(h_1,h_2)\\
&=\overrightarrow{N_{h_2}}(h_1,h_2),
\end{align*}
so $\Psi(h_1)^{-1} \Psi(h_2)$ fixes $\xi_0(z)$. 

{\it Case 2.} By Proposition~\ref{prop: 3 leaves}, there exists a plaque $T$ that separates $g_{\bf x}^{h_1}$ from $h_1$, separates $g_{\bf x}^{h_2}$ from $h_2$, and separates $h_1$ from $h_2$. See Figure \ref{fig:case 2 bending}. Note that $g_T^{h_1}$, $g_T^{h_2}$, and $g_T^{\bf x}$ are the three edges of $T$, so equation \eqref{eq: bending map on triangles} implies that $\Psi(g_T^{h_i})=\Psi(g_T^{\bf x})\, N_{h_i}(T)$ for both $i=1,2$. It follows that
\begin{align*}
\Psi(h_1)^{-1} \Psi(h_2) & = \Psi(h_1)^{-1} \Psi(g_T^{h_1})\, \Psi(g_T^{h_1})^{-1} \Psi(g_T^{h_2}) \,\Psi(g_T^{h_2})^{-1} \Psi(h_2) \\
& = \Psi(h_1)^{-1} \Psi(g_T^{h_1})\, N_{h_1}(T)^{-1} N_{h_2}(T)\, \Psi(g_T^{h_2})^{-1} \Psi(h_2).
\end{align*}
For both $i = 1, 2$, $z$ is the common endpoint of $g_T^{h_i}$ and $h_i$, and $g_T^{h_i} \in Q(g_{\bf x}^{h_i},h_i)$, so Case 1 implies that $\Psi(h_i)^{-1}\Psi(g_T^{h_i})$ fixes the flag $\xi_0(z)$. Thus, it suffices to show that $N_{h_1}(T)^{-1} N_{h_2}(T)$ fixes the flag $\xi_0(z)$. 
\begin{figure}[h!]
\includegraphics[width=10cm]{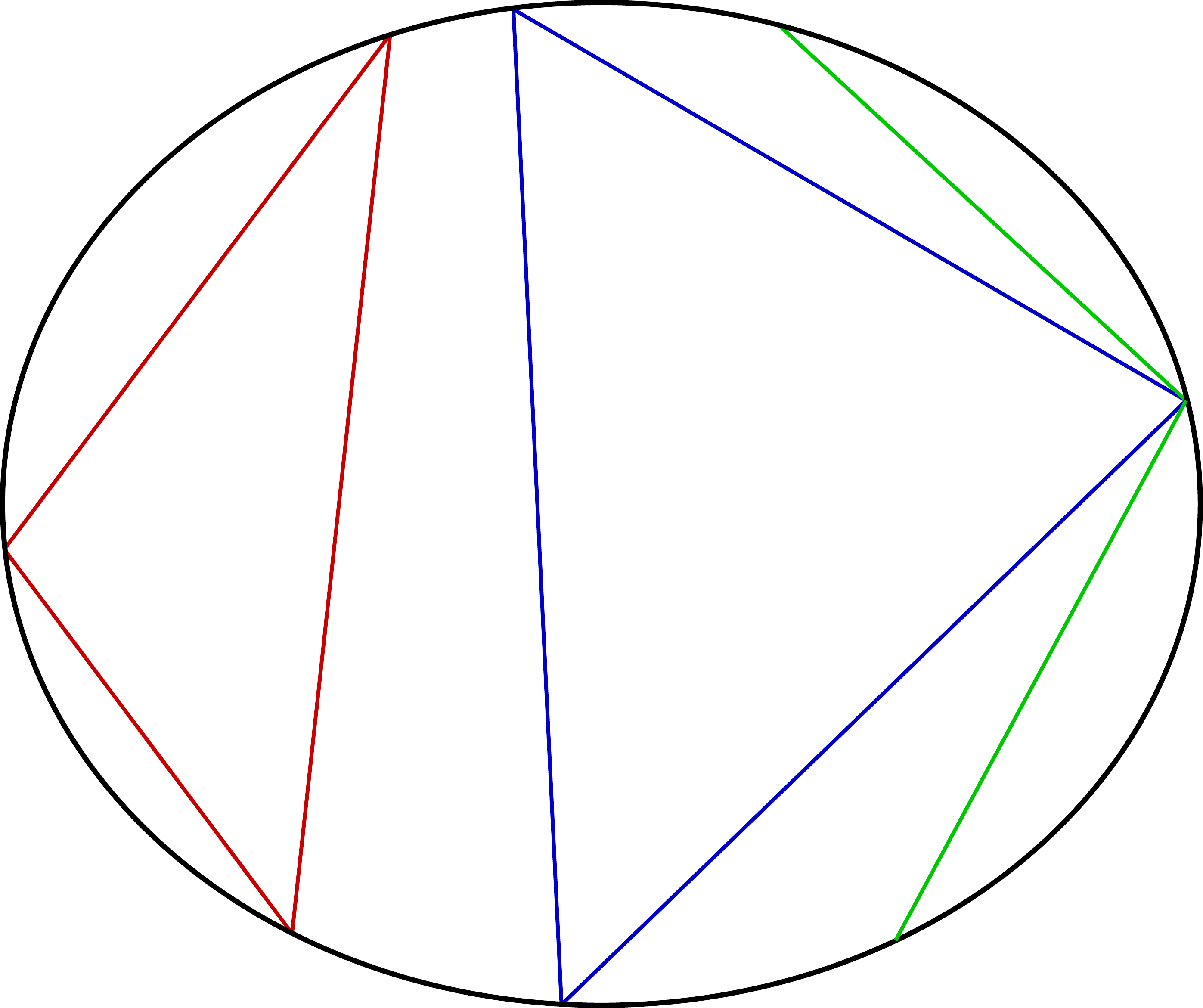}
\put(-50,187){\small $h_1$}
\put(-40,70){\small $h_2$}
\put(-100,130){\small $T$}
\put(-2,141){\small $z$}
\put(-83,175){\small $g_T^{h_1}$}
\put(-220,95){\small $g_{\bf x}^{h_1}$}
\put(-155,105){\small $g^{\bf x}_{T}$}
\put(-240,120){\small $T_{\bf x}$}
\put(-99,63){\small $g_{T}^{h_2}$}
\caption{\small The setup for Case 2 in the proof of the first property of the bending map.}
\label{fig:case 2 bending}
\end{figure}

For both $i=1,2$, let $\mathsf{g}_i$ be the leaf $g_T^{h_i}$ equipped with the right-to-left orientation with respect to the pair of leaves $(g_{\bf x}^{h_i},h_i)$, and let $\mathsf{g}$ be the leaf $g_T^{\bf x}$ equipped with the right-to-left orientation with respect to both pairs $(g_{\bf x}^{h_1},h_1)$ and $(g_{\bf x}^{h_2},h_2)$. Let ${\bf z}=(z_1,z_2,z_3)$ be the coherent labeling of the vertices of $T$ with respect to $(h_1,h_2)$. Then observe that $z_2=z$, ${\bf z}_-=(z_3,z_1,z_2)$ is the coherent labeling of the vertices of $T$ with respect to $(g_{\bf x}^{h_1},h_1)$, and $\wh{{\bf z}_-}=(z_1,z_3,z_2)$ is the coherent labeling of the vertices of $T$ with respect to $(g_{\bf x}^{h_2},h_2)$. By Proposition \ref{prop: eruption combination} part (3),
\begin{align*}
a_{{\bf z}_-}(\eta({\bf z}_-))^{-1}a_{\wh{{\bf z}_-}}(\eta(\wh{{\bf z}_-}))\cdot\xi(z)=a_{{\bf z}_-}(\eta({\bf z}_-))^{-1} a_{{\bf z}_-}(\eta({\bf z}_-))\cdot\xi(z)=\xi(z),
\end{align*}
and, since $z$ is an endpoint of both $\mathsf{g}_1$ and $\mathsf{g}_2$, the transformations $c_{\mathsf{g}_1}(-\beta(T_{\bf x},T)-\nu({\bf z}_-))$ and $c_{\mathsf{g}_2}(-\beta(T_{\bf x},T)-\nu(\wh{{\bf z}_-}))$ both fix the flag $\xi(z)$ as well ($\nu$ was defined by identity \eqref{eqn: nu}). By definition,
\begin{align*}
&\quad N_{h_1}(T)^{-1} N_{h_2}(T)\\
& = c_{\mathsf{g}_1}(-\beta(T_{\bf x},T)-\nu({\bf z}_-))^{-1}a_{{\bf z}_-}(\eta({\bf z}_-))^{-1}a_{\wh{{\bf z}_-}}(\eta(\wh{{\bf z}_-}))\,c_{\mathsf{g}_2}(-\beta(T_{\bf x},T)-\nu(\wh{{\bf z}_-})),
\end{align*}
so it follows that $N_{h_1}(T)^{-1} N_{h_2}(T)$ fixes the flag $\xi_0(z)$.

{\it Case 3.} Observe that for both $i=1,2$, every plaque $T$ that separates $g_{\bf x}^{h_i}$ and $h_i$ has $z$ as a common vertex. Hence, by Proposition \ref{prop: lift} part (1), $N_h(T) \cdot \xi_0(z) = \xi_0(z)$ for every plaque $T \in P(g_{\bf x}^{h_i},h_i)$, so Proposition \ref{prop: productable fixed} part (3) implies that $\overrightarrow{N_{h_i}}(g_{\bf x}^{h_i},h_i)$ fixes the flag $\xi_0(z)$. Since
\[\Psi(h_1)^{-1} \Psi(h_2)= \overrightarrow{N_{h_1}}(g_{\bf x}^{h_1},h_1)^{-1}v({\bf x}, h_1)^{-1}v({\bf x}, h_2)\,\overrightarrow{N_{h_2}}(g_{\bf x}^{h_2},h_2),\]
it suffices to show that $v({\bf x}, h_1)^{-1}v({\bf x}, h_2)$ fixes the flag $\xi_0(z)$. Since $v({\bf x}, h_{{\bf x}})=\id$, this is equivalent to verifying the identities
\begin{align*}
&v({\bf x}, h_{{\bf x}_-}) \cdot \xi_0(x_1) = \xi_0(x_1)\quad\text{if }z=x_1,\\
&v({\bf x}, h_{{\bf x}_+}) \cdot \xi_0(x_2) = \xi_0(x_2)\quad\text{if }z=x_2, \,\,\text{ and }\\
&v({\bf x}, h_{{\bf x}_-}) \cdot \xi_0(x_3) = v({\bf x}, h_{{\bf x}_+}) \cdot \xi_0(x_3)\quad\text{if }z=x_3.
\end{align*}
The first two identities follow from Proposition \ref{prop: eruption combination} part (1), while the third is implied by Proposition \ref{prop: eruption combination} part (3).\\

\noindent \textbf{Property (ii).}
Pick a leaf $\ell_0 \in \widetilde{\Lambda}$, and observe that there is some leaf $\ell\in\wt\Lambda$ such that the subset of leaves $Q(g_{\bf x}^{\ell},\ell)$ contains an open neighborhood of $\ell_0$ in $\wt \Lambda$. Thus, to prove Property (ii), it suffices to show that there are constants $A,\nu>0$ such that
\begin{align*}
d_{\mathcal{F}}\left(\xi(\mathsf{h}_1), \xi(\mathsf{h}_2)\right)  \leq A \, d_\infty(\mathsf{h}_1, \mathsf{h}_2)^{\nu}
\end{align*}
for all leaves $\mathsf{h}_1$ and $\mathsf{h}_2$ that are oriented in parallel and separate $g_{\bf x}^{\ell}$ and $\ell$.

First, we show that $\Psi|_{Q(g_{\bf x}^{\ell},\ell)}$ is H\"older continuous. Let $\norm{\cdot}$ denote the operator norm on $\End(\mathbb{C}^d)$ with respect to the standard Hermitian inner product on $\mathbb{C}^d$.
\begin{lemma}\label{lem: Holder Psi}
There are constants $A',\nu'>0$ such that 
\[\norm{\Psi(h_1)-\Psi(h_2)}\leq A' d_\infty(h_1,h_2)^{\nu'}\]
for all leaves $h_1,h_2\in Q(g_{\bf x}^{\ell},\ell)$.
\end{lemma}
\begin{proof} Note that for all leaves $h\in Q(g_{\bf x}^{\ell},\ell)$, 
\[\Psi(h)=	v({\bf x},h)\,\overrightarrow{N_{h}}(g_{\bf x}^{h},h)=v({\bf x},\ell)\,\overrightarrow{N_{\ell}}(g_{\bf x}^{\ell},h).\]
By the H\"older continuity of extensions of H\"older extendable maps, see Proposition \ref{prop: Holder productable}, there are constants $A'',\nu'>0$ such that
\[\norm{\overrightarrow{N_{\ell}}(g_{\bf x}^{\ell},h_1)-\overrightarrow{N_{\ell}}(g_{\bf x}^{\ell},h_2)}\leq A'' d_\infty(h_1,h_2)^{\nu'}\]
for all $h_1,h_2\in Q(g_{\bf x}^{\ell},\ell)$. The lemma now follows by setting $A':= A''\norm{v({\bf x},\ell)}$.
\end{proof}

Let $A',\nu'>0$ be the constants given by Lemma \ref{lem: Holder Psi}. Since $Q(g_{\bf x}^{\ell},\ell)$ is compact, the continuity of $\Psi$ implies that 
$$K :=\{ \Psi(g) \mid g \in Q(g_{\bf x}^{\ell},\ell)\} \subset \mathsf{PGL}_d(\mathbb{C})$$
is compact. Since $\mathsf{SL}_d(\mathbb{C})$ acts smoothly on the compact flag variety $\mathcal{F}(\mathbb{C}^d)$, it follows that the map 
\[G\co K \times \mathcal{F}(\mathbb{C}^d) \to \mathcal{F}(\mathbb{C}^d)\]
given by $G(H,F)=H\cdot F$ is Lipschitz in both entries. Thus, by enlarging $A'$ if necessary, we may assume that
\[d_{\mathcal{F}}\left(\Psi(h_1)\cdot F_1, \Psi(h_1)\cdot F_2\right)\leq A' \, d_{\mathcal{F}}\left(F_1, F_2\right)\]
and
\[d_{\mathcal{F}}\left(\Psi(h_1)\cdot F_1, \Psi(h_2)\cdot F_1\right)\leq A' \, \norm{\Psi(h_1)-\Psi(h_2)}.\]
for all leaves $h_1, h_2 \in Q(g_{\bf x}^{\ell},\ell)$ and all flags $F_1, F_2 \in \mathcal{F}(\mathbb{C}^d)$.

Recall from Theorem \ref{thm: Wang} part (1) that $\xi_0$ is locally $\lambda$--H\"older continuous, see Definition \ref{def: lambda continuous}. Since $Q(g_{\bf x}^{\ell},\ell)$ is compact, by further enlarging $A'$ and shrinking $\nu'$ if necessary, we may assume that 
\[
d_{\mathcal{F}}(\xi_0(\mathsf{h}_1), \xi_0(\mathsf{h}_2)) \leq A' \, d_\infty(\mathsf{h}_1,\mathsf{h}_2)^{\nu'}
\]
for all leaves $\mathsf{h}_1$ and $\mathsf{h}_2$ that are oriented in parallel and separate $g_{\bf x}^{\ell}$ and $\ell$. Then by Lemma~\ref{lem: Holder Psi},
\begin{align*}
d_{\mathcal{F}}\left(\xi(\mathsf{h}_1), \xi(\mathsf{h}_2)\right) &= d_{\mathcal{F}}\left(\Psi(h_1)\cdot \xi_0(\mathsf{h}_1), \Psi(h_2)\cdot \xi_0(\mathsf{h}_2)\right) \\
&\leq d_{\mathcal{F}}\left(\Psi(h_1)\cdot \xi_0(\mathsf{h}_1), \Psi(h_1)\cdot \xi_0(\mathsf{h}_2)\right) + d_{\mathcal{F}}\left(\Psi(h_1) \cdot \xi_0(\mathsf{h}_2), \Psi(h_2) \cdot \xi_0(\mathsf{h}_2)\right)\\
& \leq A' \, d_{\mathcal{F}}(\xi_0(\mathsf{h}_1), \xi_0(\mathsf{h}_2)) + A' \, \| \Psi(h_1) - \Psi(h_2)\|  \\
& \leq A \, d_\infty(\mathsf{h}_1, \mathsf{h}_2)^{\nu},
\end{align*}
for some $A>0$, where for both $i=1,2$, $h_i$ is the oriented leaf $\mathsf{h}_i$ without its orientation and $\nu:=\nu'$. \\

\noindent \textbf{Property (iii).} This is obvious if $h=g_{\bf x}^h$ (i.e. if $h$ is an edge of $T_{\bf x}$), so we may assume that $h\neq g_{\bf x}^h$. Then define the map
\[M:P(g_{\bf x}^h,h)\to\mathsf{SL}_d(\mathbb{C})\]
by 
\[M(T):= \overrightarrow{N_h}(g_{\bf x}^h,g_T^{\bf x})\,\Sigma^0(g_T^{\bf x},g_T^h)\,\overrightarrow{N_h}(g_{\bf x}^h,g_T^h)^{-1}=\overrightarrow{N_h}(g_{\bf x}^h,g_T^{\bf x})\,\Sigma^0(g_T^{\bf x},g_T^h)\,\overrightarrow{N_h}(g_T^h,g_{\bf x}^h).\] 
Proposition \ref{prop: renormalization} implies that $M$ is H\"older extendable, and its extension $\overrightarrow{M}$ satisfies
\begin{align}\label{eqn: 69}
\overrightarrow{M}(g_{\bf x}^h,h)=\Sigma^0(g_{\bf x}^h,h)\,\overrightarrow{N_h}(h,g_{\bf x}^h)=\Sigma^0(g_{\bf x}^h,h)\,\overrightarrow{N_h}(g_{\bf x}^h,h)^{-1}.
\end{align}

Let $T\in P(g_{\bf x}^h,h)$ be a plaque, and let ${\bf x}_T=(x_{T,1},x_{T,2},x_{T,3})$ be the coherent labeling of the vertices of $T$ with respect to $(g_{\bf x}^h,h)$. Observe that 
\[v({\bf x},h)\,M(T)\,v({\bf x},h)^{-1}=\Psi(g_T^{\bf x})\,\Sigma^0(g_T^{\bf x},g_T^h)\,\Psi(g_T^h)^{-1},\]
so $v({\bf x},h)\,M(T)\,v({\bf x},h)^{-1}$ fixes the flag $\xi(x_{T,2})$ and sends the flag $\xi(x_{T,3})$ to the flag $\xi(x_{T,1})$. Also, by equation \eqref{eq: bending map on triangles},
\begin{align*}
v({\bf x},h)\,M(T)\,v({\bf x},h)^{-1}=\Psi(g_T^{\bf x})\,\Sigma^0(g_T^{\bf x},g_T^h)\,N_h(T)^{-1}\Psi(g_T^{\bf x})^{-1}.
\end{align*}
Since $\Sigma^0(g_T^{\bf x},g_T^h)$ and $N_h(T)$ are both unipotent and fix the flag $\xi_0(x_{T,2})$, it follows that $v({\bf x},h)\,M(T)\,v({\bf x},h)^{-1}$ is unipotent, so
\[\Sigma(g_T^{\bf x},g_T^h)=v({\bf x},h)\,M(T)\,v({\bf x},h)^{-1}\]
because $\Sigma$ is the slithering map associated to $\xi$. By taking the extensions on both sides and applying equation \eqref{eqn: 69}, we have
\begin{align*}
\Sigma(g_{\bf x}^h,h)&=v({\bf x},h)\,\overrightarrow{M}(g_{\bf x}^h,h)\,v({\bf x},h)^{-1}\\
&=v({\bf x},h)\,\Sigma^0(g_{\bf x}^h,h)\,\overrightarrow{N_h}(g_{\bf x}^h,h)^{-1}v({\bf x},h)^{-1}\\
&=v({\bf x},h)\,\Sigma^0(g_{\bf x}^h,h)\,\Psi(h)^{-1},
\end{align*}
so Property (iii) holds.\\

\noindent \textbf{Property (iv).} This is immediate from Proposition \ref{prop: eruption combination} part (1). \\

\noindent \textbf{Property (v).} By the symmetries of the triple ratio, it suffices to verify this property for the labeling ${\bf y}=(y_1,y_2,y_3) \in \wt{\Delta}^o$ such that $y_1\prec y_2\prec y_3$, and the leaf $g_{\bf y}^{\bf x}$ has  $y_1$ and $y_2$ as its endpoints. If $h$ is the edge of $T_{\bf y}$ with endpoints $y_2, y_3$, then identity \eqref{eq: bending map on triangles} implies that
\begin{align*}
(\xi(y_1), \xi(y_2), \xi(y_3)) & = \Psi(g_{\bf y}^{\bf x}) \cdot (\xi_0(y_1), \xi_0(y_2), N_{h}(T_{\bf y}) \cdot \xi_0(y_3)).
\end{align*}
Then by the projective invariance of the triple ratios and Proposition \ref{prop: lift} part (1),
\begin{align*}
\theta_\xi^{\bf j}({\bf y})&=\tau^{\bf j}(\xi(y_1), \xi(y_2), \xi(y_3))\\
& = \tau^{\bf j}(\xi_0(y_1), \xi_0(y_2), N_{h}(T)\cdot \xi_0(y_3)) \\
& = \tau^{\bf j}(\xi_0(y_1), \xi_0(y_2), \xi_0(y_3)) + i \eta^{\bf j}({\bf y}) \\
& = \theta^{\bf j}_{\rho_0}({\bf y}) + i \eta^{\bf j}({\bf y}).
\end{align*}

\noindent \textbf{Property (vi).} The proof of Property (vi) proceeds in the following four cases:
\begin{itemize}
\item[Case 1:] $T_1=T_{\bf x}$,
\item[Case 2:] $T_2=T_{\bf x}$,
\item[Case 3:] One of $T_1$, $T_2$, and $T_{\bf x}$ separate the other two,
\item[Case 4:] $T_1$, $T_2$, $T_{\bf x}$ are pairwise distinct and none of them separate the other two.
\end{itemize}

\smallskip

\noindent {\it Case 1.} Fix ${\bf i}\in\mathcal{A}$, and let $h$ be an edge of the plaque $T:=T_2$ that is not $g_T^{\bf x}$, and let $z$ be the vertex of the plaque $T$ that is not an endpoint of the edge $g_T^{\bf x}$. Then 
\begin{align}\label{eqn: vi1}
\Sigma(g_{\bf x}^T, g_T^{\bf x})\cdot\xi(z)&= v({\bf x},h)\,\Sigma^0(g_{\bf x}^T, g_T^{\bf x})\,\Psi(g_T^{\bf x})^{-1}\cdot\xi(z)\nonumber\\
& = v({\bf x},h)\,\Sigma^0(g_{\bf x}^T, g_T^{\bf x})\,\Psi(g_T^{\bf x})^{-1}\Psi(h)\cdot\xi_0(z)\\
&= v({\bf x},h)\,\Sigma^0(g_{\bf x}^T, g_T^{\bf x})\,N_h(T) \cdot\xi_0(z)\nonumber
\end{align}
where the first equality follows from Property (iii), the second equality is the observation that $\xi(z)  = \Psi(h) \cdot \xi_0(z)$ because $z$ is an endpoint of $h$, and the third equality follows from equation \eqref{eq: bending map on triangles}.

Let $\mathsf{g}_T^{\bf x}$ and $\mathsf{g}_{\bf x}^T$ be respectively the leaves $g_T^{\bf x}$ and $g_{\bf x}^T$ equipped with the right-to-left orientation with respect to $(g_{\bf x}^h,h)$, let $z_1$ and $z_2$ be the backward and forward endpoints of $\mathsf{g}_T^{\bf x}$ respectively, let $y_1$ and $y_2$ be the backward and forward endpoints of $\mathsf{g}_{\bf x}^T$ respectively, 
and let $y_3$ be the vertex of $T_{\bf x}$ that is not an endpoint of $g_{\bf x}^T$. Observe from the definition of $N_h(T)$ that
\[N_h(T)\cdot\xi_0(z)^1=c_{\mathsf{g}_T^{\bf x}}(\beta(T_{\bf x},T))\cdot\xi_0(z)^1.\]
Since the double ratio $\sigma^{\bf i}(F_1,F_2,F_3,F_4)$ is a projective invariant that only depends on the flags $F_1$ and $F_2$ and the lines $F_3^1$ and $F_4^1$, identity \eqref{eqn: vi1} implies that
\begin{align*}
\alpha^{\bf i}_{\xi}(T_{\bf x},T)& = \sigma^{\bf i}\left(\xi(y_2),\xi(y_1),\xi(y_3), v({\bf x},h)\,\Sigma^0(g_{\bf x}^T, g_T^{\bf x})\,N_h(T) \cdot \xi_0(z)\right)\\
& = \sigma^{\bf i}\left(\xi_0(y_2),\xi_0(y_1),\xi_0(y_3), \Sigma^0(g_{\bf x}^T, g_T^{\bf x})\,c_{\mathsf{g}_T^{\bf x}}(\beta(T_{\bf x},T)) \cdot \xi_0(z)\right)\\
& = \sigma^{\bf i}\left(\xi_0(z_2),\xi_0(z_1),\Sigma^0(g_T^{\bf x},g_{\bf x}^T)\cdot\xi_0(y_3), c_{\mathsf{g}_T^{\bf x}}(\beta(T_{\bf x},T))\cdot \xi_0(z)\right)\\
& = \sigma^{\bf i}\left(\xi_0(z_2),\xi_0(z_1),\Sigma^0(g_T^{\bf x},g_{\bf x}^T)\cdot\xi_0(y_3), \xi_0(z)\right)+i\beta^{\bf i}(T_{\bf x},T)\\
& = \sigma^{\bf i}\left(\xi_0(y_2),\xi_0(y_1),\xi_0(y_3), \Sigma^0(g_{\bf x}^T,g_T^{\bf x})\cdot\xi_0(z)\right)+i\beta^{\bf i}(T_{\bf x},T)\\
&= \alpha_{\rho_0}^{\bf i}(T_{\bf x},T)+i\beta^{\bf i}(T_{\bf x},T),
\end{align*}
where the fourth equality is a consequence of Part (1) of Proposition \ref{prop: shear combination}. This proves Property (vi) in the case when $T_1=T_{\bf x}$.

\smallskip

\noindent {\it Case 2.} By Part (1) of Proposition~\ref{prop: shear bend map} and Case 1,
\[\alpha_{\xi}(T_1,T_2)=-\wh{\alpha_{\xi}(T_{\bf x},T_1)}=-\wh{\alpha_{\rho_0}(T_{\bf x},T_1)}-i\wh{\beta(T_{\bf x},T_1)}=\alpha_{\rho_0}(T_1,T_2)+i\beta(T_1,T_2),\]
where $\wh{\alpha_{\xi}(T_{\bf x},T_1)}$, $\wh{\alpha_{\rho_0}(T_{\bf x},T_1)}$, and $\wh{\beta(T_{\bf x},T_1)}$ are defined by equation \eqref{eqn: permute 2}. 

\smallskip

\noindent {\it Case 3.} We will give the proof only in the situation when the plaque $T_{\bf x}$ separates the plaques $T_1$ and $T_2$; the other two possibilities are similar. 

For any $\lambda$--cocyclic pair $(\alpha,\theta)\in\mathcal{Y}_d(\lambda;G)$ ($G$ is any Abelian group) and any labeling ${\bf x}\in\widetilde\Delta^o$, we set $s_{\theta({\bf x})} = (s_{\theta({\bf x})}^{\bf i})_{{\bf i}\in \mathcal{A}}$, where
\[s_{\theta({\bf x})}^{\bf i} :=\left\{\begin{array}{ll}
\displaystyle\sum_{{\bf j} \in \mathcal{B} \mid j_2 = i_1} \theta({\bf x},{\bf j})&\text{if }x_3'\prec x_2'\prec x_1',\\
\displaystyle-\sum_{{\bf j} \in \mathcal{B} \mid j_2 = i_2} \theta({\bf x},{\bf j})&\text{if }x_1'\prec x_2'\prec x_3'.
\end{array}
\right. , \]
for any ${\bf i }=(i_1,i_2) \in \mathcal{A}$. Let ${\bf x}'=(x_1',x_2',x_3')$ be the coherent labeling of the vertices of $T_{\bf x}$ with respect to $(g_{T_1}^{\bf x},g_{T_2}^{\bf x})$. 
Then by Part (3) of Proposition \ref{prop: shear bend map}, the previous two cases, and Property (v),
\begin{align*}
\alpha_\xi(T_1,T_2)&=\alpha_{\xi}(T_1,T_{\bf x})+\alpha_{\xi}(T_{\bf x},T_2)+ s_{\theta_{\xi}({\bf x}')}\\
&=\alpha_{\rho_0}(T_1,T_{\bf x})+i\beta(T_1,T_{\bf x})+\alpha_{\rho_0}(T_{\bf x},T_2)+i\beta(T_{\bf x},T_2)+ s_{\theta_{\rho_0}({\bf x}')}+ i s_{\eta(\mathbf{x'})}\\
&=\alpha_{\rho_0}(T_1,T_2)+i\beta(T_1,T_2).
\end{align*}

\smallskip

\noindent {\it Case 4.} Proposition \ref{prop: 3 leaves} implies that there is some plaque $T$ that separates every pair of plaques in $\{T_1, T_{\bf x}, T_2\}$. By Case 3, we have
\[\alpha_\xi(T_1,T)=\alpha_{\rho_0}(T_1,T)+i\beta(T_1,T)\quad\text{and}\quad\alpha_\xi(T,T_2)=\alpha_{\rho_0}(T,T_2)+i\beta(T,T_2).\]
Let ${\bf x}_T$ denote the coherent labeling of the vertices of $T$ with respect to $(g_{T_1}^{T},g_{T_2}^T)$. Then by Part (3) of Proposition \ref{prop: shear bend map} and Property (v),
\begin{align*}
\alpha_\xi(T_1,T_2)&=\alpha_{\xi}(T_1,T)+\alpha_{\xi}(T,T_2)+ s_{\theta_{\xi}({\bf x}_T)}\\
&=\alpha_{\rho_0}(T_1,T)+i\beta(T_1,T)+\alpha_{\rho_0}(T,T_2)+i\beta(T,T_2)+ s_{\theta_{\rho_0}({\bf x}_T)}+ i s_{\eta(\mathbf{x}_T)}\\
&=\alpha_{\rho_0}(T_1,T_2)+i\beta(T_1,T_2).
\end{align*}

This concludes the proof that $\Psi$ is a bending map.

\subsubsection{The map $N$ is H\"older extendable}\label{subsub:bending productable}
To prove Proposition \ref{prop:bending productable}, we first prove the following linear algebra result, which will shed light on the reason behind the hypothesis in Theorem \ref{thm_exist-psi} that $(\rho_0,\xi_0)$ satisfy property-$\star$.

\begin{lemma}\label{lem:some linear algebra}
Let ${\bf F} = (F_1,F_2,F_3)$ be a triple of flags in  $\mathcal{F}(\mathbb{C}^d)$ that is in general position and satisfies property-$\star$, and let ${\bf E}:= (F_1,F_2)$. Then:
\begin{enumerate}
	\item A sequence $(B_n)_n \subset \mathrm{Stab}({\bf E})$ verifies $B_n^{-1} \cdot F_3 \to F_1$ if and only if
	\[
	\lim_{n\to\infty} \max_{i = 1, \dots, d} \abs{\frac{\mu_{i+1}(B_n)}{\mu_i(B_n)}} = 0 ,
	\]
	where $\mu_i(B_n)$ denotes the eigenvalue of $B_n$ corresponding to the eigenline $F_2^i \cap F_1^{d - i + 1}$, for any $i \in \{1,\dots, d\}$.
	
	\item There exists a constant $C > 0$ such that for any sequence $(B_n)_n \subset \mathrm{Stab}({\bf E})$ that satisfies $B_n^{-1}\cdot F_3 \to F_1$ it holds that
	\[
	C^{-1} \max_{i = 1, \dots, d} \abs{\frac{\mu_{i+1}(B_n)}{\mu_i(B_n)}} \leq d_\mathcal{F}(B^{-1}_n\cdot F_3, F_1) \leq C \max_{i = 1, \dots, d} \abs{\frac{\mu_{i+1}(B_n)}{\mu_i(B_n)}}
	\]
	for every $n$ sufficiently large.
\end{enumerate}
\end{lemma}

\begin{proof}
Let $\mathcal{O} = \mathcal{O}_{F_2}$ be the set of flags that are transverse to $F_2$, and let $\mathcal{U} = \mathcal{U}_{F_2}$ be the group of unipotent elements of $\mathsf{SL}_d(\mathbb{C})$ that fix $F_2$. Denote by $V \in \mathcal{U}$ the unique unipotent transformation that fixes $F_2$ and that sends $F_1$ into $F_3$, and let $\mathsf{b}=(b_1, \dots, b_d)$ be a basis of $\mathbb{C}^d$ such that $b_i \in F_2^i \cap F_1^{d - i + 1}$ for every integer $i \in \{1, \dots, d\}$. Then the following properties hold:
\begin{enumerate}[label=(\Roman*)]
	\item The natural action $\mathcal{U} \times \mathcal{O} \to \mathcal{O}$ is smooth, free, and transitive. Thus, the orbit map $U \mapsto U\cdot F_1$ is a diffeomorphism between $\mathcal{U}$ and $\mathcal{O}$.
	\item The unipotent transformation $B_n^{-1} V B_n$ fixes $F_2$ and sends $F_1$ into $B_n^{-1} \cdot F_3$ for every positive integer $n$.
	\item Let $M$ and $M_n$ denote the upper triangular matrices that represent $V$ and $B_n^{-1} V B_n$, respectively, with respect to the basis $\mathsf{b}$. Then for every $i,j \in \{1, \dots, d\}$ we have
	\[
	(M_n)_{i,j} = \frac{\mu_j(B_n)}{\mu_i(B_n)} M_{i,j} .
	\]
\end{enumerate}
\noindent We now have all the elements to address the desired properties. \\

\textit{Part (1).}  By properties (I) and (II), $B_n^{-1} V B_n \cdot F_1=B_n^{-1} \cdot F_3\to F_1$ if and only if $B_n^{-1} V B_n\to\id$, so it suffices to show that 
\[\frac{\mu_{i+1}(B_n)}{\mu_i(B_n)} \to 0\,\,\text{ for all $i\in\{1,\dots,d-1\}$ if and only if }\,\,B_n^{-1} V B_n\to\id.\]

First, suppose that $\frac{\mu_{i+1}(B_n)}{\mu_i(B_n)} \to 0$ for all integers $i\in\{1,\dots,d-1\}$. It follows that $\frac{\mu_j(B_n)}{\mu_i(B_n)} \to 0$ for all $j>i$.
By property (III), this implies that $(M_n)_{i,j}\to 0$ for all $j>i$. Since $M_n$ is upper triangular, it follows that $B_n^{-1} V B_n\to\id$.

Conversely, suppose that $B_n^{-1} V B_n\to\id$. Then property (III) implies that 
\[\frac{\mu_{i+1}(B_n)}{\mu_i(B_n)} M_{i,i+1}=(M_n)_{i,i+1}\to 0\]
for all integers $i\in\{1,\dots,d-1\}$. Since $(F_1,F_2,F_3)$ satisfy property-$\star$, $M_{i,i + 1} \neq 0$ for all integers $i\in\{1,\dots,d-1\}$ (see Remark~\ref{rem: star reformulation}), so
$$ \frac{\mu_{i+1}(B_n)}{\mu_i(B_n)} \to 0$$
for all integers $i \in \{1, \dots, d-1\}$. \\

\textit{Part (2).} 
Define a norm $\norm{\cdot}_{\mathsf{b}}'$ on $\End(\mathbb{C}^d)$ by
\[
\norm{A}_{\mathsf{b}}':= \sum_{i,j} \abs{M(A)_{i,j}}
\]
where $M(A)$ denotes the matrix that represents $A$ in the basis $\mathsf{b}$. Fix some compact neighborhood $\mathcal{W}$ of $\mathrm{id} \in \mathsf{SL}_d(\mathbb{C})$ and some Riemannian distance $d_\mathcal{F}$ on $\mathcal{F}(\mathbb{C}^d)$. By property (I), there exists $C_0 > 1$ such that
\begin{equation}\label{eq:bbla}
	C_0^{-1} \norm{U - \mathrm{id}}_{\mathsf{b}}' \leq d_{\mathcal{F}}(U \cdot F_1, F_1) \leq C_0 \, \norm{U - \mathrm{id}}_{\mathsf{b}}'
\end{equation}
for all $U \in \mathcal{W}$.

Let $(B_n)_n$ be a sequence in $\mathrm{Stab}({\bf E})$ such that $B_n^{-1}\cdot F_3 \to F_1$. Then properties (I) and (II) imply that for all sufficiently large $n$, the unipotent element $B_n^{-1} V B_n$ lies in $\mathcal{W}$, and so equation \eqref{eq:bbla} specializes to
\begin{equation}\label{eq:bla}
	C_0^{-1} \norm{B_n^{-1} V B_n - \mathrm{id}}_{\mathsf{b}}' \leq d_{\mathcal{F}}(B_n^{-1} \cdot F_3, F_1) \leq C_0 \, \norm{B_n^{-1} V B_n - \mathrm{id}}_{\mathsf{b}}'.
\end{equation}

Since $(F_1,F_2,F_3)$ satisfy property-$\star$, we see by Remark~\ref{rem: star reformulation} that
\[C_1:= \min_{i\in\{1,\dots,d-1\}} \abs{M_{i,i+1}} > 0.\] 
Thus, for all sufficiently large $n$, we have
\begin{align}\label{eqn: blabla}
	\max_{i = 1, \dots, d-1}\abs{\frac{\mu_{i+1}(B_n)}{\mu_i(B_n)}}=\max_{i = 1, \dots, d-1}\abs{\frac{(M_n)_{i,i+1}}{M_{i,i+1}}}  &\leq C_1^{-1} \, \norm{B_n^{-1} V B_n - \mathrm{id}}_{\mathsf{b}}'
\end{align}
where the equality follows from property (III), and the inequality is a consequence of the observation that 
\[\norm{B_n^{-1} V B_n - \mathrm{id}}_{\mathsf{b}}'=\sum_{j>i}\abs{(M_n)_{i,j}}.\]

By part (1), $\abs{\frac{\mu_{i+1}(B_n)}{\mu_i(B_n)}}<1$ for all integers $i\in\{1,\dots,d-1\}$ when $n$ is sufficiently large. Together with property (III), this implies that
\begin{align} \label{eq:cacca}
	\norm{B_n^{-1} V B_n - \mathrm{id}}_{\mathsf{b}}'&\leq\sum_{j>i}\abs{\frac{\mu_{i+1}(B_n)}{\mu_i(B_n)}\,M_{i,j}}\leq \norm{V}_{\mathsf{b}}' \left( \max_{i = 1, \dots, d-1} \abs{\frac{\mu_{i+1}(B_n)}{\mu_i(B_n)}} \right)
\end{align}
for all sufficiently large $n$. 

Set $C:= \max\{C_0 \norm{V}_{\mathsf{b}}', C_0 C_1^{-1} \}$, and observe that inequalities \eqref{eq:bla}, \eqref{eqn: blabla}, and \eqref{eq:cacca} together give the required bounds.
\end{proof}

Using Lemma \ref{lem:some linear algebra}, we can now prove Proposition \ref{prop:bending productable}.

\begin{proof}[Proof of Proposition \ref{prop:bending productable}]
Fix a leaf $h$ of $\wt\lambda$. We first set up some notation that we will use in the rest of this proof.
\begin{itemize}
	\item For any plaque $T \in P(g_{\bf x}^h,h)$, denote
	\[{\bf F}_T=(F_{T,1},F_{T,2},F_{T,3}):= \Sigma^0(g_{\bf x}^h, g_T^{\bf x}) \cdot \xi_0(\mathbf{x}_{T}),\]
	where ${\bf x}_T=(x_{T,1},x_{T,2},x_{T,3})$ is the coherent labeling of the vertices of $T$ with respect to $(g_{\bf x}^h,h)$. To simplify notation, we will also denote 
	\[\eta_T:=\eta({\bf x}_T)\,\,\text{ and }\,\,\beta_T:=\beta(T_{\bf x},T).\] 
	\item Let ${\bf y}=(y_1,y_2,y_3)\in\wt\Delta^o$ be the labeling of the vertices of $T_{\bf x}$ such that $y_1\prec y_2\prec y_3$, and $y_1$ and $y_2$ are the endpoints of $g_{\bf x}^h$. Fix a basis 
	\[\mathsf{b}=(b_1,\dots,b_d)\] 
	of $\mathbb{C}^d$ such that $b_j\in L_j:= \xi_0(y_1)^j \cap \xi_0(y_2)^{d-j+1}$ for all $j\in\{1,\dots,d\}$. Let $\norm{\cdot}_{\mathsf{b}}$ be the operator norm on $\End(\mathbb{C}^d)$ induced by the Hermitian inner product on $\mathbb{C}^d$ for which $\mathsf{b}$ is an orthonormal basis. Let $\norm{\cdot}_{\mathsf{b}}'$ be the norm on $\End(\mathbb{C}^d)$ defined by 
	\[
	\norm{A}_{\mathsf{b}}':= \sum_{i,j} \abs{M(A)_{i,j}} ,
	\]
	where $M(A)$ denotes the matrix that represents $A$ in the basis $\mathsf{b}$.	
	\item Recall that $u_{{\bf F}}(\zeta,\upsilon)$ denotes the unipotent eruption of ${\bf F}$ by $(\zeta,\upsilon)$, see Section \ref{sec: unipotent eruption}. Let $U_{{\bf F}}(\zeta,\upsilon)\in\mathsf{SL}_d(\mathbb{C})$ be the unipotent representative of $u_{\bf F}(\zeta,\upsilon)$, and set 
	\[U_{{\bf F}}(\zeta):=U_{{\bf F}}(\zeta,{\bf 0}),\]
	where ${\bf 0}$ denotes the identity element in the abelian group $(\mathbb{C}/2\pi i \mathbb{Z})^{\mathcal{A}}$.
\end{itemize} 

We divide the rest of the proof into two main steps:
\begin{enumerate}
	\item[\it{Step 1.}] Show that if
	\begin{equation}\label{eq:supremum}
		\sup_{T \in P(g_{\bf x}^h,h)} \frac{\norm{U_{{\bf F}_T}(i\eta_T) - \mathrm{id}}_{\mathsf{b}}}{d_\mathcal{F}(F_{T,1}, F_{T,3})}<\infty,
	\end{equation}
	then the map $N_h$ is H\"older extendable. 
	\item[\it{Step 2.}] Show that condition \eqref{eq:supremum} holds. 
\end{enumerate}
We will proceed in order.\\

{\it Proof of Step 1.} 
First, observe that the identity map on $\End(\mathbb{C}^d)$ is bi-Lipschitz with respect to any pair of norms on $\End(\mathbb{C}^d)$. Thus, to show that $N_h$ is H\"older extendable, it suffices to prove that there exists constants $A,\nu>0$ such that
\[\norm{N_h(T)-\id}_{\mathsf{b}}\le A\,d_\infty(x_{T,1},x_{T,3})^\nu\]
for all $T\in P(g_{\bf x}^h,h)$.

Observe from Proposition \ref{prop: lift} that 
\begin{align}\label{eqn: junk0}
	N_h(T)=\Sigma^0(g_{\bf x}^h, g_T^{\bf x})^{-1} \, U_{{\bf F}_T}(i\eta_T,i\beta_T) \, \Sigma^0(g_{\bf x}^h, g_T^{\bf x}),
\end{align}
(recall that $\Sigma^0$ denotes the slithering map associated to the $\lambda$--limit map $\xi_0$ of $\rho_0$). Furthermore, if we set ${\bf E}_T:=(F_{T,1},F_{T,2})$, ${\bf G}_T:=(F_{T,2},F_{T,3})$, and
\[V_T:= \Sigma^0(g_{\bf x}^h, g_T^{\bf x}) \, \Sigma^0(g_T^h, g_T^{\bf x}) \, \Sigma^0(g_{\bf x}^h,g_T^{\bf x})^{-1}= \Sigma^0(g_{\bf x}^h, g_T^{\bf x}) \, \Sigma^0(g_T^h, g_{\bf x}^h)\in\mathsf{SL}_d(\mathbb{C}),\]
then $V_T$ is the unipotent linear transformation that sends ${\bf E}_T$ to ${\bf G}_T$. Thus, by the definition of unipotent eruptions, there are linear representatives $C_{{\bf E}_T}(i\beta_T)$ and $C_{{\bf G}_T}(-i\beta_T)$ in $\mathsf{SL}_d(\mathbb{C})$ of $c_{{\bf E}_T}(i\beta_T)$ and $c_{{\bf G}_T}(-i\beta_T)$ respectively, such that
\begin{align}\label{eqn: junk1}
	U_{{\bf F}_T}(i\eta_T,i\beta_T)&=C_{{\bf E}_T}(i\beta_T)\, U_{{\bf F}_T}(i\eta_T)\, C_{{\bf G}_T}(-i\beta_T)\\
	&=C_{{\bf E}_T}(i\beta_T)\, U_{{\bf F}_T}(i\eta_T)\, V_T\,C_{{\bf E}_T}(i\beta_T)^{-1}V_T^{-1}.\nonumber
\end{align}

Since $ Q(g_{\bf x}^h,h)$ is compact and $\Sigma^0$ is continuous,
\[\{\Sigma^0(g, g') \mid g, g' \in Q(g_{\bf x}^h,h)\} \subset \mathsf{SL}_d(\mathbb{C})\] 
is compact, so there is some constant $M>0$ such that
$$\norm{\Sigma^0(g, g')}_{\mathsf{b}}\leq M$$
for all leaves $g,g' \in Q(g_{\bf x}^h,h)$. Also, 
\begin{align*}
	\norm{C_{{\bf E}_T}(i\beta_T)}_{\mathsf{b}}=\norm{C_{{\bf E}_T}(i\beta_T)^{-1}}_{\mathsf{b}} = 1
\end{align*}
because $C_{{\bf E}_T}(i\beta_T)$ has eigenvalues of unit modulus, and preserves the line decomposition $L_1,\dots,L_d$. Thus, identities \eqref{eqn: junk0} and \eqref{eqn: junk1} imply that     	
\begin{align}\label{eqn: junk2}
	\norm{N_h(T) - \mathrm{id}}_{\mathsf{b}}&\leq M^2\norm{U_{{\bf F}_T}(i\eta_T,i\beta_T) - \mathrm{id}}_{\mathsf{b}}\nonumber \\
	& \leq  M^2 \left(\norm{C_{{\bf E}_T}(i\beta_T) \, (U_{{\bf F}_T}(i\eta_T) - \mathrm{id}) \, V_T \, C_{{\bf E}_T}(i\beta_T)^{-1} V_T^{-1}}_{\mathsf{b}}\right. \nonumber \\
	& \quad \left.+ \norm{C_{{\bf E}_T}(i\beta_T) \, (V_T - \mathrm{id}) \, C_{{\bf E}_T}(i\beta_T)^{-1} V_T^{-1}}_{\mathsf{b}}  + \norm{V_T^{-1} - \mathrm{id}}_{\mathsf{b}}\right) \\
	& \leq M^6 \norm{U_{{\bf F}_T}(i\eta_T) - \mathrm{id}}_{\mathsf{b}} + M^4 \norm{V_T - \mathrm{id}}_{\mathsf{b}} + M^2\norm{V_T^{-1} - \mathrm{id}}_{\mathsf{b}} \nonumber
\end{align}

Since $ Q(g_{\bf x}^h,h)$ is compact and $\xi_0$ is locally $\lambda$--H\"older continuous, see Definition~\ref{def: lambda continuous}, there exist $A', \nu' > 0$ such that 
\[
d_{\mathcal{F}}(\xi_0(\mathsf{g}),\xi_0(\mathsf{g}')) \leq A' \, d_\infty(\mathsf{g},\mathsf{g}')^{\nu'}
\]
for every pair of leaves $\mathsf{g}$, $\mathsf{g}'$ of $\widetilde\lambda$ that are oriented in parallel and separates the pair of leaves $g_{\bf x}^h$ and $h$. 
Furthermore, since $\{\Sigma^0(g, g') \mid g, g' \in Q(g_{\bf x}^h,h)\}$ is compact and the $\mathsf{SL}_d(\mathbb{C})$--action on the compact flag variety $\mathcal{F}(\mathbb{C}^d)$ is smooth, there exists a constant $L > 0$ such that $\Sigma^0(g,g')$ acts as an $L$--bi-Lipschitz diffeomorphism on $\mathcal{F}(\mathbb{C}^d)$ for all leaves $g,g' \in Q(g_{\bf x}^h,h)$. Together, these imply that for every plaque $T \in P(g_{\bf x}^h,h)$
\begin{align}\label{eq:bound1}
	\begin{split}
		d_\mathcal{F}(F_{T,1},F_{T,3})
		&= d_\mathcal{F}(\Sigma^0(g_{\bf x}^h, g_T^{\bf x}) \cdot \xi_0(\mathsf{g}_T^{\bf x}), \Sigma^0(g_{\bf x}^h, g_T^{\bf x}) \cdot \xi_0(\mathsf{g}_T^h)) \\
		& \leq L \, d_{\mathcal{F}}(\xi_0(\mathsf{g}_T^{\bf x}),\xi_0(\mathsf{g}_T^h))\\
		&\leq LA'\,d_\infty(x_{T,1},x_{T,3})^{\nu'} ,
	\end{split}
\end{align}
where $\mathsf{g}_T^{\bf x}$ and $\mathsf{g}_T^h$ respectively denote the edges $g_T^{\bf x}$ and $g_T^h$ of $T$ equipped with the right-to-left orientation with respect to a path with backward and forward endpoints in $g_{\bf x}^h$ and $h$ respectively. It follows from assumption \eqref{eq:supremum} that there is some constant $M'>0$ such that
\begin{align}\label{eqn: junk4}
	\norm{U_{{\bf F}_T}(i\eta_T) - \mathrm{id}}_{\mathsf{b}}\leq M'd_\infty(x_{T,1},x_{T,3})^{\nu'}
\end{align}
for all plaques $T\in P(g_{\bf x}^h,h)$.

At the same time, since $\Sigma^0|_{Q(g_{\bf x}^h,h)}$ is H\"older continuous, by enlarging $A'$ and shrinking $\nu'$ if necessary, we may assume that 
\[\norm{\Sigma^0(\ell_1,\ell_2)-\id}_{\mathsf{b}}\leq A'd_\infty(\ell_1,\ell_2)^{\nu'}\]
for all leaves $\ell_1,\ell_2\in Q(g_{\bf x}^h,h)$. Thus, for every plaque $T \in P(g_{\bf x}^h,h)$ 
\begin{align}\label{eqn: junk5}
	\norm{V_T - \mathrm{id}}_{\mathsf{b}} \leq M^2 \norm{\Sigma^0(g_T^h, g_T^{\bf x}) - \mathrm{id}}_{\mathsf{b}}\leq M^2A'd_\infty(x_{T,1},x_{T,3})^{\nu'}.
\end{align}
Similarly, 
\begin{align}\label{eqn: junk6}
	\norm{V_T^{-1} - \mathrm{id}}_{\mathsf{b}} \leq M^2A'd_\infty(x_{T,1},x_{T,3})^{\nu'}.
\end{align} 

Together, inequalities \eqref{eqn: junk2}, \eqref{eqn: junk4}, \eqref{eqn: junk5} and \eqref{eqn: junk6} imply that
\begin{align*}
	\norm{N_h(T) - \mathrm{id}}_{\mathsf{b}}&\leq M^6 \norm{U_{{\bf F}_T}(i\eta_T) - \mathrm{id}}_{\mathsf{b}} + M^4 \norm{V_T - \mathrm{id}}_{\mathsf{b}} + M^2\norm{V_T^{-1} - \mathrm{id}}_{\mathsf{b}}\\
	&\leq A\,d_\infty(x_{T,1},x_{T,3})^{\nu},
\end{align*}
where $A:=M^6M'+(M^4+M^2)M^2A'$ and $\nu:=\nu'$. This finishes Step 1.\\

{\it Proof of Step 2.} 
Since the identity map on $\End(\mathbb{C}^d)$ is bi-Lipschitz with respect to any pair of norms on $\End(\mathbb{C}^d)$, it suffices to show that 
\[\sup_{T \in P(g_{\bf x}^h,h)} \frac{\norm{U_{{\bf F}_T}(i\eta_T) - \mathrm{id}}'_{\mathsf{b}}}{d_\mathcal{F}(F_{T,1}, F_{T,3})}<\infty.\]   
Also, we may assume that there exists a sequence of pairwise distinct plaques $(T_n)_{n=1}^\infty$ in $P(g_{\bf x}^h,h)$ such that 
\[
\lim_{n \to \infty} \frac{\norm{U_{{\bf F}_{T_n}}(\eta_{T_n}) - \mathrm{id}}'_{\mathsf{b}}}{d_\mathcal{F}(F_{T_n,1}, F_{T_n,3})} = \sup_{T \in P(g_{\bf x}^h,h)} \frac{\norm{U_{{\bf F}_T}(i\eta_T) - \mathrm{id}}'_{\mathsf{b}}}{d_\mathcal{F}(F_{T,1}, F_{T,3})}.
\]
Indeed, if no such sequence exists, then the supremum above must be realized by some plaque $T \in P(g_{\bf x}^h,h)$ and hence is finite. 

For the rest of this proof, we will simplify notation and write 
\begin{itemize}
	\item $\eta_{T_n}=i\eta({\bf x}_{T_n})$ as $\eta_n$,
	\item$ {\bf F}_{T_n}=(F_{T_n,1},F_{T_n,2},F_{T_n,3})$ as ${\bf F}_n=(F_{n,1},F_{n,2},F_{n,3})$, and
	\item ${\bf x}_{T_n}=(x_{T_n,1},x_{T_n,2},x_{T_n,3})$ as ${\bf x}_n=(x_{n,1},x_{n,2},x_{n,3})$.
\end{itemize}	

Observe that for all plaques $T\in P(g_{\bf x}^h,h)$,
\[(F_{T,1},F_{T,2})=\left\{\begin{array}{ll}
	\left(\xi_0(y_1),\xi_0(y_2)\right)&\text{if }x_{T,3}\prec x_{T,2}\prec x_{T,1},\\
	\left(\xi_0(y_2),\xi_0(y_1)\right)&\text{if }x_{T,1}\prec x_{T,2}\prec x_{T,3},
\end{array}\right.\]
so by taking a subsequence of $(T_n)_{n=1}^\infty$, we may assume that there are flags $(F_1,F_2)$ such that $(F_{n,1},F_{n,2})=(F_1,F_2)$ for all $n=1,\dots,\infty$. Let 
\[\mathsf{f}:=(f_1,\dots,f_d)\] 
be the basis of $\mathbb{C}^d$ defined by 
\[f_j=\left\{\begin{array}{ll}
	b_{d-j+1}&\text{if }(F_1,F_2)=\left(\xi_0(y_1),\xi_0(y_2)\right),\\
	b_j&\text{if }(F_1,F_2)=\left(\xi_0(y_2),\xi_0(y_1)\right)
\end{array}\right. \]
for all integers $j\in\{1,\dots,d\}$. Then observe that $f_j\in F_2^j\cap F_1^{d-j+1}$ for all integers $j\in\{1,\dots,d\}$, and that for any $A\in\mathsf{SL}_d(\mathbb{C})$, we have
\[
\norm{A}_{\mathsf{b}}' = \sum_{i,j} \abs{R(A)_{i,j}} ,
\]
where $R(A)$ is the matrix representing $A$ in the basis $\mathsf{f}$.	

Since the limit map $\xi_0$ is $\rho_0$--equivariant, the set 
\[\{{\bf F}_n\in\mathcal{F}(\mathbb{C}^d)^3\mid n\in\mathbb{Z}^+\}\]
lies in finitely many projectively equivalence classes in $\mathcal{F}(\mathbb{C}^d)^3$. Also, since the $\lambda$--cocyclic pair $(\beta,\eta)$ is $\Gamma$--invariant, the set
\[\{\eta_n\mid n\in\mathbb{Z}^+\}\] 
is finite. Thus, by taking a further subsequence of $(T_n)_{n=1}^\infty$, we can assume that $\eta_i=\eta_j$ for all positive integers $i$ and $j$. Hence, there exists $B_n \in \mathsf{SL}_d(\mathbb{C})$ such that $B_n \cdot {\bf F}_1 = {\bf F}_n$ for all positive integers $n$. By Proposition \ref{prop: lift}, this implies that
\begin{equation}\label{eq:cacca2}
	U_{{\bf F}_n}(\eta_n) = B_n U_{{\bf F}_1}(\eta_1) B_n^{-1}.
\end{equation}
Observe also that each $B_n$ is represented in the basis $\mathsf{f}$ by a diagonal matrix. Let $\mu_i(B_n)$ denote the $i$--th diagonal entry of this matrix.

Since the set of plaques
\[
\{ T \in P(g_{\bf x}^h,h) \mid d_{\mathcal{F}}(\xi_0(x_{T,1}),\xi_0(x_{T,3})) \geq \varepsilon \}
\]
is finite for all $\varepsilon > 0$, it follows that $d_{\mathcal{F}}(\xi_0(x_{n,1}),\xi_0(x_{n,3}))\to 0$. Then by inequality \eqref{eq:bound1}, $B_n\cdot F_{1,3}=F_{n,3} \to F_1$, so Lemma \ref{lem:some linear algebra} part (1) implies that
\[
\lim_{n\to\infty} \max_{i = 1, \dots, d} \abs{\frac{\mu_{i+1}(B_n)}{\mu_i(B_n)}} = 0.
\]
In particular, if $n$ is sufficiently large, then $\abs{\frac{\mu_{i+1}(B_n)}{\mu_i(B_n)}}<1$ for all integers $i\in\{1,\dots,d-1\}$.

Since $U_{{\bf F}_1}(\eta_1)$ is a unipotent element that fixes $F_2$, the matrix $R$ that represents $U_{{\bf F}_1}(\eta_1)$ in the basis $\mathsf{f}$ is upper triangular. Then by identity \eqref{eq:cacca2},
\begin{align*}
	\norm{U_{{\bf F}_n}(\eta_n) - \mathrm{id}}_{\mathsf{b}}' &= \norm{B_n\, U_{{\bf F}_1}(\eta_1)\, B_n^{-1} - \mathrm{id}}_{\mathsf{b}}' \\
	&= \sum_{j>i}\abs{R_{i,j}\frac{\mu_j(B_n)}{\mu_i(B_n)}}\\
	&\leq \sum_{j>i}\abs{R_{i,j}\frac{\mu_{i+1}(B_n)}{\mu_i(B_n)}}\\
	&\leq \norm{U_{{\bf F}_1}(\eta_1)}_{\mathsf{b}}' \left( \max_{i = 1, \dots, d} \abs{\frac{\mu_{i+1}(B_n)}{\mu_i(B_n)}} \right)
\end{align*}
for all sufficiently large $n$. Since Lemma \ref{lem:some linear algebra} part (2) implies that there is some $C>0$ such that
\[
\max_{i = 1, \dots, d} \abs{\frac{\mu_{i+1}(B_n)}{\mu_i(B_n)}} \leq C \, d_{\mathcal{F}}(F_1, F_{n,3})
\]
for all sufficiently large $n$, it follows that
\[
\sup_{T \in P(g_{\bf x}^h,h)} \frac{\norm{U_{{\bf F}_T}(i\eta_T) - \mathrm{id}}'_{\mathsf{b}}}{d_\mathcal{F}(F_{T,1}, F_{T,3})}=\lim_{n\to\infty}\frac{\norm{U_{{\bf F}_n}(\eta_n) - \mathrm{id}}_{\mathsf{b}}'}{d_\mathcal{F}(F_1, F_{n,3})} \leq C \norm{U_{{\bf F}_1}(\eta_1)}_{\mathsf{b}}'.
\]
This concludes the proof of Step 2.
\end{proof}

\subsection{Equivariance of the bent limit map} \label{sec: bending cocycle}
Our next goal is to prove Theorem~\ref{thm: bending Hitchin properties}. For the rest of this section, fix a $d$--Hitchin representation $\rho_0:\Gamma\to\mathsf{PGL}_d(\mathbb{R})\subset\mathsf{PGL}_d(\mathbb{C})$ with $\lambda$--limit map $\xi_0:\partial\wt\lambda\to\mathcal{F}(\mathbb{C}^d)$, and fix a $\lambda$--cocylic pair $(\beta, \eta)\in\mathcal{Y}_d(\lambda;\mathbb{R}/2\pi\mathbb{Z})$. For each labeling ${\bf x}\in\wt\Delta^o$, we will denote by:
\begin{itemize}
\item $\Psi_{\bf x}:= \Psi_{\rho_0, \xi_0, \beta, \eta, {\bf x}} : \widetilde{\Lambda} \to \mathsf{PGL}_d(\mathbb{C})$ the $(\beta, \eta)$--bending map for $(\rho_0,\xi_0)$ with base ${\bf x} \in \widetilde{\Delta}^o$. (This exists uniquely by Theorem \ref{thm: bending Hitchin}),
\item $\xi_{\bf x} : \partial \widetilde{\lambda} \to \mathcal{F}(\mathbb{C}^d)$ the $\Psi_{\bf x}$--bent limit map.
\end{itemize}

To prove Theorem \ref{thm: bending Hitchin properties}, it is convenient to introduce the notion of a \emph{bending cocycle}.

\begin{definition}\label{def:bending cocycle in PSLd}
The \emph{$\mathsf{PGL}_d(\mathbb{C})$--valued bending cocycle} associated to $(\rho_0,\xi_0,\beta,\eta)$ is the  map
\[
B : \widetilde{\Delta}^o \times \widetilde{\Delta}^o \longrightarrow \mathsf{PGL}_d(\mathbb{C})
\]
with the defining property that for any labeling ${\bf y}=(y_1,y_2,y_3)\in \widetilde{\Delta}^o$
\[
B({\bf x}, {\bf y}) \cdot \big(\xi_0(y_1), \xi_0(y_2), \xi_0(y_3)^1\big) = \big(\xi_{\bf x}(y_1), \xi_{\bf x}(y_2), \xi_{\bf x}(y_3)^1\big).
\]
\end{definition}

The terminology for the map $B$ was chosen as a reference to the notion of \emph{quakebend cocycle} from \cite[Section~II.3.5]{EM06}, developed in the context of classical pleated surfaces. This choice is motivated by the properties (2), (3), and (4) described in Proposition \ref{prop:properties of bending cocycle} (compare in particular with \cite[Definition~II.3.5.2]{EM06}). In order to state this result, we recall the necessary notations.

Given any labeling ${\bf y} = (y_1,y_2,y_3) \in \widetilde{\Delta}^o$, we denote: the plaque of $\wt\lambda$ with vertices $y_1, y_2, y_3$ by $T_{\bf y}$; the edge of $T_{\bf y}$ with vertices $y_1, y_2$ by $h_{\bf y}$; the edge of $T_{\bf y}$ that separates $T_{\bf y}$ and $T_{\bf x}$ by $g_{\bf y}^{\bf x}$ (whenever $T_{\bf x} \neq T_{\bf y}$). If $h_{\bf y}\neq g_{\bf y}^{\bf x}$, let ${\bf x}_{T_{\bf y}}$ be the coherent labeling of the vertices of $T_{\bf y}$ with respect to the pair of leaves $(g_{\bf y}^{\bf x},h_{\bf y})$, and recall that the element $\nu({\bf y})  \in (\mathbb{C}/2\pi i \mathbb{Z})^{\mathcal{A}}$, introduced in equation \eqref{eqn: nu}, satisfies 
\begin{align}\label{eq:recall nu}
\nu({\bf y})^{\bf i} = 
\begin{cases}
	\sum_{{\bf j} \in \mathcal{B} : j_2=i_1}\eta^{\bf j}({\bf y}) & \text{if $y_{3} \prec y_{2} \prec y_{1}$,} \\
	- \sum_{{\bf j} \in \mathcal{B} : j_2=i_2}\eta^{\bf j}({\bf y}) & \text{if $y_{1} \prec y_{2} \prec y_{3}$,}
\end{cases}
\end{align}
for any pairs ${\bf i} \in \mathcal{A}$.

\begin{proposition}\label{prop:properties of bending cocycle}
The $\mathsf{PGL}_d(\mathbb{C})$--valued bending cocycle associated to $(\rho_0,\xi_0,\beta,\eta)$ satisfies the following identities: for all labelings  ${\bf x},{\bf y}, {\bf z} \in \wt{\Delta}^o$, we have
\begin{enumerate}
	\item $B({\bf x}, {\bf y})=\Psi_{\bf x}( h_{\bf y}) \, c_{\mathsf{h}_{\bf y}}(\delta_{{\bf x},{\bf y}})$
	where
	\[
	\delta_{{\bf x}, {\bf y}} =
	\begin{cases}
		0&\text{if }T_{\bf x}=T_{\bf y},\\
		\beta(T_{\bf x},T_{\bf y}) & \text{if $T_{\bf x}\neq T_{\bf y}$ and $ h_{\bf y} = g_{\bf y}^{\bf x}$,} \\
		\beta(T_{\bf x},T_{\bf y})+\nu({\bf x}_{T_{\bf y}})& \text{if $T_{\bf x}\neq T_{\bf y}$ and $ h_{\bf y} \neq g_{\bf y}^{\bf x}$,} \\
	\end{cases}
	\]
	\item $B({\bf x}, {\bf z}) = B({\bf x}, {\bf y}) \, B({\bf y}, {\bf z})$,
	\item $B({\bf x}, {\bf y}) = B({\bf y}, {\bf x})^{-1}$,
	\item $B(\gamma {\bf x}, \gamma {\bf y}) = \rho_0(\gamma) B({\bf x}, {\bf y}) \, \rho_0(\gamma)^{-1}$ for every $\gamma \in \Gamma$.
\end{enumerate}
\end{proposition}

Delaying the proof of Proposition \ref{prop:properties of bending cocycle}, we will first use it to prove Theorem~\ref{thm: bending Hitchin properties}.

\begin{proof}[Proof of Theorem~\ref{thm: bending Hitchin properties}]
Using the $\mathsf{PGL}_d(\mathbb{C})$--valued bending cocycle associated to $(\rho_0,\xi_0,\beta, \eta)$, define the map 
\[\rho=\rho_{\bf x}:\Gamma\to\mathsf{PGL}_d(\mathbb{C})\] 
by $\rho(\gamma):=B({\bf x},\gamma\cdot{\bf x})\,\rho_0(\gamma)$. We first show that $\rho$ is a homomorphism. By Proposition \ref{prop:properties of bending cocycle} parts (2) and (4), 
\begin{align*}
	\rho(\gamma_1 \gamma_2) & = B({\bf x}, \gamma_1 \gamma_2 \cdot {\bf x})\, \rho_0(\gamma_1 \gamma_2)\\
	& = B({\bf x}, \gamma_1 \cdot {\bf x})\, B(\gamma_1 \cdot {\bf x}, \gamma_1 \gamma_2 \cdot {\bf x}) \,\rho_0(\gamma_1)\, \rho_0(\gamma_2)\\
	& = B({\bf z}, \gamma_1 \cdot {\bf z}) \rho_0(\gamma_1) B({\bf z}, \gamma_2 \cdot {\bf z}) \rho_0(\gamma_2)\\
	& = \rho(\gamma_1) \rho(\gamma_2) .
\end{align*}
for every $\gamma_1, \gamma_2 \in \Gamma$.

Next, we show that the map $\xi=\xi_{\bf x}$ is $\rho$--equivariant, i.e.
$$\xi(\gamma \cdot x) = \rho(\gamma)\cdot \xi(x)$$
for any $\gamma \in \Gamma$ and point $x \in \partial \wt \lambda$, or equivalently,
\begin{equation}\label{eq:equivariant on leaves}
	\xi(\gamma \cdot \mathsf{g}) = \rho(\gamma) \cdot \xi(\mathsf{g})
\end{equation}
for all oriented geodesics $\mathsf{g} \in \widetilde{\Lambda}^o$ and all $\gamma \in \Gamma$. Since the lamination $\widetilde{\lambda}$ is a zero-measure subset of $\widetilde{S} \cong \mathbb{H}^2$, every leaf $g\in\wt\Lambda$ can be approximated by a sequence $(g_n)_n \subset \widetilde{\Lambda}$ consisting entirely of edges of plaques of $\widetilde{\lambda}$ (see e.g. \cite[Section~I.4.1]{CEG06}). Hence, the fact that $\xi$ is locally $\lambda$--H\"older continuous (see Definition \ref{def: bending map} part (ii)), implies that it is sufficient to show that relation \eqref{eq:equivariant on leaves} holds in the case when $\mathsf{g}$ is an edge of some plaque of $\widetilde{\lambda}$.

Pick a plaque $T$ of $\wt\lambda$, let $g$ be an edge of $T$, let $\mathsf{g}$ be the edge $g$ equipped with an orientation, and let ${\bf y} = (y_1,y_2,y_3) \in \wt{\Delta}^o$ be a labelling of the vertices of $T$, such that the backward and forward endpoints of $\mathsf{g}$ are $y_1$ and $y_2$, respectively. By Proposition~\ref{prop:properties of bending cocycle} parts (1), (2), and (4), we deduce that
\begin{align*}
	\rho(\gamma) \cdot \xi(\mathsf{g}) & = B({\bf x}, \gamma \cdot {\bf x}) \, \rho_0(\gamma) \, \Psi_{\bf x}(g) \cdot \xi_0(\mathsf{g}) \\
	& = B({\bf x}, \gamma \cdot {\bf x}) \, \rho_0(\gamma) \, B({\bf x}, {\bf y}) \cdot \xi_0(\mathsf{g}) \\
	& = B({\bf x}, \gamma \cdot {\bf x}) \, B(\gamma \cdot {\bf x}, \gamma \cdot {\bf y}) \, \rho_0(\gamma)  \cdot \xi_0(\mathsf{g}) \\
	& = B({\bf x}, \gamma \cdot {\bf y}) \cdot \xi_0(\gamma \cdot \mathsf{g}) \\
	& = \Psi_{\bf x}(\gamma\cdot g) \cdot \xi_0(\gamma \cdot \mathsf{g}) \\
	& = \xi(\gamma \cdot \mathsf{g}) 
\end{align*}
for every $\gamma \in \Gamma$. (For the second and fifth equality, recall that $c_{\mathsf{g}}(\delta_{\bf x,\bf y})\cdot\xi_0(\mathsf{g})=\xi_0(\mathsf{g})$.) Thus, the map $\xi$ is $\rho$--equivariant.

Finally, we prove that $\rho$ is $\lambda$--Borel Anosov. Let 
\[
T^1\wt\lambda\times\mathbb{C}^d = L^0_1+\dots+L^0_d\,\,\text{ and }\,\, T^1\wt\lambda\times\mathbb{C}^d = L_1+\dots+L_d,
\]
be the splittings induced by $\xi_0$ and $\xi$ respectively. Then for all pairs ${\bf i}\in\mathcal{A}$, let $\wh H^0_{\bf i}$ and $\wh H_{\bf i}$ denote the ${\bf i}$--th homomorphism bundles induced by $\xi_0$ and $\xi$ respectively, see Section~\ref{ssec:BorelAnosov}. Since the bending map $\Psi=\Psi_{\bf x}$ has the property that $\xi(y) = \Psi(h)\cdot\xi_0(y)$ for all leaves $h\in\wt\Lambda$ that have $y$ as an endpoint, it holds that
\[
\Psi(h)\cdot L_i^0|_v=L_i|_v
\]
for all integers $i\in\{1,\dots,d\}$ and $v \in T^1 \wt\lambda$, where $h = h_v$ is the leaf of $\wt\lambda$ that $v$ is tangent to. Thus, for all pairs ${\bf i}=(i_1,i_2)\in\mathcal{A}$, we may define the isomorphism of line bundles
\[
j_{\bf i}: H_{\bf i}:= \Hom(L_{i_1+1},L_{i_1}) \to H_{\bf i}^0:= \Hom(L_{i_1+1}^0,L_{i_1}^0) ,
\]
given by $$j_{\bf i}(f):= \Psi(g)^{-1} \circ f \circ \Psi(g).$$
By construction, the map $j_{\bf i}$ intertwines the action of the lift of the geodesic flow $\varphi_t$ on $H_{\bf i}$ and $H_{\bf i}^0$. However, in general, it does not intertwine the $\Gamma$--actions on $H_{\bf i}$ and $H_{\bf i}^0$.

Fix a pair ${\bf i}=(i_1,i_2) \in \mathcal{A}$, and choose a continuous family of Hermitian norms $(\norm{\cdot}_{0,v})_{v\in T^1\lambda}$ on the fibers of $\wh H_{\bf i}^0$. This lifts to a continuous family of Hermitian norms $(\norm{\cdot}_{0,v})_{v\in T^1\wt\lambda}$ on the fibers of $H_{\bf i}^0$, which then pulls-back via $j_{\bf i}$ to a continuous family of Hermitian norms $(\norm{\cdot}_{v})_{v\in T^1\wt\lambda}$ on the fibers of $H_{\bf i}$, i.e. 
\[
\norm{f}_v:= \norm{j_{\bf i}(f)}_{0,v}
\]
for every $v \in T^1 \widetilde{\lambda}$ and $f \in H_{\bf i}\big|_v$. Despite the fact that $j_{\bf i}$ does not necessarily intertwine the actions of $\Gamma$ on $H_{\bf i}$ and $H_{\bf i}^0$, the following lemma holds.

\begin{lemma}\label{norm-equal}
	For all $\gamma\in\Gamma$, $v\in T^1\wt\lambda$, and $f\in H_{\bf i}\big|_v$ we have
	\[\norm{\gamma\cdot f}_{\gamma\cdot v}=\norm{f}_v.\]
	In particular, $(\norm{\cdot}_v)_{v \in T^1 \widetilde{\lambda}}$ descends to a continuous family of Hermitian norms $(\norm{\cdot}_v)_{v \in T^1\lambda}$ on the fibers of the line bundle $\widehat{H}_{\bf i} \to T^1 \lambda$. 
\end{lemma}

\begin{proof}
	Since the leaves of $\widetilde{\lambda}$ that are edges of plaques are dense inside $\wt\Lambda$, it is sufficient to verify the statement when $v$ is tangent to the edge of some plaque of $\wt\lambda$. Pick any such $v$, and let ${\bf y} = (y_1, y_2, y_3) \in \wt\Delta^o$ be a labeling so that $v$ is tangent to $h_{\bf y}$, which is the leaf whose endpoints are $y_1$ and $y_2$. Proposition \ref{prop:properties of bending cocycle} part (1) then implies
	\begin{align}\label{eq:Bs}
		B({\bf x}, {\bf y}) & = \Psi(h_{\bf y}) \, c_{\mathsf{h}_{\bf y}}(\delta_{{\bf x},{\bf y}}) \,\,\text{ and }\,\, B({\bf x}, \gamma \cdot {\bf y}) = \Psi(\gamma\cdot h_{\bf y}) \, c_{\gamma\cdot \mathsf{h}_{\bf y}}(\delta_{{\bf x},\gamma\cdot {\bf y}}).
	\end{align}
	
	Observe that  every linear representative in $\mathsf{SL}_d(\mathbb{C})$ of $c_{\mathsf{h}_{\bf y}}(\delta_{{\bf x},{\bf y}})$ (respectively, $c_{\gamma\cdot \mathsf{h}_{\bf y}}(\delta_{{\bf x},\gamma\cdot {\bf y}})$) acts on the line $L_{i_1}^0\big|_v$ (respectively, $L_{i_1}^0\big|_{\gamma\cdot v}$) by scaling by a complex number of modulus $1$. Since $\norm{\cdot}_{0,v}$ and $\norm{\cdot}_{0,\gamma\cdot v}$ are Hermitian norms, 
	equation \eqref{eq:Bs} implies that
	\begin{align}\label{eq:norm is preserved by cs}
		\norm{f}_{v} = \norm{j_{\bf i}(f)}_{0,v} & = \norm{B({\bf x},{\bf y})^{-1} \circ f \circ B({\bf x},{\bf y})}_{0,v} , 
	\end{align}
	for any $f \in H_{\bf i}\big|_v$, and
	\begin{align}\label{eq:norm is preserved by cs1}
		\norm{f'}_{\gamma \cdot v} = \norm{j_{\bf i}(f')}_{0,\gamma \cdot v} & = \norm{B({\bf x}, \gamma \cdot {\bf y})^{-1} \circ f' \circ B({\bf x}, \gamma \cdot {\bf y})}_{0,\gamma \cdot v} ,
	\end{align}
	for any $f' \in H_{\bf i}\big|_{\gamma \cdot v}$. 
	
	Thus, for every $f \in H_{\bf i}\big|_v$ and $\gamma \in \Gamma$, the following identities hold:
	\begin{align*}
		\norm{\gamma \cdot f}_{\gamma \cdot v} & = \norm{\rho(\gamma) \circ f \circ \rho(\gamma)^{-1}}_{\gamma \cdot v} \\
		& = \norm{B({\bf x},\gamma \cdot {\bf x}) \, \rho_0(\gamma) \circ f \circ \rho_0(\gamma)^{-1} B({\bf x},\gamma \cdot {\bf x})^{-1}}_{\gamma \cdot v} \\
		& = \norm{B({\bf x}, \gamma \cdot {\bf y})^{-1} B({\bf x},\gamma \cdot {\bf x}) \, \rho_0(\gamma) \circ f \circ \rho_0(\gamma)^{-1} B({\bf x},\gamma \cdot {\bf x})^{-1} B({\bf x}, \gamma \cdot {\bf y})}_{0, \gamma \cdot v} \\
		& = \norm{B(\gamma \cdot {\bf x}, \gamma \cdot {\bf y})^{-1} \, \rho_0(\gamma) \circ f \circ \rho_0(\gamma)^{-1} B(\gamma \cdot {\bf x}, \gamma \cdot {\bf y})}_{0, \gamma \cdot v} \\
		& = \norm{\rho_0(\gamma) \, B({\bf x}, {\bf y})^{-1} \circ f \circ B({\bf x}, {\bf y}) \, \rho_0(\gamma)^{-1}}_{0, \gamma \cdot v} \\
		& = \norm{ B({\bf x}, {\bf y})^{-1} \circ f \circ B({\bf x}, {\bf y}) }_{0,v} = \norm{f}_{v} .
	\end{align*}
	Here, the third equality holds by relation \eqref{eq:norm is preserved by cs1}; the fourth equality holds by Proposition \ref{prop:properties of bending cocycle} parts (2) and (3), the fifth equality holds by Proposition \ref{prop:properties of bending cocycle} part (4), the sixth equality holds by the invariance of $\norm{\cdot}_0$ with respect to the $\Gamma$--action induced by $\rho_0$, and the last equality holds by relation \eqref{eq:norm is preserved by cs}.
\end{proof}

By Theorem \ref{thm: Labourie}, $\rho_0$ is $\lambda$--Borel Anosov, so there exist $C,\alpha>0$ such that 
\begin{equation}\label{eq:rho0 is anosov}
	\|\varphi_t(f)\|_{0,\varphi_t(v)}\leq Ce^{-\alpha t}\|f\|_{0,v}
\end{equation}
for all $f\in\wh H_{\bf i}^0\big|_v$, $v\in T^1\lambda$, and $t \geq 0$. Here, recall the usual abuse of notation: $\varphi_t$ denotes the flow on $T^1\lambda$, as well as the induced flows on $\wh H^0_{\bf i}$, and $\wh H_{\bf i}$. Then
\begin{align*}
	\norm{\varphi_t(f)}_{\varphi_t(v)} & = \norm{j_{\bf i}(\varphi_t(f))}_{0,\varphi_t(v)} = \norm{\varphi_t(j_{\bf i}(f))}_{0,\varphi_t(v)}  \leq Ce^{-\alpha t}\norm{j_{\bf i}(f)}_{0,v} = Ce^{-\alpha t}\norm{f}_v
\end{align*}
for all $v\in T^1\lambda$, $f\in H_{\bf i}$ and $t \geq 0$, where the second equality holds because $j_{\bf i}$ intertwines the actions of the flows on $\wh H^0_{\bf i}$ and $\wh H_{\bf i}$, and the inequality is due to relation \eqref{eq:rho0 is anosov}. This proves that $\rho$ is also $\lambda$--Borel Anosov, and thus proves the first statement of the theorem.
\end{proof}

It remains to prove Proposition \ref{prop:properties of bending cocycle}.

\begin{proof}[Proof of Proposition \ref{prop:properties of bending cocycle}]
{\it Part (1).} We need to verify that
\[\big(\xi_{\bf x}(y_1),\xi_{\bf x}(y_2),\xi_{\bf x}(y_3)^1\big)=\Psi_{\bf x}(h_{\bf y})\,c_{\mathsf{h}_{\bf y}}(\delta_{{\bf x},{\bf y}})\cdot\big(\xi_0(y_1),\xi_0(y_2),\xi_0(y_3)^1\big).\]
Since $c_{\mathsf{h}_{\bf y}}(\delta_{{\bf x},{\bf y}})$ fixes both $\xi_0(y_1)$ and $\xi_0(y_2)$, it follows that
\[\Psi_{\bf x}(h_{\bf y})\,c_{\mathsf{h}_{\bf y}}(\delta_{{\bf x},{\bf y}})\cdot\big(\xi_0(y_1),\xi_0(y_2)\big)=\Psi_{\bf x}(h_{\bf y})\cdot\big(\xi_0(y_1),\xi_0(y_2)\big)=\big(\xi_{\bf x}(y_1),\xi_{\bf x}(y_2)\big).\]
Thus, it remains to show that $\xi_{\bf x}(y_3)^1=\Psi_{\bf x}(h_{\bf y})\,c_{\mathsf{h}_{\bf y}}(\delta_{{\bf x},{\bf y}})\cdot\xi_0(y_3)^1$.

First, suppose that $T_{\bf x}=T_{\bf y}$. Then by definition, $c_{\mathsf{h}_{\bf y}}(\delta_{{\bf x},{\bf y}})=\id$, so
\[\xi_{\bf x}(y_3)^1=\xi_0(y_3)^1=v({\bf x},h_{\bf y})\cdot\xi_0(y_3)^1=\Psi_{\bf x}(h_{\bf y})\,c_{\mathsf{h}_{\bf y}}(\delta_{{\bf x},{\bf y}})\cdot\xi_0(y_3)^1,\]
where $v({\bf x},h_{\bf y})$ was defined by equation \eqref{eq: definition of v}.

Next, suppose that $T_{\bf x}\neq T_{\bf y}$ and $h_{\bf y}=g_{\bf y}^{\bf x}$. Let $h$ be an edge of the plaque $T_{\bf y}$ that is not $h_{\bf y}$. Then let $\mathsf{h}$ and $\mathsf{h}_{\bf y}$ respectively denote the right-to-left orientation on $h$ and $h_{\bf y}$ with respect to $(g_{\bf x}^{\bf y},h)$, and let ${\bf x}_{T_{\bf y}}$ be the coherent labeling of the vertices of the plaque $T_{\bf y}$ with respect to $(g_{\bf x}^{\bf y},h)$. By equation \eqref{eq: bending map on triangles}, we have
\begin{align*}
	\Psi_{\bf x}(h)=\Psi_{\bf x}(h_{\bf y})\,N_h(T_{\bf y})=\Psi_{\bf x}(h_{\bf y})\,c_{\mathsf{g}_{\bf y}^{\bf x}}(\delta_{{\bf x},{\bf y}})\,a_{{\bf x}_{T_{\bf y}}}(\eta({\bf x}_{T_{\bf y}}))\,c_{\mathsf{h}}(-\delta_{{\bf x},{\bf y}}-\nu({\bf x}_{T_{\bf y}})).
\end{align*}
Since $y_3$ is an endpoint of $h$, it follows that $a_{{\bf x}_{T_{\bf y}}}(\eta({\bf x}_{T_{\bf y}}))$ and $c_{\mathsf{h}}(-\delta_{{\bf x},{\bf y}}-\nu({\bf x}_{T_{\bf y}}))$ both fix $\xi_0(y_3)^1$, so
\begin{align*}
	\xi_{\bf x}(y_3)^1=\Psi_{\bf x}(h)\cdot\xi_0(y_3)^1=\Psi_{\bf x}(h_{\bf y})\,c_{\mathsf{h}_{\bf y}}(\delta_{{\bf x},{\bf y}})\cdot\xi_0(y_3)^1.
\end{align*}

Finally, suppose that $T_{\bf x}\neq T_{\bf y}$ and $h_{\bf y}\neq g_{\bf y}^{\bf x}$. Then let $\mathsf{h}_{\bf y}$ and $\mathsf{g}_{\bf y}^{\bf x}$ respectively denote the right-to-left orientation on $h_{\bf y}$ and $g_{\bf y}^{\bf x}$ with respect to $(g_{\bf x}^{\bf y},h_{\bf y})$, and let ${\bf x}_{T_{\bf y}}$ be the coherent labeling of the vertices of the plaque $T_{\bf y}$ with respect to $(g_{\bf x}^{\bf y},h_{\bf y})$. By equation \eqref{eq: bending map on triangles}, we have 
\begin{align*}
	\Psi_{\bf x}(g_{\bf y}^{\bf x})=\Psi_{\bf x}(h_{\bf y})N_{h_{\bf y}}(T_{\bf y})^{-1}=\Psi_{\bf x}(h_{\bf y})\,c_{\mathsf{h}_{\bf y}}(\delta_{{\bf x},{\bf y}})\,a_{{\bf x}_{T_{\bf y}}}(\eta({\bf x}_{T_{\bf y}}))^{-1}\,c_{\mathsf{g}_{\bf y}^{\bf x}}(-\delta_{{\bf x},{\bf y}}+\nu({\bf x}_{T_{\bf y}})).
\end{align*}
Since $y_3$ is an endpoint of the edge $g_{\bf y}^{\bf x}$, it follows that $a_{{\bf x}_{T_{\bf y}}}(\eta({\bf x}_{T_{\bf y}}))$ and $c_{\mathsf{g}_{\bf y}^{\bf x}}(-\delta_{{\bf x},{\bf y}}+\nu({\bf x}_{T_{\bf y}}))$ both fix $\xi_0(y_3)^1$, so
\begin{align*}
	\xi_{\bf x}(y_3)^1=\Psi_{\bf x}(g_{\bf y}^{\bf x})\cdot\xi_0(y_3)^1=\Psi_{\bf x}(h_{\bf y})\,c_{\mathsf{h}_{\bf y}}(\delta_{{\bf x},{\bf y}})\cdot\xi_0(y_3)^1.
\end{align*}

{\it Part (2).} Fix labelings ${\bf x} , {\bf y} \in \widetilde{\Delta}^o$, let $\xi':=B({\bf x},{\bf y})^{-1}\cdot\xi_{\bf x}:\partial\wt\lambda\to\mathcal{F}(\mathbb{C}^d)$, and let $\Sigma^{\bf x}$ denote the slithering map compatible with $\xi_{\bf x}$. Then observe that $\xi'$ is locally $\lambda$--H\"older continuous, $\lambda$--transverse, and $\lambda$--hyperconvex, and the map given by
\[(h_1,h_2)\mapsto B({\bf x},{\bf y})^{-1}\,\Sigma^{\bf x}(h_1,h_2)\,B({\bf x},{\bf y})\] 
is the slithering map compatible with $\xi'$. From this, we may deduce that
\[\alpha_{\xi'}=\alpha_{\xi_{\bf x}}=\alpha_{\xi_{\bf y}}\,\,\text{ and }\,\,\theta_{\xi'}=\theta_{\xi_{\bf x}}=\theta_{\xi_{\bf y}}.\]
Furthermore, observe that 
\begin{align*}
	\big(\xi'(y_1),\xi'(y_2),\xi'(y_3)^1\big)&=B({\bf x},{\bf y})^{-1}\cdot \big(\xi_{\bf x}(y_1),\xi_{\bf x}(y_2),\xi_{\bf x}(y_3)^1\big)\\
	&=\big(\xi_0(y_1),\xi_0(y_2),\xi_0(y_3)^1\big)\\
	&=\big(\xi_{\bf y}(y_1),\xi_{\bf y}(y_2),\xi_{\bf y}(y_3)^1\big),
\end{align*}
so Lemma \ref{lem: injxi} implies that $\xi'=\xi_{\bf y}$. Therefore,
\begin{align*}
	B({\bf x},{\bf y})^{-1}B({\bf x},{\bf z})\cdot\big(\xi_0(z_1),\xi_0(z_2),\xi_0(z_3)^1\big)&=B({\bf x},{\bf y})^{-1}\cdot\big(\xi_{\bf x}(z_1),\xi_{\bf x}(z_2),\xi_{\bf x}(z_3)^1\big)\\
	&=\big(\xi'(z_1),\xi'(z_2),\xi'(z_3)^1\big)\\
	&=\big(\xi_{\bf y}(z_1),\xi_{\bf y}(z_2),\xi_{\bf y}(z_3)^1\big)\\
	&=B({\bf y},{\bf z})\cdot \big(\xi_0(z_1),\xi_0(z_2),\xi_0(z_3)^1\big),
\end{align*}
which implies that $B({\bf x},{\bf z})=B({\bf x},{\bf y})\, B({\bf y},{\bf z})$.

{\it Part (3).} This is immediate from (2) and the observation that $B({\bf x},{\bf x})=\id$ for all labelings ${\bf x}\in\wt\Delta^o$.

{\it Part (4).} Fix $\gamma \in \Gamma$ and let $\xi':=\rho_0(\gamma)\cdot \xi_{\bf x}:\partial\wt\lambda\to\mathcal{F}(\mathbb{C}^d)$, and let $\xi'':=\xi_{\gamma\cdot{\bf x}}\circ\gamma:\partial\wt\lambda\to\mathcal{F}(\mathbb{C}^d)$. Observe that both $\xi'$ and $\xi''$ are locally $\lambda$--H\"older continuous, $\lambda$--transverse, and $\lambda$--hyperconvex. Furthermore, the map given by
\[(h_1,h_2)\mapsto \rho_0(\gamma)\,\Sigma^{\bf x}(h_1,h_2)\,\rho_0(\gamma)^{-1}\] 
is the slithering map compatible with $\xi'$, while the map given by
\[(h_1,h_2)\mapsto \Sigma^{\gamma\cdot \bf x}(\gamma\cdot h_1,\gamma\cdot h_2)\] 
is the slithering map compatible with $\xi''$. This, combined with the $\Gamma$--invariance of $(\alpha,\theta)$, imply that
\[\alpha_{\xi'}=\alpha_{\xi_{\bf x}}=\alpha_{\xi_{\gamma\cdot{\bf x}}}=\alpha_{\xi''}\,\,\text{ and }\,\,\theta_{\xi'}=\theta_{\xi_{\bf x}}=\theta_{\xi_{\gamma\cdot {\bf x}}}=\theta_{\xi''}.\]
Furthermore, 
\begin{align*}
	\big(\xi'(x_1),\xi'(x_2),\xi'(x_3)^1\big)&=\rho_0(\gamma)\cdot\big(\xi_{\bf x}(x_1),\xi_{\bf x}(x_2),\xi_{\bf x}(x_3)^1\big)\\
	&=\rho_0(\gamma)\cdot\big(\xi_0(x_1),\xi_0(x_2),\xi_0(x_3)^1\big) \\
	&=\big(\xi_0(\gamma\cdot x_1),\xi_0(\gamma\cdot x_2),\xi_0(\gamma\cdot x_3)^1\big) \\
	&=\big(\xi_{\gamma\cdot{\bf x}}(\gamma\cdot x_1),\xi_{\gamma\cdot{\bf x}}(\gamma\cdot x_2),\xi_{\gamma\cdot{\bf x}}(\gamma\cdot x_3)^1\big)\\
	&=\big(\xi''(x_1),\xi''(x_2),\xi''(x_3)^1\big), 
\end{align*}
so Lemma \ref{lem: injxi} implies that $\xi'=\xi''$. As such,
\begin{align*}
	\rho_0(\gamma)\, B({\bf x}, {\bf y})\cdot\big(\xi_0(y_1),\xi_0(y_2),\xi_0(y_3)^1\big)&=\rho_0(\gamma)\cdot\big(\xi_{\bf x}(y_1),\xi_{\bf x}(y_2),\xi_{\bf x}(y_3)^1\big)\\
	&=\big(\xi'(y_1),\xi'(y_2),\xi'(y_3)^1\big)\\
	&=\big(\xi''(y_1),\xi''(y_2),\xi''(y_3)^1\big)\\
	&=\big(\xi_{\gamma\cdot{\bf x}}(\gamma\cdot y_1),\xi_{\gamma\cdot{\bf x}}(\gamma\cdot y_2),\xi_{\gamma\cdot{\bf x}}(\gamma\cdot y_3)^1\big)\\
	&=B(\gamma\cdot{\bf x},\gamma\cdot{\bf y})\cdot\big(\xi_0(\gamma\cdot y_1),\xi_0(\gamma\cdot y_2),\xi_0(\gamma\cdot y_3)^1\big)\\
	&=B(\gamma\cdot{\bf x},\gamma\cdot{\bf y})\rho_0(\gamma)\cdot\big(\xi_0(y_1),\xi_0(y_2),\xi_0(y_3)^1\big),
\end{align*}
which implies that $\rho_0(\gamma)\, B({\bf x}, {\bf y})\,\rho_0(\gamma)^{-1}=B(\gamma\cdot{\bf x},\gamma\cdot{\bf y})$.
\end{proof}

\section{Homological interpretation of \texorpdfstring{$\lambda$}{l}--cocylic pairs}\label{sec: intersection}

The main goal of this section is to define the natural complex structure (see Section \ref{sec: complex structure}) on the set $\mathcal{Y}_d(\lambda;\mathbb{C}/2\pi i\mathbb{Z})$ of $\mathbb{C}/2\pi i\mathbb{Z}$--valued $\lambda$--cocyclic pairs. This is necessary for our discussion of the holomorphicity of the shear-bend map in Section~\ref{sec:holomorphicity}. To do so, we require a homological interpretation of $\lambda$--cocyclic pairs, which was first described by Bonahon-Dreyer \cite{BoD} over $\mathbb{R}$ (see Section \ref{sub:complex}). This homological interpretation also allows us to apply the usual intersection pairing on homology to $\lambda$--cocyclic pairs, which is used in Section~\ref{length} to characterize the image of the shear-bend map. In Sections~\ref{sec: train tracks}, \ref{sec: homology background} and \ref{sec:representatives} we collect necessary preliminaries on train track neighborhoods and the homology of train track neighborhoods relative their boundaries.

\subsection{Train-track neighborhoods}\label{sec: train tracks}
We now recall the notion of a train track neighborhood for the fixed maximal geodesic lamination $\lambda$, and related terminology, following the work of Bonahon and Dreyer \cite[Section~4.2]{BoD}. Other classical references on the subject (which use slightly different definitions) include Thurston \cite[Section~8.9]{thurston-notes} and Penner-Harer \cite{penner-harer}.

Let $r\co[0,1]\times[0,1]\to S$ (respectively, $r\co[0,1]\times[0,1]\to \wt S$) be the restriction of a smooth embedding from a neighborhood of $[0,1]\times [0,1]$ in $\mathbb{R}^2$ to $S$ (respectively $\wt S$). We refer to the image $R$ of $r$ as a \emph{rectangle} in $S$ (respectively, $\wt S$). The boundary $\partial R$ of the rectangle $R$ can be divided into the \emph{horizontal boundary} $\partial_hR:=r([0,1]\times\{0,1\})$ and the {\em vertical boundary} $\partial_vR:=r(\{0,1\}\times[0,1])$. A \emph{tie} of $R$ is a subset of the form $r(\{x\}\times[0,1])$ for some $x\in [0,1]$. The points
\[r(0,0),\, r(0,1),\, r(1,0),\text{ and }r(1,1)\]
are called the \emph{vertices} of $R$.

A {\em(trivalent) train track neighborhood} for $\lambda$ is a closed neighborhood $N\subset S$ of $\lambda$ which can be written as a union of finitely many rectangles $\{R_1,\dots,R_n\}$, such that the following conditions are satisfied:
\begin{enumerate}
	\item If two rectangles $R_i$ and $R_j$ intersect, then, up to switching the roles of $i$ and $j$, every component of $R_i\cap R_j$ is a vertical boundary component of $R_i$, lies in a vertical boundary component of $R_j$, and contains exactly one vertex of $R_j$. Every vertical component of a rectangle $R_j$ that satisfies the properties above (with respect to some rectangle $R_i$) is called a \emph{switch} of $N$. 
	\item For each rectangle $R_i$, every vertex of $R_i$ is contained in some $R_j$ different from $R_i$.
	\item For each rectangle $R_i$, every tie of $R_i$ intersects some leaf of $\lambda$, and any such intersection is transverse.
	\item Each component of $S-N$ is a topological cell, and its boundary is the union of six smooth curves, three of which each lie in some switch, and the other three are each a union of horizontal boundary components of rectangles. The union of all (closed) segments in $\partial N$ that satisfy the former (respectively, latter) is the \emph{vertical boundary} (respectively, \emph{horizontal boundary}) of $N$, which we denote by $\partial_h N$ (respectively, $\partial_v N$). 
\end{enumerate}

Observe that every switch of the train track neighborhood $N$ contains exactly three vertical boundary components of rectangles in $N$. Of these three, the switch is equal to exactly one vertical boundary component, and properly contains the other two, which are disjoint. 

Any train track neighborhood $N$ of $\lambda$ lifts to a closed neighborhood $\wt N\subset \wt S$ of $\wt\lambda$, which we refer to as a $\Gamma$--invariant \emph{train track neighborhood} of $\wt\lambda$. Clearly, $\wt N$ can be written as countable union of rectangles, each of which is a lift of one of the rectangles of $\wt N$. Also, the boundary of $\wt N$ can be written as the union of the \emph{vertical boundary} $\partial_v\wt N$ and \emph{horizontal boundary} $\partial_h\wt N$, which descend via the covering map to $\partial_v N$ and $\partial_h N$ respectively.

\subsubsection{Orientation covers}\label{or-cov-tt}

If $N$ is a train track neighborhood of the geodesic lamination $\lambda$ in $S$, note that the ties of the rectangles in $N$ define a foliation of $N$, so we may define the \emph{orientation cover} $N^o$ of $N$ to be the set of pairs $(x,u)$, where $x\in N$ and $u$ is an orientation of the tie of $N$ that contains $x$. We may endow $N^o$ with a topology via the embedding of $N^o$ into $T^1S$ which sends $(x,u)$ to the unit tangent vector at $x$ to the tie containing $x$ in the direction $u$.

With this topology, the natural covering map $\pi_N \co N^o\to N$ given by $\pi_N\co (x,u)\mapsto x$ is a $2$--fold cover.  We denote $\partial_h N^o:=\pi_N^{-1}(\partial_h N)$ and $\partial_v N^o:=\pi_N^{-1}(\partial_v N)$, and refer to these as the \emph{horizontal boundary} and \emph{vertical boundary} of $N^o$ respectively.

Recall that $\lambda^o$ is the orientation cover of $\lambda$, see Section \ref{subsec:length_functions} for its definition. There is a natural embedding of $\lambda^o\to N^o$ which sends each pair $(x,\ast)\in\lambda^o$ to the pair $(x,u)\in N^o$, where $u$ is the orientation on the tie containing $x$ so that it passes from the left to the right (with respect to the orientation on $N^o$ induced by the orientation on $S$) of the leaf of $\lambda$ that contains $x$ equipped with the orientation $\ast$. Via this embedding, $\lambda^o$ is a geodesic lamination in $N^o$, $N^o$ is a train track neighborhood of $\lambda^o$, and the $2$--fold cover $N^o\to N$ restricts to the $2$--fold cover $\lambda^o\to\lambda$. Note that the leaves of $\lambda^o$ are naturally oriented, and denote the set of leaves of $\lambda^o$ by $\Lambda^o$, see Section \ref{back}. Similarly, the ties of $N^o$ are naturally oriented, and the preimage of every tie $h$ of $N$ under the $2$--fold cover $N^o\to N$ are two ties whose natural orientations induce opposite orientations on $h$.

\subsection{Relative homology of \texorpdfstring{$N^o$}{No}}\label{sec: homology background}

Let $N$ be a train track neighborhood of $\lambda$, and let $\lambda^o$ and $N^o$ be the orientation covers of $\lambda$ and $N$ respectively. Let $G$ be either an Abelian group or a real vector space, i.e. $G$ is an $\mathbb{A}$--module, where $\mathbb{A}=\mathbb{Z}$ or $\mathbb{A}=\mathbb{R}$. We will now collect some basic facts about the relative homology groups $H_1(N^o,\partial_vN^o;G)$ and $H_1(N^o,\partial_h N^o;G)$. For our purposes, we will only use the results in this section for the Abelian groups $(\mathbb R/2\pi\mathbb{Z})^k$ and $(\mathbb C/2\pi i\mathbb{Z})^k$ for some positive integer $k$, and for the vector spaces $\mathbb{R}^k$ for some positive integer $k$.

\subsubsection{Bases for \texorpdfstring{$H_1(N^o,\partial_vN^o;\mathbb{A})$}{H1(No,partial-v-No,G)} and \texorpdfstring{$H_1(N^o,\partial_h N^o;\mathbb{A})$}{H1(No,partial-h-No,G}}\label{sec: bases} 

Consider the subdivision of $N$ into rectangles $R_1,\dots,R_n$ (see Section \ref{sec: train tracks}), and let $R^o_1,\dots, R^o_{2n}$ be the rectangles of $N^o$ such that for each $l\in\{1,\dots,n\}$, $R^o_l$ and $R^o_{n+l}$ both project to $R_l$. Choose a tie $k_l$ in $R_l$, and let $\mathsf{k}_l$ and $\mathsf{k}_{n+l}$ be the lifts of $k_l$ in $R^o_l$ and $R^o_{n+l}$ respectively. Since the ties of $N^o$ admit natural orientations, for each $l\in\{1,\dots,2n\}$, we may view $\mathsf{k}_l$ as a $1$--cycle of $N^o$ relative to $\partial_hN^o$, and so it represents a relative homology class $[\mathsf{k}_l]$ in $H_1(N^o, \partial_h N^o; \mathbb{A})$, see Figure \ref{fig:homology}. 

Similarly, for each $l\in\{1,\dots,n\}$, choose an unoriented path $r_l^*$ in $R_l$ that is transverse to every tie of $R_l$, and has endpoints in each of the two boundary ties of $R_l$. For each endpoint $p$ of $r_l^*$, let $s_p$ be a subsegment of the switch containing $p$, so that one of its endpoints is $p$ and its other endpoint lies in the vertical boundary component of $N$ contained in this switch. Then let $k_l^*:= r_l^*\cup s_{p_1}\cup s_{p_2}$, where $p_1$ and $p_2$ are the endpoints of $r_l^*$, and let $\mathsf{k}_l^*$ and $\mathsf{k}_{n+l}^*$ be the lifts in $N^o$ of $k_l^*$ that intersect the interior of the rectangles $R^o_l$ and $R^o_{n+l}$, respectively. Since $\mathsf{k}_l^*$ intersects $\mathsf{k}_l$ transversely for all $l\in\{1,\dots,2n\}$, $\mathsf{k}_l^*$ can be oriented so that it passes from the right to the left of $\mathsf{k}_l$. Thus, we may also view $\mathsf{k}_l^*$ as a $1$--cycle of $N^o$ relative to $\partial_v N^o$, and so it represents a relative homology class $[\mathsf{k}_l^*]$ in $H_1(N^o, \partial_v N^o; \mathbb{A})$, see Figure \ref{fig:homology}.

\begin{figure}[h!]
	\includegraphics[width=10cm]{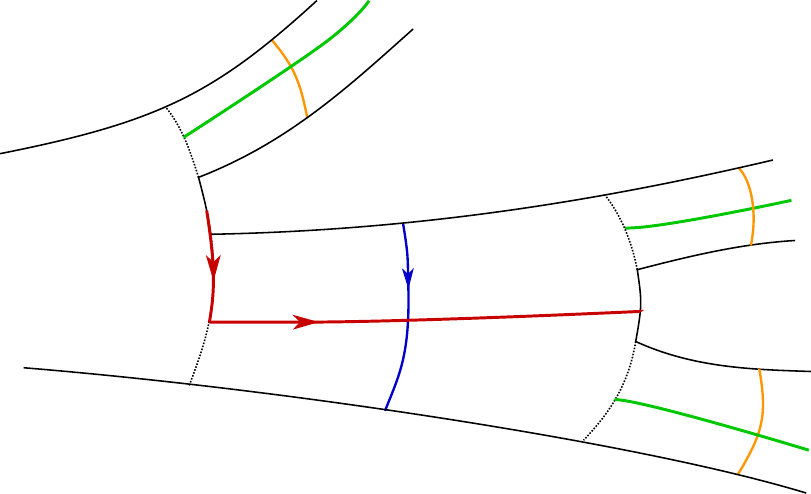}
	\put(-206, 129){\small $r_j^*$}
	\put(-178, 141){\small $k_j$}
	\put(-81,54){\small $\mathsf{k}_l^*$}
	\put(-145,32){\small $\mathsf k_l$}
	\put(-222,86){\small $s_{p_1}$}
	\put(-220,60){\small $p_1$}
	\put(-59,60){\small $p_2$}
	\caption{\small The unoriented paths $r_j^*$ (in green), $\mathsf{k}_l \in H_1(N^o, \partial_v N^o; \mathbb{A})$ (in blue) and $\mathsf{k}^*_l\in H_1(N^o, \partial_v N^o; \mathbb{A})$ (in red) in $N^o$.}
	\label{fig:homology}
\end{figure}

It is straightforward to see that for all $l\in\{1,\dots,2n\}$, $[\mathsf{k}_l]$ and $[\mathsf{k}_l^*]$ do not depend on any of the choices made in their construction. The following proposition is a consequence of repeated applications of the relative Mayer-Vietoris exact sequence (see end of \cite[Section 3.1]{hat_alg}). 

\begin{proposition}\label{prop:G_intersection}
	The homology group $H_1(N^o,\partial_h N^o; \mathbb{A})$ is freely generated (as an $\mathbb{A}$--module) by $[\mathsf{k}_1],\dots,[\mathsf{k}_{2n}]$, and $H_1(N^o,\partial_v N^o; \mathbb{A})$ is freely generated by $[\mathsf{k}_1^*],\dots,[\mathsf{k}_{2n}^*]$.
\end{proposition}

Henceforth, we will refer to $\left\{[\mathsf{k}_1],\dots,[\mathsf{k}_{2n}]\right\}$ and $\left\{[\mathsf{k}_1^*],\dots,[\mathsf{k}_{2n}^*]\right\}$ as the \emph{standard generating sets} of the relative homology groups $H_1(N^o,\partial_h N^o; \mathbb{A})$ and $H_1(N^o,\partial_v N^o; \mathbb{A})$, respectively.

Since $H_0(N^o, \partial_v N^o; \mathbb{A})$ is trivial, the following proposition is a consequence of the universal coefficient theorem for homology and cohomology groups, see \cite[Corollary 3A.4]{hat_alg} and Proposition \ref{prop:G_intersection}.

\begin{proposition} \label{prop:universal_coeff} Let $G$ be an $\mathbb{A}$--module. Then as $\mathbb{A}$--modules,
	\[H_1(N^o, \partial_v N^o; G) = H_1(N^o, \partial_v N^o; \mathbb{A}) \otimes_{\mathbb{A}} G \cong\bigoplus_{l=1}^{2n}G\cdot[\mathsf{k}_l^*].\]
\end{proposition}

\subsubsection{Intersection pairing in homology}\label{int-pair}

The orientation on $N^o$ gives us a generator $[N^o]\in H_2(N^o,\partial N^o;\mathbb{A})$ called the \emph{relative fundamental class}. Since $\partial_v N^o\cup\partial_h N^o=\partial N^o$ and $\partial_v N^o\cap\partial_h N^o$ is equal to the boundaries of both $\partial_vN^o$ and $\partial_hN^o$, we may apply Lefschetz-Poincar\'e duality (see \cite[Theorem 3.43]{hat_alg}) to deduce the following. 

\begin{proposition}  \label{prop: duality}
	Let $G$ be an $\mathbb{A}$--module. The cap product with $[N^o]$ gives isomorphisms 
	\[D\colon H^1(N^o,\partial_v N^o;G)\to H_1(N^o,\partial_h N^o;G),\]
	and
	\[
	D'\colon H^1(N^o,\partial_h N^o;G)\to H_1(N^o,\partial_v N^o;G).\]
	Furthermore, the isomorphisms $D, D'$ and the cup product $\smile$ on cohomology classes determine a morphism of $G$--modules
	\[
	\mathbb{I}_G: H_1(N^o, \partial_h N^o; \mathbb{A}) \times H_1(N^o, \partial_v N^o; G) \longrightarrow G
	\]
	given by $\mathbb{I}_G(c,c') := D^{-1}(c) \smile (D')^{-1}(c') \in G= H^2(N^o,\partial N^o;G)$.
\end{proposition}

In terms of the standard generating sets of $H_1(N^o, \partial_h N^o; \mathbb{A})$ and $H_1(N^o, \partial_v N^o; \mathbb{A})$, if 
\[c'=\sum_{j=1}^{2n}a_j[\mathsf{k}_j]\in H_1(N^o, \partial_h N^o; \mathbb{A})\,\,\text{ and }\,\,c=\sum_{l=1}^{2n}b_l[\mathsf{k}_l^*]\in H_1(N^o, \partial_v N^o; G),\] 
then
\[\mathbb{I}_G(c',c)=\sum_{j=1}^{2n}a_j \cdot b_j\in G.\]

\subsection{Relative \texorpdfstring{$\lambda^o$}{lo}--cocycle representatives of \texorpdfstring{$H_1(N^o,\partial_v N^o;G)$}{H1(No,G)}}\label{sec:representatives}

Let $G$ be an $\mathbb{A}$--module and $N$ a train track neighborhood of $\lambda$. Following Bonahon-Dreyer \cite[Sections 4.4 and 4.5]{BoD}, we now describe, for each relative cohomology class in $H_1(N^o,\partial_v N^o;G)$, a canonical relative cocycle representative. 

Recall that
\begin{itemize}
	\item $\wt\Delta^{2*}$ denotes the set of distinct pairs of plaques of $\wt\lambda$,
	\item $\wt{\Delta}^o$ denotes the set of labelings, i.e. triples of points in $\partial\wt\lambda$ that are the vertices of a plaque of $\wt\lambda$,
	\item $\preceq$ denotes the counter-clockwise cyclic order on $\partial\wt\lambda \subset \partial\wt S$ (with respect to the orientation on $S$), and
	\item For any plaques $T_1$ and $T_2$ of $\wt\lambda$ and any plaque $T$ that separates $T_1$ and $T_2$, $(x_{T,1},x_{T,2},x_{T,3})$ denotes the coherent labeling of the vertices of $T$ with respect to $(T_1,T_2)$, i.e. if $h_{T,-}$ and $h_{T,+}$ are the two edges of $T$ that separate $T_1$ and $T_2$, so that $h_{T,-}$ separates $T_1$ and $h_{T,+}$, then $x_{T,2}$ is the common endpoint of $h_{T,-}$ and $h_{T,+}$, and $x_{T,1}$ and $x_{T,3}$ are respectively the endpoints of $h_{T,-}$ and $h_{T,+}$ that are not $x_{T,2}$, see Figure \ref{fig:notations3}.
\end{itemize}

\begin{definition}[$\mathcal{Z}(\lambda^o, \mathrm{slits};G)$]\label{def:oriented relative cocycle}
	A \textit{relative $\lambda^o$--cocycle with values in $G$} is a functiuon $\eta \co \wt\Delta^{2*} \to G$ satisfying the following properties:
	\begin{enumerate}
		\item($\Gamma$--equivariance) for all pairs of plaques ${\bf T} \in \wt\Delta^{2*}$ and all elements $\gamma\in\Gamma$, 
		$$\eta\left(\gamma\cdot{\bf T}\right) = \eta\left({\bf T}\right);$$
		\item(Cocycle boundary condition) there exists a $\Gamma$--invariant function
		\[
		\partial \eta : \wt{\Delta}^o \longrightarrow G,
		\]
		called the \emph{boundary} of $\eta$, such that for any pairwise distinct plaques $T_1$, $T_2$, $T$ of $\wt\lambda$ with $T$ separating $T_1$ and $T_2$, we have
		\[
		\begin{cases}
			\eta(T_1,T_2) = \eta(T_1,T) + \eta(T,T_2) -\partial \eta(\mathbf{x}_T),  &\quad\text{if }x_{T,3} \prec x_{T,2} \prec x_{T,1},\\
			\eta(T_1,T_2) = \eta(T_1,T) + \eta(T,T_2) + \partial \eta(\mathbf{x}_T), &\quad\text{if } x_{T,1} \prec x_{T,2} \prec x_{T,3},\\
		\end{cases}
		\]		
		where ${\bf x}_T=(x_{T,1},x_{T,2},x_{T,3})$ denotes the coherent labeling of the vertices of $T$ with respect to $(T_1,T_2)$.
	\end{enumerate}
	We denote the $\mathbb{A}$--module of relative $\lambda^o$--cocycles with values in $G$ by $\mathcal{Z}(\lambda^o, \mathrm{slits}; G)$.
\end{definition} 

Although the notion of a slit does not appear in this paper, we use the notation $\mathcal{Z}(\lambda^o, \mathrm{slits}; G)$ to be consistent with the notation in Bonahon-Dreyer \cite{BoD}. 

To relate $\mathcal{Z}(\lambda^o, \mathrm{slits};G)$ to $H_1(N^o,\partial_v N^o;G)$, we need to recall the notion of a \emph{tightly transverse path} from \cite{BoD}. If $\pi_N : N^o \to N$ is the orientation cover of $N$ and $\pi_S : \widetilde{S} \to S$ is the universal cover of $S$, then we denote $\widetilde{N}:= \pi_S^{-1}(N)$ and define $\widetilde{N}^o$ to be the set of pairs in $\widetilde{N} \times N^o$ that project to the same point, that is 
\[
\widetilde{N}^o:= \{ (p,q) \in \widetilde{N} \times N^o \mid \pi_S(p) = \pi_N(q) \}.
\]
We also denote by $\pi_{\widetilde{N}} : \widetilde{N}^o \to \widetilde{N}$ the natural projection $\pi_{\widetilde{N}}(p,q):= p$.

\begin{definition}\label{def:tightly transverse}
	An unoriented path $k$  in $\wt N^o$ is \emph{tightly transverse} to $\wt \lambda^o$ if $k$ intersects $\wt \lambda^o$ transversely at least once, and every connected component $e$ of $k-\wt\lambda^o$ satisfies one of the following properties, see Figure \ref{fig:tightlytransverse}:
	\begin{itemize}
		\item the component $e$ contains an endpoint of $k$,
		\item the projection $\pi_{\widetilde{N}}(e)$ is contained inside the intersection $\widetilde{N} \cap T$ for some plaque $T \in \widetilde{\Delta}$, and the endpoints of $\pi_{\widetilde{N}}(e)$ lie in a pair of distinct edges $g_1, g_2$ of $T$. Moreover, $\pi_{\widetilde{N}}(e)$ separates $\widetilde{N} \cap T$ into two connected components, one of which contains no vertical boundary components of $\widetilde{N}$ and has boundary equal to $\pi_{\widetilde{N}}(e) \cup r_1 \cup r_2$, where $r_1$ and $r_2$ are asymptotic geodesic rays contained inside $g_1$ and $g_2$ respectively.
	\end{itemize}
	
	An unoriented path $k$ in $N^o$ is \emph{tightly transverse} if it lifts to a tightly transverse path in $\wt N^o$. Any connected component of $k-\lambda^o$ (respectively, $k-\wt\lambda^o$) that contains an endpoint of $k$ is called an \emph{exterior subsegment} of $k$. Then the \emph{interior subsegment} of $k$, denoted $k'$, is the complement in $k$ of the exterior subsegments of $k$.
\end{definition}

\begin{figure}[h!]
	\includegraphics[width=7cm]{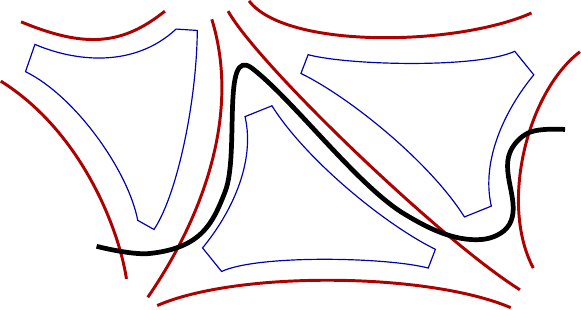}
	\put(-3,60){\small $\pi_{\wt N}(k)$}
	\caption{\small The projection to $\wt N$ of a tightly transverse path $k$. In this case, $k-\wt\lambda^o$ projects to five connected components $\pi_{\wt N}(e)$.}
	\label{fig:tightlytransverse}
\end{figure}

\subsubsection{Homological interpretation of relative \texorpdfstring{$\lambda^o$}{lambda-o}--cocycles}\label{sec: cohom int}
Let $\pi_N:N^o\to N$ denote the double cover. 

In Bonahon-Dreyer \cite[Section 4.4]{BoD}, a relative $\lambda^o$--cocycle $\eta$ is interpreted as a function 
\[
F_\eta : \{\text{unoriented paths in $N^o$ that are tightly transverse to $\lambda^o$}\} \longrightarrow G,
\]
which is defined as follows. (In \cite{BoD}, the authors consider only the case where $\mathbb{A}=\mathbb{R}$, but their argument applies to the case where $\mathbb{A}=\mathbb{Z}$.)
Any path $k$ in $N^o$ that is tightly transverse to $\lambda^o$ admits a natural orientation $\mathsf{k}$, with respect to which all the leaves of $\lambda^o$ that intersect $k$ are oriented from right to left with respect to $\mathsf{k}$. Select now $\widetilde{\mathsf{k}}$, a lift of $\pi_N(\mathsf{k})$ inside $\widetilde{S}$, and let $T_1$ and $T_2$ be the plaques of $\widetilde{\lambda}$ that contain the backward and forward endpoints of $\widetilde{\mathsf{k}}$. The function $F_\eta$ then satisfies
\[
F_\eta(k) = \eta(T_1, T_2) \in G .
\]
Notice that $F_\eta(k)$ does not depend on the choice of the lift $\widetilde{\mathsf{k}}$ since $\eta$ is $\Gamma$--equivariant. Also, the function $F_\eta$ descends to a function on the set of unoriented paths in $N$ that are tightly transverse to $\lambda$ only when $\eta$ satisfies $\eta(T_1,T_2) = \eta(T_2, T_1)$ for every pair of distinct plaques $T_1, T_2$ of $\widetilde{\lambda}$.

Viewing the relative $\lambda^o$--cocycle $\eta$ as the $G$--valued function $F_\eta$ allows us to think of them as cocycles for the cohomology classes in $H^1(N^o,\partial_h N^o;G)$, which we can then relate to homology classes in $H_1(N^o,\partial_v N^o;G)$ via Poincar\'e duality. More precisely, consider the morphism of $\mathbb{A}$--modules
\begin{equation}\label{eqn: cocycle homology} \varsigma_G:\mathcal{Z}(\lambda^o,\mathrm{slits};G) \to H_1(N^o,\partial_v N^o;G)\end{equation}
that sends every $\eta \in \mathcal{Z}(\lambda^o,\mathrm{slits};G)$ to the homology class 
\begin{equation}\label{eqn: intersection}
	[\eta]:= \sum_{j = 1}^{2 n} F_\eta(k_j) [\mathsf{k}_j^*]
	\in H_1(N^o, \partial_v N^o; G),
\end{equation}
where $\left\{[\mathsf{k}_1], \dots, [\mathsf{k}_{2n}]\right\}$ and $\left\{[\mathsf{k}_1^*], \dots, [\mathsf{k}_{2 n}^*]\right\}$ are standard generating sets of $H_1(N^o,\partial_h N^o;\mathbb{Z})$ and $H_1(N^o,\partial_v N^o;\mathbb{Z})$ respectively (see Proposition~\ref{prop:G_intersection}), and $k_j$ is the unoriented path in $N^o$ obtained by forgetting the orientation on the path $\mathsf{k}_j$.

Observe that by construction, the element $[\eta]\in H_1(N^o, \partial_v N^o; G)$ defined above satisfies
\[
\mathbb{I}_G([\mathsf{k}_j], [\eta])= F_\eta(k_j)
\]
for every $j \in \{1, \dots, 2n\}$.

The next proposition states that the morphism $\varsigma_G$ is in fact an isomorphism. 

\begin{proposition}\label{prop: cocycles and cohomology}
	The map
	\[\varsigma_G:\mathcal{Z}(\lambda^o,\mathrm{slits};G) \longrightarrow H_1(N^o,\partial_v N^o;G)\]
	is an isomorphism of $\mathbb{A}$--modules, and is uniquely characterized by the following property. Consider a path $k$ in $N^o$ that is tightly transverse to $\lambda^o$, has endpoints inside $\partial_h N^o$, and whose exterior subsegments are subpaths of ties of $N^o$. Then for all $\eta\in \mathcal{Z}(\lambda^o,\mathrm{slits};G)$, we have
	\[
	\mathbb{I}_G([\mathsf{k}], [\eta]) = F_\eta(k) ,
	\]
	where $\mathsf{k}$ denotes the orientation of $k$ with respect to which the leaves of $\lambda^o$ cross $\mathsf{k}$ from right to left. 
\end{proposition}

Proposition \ref{prop: cocycles and cohomology} was proven by Bonahon and Dreyer \cite[Proposition 4.5]{BoD} in the case when $\mathbb{A}=\mathbb{R}=G$, but their proof generalizes essentially verbatim. For this reason, we omit the proof, and we refer to their paper for the details.

Via the isomorphism $\mathcal{Z}(\lambda^o,\mathrm{slits};G) \cong H_1(N^o,\partial_v N^o;G)$, Bonahon and Dreyer \cite{BoD} also observed that the boundary map $\partial\eta:\widetilde\Delta^o\to G$ associated to every relative $\lambda^o$--cocycle $\eta\in\mathcal{Z}(\lambda^o,\mathrm{slits};G)$ has a homological interpretation. (Indeed, this is the main advantage one gains from thinking of $\mathcal{Z}(\lambda^o,\mathrm{slits};G)$ as associated to homology classes instead of cohomology classes.) To describe this, we introduce the following terminology. For every vertical boundary component $\mathsf t$ of $N^o$, let $\wt{\mathsf t}$ be a lift to $\wt S$ of $\pi_N(\mathsf t)$. Since $\mathsf t$ lies in a tie of $N^o$, it has a natural orientation, which induces an orientation on $\wt{\mathsf t}$. Note that $\wt{\mathsf t}$ lies in an oriented tie $\wt{\mathsf k}$ of $\wt N$ and a plaque $T$ of $\wt\lambda$, and the endpoints of $\wt{\mathsf k}\cap T$ lie in two different edges of $T$. Thus, we may label the vertices of $T$ by ${\bf x}=(x_1,x_2,x_3)$ so that the edge of $T$ with endpoints $x_1$ and $x_2$ contains the backward endpoint of $\wt{\mathsf k}\cap T$, while the edge of $T$ with endpoints $x_2$ and $x_3$ contains the forward endpoint of $\wt{\mathsf k}\cap T$. We refer to any triple in the $\Gamma$--orbit of ${\bf x}$ as a \emph{triple associated to $\mathsf t$}. Note that the set of triples associated to $\mathsf t$ does not depend on the choice of $\wt{\mathsf t}$. See Figure~\ref{fig:triple associated}.

\begin{figure}[h!]
	\includegraphics[width=7cm]{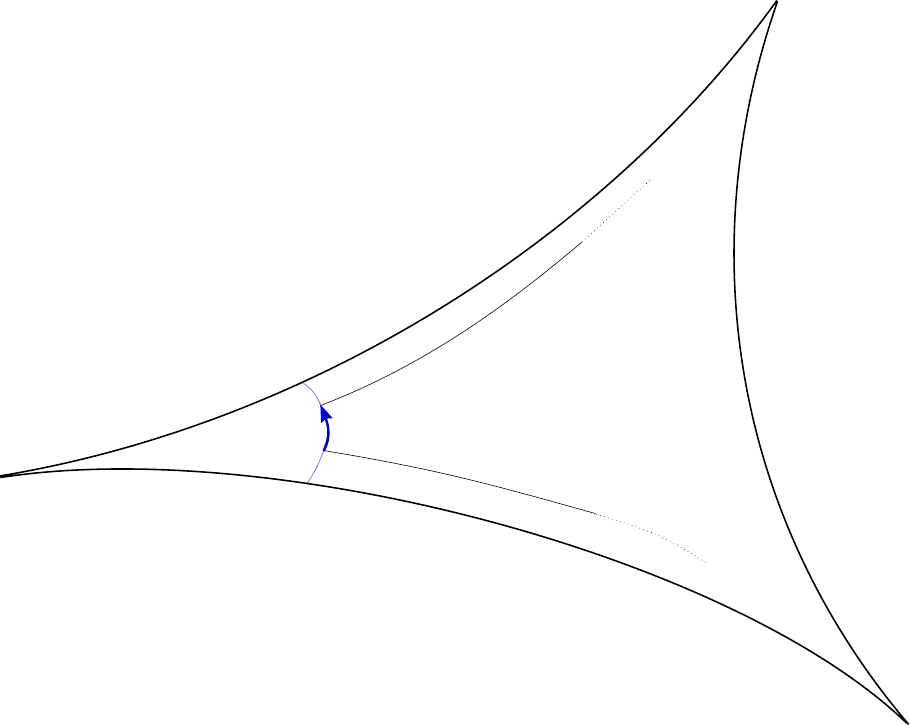}
	\put(-210,55){\small $x_2$}
	\put(-135,62){\small $\wt{\mathsf t}$}
	\put(-0,-4){\small $x_1$}
	\put(-32,160){\small $x_3$}
	\caption{\small The lift $\wt{\mathsf t}$ of the oriented vertical boundary component $\mathsf t$ and the triple $(x_1,x_2,x_3)$ associated to $\mathsf t$.}
	\label{fig:triple associated}
\end{figure}

The following proposition is a consequence of Equation \eqref{eqn: intersection} and the cocycle boundary condition in Definition \ref{def:oriented relative cocycle}. 

\begin{proposition}\label{prop: homological boundary}
	If the map $\partial:H_1(N^o,\partial_v N^o;G)\to H_0(\partial_v N^o;G)$ denotes the natural boundary map, then for all $\eta\in \mathcal{Z}(\lambda^o,\mathrm{slits};G)$, we have
	\[\partial[\eta]=\sum_{\mathsf t\in\pi_0(\partial_v N^o)}\partial\eta({\bf x}_{\mathsf t})\cdot [\mathsf t]\in H_0(\partial_v N^o;G),\] 
	where ${\bf x}_{\mathsf t}$ is some (any) triple associated to the vertical boundary component $\mathsf t$ of $N^o$. 
\end{proposition}

\subsubsection{\texorpdfstring{$\lambda^o$}{lambda-o}--cocycles}\label{sec: cocycles}
We will also use the following specialization of relative $\lambda^o$--cocycles with values in $G$.

\begin{definition}[$\mathcal{Z}(\lambda^o; G)$]\label{def: non-relative cocycles}
	A \emph{$\lambda^o$--cocycle with values in $G$} is an element $\eta\in\mathcal{Z}(\lambda^o,\mathrm{slits};G)$ with the property that $\partial\eta=0$. We denote the group of $\lambda^o$--cocycles with values in $G$ by $\mathcal{Z}(\lambda^o; G)$.
\end{definition}

As a consequence of Propositions \ref{prop: cocycles and cohomology} and \ref{prop: homological boundary}, $\varsigma_G$ restricts to an isomorphism
\[\varsigma_G|_{\mathcal{Z}(\lambda^o;G)}:\mathcal{Z}(\lambda^o;G)\to\{c\in H_1(N^o,\partial_v N^o;G):\partial c=0\}.\]
Since $H_1(\partial_v N^o;G)=0$ and $H_1(\partial_h N^o;G)=0$, the long exact sequences of relative homology give the short exact sequences
\[0\to H_1(N^o;G)\xrightarrow[]{I_v} H_1(N^o,\partial_v N^o;G)\xrightarrow[]{\partial} H_0(\partial_v N^o;G)\to 0\]
and
\[0\to H_1(N^o;G)\xrightarrow[]{I_h} H_1(N^o,\partial_h N^o;G)\xrightarrow[]{\partial} H_0(\partial_h N^o;G)\to 0,\]
where $I_v$ and $I_h$ are induced by the natural inclusion maps of pairs $(N^o,\emptyset) \subset (N^o,\partial_v N^o)$ and $(N^o,\emptyset) \subset (N^o,\partial_h N^o)$, respectively. The first short exact sequence implies that 
\[I_v^{-1}\circ \varsigma_G|_{\mathcal{Z}(\lambda^o;G)}:\mathcal{Z}(\lambda^o;G)\to H_1(N^o;G).\]
is a well-defined isomorphism, while the second short exact sequence implies that 
\[I_h\circ I_v^{-1}\circ \varsigma_G|_{\mathcal{Z}(\lambda^o;G)}:\mathcal{Z}(\lambda^o;G)\to H_1(N^o,\partial_h N^o;G).\]
is an embedding.

In light of the discussion above, given $\mu\in\mathcal{Z}(\lambda^o;G)$, we will abuse notation and denote its images in $H_1(N^o,\partial_v N^o;G)$, $H_1(N^o;G)$, and $H_1(N^o,\partial_h N^o;G)$ all by $[\mu]$; it will be clear from context which we are referring to. 

\subsubsection{Intersection pairing between relative cocycles}\label{length_cyclic} 

By Proposition \ref{prop: cocycles and cohomology}, we have identified $\mathcal{Z}(\lambda^o,{\rm slits},G)$ with $H_1(N^o,\partial_vN^o;G)$ as $\mathbb{A}$--modules. Also, we have described in Section~\ref{sec: cocycles} how one can identify $\mathcal{Z}(\lambda^o;\mathbb{A})$ with an $\mathbb{A}$--submodule of $H_1(N^o,\partial_hN^o;\mathbb{A})$. Notice however, that $\mathcal{Z}(\lambda^o,{\rm slits},G)$ and $\mathcal{Z}(\lambda^o;G)$ do not depend on the choice of a train track neighborhood $N$ of $\lambda$. We will now verify that neither does the map 
\begin{align}\label{eqn: no N}\mathcal{Z}(\lambda^o;\mathbb{A})\times \mathcal{Z}(\lambda^o,{\rm slits},G)\to\mathbb{A}\end{align}
given by $(\mu,\eta)\mapsto \mathbb{I}_G([\mu],[\eta])$, where 
\begin{itemize}
	\item $[\eta]:=\varsigma_G(\eta)\in H_1(N^o, \partial_v N^o; G)$, and $\varsigma_G$ is the map given by Equation \eqref{eqn: cocycle homology},
	\item $[\mu]\in H_1(N^o, \partial_h N^o; \mathbb{A})$ is the relative homology class determined by the $\lambda^o$--cocycle $\mu\in\mathcal{Z}(\lambda^o,{\rm slits},\mathbb{A})$ described in Section \ref{sec: cocycles} (with $G=\mathbb{A}$), and
	\item $\mathbb{I}_G :  H_1(N^o, \partial_h N^o; \mathbb{A}) \times H_1(N^o, \partial_v N^o; G) \to G$ is the $G$--valued intersection pairing given by Proposition \ref{prop: duality}.
\end{itemize}

		\begin{proposition}
			The map given by \eqref{eqn: no N} does not depend on the train track neighborhood $N$.
		\end{proposition}
		
		\begin{proof}
			Assume that $N_1$ and $N_2$ are two train track neighborhoods that carry the lamination $\lambda$ and such that $N_1 \subset N_2$. For both $i=1,2$ and for each vertical boundary component $t$ of $N_i$, let $k_i(t)$ denote the tie of $N_i$ that contains $t$. We can then find a deformation retraction $r : N_2 \to N_1$ that satisfies the following properties:
			\begin{enumerate}
				\item $r(\partial_v N_2) = \partial_v N_1$ and $r(\partial_h N_2) = \partial_h N_1$;
				\item for each plaque $T$ of $\widetilde\lambda$, $r$ sends each component $t$ of $\partial_vN_2\cap T$ to the unique component $t'$ in $\partial_v N_1 \cap T$ for which the points $k_1(t) \cap T$ and $k_2(t') \cap T$ lie in the same pair of distinct edges of $T$.
			\end{enumerate}
			Such a retraction determines homotopy equivalences on the associated orientation covers
			\[
			r^o : (N_2^o,\partial_v N_2^o) \longrightarrow (N_1^o, \partial_v N_1^o), \quad r^o : (N_2^o, \partial_h N_2^o) \longrightarrow (N_1^o, \partial_h N_1^o), 
			\]
			and induces a natural bijective correspondence between the vertical boundary components of $N_1^o$ and of $N_2^o$. 
			
			By the naturality of the intersection pairings
			\[
			\mathbb I_{G, j} : H_1(N_j^o, \partial_h N_j^o;\mathbb{A}) \times H_1(N_j^o, \partial_v N_j^o;G) \longrightarrow G
			\]
			for $j = 1, 2$, we have $\mathbb I_{G,2}(c,c') = \mathbb I_{G,1}(r^o_*(c), r^o_*(c'))$, where $r^o_*$ denote the maps in homology induced by $r^o$. For any $\lambda^o$--cocycle $\mu \in \mathcal{Z}(\lambda^o;\mathbb{A})$, let
			\[
			[\mu]_1 \in H_1(N_1^o, \partial_h N_1^o;\mathbb{R}), \quad [\mu]_2 \in H_1(N_2^o, \partial_h N_2^o;\mathbb{R})
			\]
			be the corresponding relative homology classes introduced in Section \ref{sec: cocycles}. Since the maps $\mu \mapsto [\mu]_1$ and $\mu \mapsto [\mu]_2$ are natural, they satisfy $[\mu]_1 = r_*^o([\mu]_2)$. It follows from the characterization in Proposition \ref{prop: cocycles and cohomology} of the isomorphisms
			\begin{gather*}
				\varsigma_{G,1} : \mathcal{Z}(\lambda^o,\text{slits};G) \to H_1(N_1^o, \partial_v N_1^o;G) , \\ 
				\varsigma_{G,2} : \mathcal{Z}(\lambda^o,\text{slits};G) \to H_1(N_2^o, \partial_v N_2^o;G)
			\end{gather*}
			that $r_*^o \circ \varsigma_{G, 2} = \varsigma_{G, 1}$. As such, if we set $[\eta]_j:= \varsigma_{G, j}(\eta)$ for both $j=1,2$, then combining what we just observed we deduce the identities
			\begin{align*}
				\mathbb I_{G,2}([\mu]_2,[\eta]_2) & = \mathbb I_{G,1}(r^o_*([\mu]_2),r^o_*([\eta]_2)) \\
				& = \mathbb I_{G,1}([\mu]_1, [\eta]_1) .
			\end{align*}
			
			To conclude, it suffices to observe that, for any pair of train track neighborhoods $N$ and $N'$ that carry the lamination $\lambda$, we can find a train track neighborhood $M$ contained in both $N$ and $N'$. The above discussion then proves that both length functions defined through the choices of $N$ and $N'$ coincide with the length functions defined via $M$.
		\end{proof}
		
		\subsection{Homological interpretation of \texorpdfstring{$\lambda$}{l}-cocyclic pairs}\label{sub:complex}
		Recall that $\mathcal{Y}_d(\lambda;G)$ denotes the space of $\lambda$--cocyclic pairs of dimension $d$ with values in $G$. We previously observed (see Remark \ref{rmk: lambda cocyclic modules}) that $\mathcal{Y}_d(\lambda;G)$ is an $\mathbb{A}$--module. We now relate $\mathcal{Y}_d(\lambda;G)$ to the $\mathbb{A}$--modules $H_1(N^o,\partial_v N^o;G^{\mathcal{A}})$ and $H_0(\partial_v N^o;G^{\mathcal{A}})$, where $G^{\mathcal{A}}$ is the $\mathbb{A}$--module given by
		\[G^\mathcal{A}: = \{ (x_\mathbf{i})_{\mathbf{i}\in\mathcal{A}} \mid x_{\mathbf{i}} \in G \text{ for all } \mathbf{i} \in \mathcal{A} \},\] 
		and $\mathcal{A}$ denotes pairs of positive integers that sum to $d$. 
		
		To do so, it is convenient to define the notion of a $\lambda$--triangle data function. Recall that $\mathcal{B}$ denotes triples of positive integers that sum to $d$, while $\wt\Delta^o$ is the set of labelings, i.e. ordered triples of points in $\partial\wt\lambda$ that are the vertices of a plaque of $\wt\lambda$. Recall also that given the labeling ${\bf x}=(x_1,x_2,x_3)\in\wt\Delta^o$, we defined the labelings
		\[\widehat{\bf x}:=(x_2,x_1,x_3),\quad{\bf x}_+:=(x_2,x_3,x_1),\quad\text{ and }\quad{\bf x}_-:=(x_3,x_1,x_2).\] 
		
		\begin{definition}\label{def:triangle data}
			A $\lambda$--\textit{triangle data function of dimension $d$ with values in $G$} is a map
			$$\theta \co \wt\Delta^o  \times \mathcal{B} \to G$$
			satisfying the following conditions: for all labelings ${\bf x}\in \wt\Delta^o$, all triples ${\bf j} \in \mathcal{B}$, and all elements $\gamma\in\Gamma$, we have
			\begin{enumerate}
				\item(Symmetry) 
				$$\theta\left({\bf x}, {\bf j}\right) = \theta\left({\bf x}_+, {\bf j}_+\right)= \theta\left({\bf x}_-, {\bf j}_-\right)=-\theta(\widehat{\bf x},\widehat{\bf j});$$
				\item($\Gamma$--equivariance)
				$$\theta\left(\gamma\cdot {\bf x}, {\bf j}\right) = \theta\left({\bf x}, {\bf j}\right).$$
			\end{enumerate}	
			We denote by $\mathcal{T}_d(\lambda;G)$ the $\mathbb{A}$--module of $\lambda$--triangle data functions of dimension $d$ with values in $G$.
		\end{definition} 
		
		Recall that 
		\[\widehat{\cdot}:\mathcal{A}\to\mathcal{A},\quad\widehat{\cdot}:\wt\Delta^{2*}\to\wt\Delta^{2*}, \quad\text{and}\quad\widehat{\cdot}:G^{\mathcal{A}}\to G^{\mathcal{A}}\] 
		are the involutions respectively defined by $\widehat{(i_1,i_2)}=(i_2,i_1)$, $\widehat{(T_1,T_2)}=(T_2,T_1)$, and $\widehat{g}_{\bf i}=g_{\widehat{\bf i}}$ for all ${\bf i}\in\mathcal{A}$. Using these, define the involution 
		\[\iota':\mathcal{Z}(\lambda^o,\mathrm{slits};G^{\mathcal{A}})\to\mathcal{Z}(\lambda^o,\mathrm{slits};G^{\mathcal{A}})\] 
		by $\iota'(\eta)(\widehat{\bf T})=\widehat{\eta({\bf T})}$. Also, for any $\lambda$--triangle data function $\theta\in\mathcal{T}_d(\lambda;G)$, define  \[s_\theta:\widetilde{\Delta}^o\to G^{\mathcal{A}}\] 
		by
		\[
		s_\theta({\bf x})^{\mathbf{i}} = 
		\begin{cases}
			\sum_{\mathbf{j} \in \mathcal{B} : j_2 = i_1} \theta(\mathbf{x},\mathbf{j}) &  \text{if $x_3 \prec x_2 \prec x_1$,}  \\
			\sum_{\mathbf{j} \in \mathcal{B} : j_2 = i_2} \theta(\mathbf{x},\mathbf{j}) & \text{if $x_1 \prec x_2 \prec x_3$},
		\end{cases}
		\]
		for all ${\bf i}\in\mathcal{A}$.
		
		Now, given any $\lambda$--cocyclic pair $(\alpha,\theta)\in\mathcal{Y}_d(\lambda;G)$, we may view $\alpha$ as the function $\alpha:\widetilde\Delta^{2*}\to G^{\mathcal{A}}$ which sends every pair ${\bf T}\in\widetilde\Delta^{2*}$ to the vector $(\alpha({\bf T},{\bf i}))_{{\bf i}\in\mathcal{A}}\in G^{\mathcal{A}}$. Notice then that Conditions (2) and (4) of Definition \ref{def_cocycle} are equivalent to requiring that $\theta\in\mathcal{T}_d(\lambda;G)$, Conditions (3) and (5) of Definition \ref{def_cocycle} are equivalent to requiring that $\alpha\in \mathcal{Z}(\lambda^o,{\rm slits},G^{\mathcal{A}})$ and $\partial\alpha=-s_\theta$, and Condition (1) of Definition~\ref{def_cocycle} is equivalent to requiring that $\iota'(\alpha)=\alpha$. In other words, 
		\[\mathcal{Y}_d(\lambda;G)=\{ (\alpha, \theta) \in \mathcal{Z}(\lambda^o,\mathrm{slits};G^\mathcal A) \times \mathcal{T}_d(\lambda;G) :\iota'(\alpha)=\alpha\text{ and }\partial \alpha = - s_\theta \}.\]
		
			Since $\mathcal{Y}_d(\lambda;G)$ is a subset of $\mathcal{Z}(\lambda^o,\mathrm{slits};G^\mathcal A) \times \mathcal{T}_d(\lambda;G)$, we may now define the morphism of $\mathbb{A}$--modules
			\begin{align}\label{eqn: E}
				E:\mathcal{Y}_d(\lambda;G)\to H_1(N^o,\partial_v N^o;G^{\mathcal{A}})\times \mathcal{T}_d(\lambda;G)
			\end{align}
			by $E(\alpha,\theta):=([\alpha],\theta)$. It follows from Proposition \ref{prop: cocycles and cohomology} that $E$ is an embedding. We will now describe the image of $E$. To do so, we will need the following maps. 
			
			\begin{itemize}
				\item Let $\iota: N^o \to N^o$ be the involution of the $2$--fold cover $\pi_N:N^o \to N$, and let 
				\begin{align}\label{eqn: iota}
					\iota_*:H_1(N^o, \partial_v N^o; G^\mathcal{A})\to H_1(N^o, \partial_v N^o; G^\mathcal{A})
				\end{align}
				be the induced involution on homology. If $\left\{[\mathsf{k}^*_1],\dots,[\mathsf{k}^*_{2n}]\right\}$ is the standard generating set for $H_1(N^o, \partial_v N^o; G^\mathcal{A})$, then $\iota_*(g\cdot [\mathsf{k}^*_l])=-g\cdot [\mathsf{k}^*_{l+n}]$ and $\iota_*(g\cdot [\mathsf{k}^*_{l+n}])=-g\cdot [\mathsf{k}^*_l]$ for all $l\in\{1,\dots,n\}$ and all $g\in G^{\mathcal{A}}$.
				\item Let 
				\begin{align}\label{eqn: hat}
					\widehat{\cdot}:H_1(N^o, \partial_v N^o; G^\mathcal{A})\to H_1(N^o, \partial_v N^o; G^\mathcal{A})
				\end{align}
				be the involution that sends $\sum_{j=1}^m g_j\cdot [\mathsf{h}_j]$ to $\sum_{j=1}^m {\widehat g_j}\cdot [\mathsf{h}_j]$, where $g_j\in G^{\mathcal{A}}$ and $[\mathsf{h}_j]\in H_1(N^o, \partial_v N^o; \mathbb{Z})$.
				\item Let \[K : \mathcal{T}_d(\lambda;G) \to H_0(\partial_v N^o;G^\mathcal{A})\] be the morphism of $\mathbb{A}$--modules defined by
				\begin{align}\label{eqn: K def}
					K(\theta):= - \sum_{\mathsf t\in\pi_0(\partial_v N^o)} s_\theta({\bf x}_{\mathsf t}) \cdot [\mathsf t] \in H_0(\partial_v N^o; G^\mathcal{A}),
				\end{align}
				where $\mathbf{x}_{\mathsf t}=(x_{\mathsf t,1}, x_{\mathsf t,2}, x_{\mathsf t,3})\in\wt\Delta^o$ is some (any) triple associated to the vertical boundary component $\mathsf t$ of $ N^o$, as defined in Section \ref{sec: cohom int}.
			\end{itemize}
			Recall that
			\[\partial:H_1(N^o,\partial_v N^o;G^{\mathcal{A}})\to H_0(\partial_v N^o;G^{\mathcal{A}})\] 
			is the usual boundary map.
			
			\begin{proposition}\label{prop: E injective}
				The image of $E$ given by Equation \eqref{eqn: E} is
				\[\{ (c, \theta) \in H_1(N^o,\partial_v N^o;G^{\mathcal{A}}) \times \mathcal{T}_d(\lambda;G) :\iota_*(c)=-\widehat{c}\text{ and }\partial c = K(\theta) \}.\]
			\end{proposition}
			
			\begin{proof}
				We established above that 
				\[\mathcal{Y}_d(\lambda;G)=\{ (\alpha, \theta) \in \mathcal{Z}(\lambda^o,\mathrm{slits};G^\mathcal A) \times \mathcal{T}_d(\lambda;G) :\iota'(\alpha)=\alpha\text{ and }\partial \alpha = - s_\theta \}.\]
				By the definition of $E$, it now suffices to verify that for any $\alpha\in\mathcal{Z}(\lambda^o,\mathrm{slits};G^\mathcal A)$ and $\theta\in\mathcal{T}_d(\lambda;G)$,
				\begin{itemize}
					\item  $\iota'(\alpha)=\alpha$ if and only if $\widehat{\iota_*}([\alpha])=-[\alpha]$, and
					\item $\partial \alpha = - s_\theta$ if and only if $\partial[\alpha] = K(\theta)$.
				\end{itemize} 
				
				For the former, observe that we have the following commuting diagram:
				\begin{displaymath}
					\xymatrix{
						{\mathcal{Z}(\lambda^o,\mathrm{slits};G^{\mathcal{A}})} \ar[r]^{\varsigma_{G^\mathcal A}} \ar[d]_{\iota'}
						& {H_1(N^o,\partial_vN^o;G^{\mathcal{A}})}\ar[d]^{-\widehat{\iota_*}}\\
						\mathcal{Z}(\lambda^o,\mathrm{slits};G^{\mathcal{A}}) \ar[r]^{\varsigma_{G^\mathcal A}} 
						&  {H_1(N^o,\partial_vN^o;G^{\mathcal{A}})}.
					}
				\end{displaymath}
				As such, for any $\alpha\in\mathcal{Z}(\lambda^o,\mathrm{slits};G^\mathcal A)$, $\iota'(\alpha)=\alpha$ if and only if $\widehat{\iota_*}([\alpha])=-[\alpha]$. For the latter, note that by Proposition~\ref{prop: homological boundary}, we have that for any $\alpha\in\mathcal{Z}(\lambda^o,\mathrm{slits};G^\mathcal A)$,
				\[\partial [\alpha]=\sum_{\mathsf t\in\pi_0(\partial_v N^o)}\partial\alpha({\bf x}_{\mathsf t})\cdot [\mathsf t].\]
				By definition, 
				\[- \sum_{\mathsf t\in\pi_0(\partial_v N^o)} s_\theta({\bf x}_{\mathsf t})\cdot [\mathsf t]= K(\theta).\]
				It follows that $\partial\alpha= - s_\theta$ if and only if $\partial [\alpha]=K(\theta)$.
			\end{proof}
			
			\subsection{Complex structure on \texorpdfstring{$\mathcal{Y}_d(\lambda;\mathbb{C}/2\pi i\mathbb{Z})$}{Y(l,d,C)}} \label{sec: complex structure}
			
			Using Proposition \ref{prop: E injective}, we define a complex structure on $\mathcal{Y}_d(\lambda;\mathbb{C}/2 \pi i \mathbb{Z})$. More precisely, we prove the following.
			
			\begin{proposition}\label{complex_cocyclic}
				The Abelian group $\mathcal{Y}_d(\lambda;\mathbb{C}/2 \pi i \mathbb{Z})$ carries the structure of a complex Lie group with the following property: For any open set $U \subseteq \mathbb{C}^k$, a map
				\[
				\begin{matrix}
					U & \longrightarrow & \mathcal{Y}_d(\lambda;\mathbb{C}/2 \pi i \mathbb{Z}) \\
					z & \longmapsto & (\alpha_z, \theta_z)
				\end{matrix}
				\]
				is holomorphic if and only if for any pair of plaques ${\bf T}\in\wt\Delta^{2*}$, labeling ${\bf x} \in \wt{\Delta}^o$, pair $\mathbf{i} \in \mathcal{A}$, and triple $\mathbf{j} \in \mathcal{B}$, the functions 
				\[
				z \longmapsto \alpha_z ({\bf T},{\bf i}) \in \mathbb{C}/2 \pi i \mathbb{Z} \qquad\text{ and }\qquad z \longmapsto \theta_z(\mathbf{x},\mathbf{j}) \in \mathbb{C}/2 \pi i \mathbb{Z}
				\]
				are holomorphic.
			\end{proposition}
			
			\begin{proof}
				In this proof, we set $G=\mathbb{C} /2 \pi i \mathbb{Z}$. First, observe that as Lie groups, $\mathcal{T}_d(\lambda;G)\cong G^{\Delta \times \mathcal{B}}$, where $\Delta$ is the set of plaques of $\lambda$. In fact, for any plaque $\check{T}$ of $\lambda$, choose a lift $T$ to $\wt S$ of $\check{T}$, and choose a counterclockwise labeling $\mathbf{x}_T \in \wt{\Delta}^o$ of the vertices of $T$. By the symmetries of the map $\theta$ and its $\Gamma$--invariance, the map
				\[
				\begin{matrix}
					\mathcal{T}_d(\lambda;G) & \longrightarrow & G^{\Delta \times \mathcal{B}} \\
					\theta & \longmapsto & (\theta(\mathbf{x}_T, \mathbf{j}))_{(\check{T},\mathbf{j}) \in \Delta \times \mathcal{B}}
				\end{matrix}
				\]
				is a well-defined group isomorphism. In particular, $\mathcal{T}_d(\lambda;G)$, and hence $H_1(N^o,\partial_v N^o;G^\mathcal{A})\times\mathcal{T}_d(\lambda;G)$, carries the structure of an Abelian complex Lie group. 
				
				Proposition \ref{prop:universal_coeff} gives the equality
				\[H_1(N^o,\partial_v N^o; G^\mathcal{A})=\bigoplus_{l=1}^{2n}G^\mathcal{A}\cdot[\mathsf{k}_l^*],\]
				where $\left\{[\mathsf{k}_l^*]\right\}^{2n}_{l=1}$ is the standard generating set of $H_1(N^o,\partial_vN^o;G^\mathcal A)$. Thus, the involutions $\iota_*$ and $\widehat{\cdot}$ defined in Equations \eqref{eqn: iota} and \eqref{eqn: hat} respectively are holomorphic, so the homomorphism 
				\[f_1:H_1(N^o,\partial_v N^o; G^\mathcal{A})\times\mathcal{T}_d(\lambda;G)\to H_1(N^o,\partial_v N^o; G^\mathcal{A})\]
				given by $f_1(c,\theta)=\widehat{c}+\iota_*(c)$ is holomorphic. Also, observe that the boundary map 
				\[\partial:H_1(N^o,\partial_v N^o;G^{\mathcal{A}})\to H_0(\partial_v N^o;G^\mathcal{A}) \cong G^{12 \abs{\mathcal{A}} \abs{\chi(S)}}\]
				and the map 
				\[K : \mathcal{T}_d(\lambda;G) \longrightarrow H_0(\partial_v N^o;G^\mathcal{A})\]
				defined by Equation \eqref{eqn: K def} are holomorphic group homomorphisms, so the homomorphism
				\[f_2:H_1(N^o,\partial_v N^o;G^\mathcal{A})\times\mathcal{T}_d(\lambda;G)\to H_0(\partial_v N^o; G^\mathcal{A})\]
				given by $f_2(c,\theta)=\partial c-K(\theta)$ is holomorphic. It follows that
				\[\mathcal{H}:=\{ (c, \theta) \in H_1(N^o,\partial_v N^o;G^\mathcal{A}) \times \mathcal{T}_d(\lambda;G) :f_1(c,\theta)=0\text{ and }f_2(c,\theta)=0\},\]
				is a complex Lie subgroup of $H_1(N^o,\partial_v N^o;G^\mathcal{A}) \times \mathcal{T}_d(\lambda;G)$.
				
				Since Proposition \ref{prop: E injective} implies that the map $E:\mathcal{Y}_d(\lambda;G)\to\mathcal{H}$ is an isomorphism of groups, we may pull back the complex structure on $\mathcal{H}$ onto $\mathcal{Y}_d(\lambda;G)$. By Equation \eqref{eqn: intersection}, we see that
				\[E(\alpha,\theta)=\left(\sum_{j=1}^{2n}F_\alpha(k_j)[\mathsf{k}_j^*],\theta\right),\]
				where $\left\{[\mathsf{k}_1], \dots, [\mathsf{k}_{2n}]\right\}$ and $\left\{[\mathsf{k}_1^*], \dots, [\mathsf{k}_{2 n}^*]\right\}$ are standard generating sets of $H_1(N^o,\partial_h N^o;\mathbb{Z})$ and $H_1(N^o,\partial_v N^o;\mathbb{Z})$ respectively, and $k_j$ is the unoriented path in $N^o$ obtained by forgetting the orientation on the path $\mathsf{k}_j$. The claimed description of the complex structure on $\mathcal{Y}_d(\lambda;G)$ then follows from this, the observation that $F_\alpha(k_j)$ varies holomorphically with $\alpha$ for each $j\in\{1,\dots,2n\}$,  and the cocycle boundary condition in Definition~\ref{def_cocycle}.
			\end{proof}
			
			\begin{remark}
				Notice that the second part of the assertion in Proposition \ref{complex_cocyclic} uniquely determines the complex structure on $\mathcal{Y}_d(\lambda;\mathbb{C}/2 \pi i \mathbb{Z})$. For any choice of a train track neighborhood $N$ carrying $\lambda$, the components of the isomorphism
				\[H_1(N^o,\partial_v N^o; G^\mathcal{A})\longrightarrow\bigoplus_{l=1}^{2n}G^\mathcal{A}\cdot[\mathsf{k}_l^*]\]
				are of the form $\alpha \mapsto \alpha({\bf T})$ for some ${\bf T} \in \widetilde{\Delta}^{2*}$, and similarly for the components of the isomorphism $\mathcal{T}_d(\lambda;G) \to G^{\Delta \times \mathcal{B}}$. In particular, the structure described above is independent of the choice of the train track neighborhood $N$.
			\end{remark}

\section{Lengths and the surjectivity of \texorpdfstring{$\mathcal{sb}_{d,{\bf x}}$}{sbdx}} \label{length}

Recall that $\mathcal{A}$ denotes the set of pairs of positive integers that sum to $d$. In Section \ref{subsec:length_functions} we defined, for every pair ${\bf i}\in\mathcal{A}$ and every $\lambda$--Borel Anosov representation $\rho:\Gamma\to\mathsf{PGL}_d(\mathbb{C})$, a real-valued length function 
\[\ell_{\bf i}^\rho:\mathsf{M}(\lambda^o)\to\mathbb{R},\]
where $\mathsf{M}(\lambda^o)$ denotes the set of transverse measures whose support lies in $\lambda^o$ (see Section~\ref{subsec:length_functions}). The goal of this section is to use the homological interpretation of $\lambda$--cocyclic pairs described in Section \ref{sub:complex} to give a homological description of the length functions $\ell^\rho_{\bf i}$. This will in turn allow us to characterize the image of the shear-bend map defined in Section \ref{sec: statement of main theorem}.

To give a more precise statement, first recall that $\mathcal{Z}(\lambda^o;\mathbb{R})$ denotes the vector space of relative $\lambda^o$--cocycles (see Definition \ref{def: non-relative cocycles}). Notice that every transverse measure $\mu\in \mathsf{M}(\lambda^o)$ determines a $\lambda^o$--cocycle in $\mathcal{Z}(\lambda^o;\mathbb{R})$, which we also denote by $\mu$, as follows. For any distinct pair $(T_1,T_2)$ of plaques of $\lambda$, let $\mathsf{k}$ be some (any) geodesic in $\widetilde S$ with backward and forward endpoints in $T_1$ and $T_2$ respectively, and let $\mathcal{G}(\mathsf{k})$ denote the set of oriented geodesics in $\widetilde S$ that intersect $\mathsf{k}$ transversely and pass from the left to the right of $\mathsf{k}$. Then
\begin{equation}\label{fsdjknkjsdf} \mu(T_1,T_2):=\mu(\mathcal{G}(\mathsf{k})),\end{equation}
where $\mu$ on the right is a transverse measure in $\mathsf{M}(\lambda^o)$ and $\mu$ on the left is the $\lambda^o$--cocycle in $\mathcal{Z}(\lambda^o;\mathbb{R})$.

Next, recall that for any Abelian group $G$, $\mathcal{Y}_d(\lambda;G)$ denotes the space of $G$--valued $\lambda$--cocyclic pairs, see Definition \ref{def_cocycle}. We previously observed (see Section \ref{sub:complex}) that for all $\lambda$--cocyclic pairs $(\alpha,\theta)\in \mathcal{Y}_d(\lambda;\mathbb{C}/2\pi i\mathbb{Z})$, we may view $\alpha$ as a relative $\lambda^o$--cocycle with values in $(\mathbb{C}/2\pi i\mathbb{Z})^{\mathcal{A}}$, and so its real part $\mathrm{Re}\;\alpha$ is a relative $\lambda^o$--cocycle with values in $\mathbb{R}^{\mathcal{A}}$. The following is the main theorem of this section.

\begin{theorem}\label{thm: two lengths}
	Let $\rho$ be a $d$--pleated surface with pleating locus $\lambda$, and let $(\alpha_\rho,\theta_\rho)\in\mathcal{Y}_d(\lambda;\mathbb{C}/2\pi i\mathbb{Z})$ be its shear bend coordinate. Then for any transverse measure $\mu\in\mathsf M(\lambda^o)$ and any pair ${\bf i}\in\mathcal{A}$, 
	\[\ell_{\bf i}^\rho(\mu)=\mathbb{I}([\mu],[\mathrm{Re}\;\alpha_\rho])_{\bf i},\]
	where for some (any) choice of train track neighborhood $N$ of $\lambda$,
	\begin{itemize}
		\item $[\mu]\in H_1(N^o,\partial_h N^o;\mathbb{R})$ is the relative homology class determined by $\mu$ viewed as a $\lambda^o$--cocycle, see Section \ref{sec: cocycles};
		\item $[\mathrm{Re}\;\alpha_\rho]\in H_1(N^o,\partial_v N^o;\mathbb{R}^{\mathcal{A}})$ is the relative homology class determined by $\alpha_\rho$ viewed as a relative $\lambda^o$--cocycle, see Section \ref{sub:complex}; and
		\item $\mathbb{I}=\mathbb{I}_{\mathbb{R}^{\mathcal{A}}}:H_1(N^o,\partial_h N^o;\mathbb{R})\times H_1(N^o,\partial_v N^o;\mathbb{R}^{\mathcal{A}})\to\mathbb{R}^{\mathcal{A}}$ is the intersection pairing defined by Proposition \ref{prop: duality}.
	\end{itemize}
\end{theorem}

Bonahon and Dreyer \cite[Theorem 7.5]{BoD} proved Theorem \ref{thm: two lengths} in the special case when $\rho$ is a $d$--Hitchin representation. This was a key step they used to characterize the image $\mathcal C_d(\lambda)$ of their parameterization $\mathfrak s_d$ of the $d$--Hitchin component ${\rm Hit}_d(S)$, see Theorem \ref{thm: BD par}. In fact, they showed that
\begin{align}\label{polytope description}
	\mathcal{C}_d(\lambda)=\left\{(\alpha,\theta)\in\mathcal{Y}_d(\lambda;\mathbb{R}):
	\begin{array}{l}
		\mathbb{I}([\mu],[\alpha])_{\bf i}>0\text{ for all }{\bf i}\in\mathcal{A}\\
		\text{ and all non-zero }\mu\in\mathsf M(\lambda^o)
	\end{array}\right\}.
\end{align}
Since the intersection pairing $\mathbb{I}$ is $\mathbb{R}$--bilinear, it that $\mathcal C_d(\lambda)\subset \mathcal{Y}_d(\lambda;\mathbb{R})$ is an open convex polyhedral cone.

Similarly, as a consequence of Theorem \ref{thm: two lengths} and Lemma \ref{lem:lengths positive}, we deduce the following corollary. Recall that $\mathfrak R_d(\lambda)$ denotes the space of conjugacy classes of $d$--pleated surfaces with pleating locus $\lambda$, and $\mathfrak{sb}_d:\mathfrak R_d(\lambda)\to \mathcal{Y}_d(\lambda;\mathbb{C}/2\pi i\mathbb{Z})$ is the shear-bend map.

\begin{corollary}\label{cor: surject} The image of $\mathfrak{sb}_d$ lies in $\mathcal{C}_d(\lambda)+i\mathcal{Y}_d(\lambda;\mathbb{R}/2\pi\mathbb{Z})$.
\end{corollary}

This finishes Step 3 in the outline of the proof of Theorem \ref{thm: main}, see Section \ref{subsec:proof structure}.
	
	The remainder of this section is dedicated to the proof of Theorem \ref{thm: two lengths}. The proof largely follows Bonahon and Dreyer's argument \cite[Section 7]{BoD} for the special case of Hitchin representations. However, we have to be much more careful in our arguments since the $\lambda$--limit maps of $d$--pleated surfaces have much weaker hyperconvexity and regularity properties than the limit maps of Hitchin representations.  See the description of the proof of Step 3 in Section \ref{subsec:proof structure} for more details. 
	
	An important tool used to prove Theorem \ref{thm: two lengths} is the notion of a topological closed $1$--form. We will discuss this notion below in Section \ref{sec: forms}, before describing a sketch of the argument and the organization of the remainder of this section.
	
	\subsection{Topological closed \texorpdfstring{$1$}{1}-forms}\label{sec: forms}
	
	To describe the proof of Theorem \ref{thm: two lengths}, we need to discuss the notion of a topological closed $1$--form. 
	
	\begin{definition}
		A \emph{topological closed $1$--form} $\Omega$ on a topological space $X$ is the data of a family $\{(V_i, F_i)\}_{i \in I}$ satisfying the following conditions:
		\begin{enumerate}
			\item $V_i$ is an open set of $X$ for every $i \in I$, and $\bigcup_{i \in I} V_i = X$;
			\item $F_i : V_i \rightarrow \mathbb{R}$ is a continuous function for every $i \in I$;
			\item For every $i, j \in I$ such that $V_i \cap V_j \neq \emptyset$, the function $F_i - F_j : V_i \cap V_j \rightarrow \mathbb{R}$ is locally constant. 
		\end{enumerate}
		We refer to $\{V_i\}_{i\in I}$ as the \emph{open cover associated to $\Omega$}. 
	\end{definition}
	
	Two topological closed $1$--forms $\{(V_i, F_i)\}_{i \in I}$ and $\{(U_j, H_j)\}_{j \in J}$ are \emph{equivalent} if for every $i \in I$ and $j \in J$ such that $V_i \cap U_j \neq \emptyset$, the function $F_i - H_j : V_i \cap U_j \rightarrow \mathbb{R}$ is locally constant. If $X$ is endowed with the structure of metric space, then a topological closed $1$--form $\{ (V_i, F_i)\}_{i \in I}$ is \emph{H\"older} if $F_i$ is H\"older continuous for every $i \in I$.
	
	A topological closed $1$--form can be thought of as a generalization of the notion of a closed $1$--form on a smooth manifold, which allows for weaker regularity properties. Indeed, if $X$ is a smooth manifold and $\omega$ is a closed $1$--form on $X$, then $\omega$ defines a topological closed $1$--form: We may take $\{V_i\}_{i\in I}$ to be a cover of $X$ by contractible open sets. Then choose $p_i\in V_i$ for each $i\in I$, and define $F_i:V_i\to\mathbb{R}$ by
	\[F_i(x)=\int_{\mathsf{k}}\omega,\]
	where $\mathsf{k}$ is any differentiable path in $V_i$ from $p_i$ to $x$. (The assumption that $\omega$ is closed implies that $F_i$ is well-defined.) 
	
	Given a topological closed $1$--form $\Omega = \{(V_i, F_i)\}_{i \in I}$ on $X$ and a continuous path $\mathsf{k} : [a,b] \rightarrow X$, let 
	\[a = t_0 < t_1 < \cdots < t_n < t_{n + 1} = b\] 
	be a partition of $[a,b]$ such that for all $j$, $\mathsf{k}([t_j, t_{j + 1}])$ is contained in $V_{i_j}$ for some $i_j \in I$, and define
	\[
	\int_{\mathsf{k}} \Omega:= \sum_{j = 0}^n \Big(F_{i_j}(\mathsf{k}(t_{j + 1})) - F_{i_j}(\mathsf{k}(t_{j}))\Big).
	\]
	It is straightforward to verify that $\int_{\mathsf{k}} \Omega$ does not depend on the choices of $V_{i_j}$ containing $\mathsf{k}([t_j, t_{j + 1}])$, nor on the selected partition of the interval. 
	
	Observe that the integral of the topological closed $1$--form $\Omega$ along any continuous path $\mathsf{k}$ in $X$ is invariant under homotopy of $\mathsf{k}$ relative to its endpoints. This implies that $\int_{\mathsf{c}}\Omega=\int_{\mathsf{c}'}\Omega$ for any pair of homotopic, oriented loops $\mathsf{c}$ and $\mathsf{c}'$ in $X$.  In particular, if we fix a train track neighborhood $N$ of the geodesic lamination $\lambda$, then every topological closed $1$--form $\Omega$ on $N^o$ defines a cohomology class 
	\[[\Omega]\in H^1(N^o;\mathbb{R})=\Hom(H_1(N^o;\mathbb{R}),\mathbb{R})\]
	by $\langle [\Omega],[\mathsf{c}]\rangle:=\int_{\mathsf{c}} \Omega$ for every $[\mathsf{c}]\in H_1(N^o;\mathbb{R})$ and some (any) oriented loop $\mathsf{c}$ in $N^o$ that represents $[\mathsf{c}]$. Note that if $\Omega$ and $\Omega'$ are equivalent topological closed $1$--forms, then their integrals along any continuous path agree, so $[\Omega]=[\Omega']$.
	
	\begin{remark}
		Recall that an \emph{$\mathbb{R}$--local system} on a topological space $X$ is a flat $\mathbb{R}$--principal bundle over $X$. While this notion will not be used in this paper, we remark that topological closed $1$--forms on $X$ are closely related to $\mathbb{R}$--local systems on $X$ by a natural bijection
		\[\left\{\begin{array}{l}\text{topological closed}\\ \text{$1$--forms on }X\end{array}\right\}_{\big/\sim}\:\:\longleftrightarrow\:\:\left\{(\mathcal{L},\sigma)\bigg|\begin{array}{l}\mathcal{L}\text{ is a $\mathbb{R}$--local system on }X,\\  \sigma \text{ is a continuous section of $\mathcal L$}\end{array}\right\}_{\big/\sim},\]
		where the equivalence on the topological closed $1$--forms on the left is as described above, and two pairs $(\mathcal{L},\sigma)$ and $(\mathcal{L}',\sigma')$ on the right are equivalent if there exists an isomorphism $F : \mathcal{L} \to \mathcal{L}'$ of $\mathbb{R}$--local systems that sends $\sigma$ to  $\sigma'$. 	
		
		The above bijection can be described as follows. For any $[(\mathcal{L},\sigma)]$ in the right, choose a representative $(\mathcal{L},\sigma)$ in this equivalence class, and a collection of flat, local trivializations $\{\phi_i : \mathcal{L}|_{U_j} \to U_j \times \mathbb{R}\}_{i \in I}$ that verify $\bigcup_{i \in I} U_i = X$. Then define $F_i : U_i \to \mathbb{R}$ so that $\phi_i(\sigma(x)) = (x,F_i(x))$ for every $x \in U_i$ and $i \in I$, and set $\Omega:=\{(U_i,F_i)\}_{i\in I}$. One can verify that $\Omega$ is a topological closed $1$--form, and that the assignment $[(\mathcal{L},\sigma)]\mapsto[\Omega]$ is a well-defined bijection.
	\end{remark}
	
	By Lefschetz-Poincar\'e duality, we have an isomorphism
	\[D:H_1(N^o;\mathbb{R})\to H^1(N^o,\partial N^o;\mathbb{R}).\]
	Furthermore, the pairing 
	\[\mathbb{I}_\ast:H^1(N^o,\partial N^o;\mathbb{R})\times H^1(N^o;\mathbb{R})\xrightarrow[]{\smile} H^2(N^o,\partial N^o;\mathbb{R})\xrightarrow[]{\langle\cdot,[N^o]\rangle} \mathbb{R},\]
	where the first map is the cup product and the second map is the pairing with the relative fundamental class $[N^o]$, satisfies
	\[\mathbb{I}_\ast(D([\alpha]),[\beta])=\langle [\alpha],[\beta]\rangle\]
	for all $[\alpha]\in H_1(N^o;\mathbb{R})$ and $[\beta]\in H^1(N^o;\mathbb{R})$. In particular, if $\mu\in \mathsf{M}(\lambda^o)$ is a transverse measure and $\Omega$ is a topological closed $1$--form on $N^o$, then
	\begin{equation}\label{eqn: two pairings}
		\mathbb{I}_\ast(D([\mu]),[\Omega])=\langle [\mu],[\Omega]\rangle,
	\end{equation}
	where $[\mu]\in H_1(N^o;\mathbb{R})$ is the homology class associated to $\mu$ viewed as a $\lambda^o$--cocycle (see Section \ref{sec: cocycles} and Equation \eqref{fsdjknkjsdf}) and $[\Omega]\in H^1(N^o;\mathbb{R})$ is the cohomology class associated to $\Omega$ described above.
	
	In light of Equation (\ref{eqn: two pairings}), we prove Theorem \ref{thm: two lengths} via the following strategy. 
	
	\begin{itemize}
		\item First, for every $\lambda$--Borel Anosov representation $\rho:\Gamma\to\mathsf{PGL}_d(\mathbb{C})$ and ${\bf i}\in\mathcal{A}$, we construct a topological closed $1$--form $\Omega^\rho_{\bf i}$ on $N^o$ so that 
		\begin{equation}\label{eqn: required condition 1}
			\mathbb{I}_\ast(D([\mu]),[\Omega^{\rho}_{\bf i}])=\ell^\rho_{\bf i}(\mu)
		\end{equation}
		for all $\mu\in\mathsf{M}(\lambda^o)$. 
		\item Then, we prove that if $\rho\in\mathcal{R}_d(\lambda)$, then
		\begin{equation}\label{eqn: required condition 2}
			\langle [\mu],[\Omega^{\rho}_{\bf i}]\rangle=\mathbb{I}([\mu],[\mathrm{Re}\;\alpha_\rho])_{\bf i}
		\end{equation}
		for all transverse measures $\mu\in\mathsf{M}(\lambda^o)$, where on the right, the intersection pairing $\mathbb{I}$ and the homology classes $[\mu]\in H_1(N^o,\partial_h N^o;\mathbb{R})$ and $[\mathrm{Re}\;\alpha_\rho]\in H_1(N^o,\partial_v N^o;\mathbb{R}^{\mathcal{A}})$ were defined in the statement of Theorem \ref{thm: two lengths}.
	\end{itemize}
	Together, Equations \eqref{eqn: two pairings}, \eqref{eqn: required condition 1}, and \eqref{eqn: required condition 2} imply Theorem~\ref{thm: two lengths}.
	
	The construction of $\Omega^{\rho}_{\bf i}$ is quite technical, and will be done in Section \ref{sec: construction 1 form} following these steps:
	\begin{enumerate}
		\item[Step 1:] Define a deformation retract $r:N\to N_{\mathrm{ret}}$ so that $N_{\mathrm{ret}}$ contains $\lambda$, and admits a foliation $\mathsf{F}$ by (subsegments of) geodesics such that the leaves of $\lambda$ are leaves of $\mathsf{F}$. Then $\mathsf{F}$ induces foliations $\wt{\mathsf{F}}$ on $\wt N_{\mathrm{ret}}:=\pi_S^{-1}(N_{\mathrm{ret}})$ and $\mathsf{F}^o$ on $N^o_{\mathrm{ret}}:=\pi_N^{-1}(N_{\mathrm{ret}})$, where $\pi_S:\wt S\to S$ and $\pi_N:N^o\to N$ are the covering maps, see Section \ref{or-cov-tt}.
		\item[Step 2:] Extend the slithering map $\Sigma:\wt\Lambda^2\to\mathsf{SL}_d(\mathbb{C})$ that is compatible with the $\lambda$--limit map of $\rho$ to a map $\widehat{\Sigma}:\wt{\mathsf{F}}{}^{(2)}\to\mathsf{SL}_d(\mathbb{C})$, where $\wt{\mathsf{F}}{}^{(2)}$ is a certain subset of $\wt{\mathsf{F}}{}^2$ that contains $\wt\Lambda^2$, see Definition \ref{def: well-positioned}.
		\item[Step 3:] Extend, using Step 2, the ${\bf i}$--th homomorphism bundle $\wh H_{\bf i}\to\lambda^o\cong T^1\lambda$ defined in Section \ref{ssec:BorelAnosov} to a line bundle $\widehat{H}_{\bf i,\mathrm{ret}}\to N^o_{\mathrm{ret}}$.
		\item[Step 4:] Define a topological closed $1$--form on $N^o_{\mathrm{ret}}$ using Step 3, which we then pullback via $r$ to get the topological closed $1$--form $\Omega^\rho_{\bf i}$ on $N^o$.
	\end{enumerate} 
	Then in Section \ref{sec: properties of 1 form} and Section \ref{sec: properties of 1 form 2}, we prove Equation \eqref{eqn: required condition 1} and Equation \eqref{eqn: required condition 2} respectively.

	\subsection{The construction of \texorpdfstring{$\Omega^\rho_{\bf i}$}{Ori}}\label{sec: construction 1 form}
	
	For the remainder of this section, fix a $\lambda$--Borel Anosov representation $\rho:\Gamma\to\mathsf{PGL}_d(\mathbb{C})$, let $\xi\colon\partial\wt \lambda\to \mathcal{F}(\mathbb{C}^d)$ denote the $\lambda$--limit map of $\rho$, and fix ${\bf i}\in\mathcal{A}$. We now construct the topological closed 1-form $\Omega=\Omega^{\rho}_{\bf i}$ following the steps outlined above.
	
	\subsubsection{Defining the deformation retract of \texorpdfstring{$N$}{N}}\label{sec: retract}
	
	First, we define the subset $N_{\mathrm{ret}}\subset N$ and describe the deformation retract 
	\[r:N\to N_{\mathrm{ret}}.\] 
	The set $N_{\mathrm{ret}}$ is constructed as the union of the lamination $\lambda$ with a collection of regions, three per each plaque $T$ of $\lambda$, delimited by the sides of $T$ and three horocyclic segments joining distinct pairs of edges.
	
	More precisely, let $\pi_N:N^o\to N$ and $\pi_S:\wt S\to S$ be the covering maps, and recall $\widetilde N:=\pi_S^{-1}(N)$. For any plaque $T$ of $\wt \lambda$, choose pairwise disjoint closed horoballs $H_{T,1}$, $H_{T,2}$, $H_{T,3}$ centered at the vertices of $T$, such that $T\cap H_{T,i}\subset T\cap \wt N$ for all $i$, and 
	\[\gamma\cdot (H_{T,1}\cup H_{T,2}\cup H_{T,3})=H_{\gamma\cdot T,1}\cup H_{\gamma\cdot T,2}\cup H_{\gamma\cdot T,3}\]
	for all $\gamma\in\Gamma$. Then set $U_T:=(\wt N\cap T)- (H_{T,1}\cup H_{T,2}\cup H_{T,3})$ and define
	\[\widetilde{N}_{\rm ret}:= \widetilde N-\bigcup_{T\in\widetilde\Delta}U_T,\]
	see Figure \ref{fig: retract}. It is clear that $\widetilde N_{\rm ret}$ is a $\Gamma$--invariant subset of $\widetilde N$, so it descends to a subset $N_{\rm ret}\subset N$. Define $N^o_{\mathrm{ret}}:=\pi_N^{-1}(N_{\mathrm{ret}})$ and 
	\[\wt N^o_{\mathrm{ret}}:=\{(p,q)\in\wt N_{\rm ret}\times N^o_{\rm ret}:\pi_S(p)=\pi_N(q)\}.\]
	Then the projection
	\[\pi:\wt N^o_{\rm ret}\to N^o_{\rm ret}\]
	given by 
	$\pi(p,q)=q$ is also a covering map. Observe that $\wt\lambda\subset\wt N_{\rm ret}$, so $\lambda\subset N_{\mathrm{ret}}$, $\lambda^o=\pi_N^{-1}(\lambda)\subset N^o_{\mathrm{ret}}$, and we may define $\wt\lambda^o:=\pi^{-1}(\lambda^o)\subset\wt N^o_{\mathrm{ret}}$.

	\begin{figure}[h!]
		\includegraphics[scale=1]{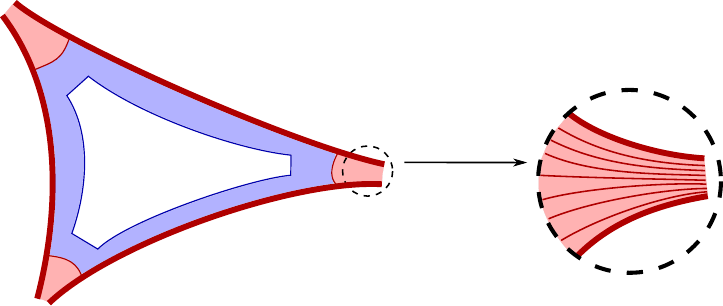}
		\put(-330,140){\small $N_{\rm ret}\cap T$}
		\put(-306,75){\small $U_T$}
		\put(-182,62){\small $H_{T,i}$}
		\caption{\small On the left, a plaque $T$ and the boundary of the horocycles $H_{T,i} \cap T$ for $i=1,2,3$. The red region corresponds to $N_{\rm ret}\cap T$, whereas the blue region is $U_T$. On the right, a wedge and its foliation in $\widetilde{\mathsf F}$.}\label{fig: retract}
	\end{figure}
	
	For each plaque $T$ of $\widetilde\lambda$, let $e_1$, $e_2$, and $e_3$ be the edges of $T$, and set 
	\[K_T:=\bigcup_{i=1}^3 (H_{T,i}\cap T)\cup\bigcup_{i=1}^3e_i\subset\widetilde N_{\rm ret}.\]
	Note that $U_T\cup K_T=\wt N\cap \overline{T}$,  where $\overline{T}$ denotes the closure of $T$, and there is a deformation retract 
	\[\wt r_T:\wt N\cap \overline T\to K_T.\] 
	We may choose a collection $\{\wt r_T:T\in\widetilde\Delta\}$ such that $\gamma\cdot \wt r_T(x)=\wt r_{\gamma\cdot T}(\gamma\cdot x)$ for all $\gamma\in\Gamma$, all plaques $T$ of $\wt\lambda$, and all $\wt N\cap \overline T$. Then the map
	\[\wt r:\wt N\to\wt N_{\rm ret}\]
	defined by $\wt r(x)=x$ if $x\in \wt N_{\rm ret}$ and $\wt r(x)=\wt r_T(x)$ if $x\in U_T$ is a $\Gamma$--equivariant deformation retract, so it descends to the required deformation retract
	\[r\colon N\to N_{\text{ret}}.\] 
	
	We refer to any connected component of $\wt N_{\mathrm{ret}}-\wt \lambda$ as a \emph{wedge} of $\wt N_{\mathrm{ret}}$. Observe that every plaque of $\wt\lambda$ contains exactly three wedges of $\wt N_{\rm ret}$, each of which is the intersection of $T$ with a horoball centered at a vertex of $T$. As such, the boundary of every wedge of $N_{\mathrm{ret}}$ is the union of a subsegment of a horocycle and a pair of asymptotic geodesic rays. 
	Each of the geodesic rays lies in a leaf of $\lambda$, which we call the \emph{boundary leaves} of the wedge. We define \emph{wedges} and {\em boundary leaves} analogously for the sets $N_{\mathrm{ret}}$, $N^o_{\mathrm{ret}}$, and $\wt N^o_{\mathrm{ret}}$. 
	
	Note that there is a unique foliation $\wt{\mathsf{F}}$ of $\wt N_{\mathrm{ret}}$ that contains $\wt{\Lambda}$, and has the property that each leaf $h\in\wt{\mathsf{F}}$ that is not in $\wt{\Lambda}$ lies in a geodesic of $\wt S$ that is asymptotic to both boundary leaves of the wedge that contains $h$. The foliation $\wt{\mathsf{F}}$ is $\Gamma$--invariant, and so descends to a foliation $\mathsf{F}$ on $N_{\mathrm{ret}}$. This in turn gives the foliations $\mathsf{F}^o:=\pi_N^{-1}(\mathsf{F})$ of $N^o_{\rm ret}$ and $\wt{\mathsf{F}}{}^o:=\pi^{-1}(\mathsf{F}^o)$ of $\wt N^o_{\rm ret}$. 
	
	\begin{proposition}\label{prop: Lipschitz foliation}
		The following properties hold:
		\begin{enumerate}
			\item The map $\mathsf{F}$ is Lipschitz, i.e. for every $x \in N_{\mathrm{ret}}$, there is a Lipschitz vector field on a neighborhood of $x$ in $N_{\mathrm{ret}}$ that is tangent to the leaves of $\mathsf{F}$.
			\item The leaves of $\mathsf{F}^o$ admit a continuous family of orientations that extend the natural orientations on the leaves in $\lambda^o$. \end{enumerate}
	\end{proposition}
	
	\begin{proof} Claim (1) follows from Epstein and Marden \cite[Corollary~II.2.5.2]{EM06}. To prove claim (2), observe that for any wedge $W$ of $N^o_{\rm ret}$, the natural orientation on the two boundary leaves of $W$ are either both oriented towards their common endpoint, or both oriented away from their common endpoint. Thus, we may orient every leaf $l$ of $\mathsf{F}^o$ that lies in $W$ towards the common endpoint of the boundary leaves of $W$ if the former holds, and away from this common endpoint if the latter holds. This, together with the natural orientations of the leaves in $\Lambda^o$, specify a continuous orientation on the leaves of $\mathsf{F}^o$.
	\end{proof}
	
	We refer to the continuous family of orientations on the leaves of $\mathsf{F}^o$ given in  Proposition \ref{prop: Lipschitz foliation} Claim (2) as the \emph{natural orientations}.
	
	\subsubsection{Extending the slithering map} Next, we extend the slithering map $\Sigma:\wt\Lambda^2\to\mathsf{SL}_d(\mathbb{C})$ associated to the limit map $\xi$ so that it is defined on a larger subset of $\wt{\mathsf{F}}{}^2$. To do so, we introduce the following notation that will be used in the rest of this section.
	
	For each leaf $h\in\wt{\mathsf{F}}$, we will associate to it a geodesic $l_h$ of $\wt S$, a pair of leaves $l_{h,1},l_{h,2}\in\wt\Lambda$, and a number $t_h\in (0,1)$ as follows. If $h\in\wt{\mathsf{F}}-\wt\Lambda$, then $l_h$ is the geodesic that contains $h$, and $l_{h,1}, l_{h,2}\in \wt\Lambda$ are the two boundary leaves of the wedge that contains $h$. Since $l_{h,1}$ and $l_{h,2}$ are asymptotic, there is a unique unipotent element $u_h\in\mathsf{SL}_2(\mathbb{R})$ that sends $l_{h,1}$ to $l_{h,2}$. Since $l_h$ is asymptotic to $l_{h,1}$ and $l_{h,2}$, we may set $t_h\in(0,1)$ to be the number such that $u_h^{t_h}$ sends $l_{h,1}$ to $l_h$. (Here, if $n$ is the unique nilpotent matrix such that $\exp(n)=u_h$, then $u_h^{t_h}:=\exp(t_hn)$.) 
	On the other hand, if $h\in\wt\Lambda$, we set $l_h:=l_{h,1}:=l_{h,2}:=h$ and $t_h:=\frac{1}{2}$ for notational convenience.
	
	\begin{remark}\label{rem: order}
		In the case when $h\in\wt{\mathsf{F}}-\wt\Lambda$, we have to choose which of the two boundary leaves of the wedge containing $h$ is $l_{h,1}$ and which is $l_{h,2}$. The quantity $t_h$ depends on this choice; if we switch the roles of $l_{h,1}$ and $l_{h,2}$, then $t_h$ will be replaced with $1-t_h$. In situations where we do not make this choice explicit, all statements we make will hold for both possible choices. The choice of $t_h=\frac{1}{2}$ when $h\in\widetilde\Lambda$ is to ensure that the above also holds in this case.
	\end{remark}
	
	\begin{definition}\label{def: well-positioned}
		A pair of leaves $(h,h')\in\wt{\mathsf{F}}{}^2$ is \emph{not well-positioned} if there exists a plaque $T$ of $\wt \lambda$ such that both $h$ and $h'$ are contained in the closure of $T$ in $\widetilde{S}$ and $h, h'$ do not share an endpoint at infinity. If this does not occur, then we say that the pair $(h,h')\in\wt{\mathsf{F}}{}^2$ is \emph{well-positioned}. Let $\wt{\mathsf{F}}{}^{(2)}$ be the set of well-positioned pairs of leaves in $\widetilde{\mathsf{F}}{}^2$.
	\end{definition} 
	
	Note that $\wt{\mathsf{F}}{}^{(2)}$ is a $\Gamma$--invariant subset of $\wt{\mathsf{F}}{}^2$ containing $\wt\Lambda^2$. Moreover, for any leaves $g_1,g_2\in\wt\Lambda$, we denote 
	\[Q_{\mathsf{F}}(g_1,g_2):=\{h\in\wt{\mathsf{F}}:l_{h,1},l_{h,2}\in Q(g_1,g_2)\}.\]
	Then observe that any pair of leaves in $Q_{\mathsf{F}}(g_1,g_2)$ is well-positioned.
	Furthermore, $Q_{\mathsf{F}}(g_1,g_2)\cap\wt\Lambda=Q(g_1,g_2)$, which is the set of leaves of $\wt\lambda$ that separate $g_1$ and $g_2$, see Section \ref{section: separation and coherence}.
	
	\begin{proposition} \label{prop: slithering_extension}
		There is a $\rho$--equivariant map
		\[\widehat{\Sigma}:\wt{\mathsf{F}}{}^{(2)}\to\mathsf{SL}_d(\mathbb{C})\]
		satisfying the following properties:
		\begin{enumerate}
			\item $\widehat{\Sigma}(h,h)=\id$ for all leaves $h\in \wt{\mathsf{F}}$, $\widehat{\Sigma}(h_2,h_1)=\widehat{\Sigma}(h_1,h_2)^{-1}$ for all pairs of leaves $(h_1,h_2)\in \wt{\mathsf{F}}{}^{(2)}$, and $\widehat{\Sigma}(h_3,h_2)\circ\widehat{\Sigma}(h_2,h_1)=\widehat{\Sigma}(h_3,h_1)$ if there are leaves $g_1$ and $g_2$ of $\wt\lambda$ such that $h_1,h_2,h_3\in Q_{\mathsf{F}}(g_1,g_2)$.
			\item For all leaves $g_1,g_2\in \wt\Lambda$, the restriction $\widehat{\Sigma}|_{Q_\mathsf{F}(g_1,g_2)}$ is H\"older continuous, i.e. there are constants $M,\mu>0$ (depending on $g_1$ and $g_2$) such that 
			\[\norm{\widehat{\Sigma}(h_2,h_1)-\widehat{\Sigma}(h_4,h_3)}\leq Md_\infty((l_{h_1},l_{h_2}),(l_{h_3},l_{h_4}))^\mu\]
			for all geodesics $h_1,h_2,h_3,h_4\in Q_\mathsf{F}(g_1,g_2)$.
			\item If $(h_1,h_2)\in\wt{\mathsf{F}}{}^{(2)}$ and the geodesics $l_{h_1}$ and $l_{h_2}$ share an endpoint, then $\widehat{\Sigma}(h_2,h_1)$ is unipotent.
			\item If $g_1,g_2\in\wt\Lambda$, then $\widehat{\Sigma}(g_2,g_1) = \Sigma(g_2,g_1)$, where $\Sigma$ is the slithering map compatible with the $\lambda$--limit map $\xi$ of $\rho$. (Recall that $\Sigma$ exists by Theorem~\ref{thm: slitherable and slithering}.)
		\end{enumerate}
	\end{proposition}
	
	We denote $\widehat{\Sigma}(h_2,h_1):=\widehat{\Sigma}_{h_2,h_1}$ when convenient.
	
	\begin{proof}
		First, we define the map $\widehat{\Sigma}$. For all $(h_1,h_2)\in\wt{\mathsf{F}}{}^{(2)}$, define 
		\[\widehat{\Sigma}(h_2,h_1):=\Sigma(l_{h_2,2},l_{h_2,1})^{t_{h_2}}\circ \Sigma(l_{h_2,1},l_{h_1,1})\circ\Sigma(l_{h_1,2},l_{h_1,1})^{-t_{h_1}}.\]
		By Condition (1) of Definition \ref{def: slithering map}, we have
		\begin{align*}
			&\hspace{.45cm}\Sigma(l_{h_2,2},l_{h_2,1})^{t_{h_2}}\Sigma(l_{h_2,1},l_{h_1,1})\Sigma(l_{h_1,2},l_{h_1,1})^{-t_{h_1}}\\
			&=\Sigma(l_{h_2,2},l_{h_2,1})^{t_{h_2}}\Sigma(l_{h_2,1},l_{h_1,2})\Sigma(l_{h_1,1},l_{h_1,2})^{t_{h_1}-1}\\
			&=\Sigma(l_{h_2,1},l_{h_2,2})^{1-t_{h_2}} \Sigma(l_{h_2,2},l_{h_1,1})\Sigma(l_{h_1,2},l_{h_1,1})^{-t_{h_1}}\\
			&=\Sigma(l_{h_2,1},l_{h_2,2})^{1-t_{h_2}} \Sigma(l_{h_2,2},l_{h_1,2})\Sigma(l_{h_1,1},l_{h_1,2})^{t_{h_1}-1},
		\end{align*}
		so we see that $\widehat{\Sigma}(h_2,h_1)$ is well-defined, i.e. it is invariant under switching the roles of $l_{h_i,1}$ with $l_{h_i,2}$ for both $i=1,2$ (compare with Remark \ref{rem: order}).  
		
		From this definition, it is obvious that (4) holds, and it is straightforward to verify that (1) and (3) follow from Conditions (1) and (3) of Definition \ref{def: slithering map} respectively. It remains to prove (2). 
		
		Let $g_1,g_2\in\wt\Lambda$ and fix a compact geodesic arc $k$ that is transverse to the lamination $\wt \lambda$, and intersects both $g_1$ and $g_2$. If $l_1$ and $l_2$ are geodesics in $\wt S$ that separate $g_1$ and $g_2$, denote the subarc of $k$ with endpoints in $k\cap l_1$ and $k\cap l_2$ by $k(l_1,l_2)$. To prove (2), it suffices to show that there are some $M_1,\mu>0$ such that 
		\begin{align}\label{eqn: reduced equation}
			\norm{\widehat{\Sigma}(h',h) - \id} \leq M_1 \ell(k(l_{h},l_{h'}))^{\mu}
		\end{align}
		for all leaves $h,h'\in Q_{\mathsf{F}}(g_1,g_2)$. Indeed, this implies that
		\[\norm{\widehat{\Sigma}(h',h)} \leq M_1 \ell(k)^{\mu}+1=:M_2\]
		for all $h,h'\in Q_{\mathsf{F}}(g_1,g_2)$, and applying Lemma \ref{lem: hyperbolic geometry} to Equation \eqref{eqn: reduced equation} gives a constant $C>0$ such that
		\[
		\norm{\widehat{\Sigma}(h',h) - \id} \leq M_1\left(2C\right)^{\mu}d_\infty(l_{h},l_{h'})^{\mu}
		\]
		for all $h,h'\in Q_{\mathsf{F}}(g_1,g_2)$. Then we have
		\begin{align*}
			\norm{\widehat{\Sigma}(h_2,h_1)-\widehat{\Sigma}(h_4,h_3)} &\leq \norm{\widehat{\Sigma}(h_2,h_4)}\norm{\widehat{\Sigma}(h_4,h_3)}\norm{\widehat{\Sigma}(h_3,h_1)-\id}\\
			&\hspace{1cm}+\norm{\widehat{\Sigma}(h_4,h_3)}\norm{\widehat{\Sigma}(h_2,h_4)-\id}\\
			&\leq M_2^2M_1(2C)^{\mu}\left(d_\infty(l_{h_1},l_{h_3})^{\mu}+d_\infty(l_{h_2},l_{h_4})^{\mu}\right)\\
			&\leq 2^{\mu+1}M_2^2M_1C^{\mu}d_\infty((l_{h_1},l_{h_2}),(l_{h_3},l_{h_4}))^{\mu},
		\end{align*}
		and Condition (2) follows by setting $M:=2^{\mu+1}M_2^2M_1C^\mu$.
		
		To prove Equation \eqref{eqn: reduced equation}, we will use the following two lemmas, whose proofs are given at the end of this section. The first estimates ``how far" two leaves $h_1$ and $h_2$ of $\wt{\mathsf F}$ are from each other in terms of the quantities $t_{h_1}$ and $t_{h_2}$.
		
		\begin{lemma}\label{lem:compare_lenghts}
			Fix a compact geodesic arc $k$ that is transverse to the geodesic lamination $\wt \lambda$. There exists a constant $B=B(k) > 0$ such that the following holds: Let $w_1$ and $w_2$ be the boundary leaves of a wedge of $\wt N_{\rm ret}$ such that $w_1$ and $w_2$ both intersect $k$. For any pair of leaves $h,h'\in Q_{\mathsf{F}}(w_1,w_2)$, if we choose $l_{h,1}=l_{h',1}$ (or equivalently, $l_{h,2}=l_{h',2}$), then we have
			\[  B^{-1} \abs{t_{h} - t_{h'}} \ell(k(w_1,w_2)) \leq \ell(k(l_{h},l_{h'})) \leq B \abs{t_{h} - t_{h'}} \ell(k(w_1,w_2)).\]
		\end{lemma}
		
		\begin{remark}
			Note that in the lemma above, even though the quantities $t_{h}$ and $t_{h'}$ depend on whether we choose $l_{h,1}=l_{h',1}$ to be $w_1$ or $w_2$, the quantity $|t_{h}-t_{h'}|$ does not depend on that choice, see Remark \ref{rem: order}.
		\end{remark}
		
		The second concerns the exponential map of $\mathsf{GL}_d(\mathbb{C})$ evaluated on nilpotent elements in $\mathfrak{gl}_d(\mathbb{C})$. Let $\mathcal{N}_d$ denote the set of nilpotent linear transformations of $\mathbb{C}^d$.
		
		\begin{lemma} \label{lem:nilp_unip}
			For any compact subset $K \subset \mathcal{N}_d$ and for any norm $\norm{\cdot}$ on the space $\End(\mathbb{C}^d)$ of endomorphisms of $\mathbb{C}^d$, there exists a constant $C = C(K, \norm{\cdot}) > 0$ satisfying 
			\[
			\norm{\exp(t X) - \id} \leq C t \norm{\exp(X) - \id}
			\]
			for every $t \in [0,1]$ and for every $X \in K$.
		\end{lemma}
		
		Assuming Lemmas \ref{lem:compare_lenghts} and \ref{lem:nilp_unip}, we first prove that Equation \eqref{eqn: reduced equation} holds for all leaves $h,h'\in Q_{\mathsf{F}}(g_1,g_2)$ that lie in the closure of the same wedge of $\wt N_{\mathrm{ret}}$. Fix a geodesic arc $k$ in $\widetilde S$ that intersects both $g_1$ and $g_2$. By Lemma \ref{lem: hyperbolic geometry} and Condition (2) of Definition~\ref{def: slithering map}, there exist $A, \nu > 0$ such that 
		\begin{equation}\label{eq:holder continuity}
			\norm{\Sigma(g',g) - \id} \leq A \ell(k(g,g'))^\nu
		\end{equation}
		for all $g,g' \in Q(g_1,g_2)$. Observe that there is a compact subset $K \subset \mathsf{SL}_d(\mathbb{C})$ that contains the unipotent elements $\Sigma(w_1,w_2)$ for all $w_1,w_2\in Q(g_1,g_2)$ that are the two boundary leaves of a wedge of $\wt N_{\mathrm{ret}}$. Thus, by Equation \eqref{eq:holder continuity} and Lemma~\ref{lem:nilp_unip}, there is some $C > 0$ such that the following holds: For any $h,h'\in Q_{\mathsf{F}}(g_1,g_2)$ that lie in the same wedge $W$ of $\wt N_{\mathrm{ret}}$ with boundary leaves $w_1,w_2$, if we choose $l_{h,1}=l_{h',1}=w_1$, then
		\begin{align}\label{eqn: reduce to lengths}
			\norm{\widehat{\Sigma}(h',h) - \id}= \norm{\widehat{\Sigma}(w_2,w_1)^{t_{h'}-t_{h}} - \id}  \leq A C\abs{t_{h}-t_{h'}} \ell(k(w_1, w_2))^\nu.
		\end{align}
		
		Set $\mu:=\min\{1,\nu\}$. Since the arc $k$ has finite length and $\mu \leq \nu$, there exists a constant $D > 0$ such that 
		\begin{equation} \label{eq:control_exponents}
			\ell(k')^\nu \leq D \ell(k')^{\mu}
		\end{equation}
		for any subarc $k' \subset k$. Also,  for any $h',h\in Q_{\mathsf{F}}(g_1,g_2)$ that lie in the same wedge of $\wt{N}_{\mathrm{ret}}$, we have $\abs{t_h-t_{h'}}\le 1$, so $\abs{t_h-t_{h'}} \leq \abs{t_h-t_{h'}}^{\mu}$. Thus, by Equation \eqref{eqn: reduce to lengths} and Lemma~\ref{lem:compare_lenghts}, there is some $B>0$ such that
		\begin{equation} \label{eq:extension_holder}
			\norm{\widehat{\Sigma}(h',h) - \id} \leq ACDB^{\mu} \ell(k(l_h, l_{h'}))^{\mu}.
		\end{equation}
		By taking limits, Equation \eqref{eq:extension_holder} also holds when $h',h\in Q_{\mathsf{F}}(g_1,g_2)$ lie in the closure of the same wedge. This proves Equation \eqref{eqn: reduced equation} in this case.
		
		Next, we prove Equation \eqref{eqn: reduced equation} when the geodesics $h,h'\in Q_{\mathsf{F}}(g_1,g_2)$ do not lie in the closure of the same wedge of $\wt{N}_{\mathrm{ret}}$. For any such pair $h,h'$, we may choose $l_{h,1}$ and $l_{h',1}$ so that $k$ intersects $l_h$, $l_{h,1}$, $l_{h',1}$, and $l_{h'}$ in this (weak) order with respect to some orientation on $k$. 
		Since $Q(g_1,g_2)$ is compact and $\Sigma$ is continuous (see Condition (2) of Definition~\ref{def: slithering map}), there is some $E>0$ such that 
		\begin{equation*}
			\norm{\Sigma(g,g')}\leq E\end{equation*}
		for all $g',g\in Q_{\mathsf{F}}(g_1,g_2)$. By enlarging $E$ if necessary, we can ensure that 
		\begin{equation*}
			ACDB^{\mu}\ell(k)^{\mu}+1\leq E.\end{equation*} 
		Then
		\begin{align*}
			\norm{\widehat{\Sigma}(h',h) - \id}
			& \leq \norm{\Sigma(l_{h',1},l_{h,1})}\norm{\widehat{\Sigma}(l_{h,1},h)} \norm{\widehat{\Sigma}(h',l_{h',1}) - \id} \\
			& \hspace{.5cm} + \norm{\widehat{\Sigma}(l_{h,1},h)} \norm{\Sigma(l_{h',1},l_{h,1}) - \id} + \norm{\widehat{\Sigma}(l_{h,1},h) - \id} \\
			& \leq E^2 \left(  \norm{\widehat{\Sigma}(h',l_{h',1}) - \id}  +\norm{\Sigma(l_{h',1},l_{h,1}) - \id}\right.\\
			& \hspace{.5cm} \left.+ \norm{\widehat{\Sigma}(l_{h,1},h) - \id}\right).
		\end{align*}
		By applying Equation \eqref{eq:holder continuity} to the term $\norm{\Sigma(l_{h',1},l_{h,1}) - \id}$ and Equation \eqref{eq:extension_holder} to the remaining terms, we deduce that
		\begin{align*}
			\norm{\widehat{\Sigma}(h',h) - \id} & \leq E^2 \left( A'\ell(k(l_{h'},l_{h',1}))^{\mu} + A \ell(k(l_{h',1},l_{h,1}))^{\nu} + A'\ell(k(l_{h,1},l_h))^{\mu} \right) \\
			& \leq E^2 \left( A'\ell(k(l_{h'},l_{h',1}))^{\mu} + AD\ell(k(l_{h',1},l_{h,1}))^{\mu} + A'\ell(k(l_{h,1},l_h))^{\mu} \right)  \\
			& \leq 3 E^2 \max\{A',AD\} \, \ell(k(l_{h'},l_h))^{\mu} ,
		\end{align*}
		where $A':= ACDB^{\mu}$. Here, the second inequality follows from Equation \eqref{eq:control_exponents}, and the third inequality is a consequence of the fact that $k(l_{h'},l_{h',1}),k(l_{h',1},l_{h,1}),k(l_{h,1},l_h)\subset k(l_{h'},l_h)$. This proves Equation \eqref{eqn: reduced equation} in this case.
	\end{proof}
	
	\begin{proof}[Proof of Lemma \ref{lem:compare_lenghts}]
		For the purpose of this proof, we will identify the universal cover $\widetilde{S}$ of $S$ with the upper half-space model of the hyperbolic plane $\mathbb{H}^2$, and its boundary at infinity $\partial\widetilde{S}$ with $\partial \mathbb{H}^2 = \mathbb{R} \cup \{\infty\}$ accordingly. We denote by $g_0$ the geodesic of $\mathbb{H}^2$ with endpoints $0, \infty \in \partial \mathbb{H}^2$, and we select a basepoint $p_0 \in g_0$.
		
		Every element $a\in\mathsf{PGL}_2(\mathbb{R})$ acts smoothly on $\partial\mathbb{H}^2$, so we may define
		\[m(a) :=\max_{z\in\partial\mathbb{H}^2}\norm{{\rm d}a_z},\] 
		where ${\rm d}a_z:T_z\partial\mathbb{H}^2\to T_{a(z)}\partial\mathbb{H}^2$ is the derivative of $a$ at $z$ and $\norm{\cdot}$ is the operator norm with respect to the Riemannian metric we fixed on $\partial\wt S$ (which we used to define $d_\infty$).
		
		Let $w_1$ and $w_2$ be the boundary leaves of a wedge of $\wt N_{\rm ret}$ such that both $w_1$ and $w_2$ intersect $k$. Then denote by $x$ the common endpoint of $w_1$ and $w_2$, and by $y_1$ and $y_2$ the other endpoints of $w_1$ and $w_2$ respectively. We also set $p_{w_1} \in \mathbb{H}^2$ to be the point of intersection between $w_1$ and $k$. Since the action of $\mathsf{PGL}_2(\mathbb{R})$ on the set of geodesics with basepoint in $\mathbb{H}^2$ is simply transitive, for every such $w_1, w_2$ there exists a unique isometry $a_{w_1} \in \mathsf{PGL}_2(\mathbb{R})$ that sends $p_{w_1}$ into $p_0 \in \mathbb{H}^2$, and the ordered pair of ideal points $(y_1,x)$ to $(0,\infty) \in \partial \mathbb{H}^2$. By construction, $a_{w_1}(w_1) = g_0$ and $a_{w_1}(w_2)$ has endpoints $\infty$ and $s_{w_1,w_2}$, for some $s_{w_1,w_2} \in \mathbb{R} \subset \partial \mathbb{H}^2$. Observe that the numbers $\abs{s_{w_1,w_2}}$ are bounded above by a constant that is independent of the wedge considered. Indeed, the geodesic $a_{w_1}(w_2)$ crosses the segment $a_{w_1}(k)$, which passes through the basepoint $p_0$. In particular,  $a_{w_1}(w_2)$ must intersect the closed metric ball of radius $\ell(k)$ and center $p_0$ in $\mathbb{H}^2$. For this to be possible, the endpoint of  $a_{w_1}(w_2)$ different from $\infty$ must vary in a compact region of $\mathbb{R} \subset \partial \mathbb{H}^2$. 
		
		Since $k$ is compact and since the action of $\mathsf{PGL}_2(\mathbb{R})$ on the set of oriented geodesics with basepoint of $\mathbb{H}^2$ is simply transitive, there exists a constant $D=D(k)>1$ (that is uniform over all possible choices of $w_1$ and $w_2$), such that
		\[\frac{1}{D}\le m(a_{w_1})\le D.\]
		
		For all $h,h'\in Q_{\mathsf{F}}(w_1,w_2)$, the geodesics $l_{h}$ and $l_{h'}$ share $x$ as a common endpoint. Denote the other endpoint of $l_{h}$ and $l_{h'}$ by $x$ and $x'$ respectively, and choose $l_{h,1}=l_{h',1}=w_1$. Then observe that $u_h=u_{h'}$, and that
		\[a_{w_1} \, u_h^{t_h} \, a_{w_1}^{-1} = 
		\begin{bmatrix}
			1&t_h \, s_{w_1,w_2} \\
			0&1
		\end{bmatrix}\,\,\text{ and }\,\,a_{w_1} \, u_{h'}^{t_{h'}} \, a_{w_1}^{-1}=\begin{bmatrix}
			1&t_{h'} \, \, s_{w_1,w_2} \\
			0&1
		\end{bmatrix}.\]
		In particular, $a_{w_1}(x)=t_h \, s_{w_1,w_2}$ and $a_{w_1}(x')=t_{h'} \, s_{w_1,w_2}$, so
		\begin{equation*}
			\frac{1}{D^2}\frac{d_\infty(t_h \, s_{w_1,w_2}, t_{h'} \, s_{w_1,w_2})}{d_\infty(0,s_{w_1,w_2})}\leq \frac{d_\infty(x,x')}{d_\infty(y_1,y_2)}\leq D^2\frac{d_\infty(t_h \, s_{w_1,w_2}, t_{h'} \, s_{w_1,w_2})}{d_\infty(0,s_{w_1,w_2})}.
		\end{equation*}
		
		Since the metric $d_\infty$ is induced by a Riemannian metric, it is bi-Lipschitz to the Euclidean metric when restricted to any closed and bounded interval. Since the numbers $\abs{s_{w_1,w_2}}$ are uniformly bounded, we can find some $E>1$ that is uniform over all $w_1$ and $w_2$, such that
		\begin{equation}\label{eqn: blahdiblah}
			\frac{1}{D^2E}\abs{t_h-t_{h'}}\leq \frac{d_\infty(x,x')}{d_\infty(y_1,y_2)}\leq D^2E\abs{t_h-t_{h'}}.
		\end{equation}
		By Lemma \ref{lem: hyperbolic geometry} there exists a  constant $C=C(k)>1$ such that 
		\begin{align*}
			\frac{1}{C}\frac{d_\infty(x,x')}{d_\infty(y_1,y_2)}&\leq \frac{\ell(k(l_{h},l_{h'}))}{\ell(k(w_1,w_2))}\leq C\frac{d_\infty(x,x')}{d_\infty(y_1,y_2)}.
		\end{align*}
		Thus, we have
		\[
		\frac{1}{CD^2E}\abs{t_{h}-t_{h'}}\leq \frac{\ell(k(l_{h},l_{h'}))}{\ell(k(w_1,w_2))} \leq CD^2E\abs{t_{h}-t_{h'}},
		\]
		so the lemma holds with $B=CD^2E$.
	\end{proof}
	
	\begin{proof}[Proof of Lemma \ref{lem:nilp_unip}]
		Let $\mathcal{U}_d$ denote the set of unipotent matrices in $\mathsf{GL}_d(\mathbb{C})$. First, we observe that the exponential map 
		\[\exp:\mathcal{N}_d\to\mathcal{U}_d\]
		is a real-analytic diffeomorphism between the space $\mathcal{N}_d$ of nilpotent linear transformations of $\mathbb{C}^d$ and $\mathcal{U}_d$. For every $X \in \mathcal{N}_d$, define $F(X):= \exp(X) - \id$. Since $\exp(X)$ is unipotent, the linear transformation $F(X)$ is nilpotent. Making use of the power series expansion of the function $x' \mapsto \log(1 + x')$ (the inverse of $x \mapsto e^x - 1$) for $|x'| < 1$, we can describe explicitly the inverse of the map $F : \mathcal{N}_d \to \mathcal{N}_d$ as polynomial with vanishing constant term in a variable $X' \in \mathcal{N}_d$. Thus $F$ is a real-analytic diffeomorphism of  $\mathcal{N}_d$ with itself. By definition, the set of unipotent matrices $\mathcal{U}_d$ coincides with the translate of $\mathcal{N}_d$ by the identity matrix, so $\exp = F + \id$ is a real-analytic diffeomorphism from $\mathcal{N}_d$ to $\mathcal{U}_d$.
		
		Since $K\subset\mathcal{N}_d$ is compact, there is some $R > 0$ such that $K$ and $F(K)$ are both contained in the ball of radius $R$ in $\End(\mathbb{C}^d)$ centered at the zero endomorphism. Since the polynomial expansions of $F$ and $F^{-1}$ both have vanishing constant terms, for any $R > 0$ we can find a constant $D=D(R)$ such that
		\[
		\norm{F(X)} \leq D \norm{X}\,\,\text{ and }\,\, \norm{F^{-1}(X)} \leq D \norm{X}
		\]
		for every $X\in K$. Thus, for all $t\in[0,1]$ and $X\in K$, 
		\begin{align*}
			\norm{\exp(t X) - \id}&=\norm{F(t X)} \leq D \norm{t X} \leq D^2 t \norm{F(X)} = D^2 t \norm{\exp(X) - \id},
		\end{align*}
		and the desired statement holds with $C=D^2$.
	\end{proof}
	
	\subsubsection{Extending the homomorphism bundle} 
	
	We next describe an extension of the ${\bf i}$--th homomorphism bundle $\wh H=\wh H_{\bf i}\to T^1\lambda\cong\lambda^o$ to a line bundle over 
	\[\widehat{H}_{\mathrm{ret}}=\widehat{H}_{{\bf i},\mathrm{ret}}\to N^o_{\mathrm{ret}}\supset\lambda^o.\] 
	
	Let $\pi_{\wt N}:\wt N^o\to\wt N$ denote the covering map. Recall that the $\lambda$--limit map $\xi$ of $\rho$ induces a $\Gamma$--invariant splitting 
	\[\wt\lambda^o\times\mathbb{C}^d= T^1\wt\lambda\times\mathbb{C}^d=L_1\oplus\dots\oplus L_d,\] 
	where $L_i|_v=\xi^i(v^+)\cap \xi^{d-i+1}(v^-)$ for all $v\in T^1\wt\lambda$ and all $i$. For any element $u\in \wt N^o_{\rm ret}$, 
	\begin{itemize}
		\item if $u$ lies in a wedge $W_u$ of $\wt N^o_{\rm ret}$, let $\mathsf{l}_{u,1}$ and $\mathsf{l}_{u,2}$ be the two boundary geodesics of $W_u$, equipped with the natural orientation as leaves of $\widetilde\lambda^o$;
		\item if $u$ lies in a leaf $\mathsf{l}_u$ of $\wt\lambda^o$, let $\mathsf{l}_{u,1}=\mathsf{l}_{u,2}=\mathsf{l}_u$,
		\item let $h_u$ be the leaf of $\wt{\mathsf{F}}$ that contains $\pi_{\wt N}(u)$, let $l_{u,1}:=\pi_{\wt N}(\mathsf{l}_{u,1})$ and let $l_{u,2}:=\pi_{\wt N}(\mathsf{l}_{u,2})$.
	\end{itemize}
	Then for all $i$, let $L_{i,\mathrm{ret}}$ be the continuous sub-bundle of $\wt N^o_{\mathrm{ret}}\times\mathbb{C}^d\to\wt N^o_{\mathrm{ret}}$ defined by
	\[L_{i,\mathrm{ret}}|_{u}:=\widehat{\Sigma}_{h_u,l_{u,1}}(L_i|_{v})=\widehat{\Sigma}_{h_u,l_{u,2}}(L_i|_{v'}),\]
	where $v$ and $v'$ are some (any) points in $\mathsf{l}_{u,1}$ and $\mathsf{l}_{u,2}$ respectively. Notice that the second equality above holds by properties (1) and (4) from Proposition \ref{prop: slithering_extension}:
	\[\widehat{\Sigma}_{h_u,l_{u,1}}(L_i|_{v})=\widehat{\Sigma}_{h_u,l_{u,2}}\circ\Sigma_{l_{u,2},l_{u,1}}(L_i|_{v})=\widehat{\Sigma}_{h_u,l_{h_u,2}}(L_i|_{v'})\] 
	for all $v\in \mathsf{l}_{\mathsf{h},1}$ and all $v'\in\mathsf{l}_{\mathsf{h},2}$.
	
	The sub-bundles $L_{1,{\rm ret}},\dots,L_{d,{\rm ret}}$ define a splitting
	\[\wt N^o_{\rm ret}\times\mathbb{C}^d=L_{1,{\rm ret}}\oplus\dots\oplus L_{d,{\rm ret}}\] 
	which in turn defines a splitting 
	\begin{align}\label{eqn: splitting big}
		\wt N^o_{\rm ret}\times\End(\mathbb{C}^d)=\bigoplus_{1\le i,j\le d}\Hom(L_{i,{\rm ret}},L_{j,{\rm ret}}).
	\end{align}
	In particular, the line bundle 
	\[H_{\mathrm{ret}}=H_{\bf i,\mathrm{ret}}:=\Hom(L_{i_1,\mathrm{ret}},L_{i_1+1,\mathrm{ret}})\to\wt N^o_{\rm ret}\]  
	is a sub-bundle of the product bundle $\wt N^o_{\rm ret}\times\End(\mathbb{C}^d)$.
	
	By the $\rho$--equivariance of $\widehat{\Sigma}$, the natural $\Gamma$--action on $\wt N^o_{\rm ret}\times\End(\mathbb{C}^d)$ (by deck transformations in the first factor and conjugation by $\rho(\Gamma)$ in the second factor) leaves $\Hom(L_{i,{\rm ret}},L_{j,{\rm ret}})$  invariant for all $1\le i,j\le d$. Thus, the line bundle $H_{\mathrm{ret}}\to\wt N^o_{\rm ret}$  
	descends to the required line bundle 
	\[\widehat{H}_{\mathrm{ret}}=\widehat{H}_{{\bf i},\mathrm{ret}}:= H_{\mathrm{ret}}/\Gamma\to N^o_{\mathrm{ret}},\]
	which is a sub-bundle of 
	\[
	\left(\wt N^o_{\rm ret}\times\End(\mathbb{C}^d)\right)/\Gamma\to N^o_{\rm ret}.\] 
	By definition, the restriction of $\wh H_{\mathrm{ret}}$ to $\lambda^o$ is $\wh H$.

	\begin{proposition} \label{prop: extended_bundle}
		Every point $u\in N^o_{\mathrm{ret}}$, has a contractible open neighborhood $V_u\subset N^o_{\mathrm{ret}}$ such that the line bundle $\wh H_{\mathrm{ret}}|_{V_u}$ admits a trivialization 
		\[\psi_u : V_u\times \wh H_{\mathrm{ret}}|_u \longrightarrow \wh H_{\mathrm{ret}}|_{V_u}\]
		with the following properties:
		\begin{enumerate}
			\item The transition functions of $\{\psi_u\}_{u\in N^o_\mathrm{ret}}$ are locally constant, i.e. $\wh H_{\mathrm{ret}}$ is flat.
			\item When viewed as a map from $V_u$ to $\End(\mathbb{C}^d)$, $\psi_u$ is H\"older continuous with respect to some (any) norm on $\End(\mathbb{C}^d)$.
			\item If $w\in V_u$ lies in a leaf of $\lambda^o$ and $a>0$ is a number such that $\varphi_t(w)\in V_u$ for all $t\in [0,a]$, then for all element $Z\in\wh H_{\mathrm{ret}}|_u$ and $t\in [0,a]$ we have
			\[\varphi_t(\psi_u(w,Z))=\psi_u(\varphi_t(w),Z).\]
		\end{enumerate}
	\end{proposition}
	
	\begin{remark}
		Proposition \ref{prop: extended_bundle} shows that the complex line bundle $\widehat{H}_{{\bf i}, \mathrm{ret}} \to N^o_{\mathrm{ret}}$ admits a structure of flat bundle. The restriction of the associated flat connection $\nabla$ to $\widehat{H}_{\bf i} \to T^1 \lambda$ generalizes the \emph{slithering connection} considered in \cite{bobb-farre} from the setting of Hitchin representations to the context of $d$--pleated surfaces. 
	\end{remark}
	
	\begin{proof}[Proof of Proposition \ref{prop: extended_bundle}]
		Since the claims in this proposition can all be verified locally, it suffices to show that the analogous statements hold for the line bundle $H_{\mathrm{ret}}\to \wt N^o_{\mathrm{ret}}$. Given a point in $u\in \wt N^o_{\mathrm{ret}}$, let $V_{u}\subset \wt N^o_{\mathrm{ret}}$ be a contractible neighborhood of $u$ with compact closure. We can find two leaves $g_1$, $g_2$ of $\wt \lambda$ such that every leaf of the foliation $\wt{\mathsf{F}}{}^o$ that intersects $V_{u}$ projects to a leaf of $\wt{\mathsf{F}}$ in $Q_\mathsf{F}(g_1, g_2)$. By the definition of $H_{\mathrm{ret}}$ and Proposition~\ref{prop: slithering_extension} Claim (1), we have
		\begin{equation}\label{eq:bundle and slithering}
			H_{\mathrm{ret}} |_{w} = \widehat{\Sigma}_{h_{w}, h_{u}} \circ H_{\mathrm{ret}}|_{u} \circ \widehat{\Sigma}_{h_{u}, h_{w}} 
		\end{equation}
		for every $w \in V_{u}$. Thus, we may define a local trivialization of $H_{\mathrm{ret}}|_{V_{u}}$ by
		\[
		\begin{matrix}
			\psi_{u} : & V_{u}\times H_{\mathrm{ret}}|_{u} & \longrightarrow & H_{\mathrm{ret}}|_{V_{u}} \\
			& (w, f) & \longmapsto & \left(w, \widehat{\Sigma}_{h_{w}, h_{u}} \circ f \circ \widehat{\Sigma}_{h_{u}, h_{w}}\right).
		\end{matrix}
		\]
		
		By Proposition \ref{prop: slithering_extension} Claim (1),  $\psi_{u'}^{-1}\circ\psi_{u}(w,f)= \left(w, \widehat{\Sigma}_{h_{u'}, h_{u}} \circ f \circ \widehat{\Sigma}_{h_{u}, h_{u'}}\right)$ for all $w\in V_{u}\cap V_{u'}$, so (1) holds. Since the map $T^1\wt S\to\partial\wt S^2$ given by $v\mapsto (v^-,v^+)$ is smooth, Proposition~\ref{prop: Lipschitz foliation} Claim (1) implies that the map $\wt N^o_{\rm ret}\to \wt{\mathsf{F}}$ given by $w\mapsto h_w$ is locally Lipschitz, so by Proposition \ref{prop: slithering_extension} Claim (2), the map $w\mapsto \widehat{\Sigma}_{h_{u}, h_{w}}$ is H\"older continuous on $V_u$. This implies that (2) holds.  Finally, (3) is obvious from the definition of the local trivializations.
	\end{proof}

	\subsubsection{Defining \texorpdfstring{$\Omega^\rho_{\bf i}$}{Omega-rho-i}}
	
	Choose a continuous family of norms $\norm{\cdot}=\norm{\cdot}_{u\in N^o_{\rm ret}}$ on the fibers of the bundle $\wh H_{\mathrm{ret}}\to N^o_{\mathrm{ret}}$ defined above. For any point $u \in  N^o_{\mathrm{ret}}$, let 
	\[\psi_u : V_u\times \wh H_{\mathrm{ret}}|_u \longrightarrow \wh H_{\mathrm{ret}}|_{V_u}\] 
	be the local trivialization given in Proposition \ref{prop: extended_bundle}. Choose a non-zero vector $Z_u\in \wh H_{\mathrm{ret}}|_u$, and define the continuous section $\sigma_u:V_u\to\wh H_{\mathrm{ret}}|_{V_u}$ by
	\begin{equation}\label{eqn: sigma} 
		\sigma_u(w):=\psi_u(w,Z_u).
	\end{equation}
	Since $\norm{ \sigma_u(w) }_w>0$ for all $w\in V_u$, we may define the continuous function
	\[F_u : V_u \to \mathbb{R}\] 
	by $F_u(w):= - \log \norm{ \sigma_u(w) }_w$, and define
	\[\Omega|_{N^o_{\mathrm{ret}}}: = \{ (V_u, F_u) \}_{u \in N^o_{\mathrm{ret}}}.\]
	
	\begin{proposition}
		$\Omega|_{N^o_{\mathrm{ret}}}$ as defined above is topological closed $1$--form on $N^o_{\mathrm{ret}}$.
	\end{proposition}
	
	\begin{proof}
		It suffices to verify that if $u$ and $u'$ are points in $N^o_{\rm ret}$ such that $V_u\cap V_{u'}$ is non-empty, then $F_u-F_{u'}$ is constant on each connected component of $V_u\cap V_{u'}$. By Property (1) of $\psi_u$ given in Proposition \ref{prop: extended_bundle} and the fact that $\wh H_{\mathrm{ret}}$ is a line bundle, it follows that for each connected component $U$ of $V_u \cap V_{u'}$ there is some $k=k(U)\in\mathbb{R}-\{0\}$ such that $k\sigma_u(w)=\sigma_{u'}(w)$ for all $w\in U$. Thus by definition, $F_u(w)-F_{u'}(w)=  \log |k|$ for all $w\in U$.
	\end{proof}
	
	Now, choose a deformation retract $r:N^o\to N^o_{\mathrm{ret}}$. Then the required topological closed $1$--form on $N^o$ is
	\[\Omega=\Omega^{\rho}_{\bf i}:=\{(r^{-1}(V_u),F_u\circ r)\}_{u \in N^o_{\mathrm{ret}}}.\] 
	
	In the above construction, even though we suppress the dependence in the notation, $\Omega$ depends on the choice of a family of local trivializations $\{\psi_u\}_{u\in N^o_{\mathrm{ret}}}$ satisfying Proposition \ref{prop: extended_bundle}, the choice of vectors $\{Z_u\}_{u\in N^o_{\mathrm{ret}}}$, the choice of a continuous family of norms $\norm{\cdot}$ on the fibers of $\wh H_{\mathrm{ret}} \to N^o_{\mathrm{ret}}$, and the choice of a deformation retract $r:N^o\to N^o_{\mathrm{ret}}$. Observe that changing the choice of $\{\psi_u\}_{u\in N^o_{\mathrm{ret}}}$ and $\{Z_u\}_{u\in N^o_{\mathrm{ret}}}$ results in equivalent topological closed $1$--forms. On the other hand, a different choice of $\norm{\cdot}$ and $r$ might give non-equivalent topological closed $1$--forms on $N^o$. However, the next proposition implies that all such topological closed $1$--forms are cohomologous as elements of $H^1(N^o;\mathbb R)$.
	
	\begin{proposition}\label{prop: norm independence}
		Let $r_1,r_2:N^o\to N^o_{\mathrm{ret}}$ be deformation retracts, let $\norm{\cdot}_1$ and $\norm{\cdot}_2$ be continuous family of norms on the fibers of $\wh H_{\mathrm{ret}} \to N^o_{\mathrm{ret}}$, and let $\Omega_1$ and $\Omega_2$ be the topological closed $1$--forms constructed as above using $\norm{\cdot}_1$ and $r_1$, and using $\norm{\cdot}_2$ and $r_2$ respectively. If $\mathsf{c}$ is an oriented loop in $N^o$, then
		\[\int_{\mathsf{c}}\Omega_1=\int_{\mathsf{c}}\Omega_2.\]
		In particular, the cohomology class $[\Omega]=[\Omega^\rho_{\bf i}]\in H^1(N^o;\mathbb{R})$ does not depend on the choice of a continuous norm $\norm{\cdot}$ on the fibers of $\wh H_{\mathrm{ret}} \to N^o_{\mathrm{ret}}$, nor on the choice of a deformation retract $r:N^o\to N^o_{\mathrm{ret}}$.
	\end{proposition}
	
	\begin{proof}
		Since equivalent topological closed $1$--forms are cohomologous, we may assume that we use the same choices of $\{\psi_u\}_{u\in N^o_{\mathrm{ret}}}$ and $\{Z_u\}_{u\in N^o_{\mathrm{ret}}}$ to define both $\Omega_1$ and $\Omega_2$. This implies that in the definition of $\Omega_1$ and $\Omega_2$, we use the same family of sections $\{\sigma_u\}_{u\in N^o_{\mathrm{ret}}}$ defined by Equation \eqref{eqn: sigma}.
		
		Note that for any oriented loop $\mathsf{c}$ in $N^o$, $r_1(\mathsf{c})$ and $r_2(\mathsf{c})$ are homotopic in $N^o_{\rm ret}$. We previously observed that for any topological closed $1$--form on $N^o_{\rm ret}$, its integral along homotopic oriented loops in $N^o_{\rm ret}$ agree. Furthermore, by definition, we have
		\[\int_{\mathsf{c}}\Omega_1=\int_{r_1(\mathsf{c})}\Omega_1\,\,\text{ and }\,\,\int_{\mathsf{c}}\Omega_2=\int_{r_2(\mathsf{c})}\Omega_2.\]
		Thus, it suffices to prove that
		\[\int_{\mathsf{c}}\Omega_1=\int_{\mathsf{c}}\Omega_2\]
		for all oriented loops $\mathsf{c}$ in $N^o_{\mathrm{ret}}$. 
		
		By definition,
		\[\Omega_1|_{N^o_{\mathrm{ret}}}=\{ (V_u, F_u) \}_{u \in N^o_{\mathrm{ret}}}\,\,\text{ and }\,\,\Omega_2|_{N^o_{\mathrm{ret}}}=\{ (V_u, H_u) \}_{u \in N^o_{\mathrm{ret}}},\]
		where $F_u,H_u : V_u \to \mathbb{R}$ are given by 
		\[F_u(v):= - \log \norm{ \sigma_u(v) }_{1,v}\,\,\text{ and }\,\,H_u(v):= - \log \norm{ \sigma_u(v) }_{2,v}.\]
		Let $\mathsf{c}$ be an oriented loop in $N^o_{\rm ret}$, and let $\mathsf{c}_1,\dots,\mathsf{c}_n$ be paths in $N^o_{\rm ret}$ such that $\mathsf{c}$ is the concatenation $\mathsf{c}_1\cdot\ldots\cdot\mathsf{c}_n$, and each $\mathsf{c}_k$ lies in $V_{u_k}$ for some $u_k\in N^o_{\mathrm{ret}}$. By replacing $\Omega_j|_{N^o_{\mathrm{ret}}}$ with an equivalent topological closed $1$--form if necessary, we may assume that for all $k\in\{1,\dots,n-1\}$, $F_{u_k}=F_{u_{k+1}}$ and $H_{u_k}=H_{u_{k+1}}$ on the connected component of on $V_{u_k}\cap V_{u_{k+1}}$ containing $\mathsf{c}_k$. Since the forward and backward endpoints of $\mathsf{c}$ agree, we may denote it by $v_0$ and get
		\begin{align*}
			\int_{\mathsf{c}}\Omega_1-\int_{\mathsf{c}}\Omega_2&=F_{u_n}(v_0)-F_{u_1}(v_0)-H_{u_n}(v_0)+H_{u_1}(v_0)\\
			&=\log \frac{\norm{ \sigma_{u_1}(v_0) }_{1,v_0}}{\norm{ \sigma_{u_1}(v_0) }_{2,v_0}}\frac{\norm{ \sigma_{u_n}(v_0) }_{2,v_0}}{\norm{ \sigma_{u_n}(v_0) }_{1,v_0}}=0,
		\end{align*}
		where the last equality holds because there is some constant $k>0$ such that $\norm{X}_{2,v_0}=k\norm{X}_{1,v_0}$ for all $X\in\wh H_{\mathrm{ret}}|_{v_0}$.
	\end{proof}
	
	Finally, we observe that by choosing $\norm{\cdot}$ and $r$ appropriately, we may ensure that $\Omega$ is H\"older continuous. 
	
	\begin{proposition}\label{prop: Holder topological closed form}
		If we choose the continuous family of norms $\norm{\cdot}$ on the fibers of $\wh H_{\rm ret}$ to be the restriction of a smooth family of norms on the fibers of the vector bundle $\left(\wt N^o\times\End(\mathbb{C}^d)\right)/\Gamma\to N^o$, and we choose the deformation retract $r:N^o\to N^o_{\rm ret}$ to be H\"older continuous, then $\Omega$ is H\"older. 
	\end{proposition}
	
	\begin{proof} By Proposition \ref{prop: extended_bundle} Claim (2), we see that for all $u\in N^o_{\rm ret}$, the map $\sigma_u$ is H\"older continuous when viewed as a map from $V_u$ to $\End(\mathbb{C}^d)$. Then our choice of $\norm{\cdot}_{u\in N^o_{\rm ret}}$ ensures that $\Omega|_{N^o_{\mathrm{ret}}}$ is H\"older continuous, in which case our choice of $r$ ensures that $\Omega$ is H\"older continuous.
	\end{proof}
	
	\subsection{Proof of Equation \texorpdfstring{\eqref{eqn: required condition 1}}{n}}\label{sec: properties of 1 form}
	
	We now prove that the cohomology class in $H^1(N^o;\mathbb{R})$ of the closed topological $1$--form $\Omega=\Omega^\rho_{\bf i}$ satisfies Equation \eqref{eqn: required condition 1}. 
	
	First, recall that in Section \ref{subsec:length_functions} we defined the function $f=f^\rho_{\bf i}:T^1\lambda\to\mathbb{R}$ as 
	\begin{align}\label{eqn: f}
		f(v):= - \left. \frac{\mathrm{d}}{\mathrm{d} t}\right|_{t = 0} \log \norm{\varphi_t(X)}_{\varphi_t(v)} 
	\end{align}
	for some (any) non-zero $X \in \wh H|_v$, where $\norm{\cdot}$ is a continuous family of norms on the fibers of $\wh H \to T^1\lambda\cong\lambda^o$ and 
	$\varphi_t$ is a flow on $\wh H$. 
	Then, the length function $\ell=\ell^\rho_{\bf i}$ is defined by
	\[
	\ell(\mu)=\int_{T^1\lambda} f \ \mathrm{d}t\times \mu =\sum_{R\text{ a rectangle of }N^o}\int_{\{\mathsf{h}\in\Lambda^o:\mathsf{h}\cap R\neq\emptyset\}}\left(\int_{\mathsf{h}\cap R} f\ \mathrm{d}t\right)\mu,
	\]
	where $\mathrm{d}t$ is the $1$--form along the leaves of $T^1\lambda$ given by the arc length. Recall that $\Lambda^o$ is the set of oriented leaves $\mathsf{h}$ of $\lambda^o$.
	
	Extend $\norm{\cdot}$ to a continuous family of norms on the fibers of $N^o_{\mathrm{ret}}$, which we also denote by $\norm{\cdot}$. By Proposition \ref{prop: norm independence}, we may assume that the topological closed $1$--form $\Omega$ was defined using this continuous family. Recall that $\mathsf{M}(\lambda^o)$ denotes the set of transverse measures whose support lies in $\lambda^o$.
	
	\begin{proposition} \label{prop: second_extension_forms}
		If $\mathsf{k}$ is an oriented geodesic arc that lies in a leaf of $\lambda^o$, then
		\[
		\int_{\mathsf{k}} \Omega = \int_{\mathsf{k}} f \,\mathrm{d}t .
		\]
		In particular, 
		\[\ell(\mu)=\sum_{R\text{ a rectangle of }N^o}\int_{\{\mathsf{h}\in\Lambda^o:\mathsf{h}\cap R\neq\emptyset\}}\left(\int_{\mathsf{h}\cap R}\Omega\right)\mu\] 
		for all $\mu\in\mathsf{M}(\lambda^o)$.
	\end{proposition}
	\begin{proof}
		By definition, we have
		\[\Omega|_{N^o_{\mathrm{ret}}}=\{ (V_u, F_u) \}_{u \in N^o_{\mathrm{ret}}}\] 
		where $F_u: V_u \to \mathbb{R}$ is given by $F_u(w):= - \log \norm{ \sigma_u(w) }_{w}$ and $\sigma_u:V_u\to\wh H_{\mathrm{ret}}|_{V_u}$ is the section defined by Equation \eqref{eqn: sigma} (it  depends on the choice of a family of local trivializations $\{\psi_u\}_{u\in N^o_{\mathrm{ret}}}$ and vectors $\{Z_u\}_{u\in N^o_{\mathrm{ret}}}$). We may assume without loss of generality that $\mathsf{k}$ lies in $V_u$ for some $u\in N^o_{\mathrm{ret}}$. Choose a parameterization $\alpha : [- \varepsilon, \varepsilon] \to \lambda^o$ of $\mathsf{k}$ by arc length. It follows from Property (3) of Proposition~\ref{prop: extended_bundle} and the definition of $\sigma_u$ that 
		\[\varphi_h(\sigma_{u}(\alpha(s)))=\sigma_{u}(\alpha(h+s)).\]
		Thus,
		\begin{align*}
			\int_{\mathsf{k}} f \,\mathrm{d}t 
			& = - \int_{- \varepsilon}^{\varepsilon} \left. \frac{\mathrm{d}}{\mathrm{d} h}\right|_{h = 0} \log \norm{\sigma_u(\alpha(s + h))}  \textrm{d}s  = F_u(\alpha(\varepsilon)) - F_u(\alpha(- \varepsilon)) = \int_{\mathsf{k}} \Omega|_{N^o_{\mathrm{ret}}}.\qedhere
		\end{align*}
	\end{proof}
	
	Via the de Rham isomorphism
	\[H^1_{\rm dR}(N^o;\mathbb{R})\cong H^1(N^o;\mathbb{R})\]
	there is a closed $1$--form $\omega$ on $N^o$ such that $[\omega]\in H^1_{\rm dR}(N^o;\mathbb{R})$ corresponds to $[\Omega]\in H^1(N^o;\mathbb{R})$. Equivalently, 
	\[\int_{\mathsf{c}}\omega=\int_{\mathsf{c}}\Omega\]
	for all oriented loops $\mathsf{c}$ in $N^o$. Since any $\mu\in\mathsf{M}(\lambda^o)$ can be realized as a limit (in the space of $\Gamma$--invariant Radon measure on the space of oriented geodesics in $\wt S$) of a sequence of point masses supported on oriented closed geodesics in $N^o$, it follows from Proposition \ref{prop: second_extension_forms} that 
	\begin{align}\label{eqn: length rep 1}
		\ell(\mu)=\sum_{R\text{ a rectangle of }N^o}\int_{\{\mathsf{h}\in\Lambda^o:\mathsf{h}\cap R\neq\emptyset\}}\left(\int_{\mathsf{h}\cap R}\omega\right)d\mu.
	\end{align}
	
	Again, via the de Rham isomorphism
	\[H^1_{\rm dR}(N^o,\partial N^o;\mathbb{R})\cong H^1(N^o,\partial N^o;\mathbb{R}),\]
	there is a closed $1$--form $\nu$ on $N^o$ whose support lies in the interior of $N^o$, such that $[\nu]\in H^1_{\rm dR}(N^o,\partial N^o;\mathbb{R})$ corresponds to $D([\mu])\in H^1(N^o,\partial N^o;\mathbb{R})$. Here, recall that
	\begin{itemize}
		\item $D:H_1(N^o;\mathbb{R})\to H^1(N^o,\partial N^o;\mathbb{R})$ is the isomorphism given by Lefschetz-Poincar\'e duality, and 
		\item $[\mu]\in H_1(N^o;\mathbb{R})$ is the homology class determined by the transverse measure $\mu\in\mathsf{M}(\lambda^o)$ viewed as a $\lambda^o$--cocycle, see Section \ref{sec: cocycles} and the discussion at the start of this section.
	\end{itemize}
	Then, by the discussion in Ruelle-Sullivan \cite[Pages 320--321]{RuSul}, we have \begin{align}\label{eqn: length rep 2}
		\sum_{R\text{ a rectangle of }N^o}\int_{\{\mathsf{h}\in\Lambda^o:\mathsf{h}\cap R\neq\emptyset\}}\left(\int_{\mathsf{h}\cap R}\omega\right)\mu=\int_{N^o}\nu\wedge\omega.
	\end{align}
	
	Finally, notice that the cohomology class $[\nu\wedge\omega]\in H^2_{\rm dR}(N^o,\partial N^o;\mathbb{R})$ corresponds, via the de Rham isomorphism 
	\[H^2_{\rm dR}(N^o,\partial N^o;\mathbb{R})\cong H^2(N^o,\partial N^o;\mathbb{R}),\] 
	to the cup product $D([\mu])\smile [\Omega]\in H^2(N^o,\partial N^o;\mathbb{R})$. It follows that
	\begin{align}\label{eqn: length rep 3}
		\int_{N^o}\nu\wedge\omega=\mathbb I_*(D([\mu]),[\Omega]),
	\end{align}
	where $\mathbb I_*$ is the intersection pairing in cohomology defined in Section \ref{sec: forms}.
	Together, Equation \eqref{eqn: length rep 1}, Equation \eqref{eqn: length rep 2}, and Equation \eqref{eqn: length rep 3} imply Equation \eqref{eqn: required condition 1}.
	
	\begin{remark}\label{rem: independence from norm}
		As a consequence of Equation \eqref{eqn: required condition 1}, we see that $\ell^\rho_{\bf i}=\ell$ does not depend on the choice of $\norm{\cdot}$, because by Proposition \ref{prop: norm independence}, the cohomology class $[\Omega]\in H^1(N^o;\mathbb{R})$ does not depend on $\norm{\cdot}$. 
	\end{remark}
	
	\subsection{Proof of Equation \texorpdfstring{\eqref{eqn: required condition 2}}{n}}\label{sec: properties of 1 form 2}
	
	Let $(\alpha_\rho,\theta_\rho)\in\mathcal{Y}_d(\lambda;\mathbb{C}/2\pi i\mathbb{Z})$ denote the shear-bend coordinates of the fixed $d$--pleated surface $\rho\in\mathcal{R}_d(\lambda)$. Now, we prove that Equation \eqref{eqn: required condition 2} holds for all transverse measures $\mu\in\mathsf M(\lambda^o)$. The argument here uses ideas from Bonahon and Dreyer \cite[Lemma 7.7 and 7.8]{BoD}.
	
	Let $\mathsf{k}$ be a path in $\wt N^o$ that is tightly transverse to $\wt\lambda^o$ (see Definition \ref{def:tightly transverse}), oriented so that the leaves of $\wt\lambda^o$, when equipped with their natural orientations, cross from the right to the left of $\mathsf{k}$. Recall that the connected components of $\mathsf{k} - \wt \lambda^o$ that contain an endpoint of $\mathsf{k}$ are the \emph{exterior subsegments} of $\mathsf{k}$. We denote by $\mathsf{k}'$ the complement in $\mathsf{k}$ of the exterior subsegments of $\mathsf{k}$. Then let $\mathsf{k}_-$ and $\mathsf{k}_+$ (respectively, $\mathsf{k}'_-$ and $\mathsf{k}'_+$) denote the backward and forward endpoints of $\mathsf{k}$ (respectively, $\mathsf{k}'$) respectively. Recall that $\pi_{\wt N}:\wt N^o\to \wt N$ is the covering map, and let $T_{\mathsf{k},1}$ (respectively, $T_{\mathsf{k},2}$) be the plaque of $\wt \lambda$ that contains $\pi_{\wt N}(\mathsf{k}_-)$ (respectively, $\pi_{\wt N}(\mathsf{k}_+)$). Then let $g^{\mathsf{k}}_j$ be the edge of $T_{\mathsf{k},j}$ that contains an endpoint of $\pi_{\wt N}(\mathsf{k}')$, and let $z^{\mathsf{k}}_j \in \partial \wt \lambda$ be the vertex of $T_{\mathsf{k},j}$ that is not an endpoint of $g^{\mathsf{k}}_j$. Since $\mathsf{k}'_-$ and $\mathsf{k}'_+$ lie in $\wt\lambda^o\cong T^1\wt\lambda$, they determine orientations on $g_1^{\mathsf{k}}$ and $g_2^{\mathsf{k}}$ respectively. Let $x_j^{\mathsf{k}},y_j^{\mathsf{k}}\in\partial\wt\lambda$ denote the backward and forward endpoints of $g_j^{\mathsf{k}}$. As usual, we denote ${\bf T}_{\mathsf{k}}:=(T_{\mathsf{k},1},T_{\mathsf{k},2})$.
	
	For now, fix $j\in\{1,2\}$. For any $i\in\{1,\dots,d\}$, set 
	\[L_{j,i}:=\xi^i(x_j^{\mathsf{k}})\cap \xi^{d-i+1}(y_j^{\mathsf{k}}),\] 
	and note that we have the line decomposition $\mathbb{C}^d=L_{j,1}+\dots+L_{j,d}$. For any pair of positive integers ${\bf i}=(i_1,i_2)$ that sum to $d$, let 
	\[X_{{\bf i},j}^{\mathsf{k}}=X_j^{\mathsf{k}}\in\Hom(L_{j,i_1+1},L_{j,i_1})\] 
	be defined as follows. Pick any nonzero vector $V \in \xi^1(z^{\mathsf{k}}_j)$, let $V_i$ be the projection of $V$ onto $L_{j,i}$  with respect to the line decomposition $\mathbb{C}^d=L_{j,1}+\dots+L_{j,d}$, and let $X_j^{\mathsf{k}}$ be the unique linear map that sends $V_{i_1 + 1}$ into $V_{i_1}$. Observe that $X_j^{\mathsf{k}}$ does not depend on the choice of $V$, and $X_j^{\mathsf{k}}$ is non-zero because $\xi$ is $\lambda$--hyperconvex. 
	
	Pick a smooth, $\Gamma$--invariant family of norms on the fibers of the product bundle $\wt N^o\times\End(\mathbb{C}^d)\to\wt N^o$, and let $\norm{\cdot}$ denote its restriction to the fibers of the bundle $H_{\rm ret} \to\wt N^o_{\rm ret}$. Also, let the retract $r:N^o\to N^o_{\rm ret}$ be a Lipschitz retract. By Proposition \ref{prop: norm independence}, we may assume that the topological closed $1$--form $\Omega=\Omega^\rho_{\bf i}$ is constructed using $\norm{\cdot}$ and $r$, in which case Proposition \ref{prop: Holder topological closed form} implies that $\Omega$ is H\"older. The next proposition is the key result needed to establish Equation \eqref{eqn: required condition 2}--it relates the integral of $\wt\Omega$ (which is the lift of the topological closed $1$--form $\Omega$ to $\wt N^o$) along any path $\mathsf{k}$ in $\wt N^o$ that is tightly transverse to $\wt\lambda^o$ to the quantity $\mathrm{Re}\;\alpha_\rho({\bf T}_{\mathsf{k}},{\bf i})$.
	
	\begin{proposition}\label{prop: the rise of the shear}
		If $\mathsf{k}$ is a path in $\wt N^o$ that is tightly transverse to $\wt \lambda^o$, then
		\[
		\mathrm{Re}\; \alpha_\rho({\bf T}_{\mathsf{k}},\mathbf{i}) = \int_{\mathsf{k}'} \wt\Omega - \log  \norm{X_1^{\mathsf{k}}}_{\mathsf{k}'_-} + \log  \norm{X_2^{\mathsf{k}}}_{\mathsf{k}'_+} .
		\]
	\end{proposition}
	
	\begin{proof}
		Let $\wt\Omega=\{ (V_u, F_u) \}_{u \in \wt N^o_{\mathrm{ret}}}$ and let $\mathsf{k}_1',\dots,\mathsf{k}_n'$ be paths in $\wt N^o_{\rm ret}$ such that $\mathsf{k}'$ is the concatenation $\mathsf{k}_1'\cdot\ldots\cdot\mathsf{k}_n'$, and each $\mathsf{k}_i'$ lies in $V_{u_i}$ for some $u_i\in \wt N^o_{\mathrm{ret}}$. By replacing $\wt\Omega$ with an equivalent topological closed $1$--form if necessary, we may assume that for all $i\in\{1,\dots,n-1\}$, $F_{u_i}=F_{u_{i+1}}$ on the connected component of $V_{u_i}\cap V_{u_{i+1}}$ that contains $\mathsf{k}_i'$. Then the collection of functions $\{F_{u_1},\dots,F_{u_k}\}$ combine into a single H\"older continuous function $F$ which is defined along $\mathsf{k}'$, and by definition, 
		\[\int_{\mathsf{k}'}\wt\Omega=F(\mathsf{k}'_+)-F(\mathsf{k}'_-).\]
		
		For any connected component $\mathsf{e}$ of $\mathsf{k}'-\wt\lambda^o$, let $\mathsf{e}_-$ and $\mathsf{e}_+$ be the backward and forward endpoints of $\mathsf{e}$ respectively, and note that 
		\[\int_{\mathsf{e}}\wt\Omega=F(\mathsf{e}_+)-F(\mathsf{e}_-).\]
		Since the lamination $\wt \lambda^o$ has Hausdorff dimension $1$ \cite[Theorem I]{BIRMAN1985217}, its intersection with $\mathsf{k}'$ has Hausdorff dimension $0$. Then by the H\"older continuity of $F$, we may apply the gap lemma \cite[Lemma 3]{bon_tra} to deduce that
		\begin{align}\label{eqn: sum well-defined}
			\int_{\mathsf{k}'} \wt\Omega= \sum_{\mathsf{e}} \int_{\mathsf{e}} \wt\Omega,
		\end{align}
		and the sum on the right, which is over all connected components $\mathsf{e}$ of $\mathsf{k}'-\wt\lambda^o$, converges.

		Let $g^{\mathsf{e}}_1$ and $g^{\mathsf{e}}_2$ be the leaves of $\wt\lambda$ that contain $\pi_{\wt N}(\mathsf{e}_-)$ and $\pi_{\wt N}(\mathsf{e}_+)$ respectively. A key calculation needed in this proof is the following lemma.
		
		\begin{lemma} \label{lem:integral_on_component_of_k}
			If $\mathsf{e}$ is a connected component of $\mathsf{k} - \wt\lambda^o$ whose closure lies in the interior of $\mathsf{k}$, then we have
			\[
			\int_{\mathsf{e}} \wt\Omega = \log \frac{\norm{X}_{{\mathsf{e}}_-}}{\norm{\Sigma(g^{\mathsf{e}}_2, g^{\mathsf{e}}_1) \circ X \circ \Sigma(g^{\mathsf{e}}_1, g^{\mathsf{e}}_2)}_{{\mathsf{e}}_+}},
			\]
			where $X$ is any non-zero vector in $H |_{{\mathsf{e}}_-}$.
		\end{lemma}
		
		\begin{proof}
			Let $\mathsf{g}^{\mathsf{e}}_1$ and $\mathsf{g}^{\mathsf{e}}_2$ be the leaves of $\wt\lambda^o$ that contain $\mathsf{e}_-$ and $\mathsf{e}_+$ respectively, and let $(\mathsf{e}_n)_{n=1}^\infty$ be a sequence of geodesic arcs in $\wt N^o$ whose backward endpoints $(\mathsf{e}_n)_-$ and forward endpoints $(\mathsf{e}_n)_+$ lie in $\mathsf{g}^{\mathsf{e}}_1$ and $\mathsf{g}^{\mathsf{e}}_2$ respectively, and whose lengths converge to zero as $n$ grows to $\infty$. Such a sequence exists because $\mathsf{g}^{\mathsf{e}}_1$ and $\mathsf{g}^{\mathsf{e}}_2$ are asymptotic as oriented geodesics in $S$. Let $\mathsf{p}_n$ (respectively, $\mathsf{q}_n$) be the geodesic arc whose backward and forward endpoints are $\mathsf{e}_-$ and $(\mathsf{e}_n)_-$ (respectively, $\mathsf{e}_+$ and $(\mathsf{e}_n)_+$). Then for all $n$, we have
			\begin{align}\label{eqn: tiny segment}
				\int_{\mathsf{e}} \wt\Omega = \int_{\mathsf{p}_n} \wt\Omega+\int_{\mathsf{e}_n} \wt\Omega-\int_{\mathsf{q}_n} \wt\Omega.
			\end{align}
			
			Let $\wt f$ be the lift to $\wt\lambda^o=T^1\wt\lambda$ of the function $f$ on $\lambda^o=T^1\lambda$ defined by Equation \eqref{eqn: f}. By Proposition \ref{prop: second_extension_forms}, 
			\begin{align*}
				\int_{\mathsf{p}_n} \wt\Omega-\int_{\mathsf{q}_n} \wt\Omega&= \int_{\mathsf{p}_n}\wt f-\int_{\mathsf{q}_n}\wt f\\
				&= \log\frac{\norm{X}_{\mathsf{e}_-}}{\norm{X}_{(\mathsf{e}_n)_-}}-\log\frac{\norm{\Sigma(g^{\mathsf{e}}_2, g^{\mathsf{e}}_1) \circ X \circ \Sigma(g^{\mathsf{e}}_1, g^{\mathsf{e}}_2)}_{\mathsf{e}_+}}{\norm{\Sigma(g^{\mathsf{e}}_2, g^{\mathsf{e}}_1) \circ X \circ \Sigma(g^{\mathsf{e}}_1, g^{\mathsf{e}}_2)}_{(\mathsf{e}_n)_+}}\\
				&= \log\frac{\norm{X}_{\mathsf{e}_-}}{\norm{\Sigma(g^{\mathsf{e}}_2, g^{\mathsf{e}}_1) \circ X \circ \Sigma(g^{\mathsf{e}}_1, g^{\mathsf{e}}_2)}_{\mathsf{e}_+}}-\log\frac{\norm{X}_{(\mathsf{e}_n)_-}}{\norm{\Sigma(g^{\mathsf{e}}_2, g^{\mathsf{e}}_1) \circ X \circ \Sigma(g^{\mathsf{e}}_1, g^{\mathsf{e}}_2)}_{(\mathsf{e}_n)_+}}.
			\end{align*}
			
			Let $K\subset \wt N^o$ be a compact set such that $\Gamma\cdot K=\wt N^o$. Since the length of $\mathsf{e}_n$ converges to $0$ as $i$ grows to $\infty$, by enlarging $K$ if necessary, we can ensure that for all $n$, there is some $\gamma_n\in\Gamma$ such that $\gamma_n\cdot \mathsf{e}_n\subset K$. By taking a subsequence, we can then ensure that $\gamma_n\cdot (\mathsf{e}_n)_\pm$ converge in $\wt N^o$ and $\gamma_n\cdot g_j^{\mathsf{e}}$ converges in the space of leaves of $\wt\lambda$ for both $j=1,2$. Then 
			\[\lim_{n\to\infty} \gamma_n\cdot (\mathsf{e}_n)_+=\lim_{n\to\infty}\gamma_n\cdot (\mathsf{e}_n)_-\,\,\text{ and }\,\,\lim_{n\to\infty}\gamma_n\cdot g_1^{\mathsf{e}}=\lim_{n\to\infty}\gamma_n\cdot g_2^{\mathsf{e}}.\]
			Since
			\[\frac{\norm{X}_{(\mathsf{e}_n)_-}}{\norm{\Sigma(g^{\mathsf{e}}_2, g^{\mathsf{e}}_1) \circ X \circ \Sigma(g^{\mathsf{e}}_1, g^{\mathsf{e}}_2)}_{(\mathsf{e}_n)_+}}=\frac{\norm{\rho(\gamma_n)\cdot X}_{\gamma_n\cdot(\mathsf{e}_n)_-}}{\norm{\Sigma(\gamma_n\cdot g^{\mathsf{e}}_2, \gamma_n\cdot g^{\mathsf{e}}_1) \circ (\rho(\gamma_n)\cdot X) \circ \Sigma(\gamma_n\cdot g^{\mathsf{e}}_1, \gamma_n\cdot g^{\mathsf{e}}_2)}_{\gamma_n\cdot(\mathsf{e}_n)_+}},\]
			it follows that
			\begin{align}\label{eqn: tiny segment 2}
				\lim_{n\to\infty} \left(\int_{\mathsf{p}_n} \wt\Omega-\int_{\mathsf{q}_n} \wt\Omega\right)=\log\frac{\norm{X}_{\mathsf{e}_-}}{\norm{\Sigma(g^{\mathsf{e}}_2, g^{\mathsf{e}}_1) \circ X \circ \Sigma(g^{\mathsf{e}}_1, g^{\mathsf{e}}_2)}_{\mathsf{e}_+}}.
			\end{align}
			At the same time, the continuity and $\Gamma$--invariance of $\wt\Omega$ implies that 
			\begin{align}\label{eqn: tiny segment 3}
				\lim_{n\to\infty}\int_{\mathsf{e}_n} \wt\Omega=\lim_{n\to\infty}\int_{\gamma_n\cdot \mathsf{e}_n} \wt\Omega=0.
			\end{align}
			Together, Equation \eqref{eqn: tiny segment}, Equation \eqref{eqn: tiny segment 2}, and Equation \eqref{eqn: tiny segment 3} imply the lemma.
		\end{proof}

		By Equation \eqref{eqn: sum well-defined} and Lemma \ref{lem:integral_on_component_of_k},we can see
		\[
		\int_{\mathsf{k}'} \wt\Omega= \sum_{\mathsf{e}}\log \frac{\norm{X_{\mathsf{e}}}_{{\mathsf{e}}_-}}{\norm{\Sigma(g^{\mathsf{e}}_2, g^{\mathsf{e}}_1) \circ X_{\mathsf{e}} \circ \Sigma(g^{\mathsf{e}}_1, g^{\mathsf{e}}_2)}_{{\mathsf{e}}_+}},
		\]
		for any non-zero $X_{\mathsf{e}}\in H|_{\mathsf{e}_-}$. If we choose $X_{\mathsf{e}}$ to be $\Sigma(g^{\mathsf{e}}_1, g^{\mathsf{k}}_1) \circ X^{\mathsf{k}}_1 \circ \Sigma(g^{\mathsf{k}}_1,g^{\mathsf{e}}_1)$ for all $\mathsf{e}$, then the composition rule of the slithering map gives
		\begin{align*}
			\int_{\mathsf{k}'} \wt\Omega = \sum_{\mathsf{e}} \log \frac{\norm{\Sigma(g^{\mathsf{e}}_1, g^{\mathsf{k}}_1) \circ X^{\mathsf{k}}_1 \circ \Sigma(g^{\mathsf{k}}_1,g^{\mathsf{e}}_1)}_{\mathsf{e}_-}}{\norm{\Sigma(g^{\mathsf{e}}_2, g^{\mathsf{k}}_1) \circ X^{\mathsf{k}}_1 \circ \Sigma(g^{\mathsf{k}}_1,g^{\mathsf{e}}_2)}_{\mathsf{e}_+}}.
		\end{align*}
		
		For any $u \in \mathsf{k} \cap \wt \lambda^o$, let $\mathsf{g}_u$ be the leaf of $\wt \lambda^o$ passing through $u$, and let $g_u:=\pi_{\wt N}(\mathsf{g}_u)$. Again by Proposition~\ref{prop: Lipschitz foliation} Claim (1), the assignment $u\mapsto g_u$ is Lipschitz. Since the slithering map is locally H\"older continuous  (Definition \ref{def: slithering map} Property (2)), the function 
		\[F:\mathsf{k} \cap \wt N^o_{\rm ret}\to\mathbb{R}\] 
		given by $F : u \mapsto -\log \norm{\widehat{\Sigma}(g_u, g^{\mathsf{k}}_1) \circ X^{\mathsf{k}}_1 \circ \widehat{\Sigma}(g^{\mathsf{k}}_1,g_u)}_u$ is H\"older continuous. Observe that $\mathsf{k} \cap \wt N^o_{\rm ret}$ is a finite union of subsegments of $\mathsf{k}$, and the intersection of $\wt\lambda^o$ with each of these subsegments has measure $0$. Thus, we may again apply the gap lemma to deduce that 
		\[
		\int_{\mathsf{k}'} \wt\Omega=\sum_{\mathsf{e}}F(\mathsf{e}_+)-F(\mathsf{e}_-)=F(\mathsf{k}'_+)-F(\mathsf{k}'_-)=\log  \frac{\norm{ X^{\mathsf{k}}_1}_{\mathsf{k}'_-}}{\norm{\Sigma(g^{\mathsf{k}}_2, g^{\mathsf{k}}_1) \circ X^{\mathsf{k}}_1 \circ \Sigma(g^{\mathsf{k}}_1,g^{\mathsf{k}}_2)}_{\mathsf{k}'_+}}.
		\]
		
		Since $H|_{\mathsf{k}'_+}$ is a $1$--dimensional vector space, there exists a $\kappa \in \mathbb{C}^*$ such that $\Sigma(g^{\mathsf{k}}_2, g^{\mathsf{k}}_1) \circ X^{\mathsf{k}}_1 \circ \Sigma(g^{\mathsf{k}}_1,g^{\mathsf{k}}_2) = \kappa  X^{\mathsf{k}}_2$. Therefore,
		\[
		\int_{\mathsf{k}'} \wt\Omega = \log  \frac{\norm{X^{\mathsf{k}}_1}_{\mathsf{k}'_-}}{\norm{X^{\mathsf{k}}_2}_{\mathsf{k}'_+}} - \log \abs{\kappa}.
		\]
		The proposition now follows once we show that $\mathrm{Re}\; \alpha_\rho({\bf T}_{\mathsf{k}}, \mathbf{i}) = - \log \abs{\kappa}$. 
		
		Let $V$ be a non-zero vector in $\Sigma_{g_2^{\mathsf{k}},g_1^{\mathsf{k}}} (\xi^1(z^{\mathsf{k}}_1))$ and let $U$ be a non-zero vector in $\xi^1(z^{\mathsf{k}}_2)$. For each $i=1,\dots,d$, let $V_i$ and $U_i$ respectively be the projections of $V$ and $U$ onto the line $L_i|_{\mathsf{k}'_+}$ with respect to the line decomposition 
		\[\mathbb{C}^d=L_1|_{\mathsf{k}'_+}+\dots+L_d|_{\mathsf{k}'_+}.\] 
		Again, since $\rho\in\mathcal{R}_d(\lambda)$, $V_i\neq 0\neq U_i$ for all $i=1,\dots,d$, so we can identify $\frac{V_i}{U_i}$ and $\frac{U_i}{V_i}$ with elements of $\mathbb K$. It is a straightforward computation to verify that
		\[
		\abs{D^{\mathbf{i}}\left(\xi(x^{\mathsf{k}}_2), \xi(y^{\mathsf{k}}_2), \Sigma_{ g_2^{\mathsf{k}},g_1^{\mathsf{k}}}(\xi(z^{\mathsf{k}}_1)),  \xi(z^{\mathsf{k}}_2)\right)} = \abs{\frac{V_{i_1 + 1}}{U_{i_1 + 1}}} \abs{\frac{U_{i_1}}{V_{i_1}}},
		\]
		where $D^{\mathbf{i}}$ denotes the double ratio defined in Section \ref{xflag}.
		
		Now pick a norm $\vertiii{\cdot}$ on $\mathbb{C}^d$. Since $H|_{u^{\mathsf{k}}_2}$ is $1$--dimensional, we can see
		\[\frac{1}{|\kappa|}=\frac{\norm{X^{\mathsf{k}}_2}_{u^{\mathsf{k}}_2}}{\norm{\Sigma(g^{\mathsf{k}}_2, g^{\mathsf{k}}_1) \circ X^{\mathsf{k}}_1 \circ \Sigma(g^{\mathsf{k}}_1,g^{\mathsf{k}}_2)}_{u^{\mathsf{k}}_2}}=\frac{\vertiii{V_{i_1+1}}}{\vertiii{V_{i_1}}}\frac{\vertiii{U_{i_1}}}{\vertiii{U_{i_1+1}}}= \abs{\frac{V_{i_1 + 1}}{U_{i_1 + 1}}} \abs{\frac{U_{i_1}}{V_{i_1}}}.\]
		It follows that
		\begin{align*}
			\mathrm{Re}\; \alpha_\rho(\mathbf{T}_{\mathsf{k}}, \mathbf{i}) & = \mathrm{Re}\; \left( \sigma^{\mathbf{i}}\left(\xi(x^{\mathsf{k}}_1), \xi(y^{\mathsf{k}}_1), \xi(z^{\mathsf{k}}_1), \Sigma_{g_1^{\mathsf{k}}, g_2^{\mathsf{k}}} (\xi(z^{\mathsf{k}}_2))\right) \right) \\
			& = \log \abs{D^{\mathbf{i}}\left(\xi(x^{\mathsf{k}}_2), \xi(y^{\mathsf{k}}_2), \Sigma_{ g_2^{\mathsf{k}},g_1^{\mathsf{k}}}(\xi(z^{\mathsf{k}}_1)),  \xi(z^{\mathsf{k}}_2)\right)}=- \log \abs{\kappa}.\qedhere
		\end{align*}
	\end{proof}
	
	Let $\mathsf{k}_1,\dots,\mathsf{k}_{2n}$ be ties of $N^o$ such that $[\mathsf{k}_1],\dots,[\mathsf{k}_{2n}]$ are the standard generators of $H_1(N^o,\partial_h N^o;\mathbb{R})$. Also, let $[\mathsf{k}_1^*],\dots,[\mathsf{k}_{2n}^*]$ be the standard generators of $H_1(N^o,\partial_v N^o;\mathbb{R})$. For each $j\in\{1,\dots,2n\}$, we may view $\mathsf{k}_j$ as an oriented tie of $N$. Choose a lift $\widetilde{\mathsf{k}}_j$ of $\mathsf{k}_j$ to $\widetilde N$ (recall that $\widetilde N:=\pi_S^{-1}(N)$, where $\pi_S:\widetilde S\to S$ is the universal cover), and let ${\bf T}_j=(T_{j,1},T_{j,2})$ be the pair of plaques of $\wt\lambda$ that contains the forward and backward endpoints of $\widetilde{\mathsf{k}}_j$. By Equation \eqref{eqn: intersection}, the homology class $[{\rm Re}\;\alpha_\rho]\in H_1(N^o,\partial_v N^o;\mathbb{R}^{\mathcal{A}})$ can be written as 
	\[[{\rm Re}\;\alpha_\rho]_{\bf i}=\sum_{j=1}^{2n}\mathrm{Re}\;\alpha_\rho({\bf T}_j)[\mathsf{k}_j^*],\]
	where ${\rm Re}\;\alpha_\rho$ in the above equation is viewed as a relative $\lambda^o$--cocycle in $\mathcal{Z}(\lambda^o;{\rm slits};\mathbb{R}^{\mathcal{A}})$. 
	
	Let $\mu\in\mathsf{M}(\lambda^o)$ be a transverse measure, let $[\mu]\in H_1(N^o,\partial_h N^o;\mathbb{R})$ be the  homology class associated to $\mu$ viewed as a $\lambda^o$--cocycle (see Section \ref{sec: cocycles} and Equation \eqref{fsdjknkjsdf}), and let $\mu_j\in\mathbb{R}$ be defined by
	\[[\mu]=\sum_{j=1}^{2n}\mu_j[\mathsf{k}_j].\]
	Then observe from Section \ref{int-pair} that
	\begin{align}\label{eqn: cocycle length formula}
		\mathbb{I}([\mu],[\alpha_\rho])=\sum_{j=1}^{2n}\mu_j\mathrm{Re}\;\alpha_\rho({\bf T}_j)\in\mathbb{R}^{\mathcal{A}}.
	\end{align}

	For each horizontal boundary component $\mathsf{c}$ of $N^o$, fix once and for all a tie of $N^o$ that has one endpoint in $\mathsf{c}$. Let $\mathsf{d}_{\mathsf{c}}$ denote the subsegment of this tie that has one endpoint $\mathsf{c}$ and whose interior lies entirely in $N^o-\lambda^o$. Note that the natural orientation on this tie induces an orientation on $\mathsf{d}_{\mathsf{c}}$. Via a homotopy relative to the leaves in $\mathsf{F}^o$, we may replace each $\mathsf{k}_j$ with a path $\mathsf{h}_j$ in $N^o$ such that 
	\begin{itemize}
		\item $\mathsf{h}_j$ is homotopic to $\mathsf{k}_j$ relative $\partial_h N^o$,
		\item $\mathsf{h}_j$ is tightly transverse to $\lambda^o$,
		\item if an endpoint of $\mathsf{h}_j$ lies in a horizontal boundary component $\mathsf{c}$ of $N^o$, then $\mathsf{d}_{\mathsf{c}}$ is a subsegment of $\mathsf{h}_j$ (that necessarily contains that endpoint).
	\end{itemize}
	Note that for all $j$, the orientation on $\mathsf{k}_j$ defines an orientation on $\mathsf{h}_j$. Let $\mathsf{h}_j'$ denote the interior subsegment of $\mathsf{h}_j$, and let $\mathsf{a}_j$ and $\mathsf{b}_j$ denote the exterior subsegments of $\mathsf{h}_j$ that contain the backward and forward endpoints of $\mathsf{h}_j$ respectively. For all $j\in\{1,\dots,2n\}$, choose lifts $\wt{\mathsf{h}}_j$ and $\wt{\mathsf{h}}_j'$ in $\wt N^o$ of $\mathsf{h}_j$ and $\mathsf{h}_j'$ respectively, such that $\wt{\mathsf{h}}_j'\subset\wt{\mathsf{h}}_j$.  
	
	By Proposition~\ref{prop: the rise of the shear} and Equation \eqref{eqn: cocycle length formula}, we have that 
	\begin{align}\label{eqn: sum is zero 0}
		\mathbb{I}([\mu],[\alpha_\rho])_{\bf i}
		&=\sum_{j=1}^{2n}\mu_j\int_{\wt{\mathsf{h}}_j'} \wt\Omega-\sum_{j=1}^{2n}\mu_j\log  \norm{X_1^{\wt{\mathsf{h}}_j}}_{(\wt{\mathsf{h}}_j')_-}+\sum_{j=1}^{2n}\mu_j\log  \norm{X_2^{\wt{\mathsf{h}}_j}}_{(\wt{\mathsf{h}}_j')_+}.
	\end{align}
	(Notice that on the right hand side, $\widetilde\Omega$, $X_1$, $X_2$, and the norms depend on $\rho$ and ${\bf i}$ even though this dependence is suppressed in the notation.) Thus, to prove Equation \eqref{eqn: required condition 2}, it now suffices to show that 
	\begin{align}\label{eqn: final reduction length}
		\sum_{j=1}^{2n}\mu_j\int_{\mathsf{a}_j}\Omega=0=\sum_{j=1}^{2n}\mu_j\int_{\mathsf{b}_j}\Omega
	\end{align}
	and 
	\begin{align}\label{eqn: final reduction length'}
		\sum_{j=1}^{2n}\mu_j\log  \norm{X_1^{\wt{\mathsf{h}}_j}}_{(\wt{\mathsf{h}}_j')_-}=0=\sum_{j=1}^{2n}\mu_j\log  \norm{X_2^{\wt{\mathsf{h}}_j}}_{(\wt{\mathsf{h}}_j')_+}.
	\end{align}
	Indeed, if we can do so, then Equation \eqref{eqn: sum is zero 0} implies that
	\[{\mathbb I}([\mu],[\alpha_\rho])_{\bf i}=\sum_{j=1}^{2n}\mu_j\int_{\mathsf{h}_j'} \Omega=\sum_{j=1}^{2n}\mu_j\int_{\mathsf{h}_j} \Omega=\langle [\mu],[\Omega]\rangle.\]
	
	For each horizontal boundary component $\mathsf{c}$ of $N^o$, denote
	\[I_{\mathsf{c}}:=\{j\in\{1,\dots,2n\}:\mathsf{d}_{\mathsf{c}}\subset\mathsf{h}_j\}.\]
	Observe that for orientation reasons, either $\mathsf{d}_{\mathsf{c}}=\mathsf{a}_j$ for  all $j\in I_{\mathsf{c}}$ or $\mathsf{d}_{\mathsf{c}}=\mathsf{b}_j$ for  all $j\in I_{\mathsf{c}}$. We say that $\mathsf{c}$ is of \emph{backward type} or \emph{forward type} if the former or latter holds respectively. Since $\partial\mu=0$, we see that
	\[\sum_{j\in I_{\mathsf{c}}}\mu_j=0\]
	for any horizontal boundary components $\mathsf{c}$ of $N^o$. It now follows that
	\begin{align*}
		\sum_{j=1}^{2n}\mu_j\int_{\mathsf{a}_j}\Omega=\sum_{\mathsf{c}\text{ backward type}}\left(\sum_{j\in I_{\mathsf{c}}}\mu_j\int_{\mathsf{d}_{\mathsf{c}}}\Omega\right)=0
	\end{align*}
	and
	\begin{align*}
		\sum_{j=1}^{2n}\mu_j\int_{\mathsf{b}_j}\Omega=\sum_{\mathsf{c}\text{ forward type}}\left(\sum_{j\in I_{\mathsf{c}}}\mu_j\int_{\mathsf{d}_{\mathsf{c}}}\Omega\right)=0,
	\end{align*}
	so Equation \eqref{eqn: final reduction length} holds.
	
	Similarly, if $\mathsf{c}$ is of backward (resp. forward) type, then $\log  \norm{X_1^{\wt{\mathsf{h}}_j}}_{(\wt{\mathsf{h}}_j')_-}$ (resp. $\log  \norm{X_1^{\wt{\mathsf{h}}_j}}_{(\wt{\mathsf{h}}_j')_+}$) agree for all $j\in I_{\mathsf{c}}$. Therefore, we may denote
	\[m_{\mathsf{c}}:=\left\{\begin{array}{ll}
		\log  \norm{X_1^{\wt{\mathsf{h}}_{j}}}_{(\wt{\mathsf{h}}_j')_-}&\text{if $\mathsf{c}$ is of backward type},\\
		\log  \norm{X_1^{\wt{\mathsf{h}}_{j}}}_{(\wt{\mathsf{h}}_j')_+}&\text{if $\mathsf{c}$ is of forward type},
	\end{array}
	\right.\]
	for all $j\in I_{\mathsf{c}}$. Then we have
	\[\sum_{j=1}^{2n}\mu_j\log  \norm{X_1^{\wt{\mathsf{h}}_j}}_{(\wt{\mathsf{h}}_j')_-}=\sum_{\mathsf{c}\text{ backward type}}\left(\sum_{j\in I_{\mathsf{c}}}\mu_jm_{\mathsf{c}}\right)=0\]
	and 
	\[\sum_{j=1}^{2n}\mu_j\log  \norm{X_2^{\wt{\mathsf{h}}_j}}_{(\wt{\mathsf{h}}_j')_+}=\sum_{\mathsf{c}\text{ forward type}}\left(\sum_{j\in I_{\mathsf{c}}}\mu_jm_{\mathsf{c}}\right)=0,\]
	so Equation \eqref{eqn: final reduction length'} holds. This concludes the proof of Theorem \ref{thm: two lengths}.

\section{The map \texorpdfstring{$\mathcal{sb}_{d,\bf x}$}{sb-x-d} is holomorphic} \label{sec:holomorphicity}

Recall that $\wt{\Delta}^o$ denotes the set of labelings, i.e. ordered triples of points in $\partial\wt\lambda$ that are the vertices of a plaque of $\wt\lambda$. Fix a labeling ${\bf x}=(x_1,x_2,x_3)\in\wt\Delta^o$, and recall also that 
\[\mathcal{sb}_{d,{\bf x}}:\mathcal{R}_d(\lambda)\to \mathcal{Y}_d(\lambda;\mathbb{C}/2\pi i\mathbb{Z})\times\mathcal{N}(\mathbb{C}^d)\] 
is the map defined by \eqref{eqn: main thm} in Section~\ref{sec: statement of main theorem}. In this section we prove: 

\begin{proposition}\label{prop:sb is holomorphic}
	The map $\mathcal{sb}_{d,\bf x}$ is a biholomorphism onto its image.
\end{proposition}

This establishes Step 4 in the outline of the proof of Theorem \ref{thm: main}, see Section~\ref{subsec:proof structure}.

\begin{proof}[Proof of Proposition \ref{prop:sb is holomorphic}]
	By Theorem \ref{thm: injectivity}, the map $\mathcal{sb}_{d,\bf x} $ is injective, and by Theorem \ref{thm: surjectivity} and Corollary \ref{cor: surject}, its image coincides with
	\[
	\left(\mathcal{C}_d(\lambda)+i\mathcal{Y}_d(\lambda;\mathbb{R}/2\pi\mathbb{Z})\right)\times\mathcal{N}(\mathbb{C}^d).
	\]
	Thus, it suffices to show that $\mathcal{sb}_{d,\bf x}^{-1}$ is holomorphic. (Recall that the complex structure on $\mathcal{C}_d(\lambda)+i\mathcal{Y}_d(\lambda;\mathbb{R}/2\pi\mathbb{Z})$ is the one inherited from being an open subset of $\mathcal{Y}_d(\lambda;\mathbb C/2\pi i\mathbb Z))$, see Section \ref{sec: complex structure}.)
	
	By definition, given an element
	\[((\alpha,\theta),(F,G,p))\in (\mathcal C_d(\lambda)+i\, \mathcal Y_d(\lambda;\mathbb R/2\pi \mathbb Z))\times \mathcal N(\mathbb C^d),\] 
	its preimage $\mathcal{sb}_{d,\bf x}^{-1}((\alpha,\theta),(F,G,p))$ coincides with the unique $d$--pleated surface $(\rho,\xi)$ whose shear-bend coordinates are equal to $(\alpha,\theta)$ and whose $\lambda$--limit map $\xi$ satisfies
	\[
	(\xi(x_1),\xi(x_2),\xi^1(x_3))=(F,G,p).
	\]
	For a fixed $(\alpha,\theta)$, the fact that $(\rho,\xi)$ varies holomorphically with respect to $(F,G,p)$ is obvious: if $(F',G',p')$ is a holomorphic variation of $(F,G,p)$ in $\mathcal{N}(\mathbb{C}^d)$ and $g\in\mathsf{PGL}_d(\mathbb{C})$ is the unique projective transformation such that $g\cdot(F,G,p)=(F',G',p')$, then $g$ is a holomorphic variation of $\id$ in $\mathsf{PGL}_d(\mathbb{C})$ and, hence, $(c_g\circ\rho,g\cdot\xi)=\mathcal{sb}_{d,\bf x}^{-1}((\alpha,\theta),(F',G',p'))$ is a holomorphic variation of $(\rho,\xi)$.
	
	Thus, it suffices to fix $(F,G,p)$ and show that $(\rho,\xi)$ varies holomorphically with respect to $(\alpha,\theta)$. Since for every $\gamma\in\Gamma$, the element $\rho(\gamma)$ is uniquely determined by its action on a projective frame, it is enough to show that for every $z\in\partial \wt\lambda$, the flag $\xi(z)$ varies holomorphically with $(\alpha,\theta)$. 
	
	Let $\Sigma$ be the slithering map compatible with $\xi$. Since $(\xi(x_1),\xi(x_2),\xi^1(x_3))$ is fixed, Lemma \ref{lem: injectivity} part (2) implies that $\Sigma$ depends only on $(\alpha,\theta)$. We write 
	\[\Sigma=\Sigma^{(\alpha,\theta)}\] 
	when we want to emphasize this dependence. Let $h$ be a leaf of $\wt\Lambda$ that has $z$ as one of its endpoints, and let $T_{\bf x}$ be the plaque of $\wt\lambda$ whose vertices are $x_1$, $x_2$, and $x_3$. Then by definition, 
	\[\xi(z)=\Sigma(h,g^h_{\bf x})\cdot\xi(z_0),\] 
	where $g^h_{\bf x}$ is the edge of $T_{\bf x}$ that separates $T_{\bf x}$ from $h$ and $z_0$ is the unique endpoint of $g^h_{\bf x}$ for which there exist orientations on the leaves $h$ and $g^h_{\bf x}$ so that $z$ and $z_0$ are their forward endpoints, respectively. Since $(F,G,p)$ is fixed, the flag $\xi(z_0)$ is constant in the cocyclic pair $(\alpha,\theta)$ when $z_0=x_1$ or $z_0=x_2$. Furthermore, if $z_0=x_3$, then $\xi(z_0)$ varies holomorphically with the triangle invariants of the base plaque $\theta({\bf x})$. Thus, it suffices to show that for any edge $h$ of $\wt\lambda$ that is not a side of $T_{\bf x}$, the transformation $\Sigma^{(\alpha,\theta)}(h,g^h_{\bf x})$ varies holomorphically with $(\alpha,\theta)$. 
		
	To do so, we use the H\"older extendable map $L=L_{(\alpha,\theta)}:P(g_{\bf x}^h,h)\to\mathsf{SL}_d(\mathbb{C})$ constructed in Section~\ref{sec: explicit expression}. Pick an exhausting sequence $(\mathcal X_n)_n$ for $P(g_{\bf x}^h,h)$. For each positive integer $n$, Lemma \ref{lem: injectivity} part (1) implies that $\overrightarrow{L}_{(\alpha,\theta)}^F(\mathcal X_n)$ varies holomorphically with $(\alpha,\theta)$. At the same time, Lemma~\ref{lem: injectivity} part (2) and Lemma \ref{lem: uniform} part (4) imply that $\overrightarrow{L}_{(\alpha,\theta)}^F(\mathcal X_n)$ converges uniformly to $\Sigma^{(\alpha,\theta)}(h,g_{\bf x}^h)$. Therefore, $\Sigma^{(\alpha,\theta)}(h,g^h_{\bf x})$ varies holomorphically with $(\alpha,\theta)$. 
	\end{proof}

\section{Non-Hausdorff points and pleated surfaces}\label{sec: nonHausdorff}
In this section, we give examples of $d$--pleated surfaces whose $\mathsf{PGL}_d(\mathbb{C})$--conjugacy class corresponds to a non-Hausdorff point in $\Hom(\Gamma,\mathsf{PGL}_d(\mathbb{C}))/\mathsf{PGL}_d(\mathbb{C})$. This justifies why, unlike Bonahon and Dreyer's parameterization theorem of the $\mathsf{PGL}_d(\mathbb{R})$--Hitchin component ${\rm Hit}_d(S)\subset\Hom(\Gamma,\mathsf{PGL}_d(\mathbb{R}))/\mathsf{PGL}_d(\mathbb{R})$ (see Theorem \ref{thm: BD par}), our main parameterization theorem (see Theorem \ref{thm: main}) is written for the space $\mathcal{R}_d(\lambda)\subset\Hom(\Gamma,\mathsf{PGL}_d(\mathbb{C}))$ of $d$--pleated surfaces with pleating locus $\lambda$  instead of the space $\mathfrak R_d(\lambda)$ of conjugacy classes of the representations in $\mathcal{R}_d(\lambda)$.

\subsection{The \texorpdfstring{$d=2$}{d=2} case}

The first examples of a conjugacy class of $2$--pleated surfaces that are not Hausdorff points in $\Hom(\Gamma,\mathsf{PGL}_d(\mathbb{C}))/\mathsf{PGL}_d(\mathbb{C})$ arise from the following theorem due to Gu\'eritaud, Kassel, and Wolff \cite{GKW}. 

\begin{theorem}[{\cite[Theorem 1.2]{GKW}}]\label{GKW}
	A representation $\rho:\Gamma\to\mathsf{PGL}_2(\mathbb{R})\subset\mathsf{PGL}_2(\mathbb{C})$ is a $2$--pleated surface if and only if its image is not abelian.
\end{theorem}

Theorem \ref{GKW}, together with a standard construction of non-Hausdorff points in $\Hom(\Gamma,\mathsf{PGL}_d(\mathbb{C}))/\mathsf{PGL}_d(\mathbb{C})$, yields the following consequence.

\begin{corollary}\label{cor: GKW}
	Let $\rho:\Gamma\to\mathsf{PGL}_2(\mathbb{R})$ be a reducible representation whose image is not abelian. Then $\rho$ is a $2$--pleated surface whose conjugacy class in $\Hom(\Gamma,\mathsf{PGL}_2(\mathbb{C}))/\mathsf{PGL}_2(\mathbb{C})$ is a non-Hausdorff point.
\end{corollary}

\begin{proof}
	By Theorem \ref{GKW}, $\rho$ is a $2$--pleated surface, so it suffices to check that its conjugacy class is a non-Hausdorff point in $\Hom(\Gamma,\mathsf{PGL}_2(\mathbb{C}))/\mathsf{PGL}_2(\mathbb{C})$. This is a standard argument that we repeat here for completeness.
	
	Since $\rho$ is reducible, by conjugating, we may assume that $\rho(\Gamma)$ lies in the subgroup  of upper triangular matrices in $\mathsf{PGL}_2(\mathbb{R})$, i.e. there are maps $a:\Gamma\to\mathbb{R}^*$ and $b:\Gamma\to\mathbb{R}$ such that for all $\gamma\in\Gamma$,
	\[\rho(\gamma)=\begin{bmatrix}
		a(\gamma)&b(\gamma)\\
		0&\frac{1}{a(\gamma)}
	\end{bmatrix}.\]
	Then for all $s\in\mathbb{R}$, set
	\[a(s):=\begin{bmatrix}
		s&0\\
		0&s^{-1}
	\end{bmatrix},\]
	and set $\rho_s:=c_{a(s)}\circ\rho$, i.e. for all $\gamma\in\Gamma$,
	\[\rho_s(\gamma)=\begin{bmatrix}
		a(\gamma)&s^2b(\gamma)\\
		0&\frac{1}{a(\gamma)}
	\end{bmatrix}.\]
	Notice that as $s\to0$, $\rho_s$ converges to a representation $\rho_0$ that lies in the diagonal subgroup of $\mathsf{PGL}_2(\mathbb{R})$, and hence is abelian. Since $\rho$ is not abelian, $\rho$ is not conjugate to $\rho_0$, but $\rho_0$ lies in the closure of the orbit of $\rho$ by conjugation. Hence, the conjugacy class of $\rho$ is not closed in $\Hom(\Gamma,\mathsf{PGL}_2(\mathbb{C}))$, and so corresponds to a non-Hausdorff point in $\Hom(\Gamma,\mathsf{PGL}_2(\mathbb{C}))/\mathsf{PGL}_2(\mathbb{C})$.
\end{proof}

By Corollary \ref{cor: GKW}, to construct examples of $2$--pleated surfaces whose conjugacy class are non-Hausdorff points in $\Hom(\Gamma,\mathsf{PGL}_2(\mathbb{C}))/\mathsf{PGL}_2(\mathbb{C})$, it now suffices to construct reducible representations from $\Gamma$ to $\mathsf{PGL}_2(\mathbb{C})$ that are not abelian. As described in the next example, these are easy to find.

\begin{example}
	Let $\Gamma=\langle \gamma_1,\eta_1,\dots,\gamma_g,\eta_g\mid[\gamma_1,\eta_1]\dots[\gamma_g,\eta_g]\rangle$, and define
	\[\rho(\gamma_i)=\begin{bmatrix}
		\sqrt{2}&0\\
		0&\frac{1}{\sqrt{2}}
	\end{bmatrix},\quad\rho(\eta_i)=\begin{bmatrix}
		1&1\\
		0&1
	\end{bmatrix}\]
	for all $i=1,\dots,g-1$, and
	\[\rho(\gamma_g)=\begin{bmatrix}
		\sqrt{g}&0\\
		0&\frac{1}{\sqrt{g}}
	\end{bmatrix},\quad\rho(\eta_g)=\begin{bmatrix}
		1&-1\\
		0&1
	\end{bmatrix}.\]
	One can then verify that $[\rho(\gamma_1),\rho(\eta_1)]\dots[\rho(\gamma_g),\rho(\eta_g)]=\id$, so this defines a representation $\rho:\Gamma\to\mathsf{PGL}_2(\mathbb{R})$. Furthermore, $\rho$ is reducible as it leaves the line $[e_1]$ invariant, and is not abelian, because $\rho(\gamma_1)$ and $\rho(\eta_1)$ do not commute. 
\end{example}

\begin{remark}
	The above examples show that there exist maximal geodesic laminations $\lambda$ such that the space $\mathfrak R_2(\lambda)$ of conjugacy classes of $2$--pleated surfaces with pleating locus $\lambda$ does not embed into the smooth part of the corresponding $\mathsf{PGL}_2(\mathbb{C})$ character variety, contrary to what is claimed in \cite[p. 280]{bonahon-toulouse}. Nevertheless, the proof in \cite{bonahon-toulouse} essentially shows that the shear-bend coordinates define a bilomorphic parameterization of $\mathcal R_2(\lambda)$ rather than of $\mathfrak R_2(\lambda)$ as stated in \cite[Theorem~D]{bonahon-toulouse}. That said, the argument of \cite[p. 280]{bonahon-toulouse} does apply to some geodesic laminations, for example to minimal maximal geodesic laminations $\lambda$ of $S$ (see \cite[Chapter 4]{CEG06}), and shows that, for this class of laminations, every point in $\mathfrak R_2(\lambda)$ is a Hausdorff point of $\Hom(\Gamma,\mathsf{PGL}_2(\mathbb{C}))/\mathsf{PGL}_2(\mathbb{C})$. 
\end{remark}

The construction of $2$--pleated surfaces whose conjugacy class are non-Hausdorff points in $\Hom(\Gamma,\mathsf{PGL}_2(\mathbb{C}))/\mathsf{PGL}_2(\mathbb{C})$ does not provide us any control of their pleating loci. This is in contrast to the case when $d=3$ which we will now discuss.

\subsection{The \texorpdfstring{$d=3$}{d=3} case}

Unlike the $d=2$ case, one can prove that, when $d=3$, the space $\mathfrak R_3(\lambda)$ does contain a non-Hausdorff point of $\Hom(\Gamma,\mathsf{PGL}_3(\mathbb{C}))/\mathsf{PGL}_3(\mathbb{C})$ \emph{for every} maximal geodesic lamination $\lambda$ of $S$. The proof of this result has two steps. The first is the following proposition, which gives a sufficient condition for  the conjugacy class of a $3$--pleated surface $\rho$ to be a non-Hausdorff point of $\Hom(\Gamma,\mathsf{PGL}_3(\mathbb{C}))/\mathsf{PGL}_3(\mathbb{C})$.

\begin{proposition}\label{nonHausdorffd=3}
	If $\rho\in\mathcal{R}_3(\lambda)$ such that $\theta_\rho({\bf x})=i\pi$ for all ${\bf x}\in\widetilde{\Delta}^o$, then the conjugacy class $[\rho]\in\Hom(\Gamma,\mathsf{PGL}_3(\mathbb{C}))/\mathsf{PGL}_3(\mathbb{C})$ is a non-Hausdorff point.
\end{proposition}

\begin{notation}
	Note that since $d=3$, the set $\mathcal{B}$ of positive triples of integers that sum to $3$ consists of only one element, namely $(1,1,1)$. As such, in the above proposition, we will write $\theta_\rho({\bf x},(1,1,1))$ simply as $\theta_\rho({\bf x})$ and $T^{(1,1,1)}({\bf F})$ simply as $T({\bf F})$. This simplification of notation will be used throughout this section.
\end{notation}

The proof of Proposition \ref{nonHausdorffd=3} relies on the following lemma.

\begin{lemma}\label{lem: non-Hausdorff 2}
	Let $\xi:\partial\widetilde\lambda\to\mathcal{F}(\mathbb{C}^3)$ be a locally $\lambda$--H\"older continuous, $\lambda$--transverse, and $\lambda$--hyperconvex map, and let $(\alpha_\xi,\theta_\xi)$ be its shear-bend coordinate. If $\theta_\xi({\bf y})=i\pi$ for all ${\bf y}\in\widetilde{\Delta}^o$, then there is a point $L\in\mathbb{C}\mathrm{P}^2$ such that if $a,b\in\partial\widetilde\lambda$ are endpoints of a leaf of $\widetilde\lambda$, then $\xi^2(a)\cap\xi^2(b)=L$. 
\end{lemma}

\begin{proof}
	First, given a triple of flags ${\bf F}=(F_1,F_2,F_3)$ in general position, we observe that the triple ratio $T({\bf F})=-1$ if and only if $\dim F_1^2 \cap F_2^2 \cap F_3^2 = 1$. To see this, first verify that if $\dim F_1^2 \cap F_2^2 \cap F_3^2 = 1$, then $T({\bf F})=-1$. This can be done by explicit computation, using a basis $\{e_1, e_2, e_3\}$ of $\mathbb{C}^3$ such that $F_i^1 = \mathrm{Span}(e_i)$ and $F_i^2 = \mathrm{Span}(e_i, e_1 + e_2 + e_3)$ for every $i = 1, 2, 3$. On the other hand, since the triple ratio is a complete invariant for projective classes of triples of flags in general position, the converse is also true.
	
	Let $\Sigma$ denote the slithering map compatible with $\xi$. Fix a base labelling ${\bf x}=(x_1,x_2,x_3)\in\widetilde\Delta^o$, and let $T_{\bf x}$ be the plaque of $\widetilde\Delta$ whose vertices are $x_1$, $x_2$, and $x_3$. By assumption, $T(\xi(x_1),\xi(x_2),\xi(x_3))=e^{\theta_\xi({\bf x})}=-1$, so we may set $L:=\xi(x_1)^2\cap\xi(x_2)^2\cap\xi(x_3)^2$. Let $h$ be any leaf of $\widetilde\lambda$, and let $a$ and $b$ be its endpoints. Up to relabelling the vertices of $T_{\bf x}$, we may assume that the edge $g_{\bf x}^h$ of $T_{\bf x}$ that separates $h$ and $T_{\bf x}$ has endpoints $x_1$ and $x_2$. Then observe that $\Sigma(h,g_{\bf x}^h)$ sends $L=\xi^2(x_1)\cap\xi^2(x_2)$ to $\xi^2(a)\cap\xi^2(b)$, so it suffices to show that $\Sigma(h,g_{\bf x}^h)$ fixes $L$. 
	
	Suppose that $F\in\mathcal{F}(\mathbb{C}^3)$ is a flag such that $(\xi(x_1),\xi(x_2),F)$ is in general position, and $T(\xi(x_1),\xi(x_2),F)=-1$. The first paragraph implies that 
	\[\xi(x_1)^2\cap\xi(x_2)^2=\xi(x_1)^2\cap F^2=\xi(x_2)^2\cap F^2.\] 
	As such, for both $i=1,2$, the unipotent linear map $U(\xi(x_i),\xi(x_{3-i}),F)$ that fixes $\xi(x_{3-i})$ and sends $\xi(x_i)$ to $F$ fixes $L$. By Proposition \ref{lem: injectivity} part (2), this implies that $\Sigma(h,g_{\bf x}^h)$ fixes $L$.
\end{proof}

\begin{proof}[Proof of Proposition \ref{nonHausdorffd=3}]
	Let $\xi:\partial\widetilde\lambda\to\mathcal{F}(\mathbb{C}^3)$ be the $\lambda$--limit map of $\rho$. By Lemma \ref{lem: non-Hausdorff 2}, there is a point $L\in\mathbb{C}\mathrm{P}^2$ such that if $a,b\in\partial\widetilde\lambda$ are endpoints of a leaf of $\widetilde\lambda$, then $\xi^2(a)\cap\xi^2(b)=L$. Since $\Gamma$ preserves the set of leaves of $\widetilde\lambda$, the $\rho$--equivariance of $\xi$ implies that $L$ is a fixed point of the group $\rho(\Gamma)$.
	
	Select arbitrarily a $2$--dimensional subspace $P\subset\mathbb{C}^3$ that is transverse to $L$, and express $\rho$ in the form
	\[
	\rho(\gamma) = \begin{bmatrix}
		t(\gamma) & b(\gamma) \\
		0 & \rho_0(\gamma)
	\end{bmatrix} \in \mathsf{PGL}(L \oplus P) = \mathsf{PGL}_3(\mathbb{C}) ,
	\]
	for some functions $t : \Gamma \to \mathsf{GL}(L), \rho_0 : \Gamma \to \mathsf{GL}(P)$, and $b : \Gamma \to \Hom(P,L)$.
	
	Let $\rho_\infty:\Gamma\to\mathsf{PGL}_3(\mathbb{C})$ be the semisimplification of $\rho$, i.e. with respect to the decomposition $\mathbb{C}^3=L+P$, $\rho_\infty$ is given by
	\[
	\rho_\infty(\gamma):= \begin{bmatrix}
		t(\gamma) & 0 \\
		0 & \rho_0(\gamma)
	\end{bmatrix} \in \mathsf{PGL}(L \oplus P) = \mathsf{PGL}_3(\mathbb{C}).
	\]
	Note that for every $s\in\mathbb{C}^*$, 
	\[
	A_s = \begin{bmatrix}
		s^{-2} & 0 \\
		0 & s \, \mathrm{id}_P
	\end{bmatrix} \in \mathsf{PGL}(L \oplus P) ,
	\]
	centralizes $\rho_\infty(\Gamma)$, so Proposition \ref{prop: pleated surfaces smooth} implies that $\rho_\infty\notin\mathcal{R}_d(\lambda)$. In particular, $\rho$ is not conjugate to $\rho_\infty$. On the other hand, we see by direct computation that if we define $\rho_s:\Gamma\to\mathsf{PGL}_3(\mathbb{C})$ by $\rho_s(\gamma)=A^n_s \circ \rho(\gamma) \circ A^{-n}_s$ for all $s\in\mathbb{R}$, then $\rho_s$ converges to $\rho_\infty$ as $s\to\infty$. This shows that in $\Hom(\Gamma, \mathsf{PGL}_3(\mathbb{C}))$, $\rho_\infty$ does not lie in the conjugacy class of $\rho$, but is in the closure of the conjugacy class of $\rho$. Thus, the conjugacy class $[\rho]$ is a non-Hausdorff point of $\Hom(\Gamma,\mathsf{PGL}_3(\mathbb{C}))/\mathsf{PGL}_3(\mathbb{C})$ as it cannot be separated from $[\rho_\infty]$ by open sets.
\end{proof}

The second step is to construct a $3$--pleated surface that satisfies the hypothesis of Proposition \ref{nonHausdorffd=3}. We will not prove this step for every maximal lamination, as it will require further understanding of the defining equations of the space $\mathcal{Y}_d(\lambda;\mathbb{C}/2\pi i\mathbb{Z})$, see Proposition \ref{prop: E injective}, which we explore in \cite{MMMZ2}. Instead, we will specify a particular maximal lamination $\lambda_{\mathcal{P}}$ for which we can construct a $3$--pleated surface with pleating locus $\lambda_{\mathcal{P}}$ whose conjugacy class is a non-Hausdorff point of $\Hom(\Gamma,\mathsf{PGL}_3(\mathbb{C}))/\mathsf{PGL}_3(\mathbb{C})$.

Given an oriented hyperbolic pair of pants $P$, orient its three boundary components so that $P$ lies to their left. Let $\mathsf{g}_1$, $\mathsf{g}_2$, and $\mathsf{g}_3$ be lifts of the three boundary geodesics of $P$ to oriented geodesics in the universal cover $\widetilde{P}$ of $P$, so that if $\gamma_i$ is the primitive element in $\pi_1(P)$ whose axis is $\mathsf{g}_i$, then $\gamma_1\gamma_2\gamma_3=\id$. For all $i=1,2,3$, let $h_i$ be the geodesic in $\widetilde{P}$ with $\mathsf{g}_{i-1}^-$ and $\mathsf{g}_{i+1}^-$ as its endpoints (arithmetic in the subscripts are done modulo $3$). Note that $\{\pi_P(h_1),\pi_P(h_2),\pi_P(h_3)\}$ are the leaves of a maximal geodesic lamination of $P$, where $\pi_P:\widetilde{P}\to P$ is the covering map. We refer to this lamination as the \emph{standard lamination} on $P$, and note that it does not depend on the choice of $\mathsf{g}_1$, $\mathsf{g}_2$ or $\mathsf{g}_3$.

Now, choose a pants decomposition $\mathcal{P}$ on $S$, i.e. a maximal collection of pairwise non-intersecting simple closed geodesics in $S$. Then note that $\mathcal{P}$ together with the standard lamination of each connected component of $S-\mathcal{P}$, form a maximal geodesic lamination, denoted $\lambda_{\mathcal{P}}$, on $S$. 

\begin{proposition}\label{lem: non-Hausdorff 1}
	There exists $\rho\in\mathcal{R}(\lambda_{\mathcal{P}},3)$ such that $\theta_\rho({\bf x})=i\pi$ for all ${\bf x}\in\widetilde{\Delta}^o$. In particular, the conjugacy class $[\rho]\in\Hom(\Gamma,\mathsf{PGL}_3(\mathbb{C}))/\mathsf{PGL}_3(\mathbb{C})$ is a non-Hausdorff point.
\end{proposition}
\begin{proof}
	The second statement follows from the first and Proposition \ref{nonHausdorffd=3}, so we will only focus on the first. We ease notations for this proof by writing $\lambda$ for $\lambda_{\mathcal{P}}$. 
	
	Let $\rho_0:\Gamma\to\mathsf{PGL}_3(\mathbb{R})$ be a Hitchin representation that lies in the Fuchsian locus. A straightforward computation shows that $\theta_{\rho_0}({\bf x})=0$ for all ${\bf x}\in\widetilde{\Delta}^o$, so Theorem \ref{thm: BD par} implies that $(\alpha_{\rho_0},0)\in\mathcal{C}_3(\lambda)$. It suffices to show that there exists a $\lambda$--cocyclic pair $(\alpha, \theta) \in \mathcal{Y}_3(\lambda;\mathbb{R}/2\pi\mathbb{Z})$ such that $\theta(\mathbf{x}) = \pi$ for every $\mathbf{x} \in \widetilde{\Delta}^o$; indeed, if that were the case, then 
	\[(\alpha_{\rho_0}+i\alpha,i\theta)\in\mathcal{C}_3(\lambda)+i\mathcal{Y}_3(\lambda;\mathbb{R}/2\pi\mathbb{Z}),\] 
	so Theorem \ref{thm: main} implies that there is some pleated surface $\rho\in\mathcal{R}_3(\lambda)$ such that 
	\[(\alpha_\rho,\theta_\rho)=(\alpha_{\rho_0}+i\alpha,i\theta).\]
	
	Choose a train track neighborhood $N$ of $\lambda$, such that every closed curve in $\mathcal{P}$ lies in two rectangles as drawn in Figure \ref{fig:traintrack}, and every component of $S-\mathcal{P}$ contains six rectangles as drawn in Figure \ref{fig:traintrack}.
	
	\begin{figure}[h!]
		\begin{minipage}{0.47\textwidth}
			\includegraphics[width=\textwidth]{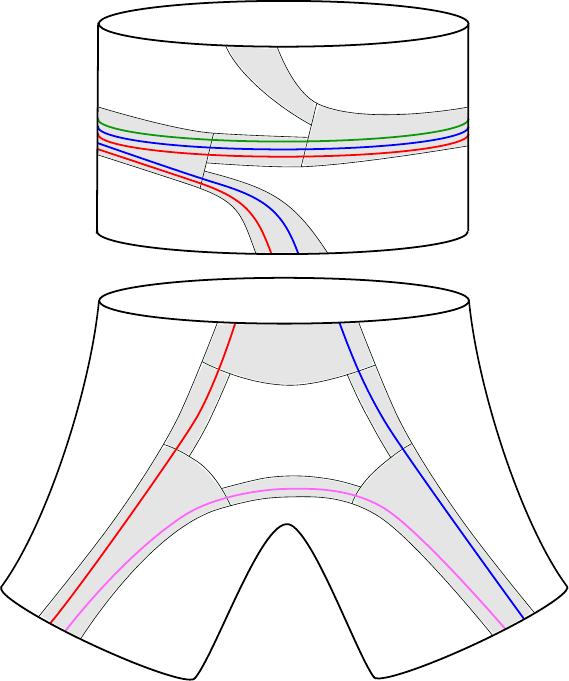}
		\end{minipage}
		\hfill
		\begin{minipage}{0.47\textwidth}
			\includegraphics[width=\textwidth]{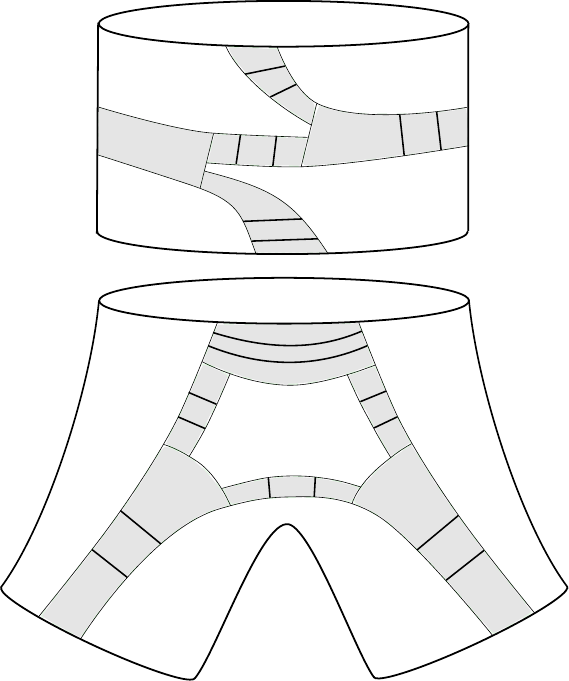}
			\put(-84,182){\small $\pi$}
			\put(-74,130){\small $\pi$}
			\put(-59,98){\small $\pi$}
			\put(-33,50){\small $\pi$}
			\put(-139,50){\small $\pi$}
		\end{minipage}
		\caption{\small {\em Left figure.} Part of the lamination $\lambda_\mathcal P$ and a train track neighborhood $N$. {\em Right figure.} Weights (in $\mathbb R/2\pi\mathbb Z$) associated to the $\lambda$--cocylic pair $(\alpha,\rho)$ from the proof of Lemma \ref{lem: non-Hausdorff 1}. Unlabeled ties have weight 0 and each plaque has triangle invariant $i\pi$.}
		\label{fig:traintrack}
	\end{figure}
	
	Let $\{[\mathsf{k}_1],\dots,[\mathsf{k}_{2n}]\}$ denote the standard generating set of $H_1(N^o,\partial_h N^o;\mathbb{Z})$ and pick representatives $\{\mathsf{k}_1,\dots,\mathsf{k}_{2n}\}$ for these homology classes. Recall that if $\pi_N:N^o\to N$ is the two-fold cover, then for all $j\in\{1,\dots,n\}$, $k_j:=\pi_N(\mathsf{k}_j)=\pi_N(\mathsf{k}_{n+j})$ is a tie of $N$, but the natural orientations on $\mathsf{k}_j$ and $\mathsf{k}_{n+j}$ induce opposite orientations on $k_j$. For each $j\in\{1,\dots,2n\}$, choose a lift $\wt k_j$ to $\wt S$ of $k_j$, and let ${\bf T}_j=(T_{j,1},T_{j,2})$ be the pair of plaques of $\wt\lambda_{\mathcal{P}}$ such that $T_{j,1}$ and $T_{j,2}$ respectively contain the backward and forward endpoint of $\wt k_j$ equipped with the orientation induced by $\mathsf{k}_j$. 
	
	Recall that 
	\[\widehat{\cdot}:\mathcal{A}\to\mathcal{A},\quad\widehat{\cdot}:\wt\Delta^{2*}\to\wt\Delta^{2*}, \quad\text{and}\quad\widehat{\cdot}:G^{\mathcal{A}}\to G^{\mathcal{A}}\] 
	are the involutions respectively defined by $\widehat{(i_1,i_2)}=(i_2,i_1)$, $\widehat{(T_1,T_2)}=(T_2,T_1)$, and $\widehat{g}_{\bf i}=g_{\widehat{\bf i}}$ for all ${\bf i}\in\mathcal{A}$, and that 
	\[\iota':\mathcal{Z}(\lambda^o,\mathrm{slits};G^{\mathcal{A}})\to\mathcal{Z}(\lambda^o,\mathrm{slits};G^{\mathcal{A}})\] 
	is the involution defined by $\iota'(\eta)(\widehat{\bf T})=\widehat{\eta({\bf T})}$. To show that the required pair $(\alpha, \theta) \in \mathcal{Y}_3(\lambda;\mathbb{R}/2\pi\mathbb{Z})$ exists, it suffices to construct $\eta \in \mathcal{Z}(\lambda^o,\mathrm{slits};(\mathbb{R}/2\pi\mathbb{Z})^2)$ such that $\iota'(\eta)=\eta$, and 
	\begin{align*}
		\partial\eta({\bf x})_{(1,2)}=\left\{\begin{array}{ll}
			\pi&\text{if $x_3 \prec x_2 \prec x_1$,} \\
			0&\text{if $x_1 \prec x_2 \prec x_3$}\end{array}\right.\,\,\text{ and }\,\,\partial\eta({\bf x})_{(2,1)}=\left\{\begin{array}{ll}
			0&\text{if $x_3 \prec x_2 \prec x_1$,} \\
			\pi&\text{if $x_1 \prec x_2 \prec x_3$}\end{array}\right.
	\end{align*}
	for every ${\bf x}\in\widetilde{\Delta}^o$. To do so, we will construct $2n$ pairs ${\bf v}_1,\dots,{\bf v}_{2n}\in (\mathbb{R}/2\pi\mathbb{Z})^2$ such that the following holds:
	\begin{enumerate}
		\item[(i)] For all $j\in\{1,\dots,n\}$, the first and second entries of ${\bf v}_j$ are the second and first entries of ${\bf v}_{n+j}$ respectively.
		\item[(ii)] If $i,j,l\in\{1,\dots,2n\}$ satisfies $T_{i,1}=T_{l,1}$, $T_{j,2}=T_{l,2}$, and $T:=T_{i,2}=T_{j,1}$, then
		\begin{align*}
			{\bf v}_l=\left\{\begin{array}{ll}
				{\bf v}_i+{\bf v}_j+(\pi,0)&\text{if $x_{T,3} \prec x_{T,2} \prec x_{T,1}$,} \\
				{\bf v}_i+{\bf v}_j+(0,\pi)&\text{if $x_{T,1} \prec x_{T,2} \prec x_{T,3}$}.\end{array}\right.
		\end{align*}
	\end{enumerate}
	It follows from Proposition \ref{prop: cocycles and cohomology} that there is a unique $\eta\in\mathcal{Z}(\lambda^o,\mathrm{slits};(\mathbb{R}/2\pi\mathbb{Z})^2)$ such that for all $j\in\{1,\dots,2n\}$, $\eta({\bf T}_j)_{(1,2)}$ and $\eta({\bf T}_j)_{(2,1)}$ are the first and second entries of ${\bf v}_j$ respectively, and it is straightforward to see that (i) and (ii) imply that $\eta$ has the required properties.
	
	We refer to the rectangles that contain the closed curves in $\mathcal{P}$ as \emph{boundary rectangles}. Of the six rectangles in each component of $S-\mathcal{P}$, notice that three of them are adjacent to boundary rectangles, while the other three are not. We refer to the former as \emph{bridging rectangles} and the latter as \emph{internal rectangles}.
	
	We now construct the pairs $\bf v_j$. If $i\in\{1,\dots,n\}$ is such that $\pi_N(\mathsf{k}_i)=\pi_N(\mathsf{k}_{n+i})$ is a tie in a boundary rectangle or internal rectangle, set both ${\bf v}_i$ and ${\bf v}_{n+i}$ to be $(0,0)$. On the other hand, if $i\in\{1,\dots,n\}$ is such that $\pi_N(\mathsf{k}_i)=\pi_N(\mathsf{k}_{n+i})$ is a tie in a bridging rectangle, notice that there is exactly one plaque $T$ that separates ${\bf T}_i$ (equivalently, separates ${\bf T}_{n+i}$). Then for both $j\in\{i,n+i\}$, if ${\bf x}_T=(x_{T,1},x_{T,2},x_{T,3})$ is the coherent labeling of the vertices of $T$ with respect to ${\bf T}_j$, set
	\begin{align*}
		{\bf v}_j:=\left\{\begin{array}{ll}
			(\pi,0)&\text{if $x_{T,3} \prec x_{T,2} \prec x_{T,1}$,} \\
			(0,\pi)&\text{if $x_{T,1} \prec x_{T,2} \prec x_{T,3}$}.\end{array}\right.
	\end{align*}
	
	It is clear that $\{{\bf v}_1,\dots,{\bf v}_{2n}\}$ as defined above satisfies (i). To finish the proof, we need only to show that they satisfy (ii). 
	
	For each $i\in\{1,\dots,2n\}$, let $R_i$ be the rectangle of $N$ that contains $\pi_N(\mathsf{k}_i)$. Observe that requiring $i,j,l\in\{1,\dots,2n\}$ to satisfy $T_{i,1}=T_{l,1}$, $T_{j,2}=T_{l,2}$, and $T:=T_{i,2}=T_{j,1}$ is equivalent to requiring that a vertical boundary component of $R_l$ is a switch of $N$ that properly contains vertical boundary components of $R_i$ and $R_j$, and $T$ is the plaque that has an edge intersecting $R_i$ and another edge intersecting $R_j$. From the definition of $\lambda$, note that there are two types of switches: 
	\begin{enumerate}
		\item[Type 1:] $R_l$ is a bridging rectangle while $R_i$ and $R_j$ are internal rectangles, 
		\item[Type 2:] $R_l$ is a boundary rectangle while one of $R_i$ and $R_j$ is a bridging rectangle and the other is a boundary rectangle.
	\end{enumerate}
	For the Type 1 switches, since $R_i$ and $R_j$ are internal rectangles, (ii) simplifies to
	\begin{align*}
		{\bf v}_l=\left\{\begin{array}{ll}
			(\pi,0)&\text{if $x_{T,3} \prec x_{T,2} \prec x_{T,1}$,} \\
			(0,\pi)&\text{if $x_{T,1} \prec x_{T,2} \prec x_{T,3}$},\end{array}\right.
	\end{align*}
	which holds because $R_l$ is a bridging rectangle. For the Type 2 switches, since one of $R_i$ and $R_j$ is a bridging rectangle and the other is a boundary rectangle, (ii) simplifies to
	\begin{align*}
		{\bf v}_l=\left\{\begin{array}{ll}
			(2\pi,0)&\text{if $x_{T,3} \prec x_{T,2} \prec x_{T,1}$,} \\
			(0,2\pi)&\text{if $x_{T,1} \prec x_{T,2} \prec x_{T,3}$}\end{array}\right.=(0,0).
	\end{align*}
	This holds because $R_l$ is a boundary rectangle.
\end{proof}

In light of these examples, we make the following conjecture.

\begin{conjecture}
	If $d\ge 3$, then for every maximal geodesic lamination $\lambda$ of $S$, $\mathfrak R_d(\lambda)$ contains a non-Hausdorff point of $\Hom(\Gamma,\mathsf{PGL}_d(\mathbb{C}))/\mathsf{PGL}_d(\mathbb{C})$.
\end{conjecture}

\bibliographystyle{amsalpha}
\bibliography{biblio}

\end{document}